\documentclass[10.5pt]{amsart}

\usepackage{amsthm}
\usepackage[top=3cm, bottom=2cm, left=3cm, right=3cm]{geometry}
\usepackage{graphicx}        % standard LaTeX graphics tool
\usepackage{multicol}        % used for the two-column index
\usepackage[bottom]{footmisc}% places footnotes at page bottom

\usepackage{booktabs}% http://ctan.org/pkg/booktabs
\newcommand{\tabitem}{~~\llap{\textbullet}~~}

\usepackage[T1]{fontenc}
\usepackage[applemac]{inputenc}
\usepackage{amsmath}
\usepackage{array}
\usepackage{amssymb}
\usepackage{amsfonts}
\usepackage[english]{babel}
\usepackage{enumerate}
 \usepackage{caption}
\usepackage[caption=false]{subfig}
\usepackage{float}
\usepackage{wrapfig}
\usepackage{etoolbox}
\usepackage{pdfsync}
\usepackage{graphicx,color}
\usepackage[all]{xy}

\usepackage{geometry}

\newtheorem{theorem}{Theorem}
\newtheorem{definition}[theorem]{Definition}
\newtheorem{proposition}[theorem]{Proposition}
\newtheorem{lemma}[theorem]{Lemma}
\newtheorem{corollary}[theorem]{Corollary}
\newtheorem{remark}[theorem]{Remark}

\newtheorem{conjecture}{Conjecture}

\numberwithin{equation}{section}
\numberwithin{theorem}{section}

\newcommand{\sign}{{\rm sign }}
\newcommand{\RR}{\mathbb{R}}

\newcommand{\AC}{\mathrm{ac}}
\newcommand{\eps}{\varepsilon}

\newcommand{\supp}{{\rm supp\ }}

\newcommand{\resc}{{\bf resc}}

\newcommand{\bW}{{\bf W}}

\newcommand{\mU}{{\mathcal U}}
\newcommand{\mW}{{\mathcal W}}
\newcommand{\mV}{{\mathcal V}}
\newcommand{\mF}{{\mathcal F}}

\newcommand{\mFr}{{\mathcal F_{k,\resc}}}
\newcommand{\mG}{{\mathcal G}}
\newcommand{\mH}{{\mathcal H}}
\newcommand{\mY}{{\mathcal Y}}

\newcommand{\la}{\left\langle}
\newcommand{\ra}{\right\rangle}
\newcommand{\osc}{{\rm osc\, }}
\newcommand{\id}{{\rm id }}
\newcommand{\bla}{\big\langle}
\newcommand{\bra}{\big\rangle}

\title[The one-dimensional fair-competition regime]{The geometry of diffusing and self-attracting particles in a one-dimensional fair-competition regime}
\author{V. Calvez, J. A. Carrillo, F. Hoffmann}

\begin{document}

\maketitle

\begin{abstract}
We consider an aggregation-diffusion equation modelling particle interaction with non-linear diffusion and non-local attractive interaction using a homogeneous kernel (singular and non-singular) leading to variants of the Keller-Segel model of chemotaxis. We analyse the \emph{fair-competition regime} in which both homogeneities scale the same with respect to dilations. Our analysis here deals with the one-dimensional case, building on the work in \cite{CCH1}, and provides an almost complete classification. In the singular kernel case and for critical interaction strength, we prove uniqueness of stationary states via a variant of the Hardy-Littlewood-Sobolev inequality. Using the same methods, we show uniqueness of self-similar profiles in the sub-critical case by proving a new type of functional inequality. Surprisingly, the same results hold true for any interaction strength in the non-singular kernel case. 
Further, we investigate the asymptotic behaviour of solutions, proving convergence to equilibrium in Wasserstein distance in the critical singular kernel case, and convergence to self-similarity for sub-critical interaction strength, both under a uniform stability condition. Moreover, solutions converge to a unique self-similar profile in the non-singular kernel case. 
Finally, we provide a numerical overview for the asymptotic behaviour of solutions in the full parameter space demonstrating the above results. We also discuss a number of phenomena appearing in the numerical explorations for the diffusion-dominated and attraction-dominated regimes.
\end{abstract}
%\tableofcontents

%%%%%%%%%%%%%%%%%%%%%%%%%%%%
%%%%%%%%%%%%%%%%%%%%%%%%%%%%
%%%%%%%%%%%%%%%%%%%%%%%%%%%%
%%%%%%%%%%%%%%%%%%%%%%%%%%%%
\section{Introduction}
\label{sec:Introduction}
%%%%%%%%%%%%%%%%%%%%%%%%%%%%
%%%%%%%%%%%%%%%%%%%%%%%%%%%%
%%%%%%%%%%%%%%%%%%%%%%%%%%%%
%%%%%%%%%%%%%%%%%%%%%%%%%%%%

Mean field macroscopic models for interacting particle systems have been derived in the literature \cite{O90,MCO} with the objective of explaining the large time behaviour, the qualitative properties and the stabilisation of systems composed by a large number of particles with competing effects such as repulsion and attraction between particles. They find natural applications in mathematical biology, gravitational collapse, granular media and self-assembly of nanoparticles, see \cite{Chan,KeSe70,CaMcCVi03,ToBeLe06,HP,review} and the references therein. These basic models start from particle dynamics in which their interaction is modelled via pairwise potentials. By assuming the right scaling between the typical interaction length and the number of particles per unit area one can obtain different mean field equations, see for instance \cite{BV}. In the mean-field scaling they lead to non-local equations with velocity fields obtained as an average force from a macroscopic density encoding both repulsion and attraction, see \cite{BCL,BCLR} and the references therein. However, if the repulsion strength is very large at the origin, one can model repulsive effects by (non-linear) diffusion while attraction is considered via non-local long-range forces \cite{MCO,ToBeLe06}.

In this work, we concentrate on this last approximation: repulsion is modelled by diffusion and attraction by non-local forces. We will make a survey of the main results in this topic exemplifying them in the one dimensional setting while at the same time we will provide new material in one dimension with alternative proofs and information about long time asymptotics which are not known yet in higher dimensions. In order to understand the interplay between repulsion via non-linear diffusion and attraction via non-local forces, we concentrate on the simplest possible situation in which both the diffusion and the non-local attractive potential are homogeneous functions. We will focus on models with a variational structure that dissipate the free energy of the system. This free energy is a natural quantity that is already dissipated at the underlying particle systems.\\

The plan for this work is twofold. In a first part we shall investigate some properties of the following class of homogeneous functionals, defined for centered probability densities $\rho(x)$,
belonging to suitable weighted $L^p$-spaces, and some interaction
strength coefficient $\chi> 0$ and diffusion power $m>0$:
\begin{align}
&\mathcal F_{m,k}[\rho] =  \int_{\RR}   U_m\left(\rho(x)\right)
\, dx + \chi  \iint_{\RR\times\RR} \rho(x)  W_k(x-y) \rho(y)\,
dxdy := \mathcal U_m[\rho] + \chi \mathcal W_k[\rho] \, ,
\label{eq:functional}\\
&\rho(x)\geq 0\, , \quad \int_{\RR}\rho(x)\, dx = 1\, , \quad
\int_{\RR} x\rho(x)\, dx = 0\,  , \nonumber
\end{align}
with
\begin{equation*}
U_m(\rho) = \left\{\begin{array}{ll} \dfrac1{m-1}\, \rho^m\, ,
&\mbox{if}\quad m\neq 1 \medskip\\ \rho\log \rho \, ,
&\mbox{if}\quad  m=1  \end{array}\right.\,, %\label{eq:ass1}
\end{equation*}
and
\begin{equation}
W_k(x) = \left\{\begin{array}{ll} \dfrac{|x|^k}k\, , &
\mbox{if}\quad k\in(-1,1)\setminus\{0\}
\medskip\\ \log |x|\, ,  & \mbox{if}\quad  k = 0
\end{array}\right.\, . \label{eq:ass2}
\end{equation}
The center of mass of the density $\rho$ is assumed to be zero since the free energy functional is invariant by translation. 
Taking mass preserving dilations, one can see that $\mU_m[\cdot]$ scales with a power $m-1$, whilst $\mW_k[\cdot]$ scales with power $-k$, indicating that the relation between the parameters $k$ and $m$ plays a crucial role here. And indeed, one observes different types of behaviour depending on which of the two forces dominates, non-linear diffusion or non-local attraction. This motivates the definition of three different regimes: the \emph{diffusion-dominated regime} $m-1>-k$, the \emph{fair-competition regime} $m-1=-k$, and the \emph{attraction-dominated regime} $m-1<-k$. We will here concentrate mostly on the fair-competition regime.\\

This work can be viewed as a continuation of the seminal paper  by
McCann \cite{McCann97} in a non-convex setting. Indeed McCann used
the very powerful toolbox of Euclidean optimal transportation to
analyse functionals like \eqref{eq:functional} in the case $m\geq0$ and for a convex interaction kernel $W_k$. He discovered that
such functionals are equipped with an underlying convexity
structure, for which the interpolant $[\rho_0,\rho_1]_t$ follows
the line of optimal transportation \cite{Villani03}. This provides
many interesting features among which a natural framework to show
uniqueness of the ground state as soon as it exists.
In this paper we deal with concave homogeneous interaction kernels
$W_k$ given by \eqref{eq:ass2} for which McCann's results \cite{McCann97}
do not apply. Actually, the conditions on $k$ imply that the interaction kernel $W_k$ is locally integrable on $\RR$ and concave on $\RR_+$, which means that $\mathcal W_k[\cdot]$ is displacement concave as shown in \cite{CarrilloFerreiraPrecioso12}. We explain in this paper how some ideas from \cite{McCann97} can be extended to some convex-concave competing effects. Our main statement is that the functional \eqref{eq:functional} -- the sum of a convex
and a concave functional -- behaves almost like a convex
functional in some good cases detailed below. In particular,
existence of a critical point implies  uniqueness (up to translations and
dilations). The bad functional contribution is somehow absorbed by
the convex part for certain homogeneity relations and parameters $\chi$.

The analysis of these free energy functionals and their respective
gradient flows is closely related to some functional inequalities
of Hardy-Littlewood-Sobolev (HLS) type
\cite{LieLo01,Gardner02,CarlenCarrilloLoss10,BCL}. To give a flavour,
we highlight the case $(m = 1,k = 0)$, called the \emph{logarithmic
case}. It is known from \cite{DoPe04,BlaDoPe06} using \cite{CaLo92,Beckner93} that the
functional $\mathcal F_{1,0}$ is bounded from below if and only if
$0< \chi \leq 1$. Moreover, $\mathcal F_{1,0}$ achieves its
minimum if and only if $\chi =1$ and the extremal functions are
mass-preserving dilations of Cauchy's density:
\begin{equation} \bar \rho_0(x) = \dfrac{1}{\pi} \left(\dfrac{1}{1+ |x|^2}\right)\, .
\label{eq:StSt log}
\end{equation}
In \cite{CaLo92} authors have proved the uniqueness (up to
dilations and translations) of this logarithmic HLS
inequality based on a competing-symmetries argument. We develop in
the present paper an alternative argument based on some accurate
use of the Jensen's inequality to get similar results in the porous medium 
case $-1<k<0$. This goal will be achieved for some variant of the HLS
inequality as in \cite{BCL}, indeed being a combination of the HLS inequality and 
interpolation estimates, see Theorem \ref{thm:HLSm}. 
The case $0<k<1$ has been a lot less studied, and we will show here that no critical interaction strength exists as there is no $\chi>0$ for which $\mF_{m,k}$ admits global minimisers. On the other hand, we observe certain similarities with the behaviour of the fast diffusion equation ($0<m<1$, $\chi=0$) \cite{VazquezFDE}. 
The mass-preserving dilation homogeneity of the functional $\mF_{m,k}$ is shared by the range of parameters $(m,k)$ with $N(m-1)+k=0$ for all  dimensions, $m>0$ and $k\in (-N,N)$. This general fair-competition regime, has recently been studied in \cite{CCH1}.\\

In a second stage, here we also tackle the behaviour of the following family
of partial differential equations modelling self-attracting
diffusive particles at the macroscopic scale,
\begin{equation}
\left\{
\begin{array}{l}
\partial_t \rho  =  \partial_{xx}\left(\rho ^m \right)+
2\chi \partial_x \left( \rho \, \partial_x S_k\right)\,
, \quad t>0\, , \quad x\in \RR\, , \smallskip\\
\displaystyle \rho(t=0,x) = \rho_0(x)\, .
\end{array}
\right. \label{eq:KS}
\end{equation}
where we define the mean-field potential $S_k(x) := W_k(x)*\rho(x)$. For $k>0$, the gradient $\partial_x S_k:= \partial_x \left(W_k \ast \rho\right)$ is well defined.
For $k<0$ however, it becomes a singular integral, and we thus define it via a Cauchy principal value. Hence, the mean-field potential gradient in equation \eqref{eq:KS} is given by
\begin{equation}\label{gradS}
 \partial_x S_k(x) :=
 \begin{cases}
  \partial_x W_k \ast \rho\, ,
  &\text{if} \, \, 0<k<1\, , \\[2mm]
  \displaystyle\int_{\RR} \partial_x W_k (x-y)\left(\rho(y)-\rho(x)\right)\, dy\, ,
  &\text{if} \, \, -1<k<0\, . 
 \end{cases}
\end{equation}
Further, it is straightforward to check that equation \eqref{eq:KS} formally preserves positivity, mass and centre of mass, and so we can choose to impose
$$
\displaystyle \rho_0(x)\geq 0\, ,\quad \int
\rho_0(x)\, dx = 1\, ,
\quad \int x \rho_0(x)\, dx=0\, .
$$
This class of PDEs are one of the prime examples for 
competition between the diffusion (possibly non-linear), and
the non-local, quadratic non-linearity which is due to the
self-attraction of the particles through the mean-field potential
$S_k(x)$. The parameter $\chi>0$ measures the
strength of the interaction.
We would like to point out that we are here not concerned with the regularity of solutions or existence/uniqueness results for equation \eqref{eq:KS},  allowing ourselves to assume solutions are 'nice' enough in space and time for our analysis to hold (for more details on regularity assumptions, see Section \ref{sec:LTA}).\\

There exists a strong link between the PDE \eqref{eq:KS} and the functional \eqref{eq:functional}. Not only is $\mathcal F_{m,k}$ decreasing along the
trajectories of the system,
but more importantly, system \eqref{eq:KS} is the formal gradient flow
of the free energy functional \eqref{eq:functional} when the
space of probability measures is endowed with the Euclidean
Wasserstein metric $\bW$:
\begin{equation}
\partial_t\rho(t)
=- \nabla_{\bf W} \mathcal F_{m,k}[\rho(t)]\, .
\label{eq:gradient flow}
\end{equation}
% where $\mathcal T_{m,k} [\rho]$ denotes the first variation of the energy functional in the set of probability densities:
% \begin{equation}\label{eq:1stvar}
%  \mathcal T_{m,k} [\rho](x) :=\frac{\delta \mF_{m,k}}{\delta \rho}[\rho](x) = \frac{m}{N(m-1)}\rho^{m-1}(x)+ 2 \chi W_k(x) \ast \rho(x)\, .
% \end{equation}
This illuminating statement has been clarified in the seminal
paper by Otto \cite{Otto}. We also refer to the books by Villani
\cite{Villani03} and Ambrosio, Gigli and Savar\'e \cite{AmGiSa05}
for a comprehensive presentation of this theory of gradient flows
in Wasserstein metric spaces, particularly in the convex case.
% Let us mention that such a gradient flow can be constructed as the
% limit of a discrete in time steepest descent scheme:}
% \begin{equation*} \label{eq:discrete Gradient flow}
% \rho(t + \Delta t) = \argmin_{\nu} \left\{ \mathcal F_{m,k}(\nu) +
% \dfrac1{2\Delta t} {\bf W}(\rho(t), \nu)^2  \right\}\, .
% \end{equation*}
Performing gradient flows of a convex functional is a
natural task, and suitable estimates from below on the Hessian of
$\mF_{m,k}$ in \eqref{eq:functional} translate into a rate of convergence
towards equilibrium for the PDE \cite{CaMcCVi03,Villani03,CaMcCVi06}.
However, performing gradient flow of functionals with convex and concave contributions
is more delicate, and one has to seek compensations. Such
compensations do exist in our case, and one can prove convergence
in Wasserstein distance towards some stationary state under
suitable assumptions, in some cases with an explicit rate of convergence. It is
of course extremely important to understand how the convex and the
concave contributions are entangled.

The results obtained in the fully convex case generally consider
each contribution separately, resp. internal energy, potential
confinement energy or interaction energy, see \cite{CaMcCVi03,Villani03,AmGiSa05,CaMcCVi06}. It happens however that
adding two contributions provides better convexity estimates. In
\cite{CaMcCVi03} for instance the authors prove exponential speed
of convergence towards equilibrium when a degenerate convex
potential $W_k$ is coupled with strong enough diffusion, see
\cite{BGG13} for improvements.\\

The family of non-local PDEs \eqref{eq:KS} has been
intensively studied in various contexts arising in physics and
biology. 
The two-dimensional logarithmic case $(m = 1,k=0)$ is the
so-called Keller-Segel system in its simplest formulation
\cite{KeSe70,KeSe71a,Nanjundiah73,JaLu92,BlaDoPe06,Perthame06}. It
has been proposed as a model for chemotaxis in cell populations.
The three-dimensional configuration $(m =
1,k = -1)$ is the so-called Smoluchowski-Poisson system arising in
gravitational physics \cite{Chan,CM,CLM}. It describes macroscopically a
density of particles subject to a self-sustained gravitational
field.

Let us describe in more details the two-dimensional Keller-Segel
system, as the analysis of its peculiar structure will serve as a guideline to understand other cases. The corresponding gradient flow is subject to a
remarkable dichotomy, see  \cite{ChiPe81,JaLu92,Nagai95,GaZa98,DoPe04,BlaDoPe06} . The density
exists globally in time if $\chi<1$ (diffusion overcomes
self-attraction), whereas blow-up occurs in finite time when $\chi
>1$ (self-attraction overwhelms diffusion). In the sub-critical case, it
has been proved that solutions decay to self-similarity solutions
exponentially fast in suitable rescaled variables \cite{CaDo12,CaDo14,EM16}. In the super-critical case, solutions blow-up in finite time with by now well studied blow-up profiles for close enough to critical cases, see \cite{Herrero-Velazquez97,Raphael-Schweyer14}.

Substituting linear diffusion by non-linear diffusion with $m>1$ in two dimensions and higher is a way of regularising the Keller-Segel model as proved in \cite{CaCa06,Sugi1} where it is shown that solutions exist globally in time regardless of the value of the parameter $\chi>0$.
It corresponds to the diffusion-dominated case in two dimensions for which the existence of compactly supported stationary states and global minimisers of the free energy has only been obtained quite recently in \cite{CHVY}. The fair-competition case for Newtonian interaction $k=2-N$ was first clarified in \cite{BCL}, see also \cite{Sugi2}, where the authors find that there is a similar dichotomy to the two-dimensional classical Keller-Segel case $(N=2,m = 1,k = 0)$, choosing the non-local term as the Newtonian potential, $(N\geq 3,m = 2- 2/N,k = 2-N)$. The main difference is that the stationary states found for the critical case are compactly supported. We will see that such dichotomy also happens for $k<0$ in our case while for $k>0$ the system behaves totally differently. In fact, exponential convergence towards equilibrium seems to be the generic behaviour in rescaled variables as observed in Figure \ref{fig:chi=07_k=-02_resc=1_intro}. 

\begin{figure}[h!]
\centering
\includegraphics[width=.5\textwidth]{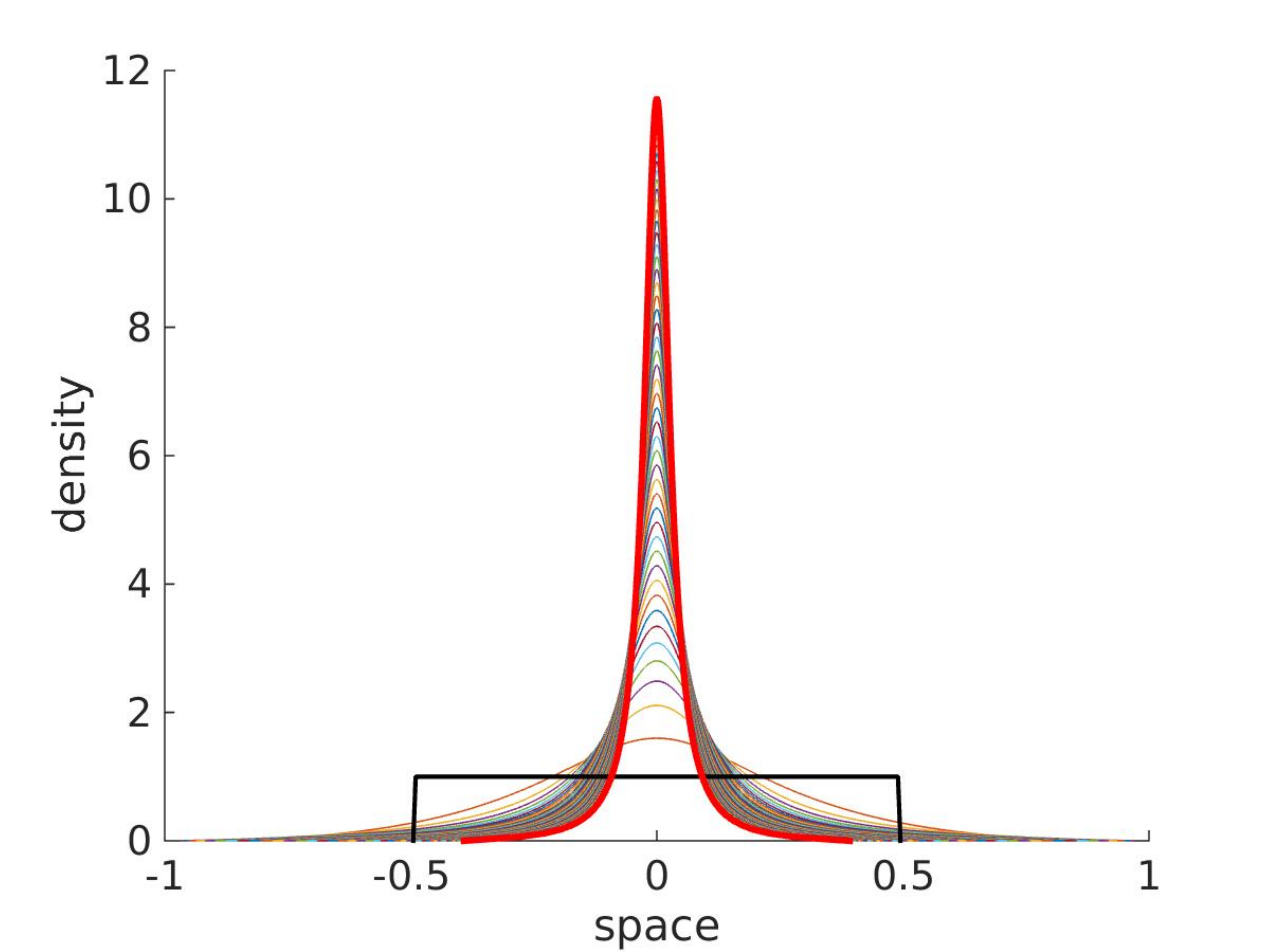}
\caption{Density evolution for parameter choices $\chi=0.7$, $k=-0.2$, $m=1.2$ following the PDE \eqref{eq:KS} in rescaled variables from a characteristic supported on $B(0,1/2)$ (black) converging to a unique stationary state (red). For more details, see Figure \ref{fig:chi=07_k=-02_resc=1} and the explanations in Section \ref{sec:numerics}.}
\label{fig:chi=07_k=-02_resc=1_intro}
\end{figure}

The paper is structured as follows: in Section \ref{sec:Preliminaries}, we give an analytic framework with all necessary definitions and assumptions. In cases where no stationary states exist for the aggreg-ation-diffusion equation \eqref{eq:KS}, we look for self-similar profiles instead. Self-similar profiles can be studied by changing variables in \eqref{eq:KS} so that stationary states of the rescaled equation correspond to self-similar profiles of the original system. Further, we give some main results of optimal transportation needed for the analysis of Sections \ref{sec: Functional inequalities} and \ref{sec:LTA}. In Section \ref{sec: Functional inequalities}, we establish several functional inequalities of HLS type that allow us to make a connection between minimisers of $\mF_{m,k}$ and stationary states of \eqref{eq:KS}, with similar results for the rescaled system. Section \ref{sec:LTA} investigates the long-time asymptotics where we demonstrate convergence to equilibrium in Wasserstein distance under certain conditions, in some cases with an explicit rate. Finally, in Section \ref{sec:numerics}, we provide numerical simulations of system \eqref{eq:KS} to illustrate the properties of equilibria and self-similar profiles in the different parameter regimes for the fair-competition regime. In Section \ref{sec:otherregimes}, we use the numerical scheme to explore the asymptotic behaviour of solutions in the diffusion- and attraction-dominated regimes.

%%%%%%%%%%%%%%%%%%%%%%%%%%%%
%%%%%%%%%%%%%%%%%%%%%%%%%%%%
%%%%%%%%%%%%%%%%%%%%%%%%%%%%
%%%%%%%%%%%%%%%%%%%%%%%%%%%%
\section{Preliminaries}
\label{sec:Preliminaries}
%%%%%%%%%%%%%%%%%%%%%%%%%%%%
%%%%%%%%%%%%%%%%%%%%%%%%%%%%
%%%%%%%%%%%%%%%%%%%%%%%%%%%%
%%%%%%%%%%%%%%%%%%%%%%%%%%%%

%%%%%%%%%%%%%%%%%%%%%%%%
%%%%%%%%%%%%%%%%%%%%%%%%
%%%%%%%%%%%%%%%%%%%%%%%%
%%%%%%%%%%%%%%%%%%%%%%%%
\subsection{Stationary States: Definition \& Basic Properties}
\label{sec:sstates}
%%%%%%%%%%%%%%%%%%%%%%%
%%%%%%%%%%%%%%%%%%%%%%%
%%%%%%%%%%%%%%%%%%%%%%%
%%%%%%%%%%%%%%%%%%%%%%%

Let us define precisely the notion of stationary states to the aggregation-diffusion equation \eqref{eq:KS}.

\begin{definition}\label{def:sstates}
 Given $\bar \rho \in L_+^1\left(\RR\right) \cap L^\infty\left(\RR\right)$ with $||\bar \rho||_1=1$, it is a \textbf{stationary state} for the evolution equation \eqref{eq:KS} if $\bar \rho^{m} \in \mathcal{W}_{loc}^{1,2}\left(\RR\right)$, $\partial_x \bar S_k\in L^1_{loc}\left(\RR\right)$, and it satisfies
 \begin{equation*}%\label{eq:steady}
  \partial_x \bar \rho^m=-2 \chi \, \bar \rho \partial_x \bar S_k
 \end{equation*}
in the sense of distributions in $\RR$. 
If $k \in (-1,0)$, we further require $\bar \rho \in C^{0,\alpha}\left(\RR\right)$ with $\alpha \in (-k,1)$.
\end{definition}

In fact, the function $ S_k$ and its gradient defined in \eqref{gradS} satisfy even more than the regularity $\partial_x  S_k\in L^1_{loc}\left(\RR\right)$ required in Definition \ref{def:sstates}. We have from \cite{CCH1}:

\begin{lemma}\label{lem:regS} Let $\rho \in  L_+^1\left(\RR\right) \cap L^\infty\left(\RR\right)$ with $||\rho||_1=1$. If $k \in (0,1)$, we additionally assume $|x|^k \rho \in L^1\left(\RR\right)$. Then the following regularity properties hold:
\begin{itemize}
\item[i)] $S_k \in L^\infty_{loc}\left(\RR\right)$ for $0<k<1$ and $ S_k \in L^{\infty}\left(\RR\right)$ for $-1<k<0$.
\item[ii)]  $\partial_x  S_k \in L^{\infty}\left(\RR\right)$ for $k \in (-1,1)\backslash\{0\}$, assuming additionally $\rho \in C^{0,\alpha}\left(\RR\right)$ with $\alpha \in (-k,1)$ in the range $-1<k< 0$. 
\end{itemize}
\end{lemma}

Furthermore, for certain cases, see \cite{CCH1}, there are no stationary states to \eqref{eq:KS} in the sense of Definition \ref{def:sstates} (for a dynamical proof of this fact, see Remark \ref{no sstates} in Section \ref{sec: LTA sub-critical k neg}), and so the scale invariance of \eqref{eq:KS} motivates us to look for self-similar solutions instead. To this end, we rescale equation \eqref{eq:KS} to a non-linear Fokker-Planck type equation as in \cite{CaTo00}. Let us define 
$$
u(t,x):= \alpha(t)\rho\left( \beta(t), \alpha(t) x\right),
$$
where $\rho(t,x)$ solves \eqref{eq:KS} and the functions $\alpha(t)$, $\beta(t)$ are to be determined. If we assume $u(0,x)=\rho(0,x)$, then $u(t,x)$ satisfies the rescaled drift-diffusion equation
\begin{equation}
\left\{
\begin{array}{l}
\partial_t u =   \partial_{xx} u ^m +
2\chi \partial_x \left( u \, \partial_x S_k \right)+ \partial_x\left(x u\right)\,
, \quad t>0\, , \quad x\in \RR\, , \smallskip\\
\displaystyle u(t=0,x) = \rho_0(x)\geq 0\, ,
\quad \int_{-\infty}^\infty \rho_0(x)\, dx = 1\, ,
\quad \int_{-\infty}^\infty x\rho_0(x)\, dx = 0\, ,
\end{array}
\right. \label{eq:KSresc}
\end{equation}
for the choices
\begin{equation}\label{scaling}
\alpha(t)= e^t, \quad \beta(t)=
\begin{cases}
 \frac{1}{2-k} \left(e^{(2-k)t} -1\right), &\text{if} \, k\neq2\, ,\\
 t, &\text{if} \, k=2\, ,
\end{cases}
\end{equation}
and with $\partial_x S_k$ given by \eqref{gradS} with $u$ instead of $\rho$. By differentiating the centre of mass of $u$, we see easily that 
\begin{equation*}\label{COM}
 \int_{\RR} x u(t,x)\, dx = e^{-t} \int_{\RR} x \rho_0(x)\, dx=0\, , \qquad \forall t>0\, ,
\end{equation*}
and so the initial zero centre of mass is preserved for all times. Self-similar solutions to \eqref{eq:KS} now correspond to stationary solutions of \eqref{eq:KSresc}. Similar to Definition \ref{def:sstates}, we state what we exactly mean by stationary states to the aggregation-diffusion equation \eqref{eq:KSresc}.

%%%%%%%%%%%%%%%%%%%%%%%%%%%%%%%%%%%%%%%%%%%%%%%%%%%%%%%
%%%%%%%%%%%%%%%%%%%%%%%%%%%%%%%%%%%%%%%%%%%%%%%%%%%%%%%

\begin{definition}\label{def:sstates resc}
Given $\bar u \in L_+^1\left(\RR\right) \cap L^\infty\left(\RR\right)$ with $||\bar u||_1=1$, it is a \textbf{stationary state} for the evolution equation \eqref{eq:KSresc} if $\bar u^{m} \in \mathcal{W}_{loc}^{1,2}\left(\RR\right)$, $\partial_x \bar S_k\in L^1_{loc}\left(\RR\right)$, and it satisfies
 \begin{equation*}%\label{eq:steady resc}
  \partial_x \bar u^m=-2 \chi \, \bar u \partial_x \bar S_k - x\, \bar u
 \end{equation*}
in the sense of distributions in $\RR$. 
% If $k>0$, we further require $|x|^k\bar u \in L^1\left(\RR\right)$.
If $-1<k<0$, we further require $\bar u \in C^{0,\alpha}\left(\RR\right)$ with $\alpha \in (-k,1)$.
\end{definition}
From now on, we switch notation from $u$ to $\rho$ for simplicity, it should be clear from the context if we are in original or rescaled variables. In fact, stationary states as defined above have even more regularity:
\begin{lemma}\label{lem:sstatesreg}
 Let $k \in (-1,1)\backslash\{0\}$ and $\chi>0$. 
 \begin{enumerate}[(i)]
  \item If $\bar \rho$ is a stationary state of equation \eqref{eq:KS} with $|x|^k\bar\rho\in L^1\left(\RR\right)$ in the case $0<k<1$, then $\bar \rho$ is continuous on $\RR$.
  \item If $\bar \rho_\resc$ is a stationary state of equation \eqref{eq:KSresc} with $|x|^k\bar\rho_\resc\in L^1\left(\RR\right)$ in the case $0<k<1$, then $\bar \rho_\resc$ is continuous on $\RR$. 
 \end{enumerate}
\end{lemma}

In the case $k<0$, we furthermore have a non-linear algebraic equation for stationary states \cite{CCH1}:

\begin{corollary}[Necessary Condition for Stationary States] \label{cor:EL}
Let $k \in (-1,0)$ and $\chi>0$.
\begin{enumerate}[(i)]
%%%%%%%%%%%%%%%%%%%%%%%%%%%%%%%%%%%%%%%%%%
 \item If $\bar \rho$ is a stationary state of equation \eqref{eq:KS}, then $\bar \rho \in \mathcal{W}^{1,\infty}\left(\RR\right)$ and it satisfies 
  \begin{equation*}%\label{eq:EL}
 \bar \rho(x)^{m-1} = \frac{(m-1)}{m} \left( C_k[\bar\rho](x)-2 \chi \, \bar S_k(x)\right)_+\, , \qquad
 \forall \, x \in \RR\, ,
\end{equation*}
 where $C_k[\bar\rho](x)$ is constant on each connected component of $\supp(\bar \rho)$. 
 %%%%%%%%%%%%%%%%%%%%%%%%%%%%%%%%%%%%%%%%%%%%
\item If $\bar \rho_{\resc}$ is a stationary state of equation \eqref{eq:KSresc}, then $\bar \rho_\resc \in \mathcal{W}^{1,\infty}_{loc}\left(\RR\right)$ and it satisfies 
 \begin{equation*}%\label{eq:ELresc}
 \bar \rho_\resc(x)^{m-1} = \frac{(m-1)}{m} \left( C_{k,\resc}[\bar\rho](x)-2 \chi \, \bar S_k(x)-\frac{|x|^2}{2}\right)_+\, , \qquad
 \forall \, x \in \RR\, ,
\end{equation*}
where $C_{k,\resc}[\bar\rho](x)$ is constant on each connected component of $\supp(\bar \rho_\resc)$. 
%%%%%%%%%%%%%%%%%%%%%%%%%%%%%%%%%%%%%%%%%%%%%%
\end{enumerate}
\end{corollary}

%%%%%%%%%%%%%%%%%%%%%%%%%%%%
%%%%%%%%%%%%%%%%%%%%%%%%%%%%
%%%%%%%%%%%%%%%%%%%%%%%%%%%%
%%%%%%%%%%%%%%%%%%%%%%%%%%%%
\subsection{Definition of the different regimes}
\label{sec:definitions}
%%%%%%%%%%%%%%%%%%%%%%%%%%%%
%%%%%%%%%%%%%%%%%%%%%%%%%%%%
%%%%%%%%%%%%%%%%%%%%%%%%%%%%
%%%%%%%%%%%%%%%%%%%%%%%%%%%%

It is worth noting that the functional $\mathcal F_{m,k}[\rho]$
possesses remarkable homogeneity properties. Indeed, the
mass-preserving dilation $\rho_\lambda(x) =
\lambda\rho(\lambda x)$ transforms the functionals as
follows:
\begin{align*}
\mU_m\left[\rho_\lambda\right] = 
 \begin{cases}
  \lambda^{(m-1)}\mU_m[\rho]\, ,  &\text{if} \quad m\neq 1\, ,\\
  \mU_m[\rho] +\log\lambda\, , &\text{if}\quad  m=1\, ,
 \end{cases}
\end{align*} 
and,
\begin{align*}
\mW_k \left[\rho_\lambda\right] =
 \begin{cases}
  \lambda^{-k} \mW_k[\rho] \, ,  &\text{if} \quad k\neq 0\, ,\\
   \mW_k[\rho] - \log \lambda\, , &\text{if}\quad  k=0\, .
 \end{cases}
\end{align*}
% \[\quad \mW_k \left[\rho_\lambda\right]  = \left\{\begin{array}{ll} \lambda^{-k} \mW_k[\rho] , \quad \, &k\neq 0 \medskip\\ \mW_k[\rho] - \log \lambda\, , \quad \mbox{if}\quad  &k = 0  \end{array}\right\}\, .
% \]
This motivates the following classification:
\begin{definition}[Three different regimes]
\
\begin{description}
\item[$\mathbf{m+k=1}$] This is the {\bf fair-competition}
regime, where homogeneities of the two competing contributions
exactly balance. If $k<0$, or equivalently $m>1$, then we will have a dichotomy according to $\chi$ (see Definition \ref{def:k} below). Some variants of the HLS inequalities are very related to this dichotomy. This was already proven in \cite{DoPe04,BlaDoPe06,CaDo14,EM16} for the Keller-Segel case in $N=2$, and in \cite{BCL} for the Keller-Segel case in $N\geq 3$. If $k>0$, that is $m<1$, no critical $\chi$ exists as we will prove in Section \ref{sec:Fast Diffusion Case k pos}. 

\item[$\mathbf{m+k>1}$] This is the {\bf
diffusion-dominated} regime. Diffusion is strong, and is expected to overcome aggregation, whatever $\chi>0$ is. This domination effect means that solutions exist globally in time and are bounded uniformly in time \cite{CaCa06,Sugi1,Sugi2}. Stationary states were found by minimisation of the free energy functional in two and three dimensions \cite{Strohmer2008,CCV, CS16} in the case of attractive Newtonian potentials. Stationary states are radially symmetric if $2-N\leq k<0$ as proven in \cite{CHVY}. Moreover, in the particular case of $N=2$, $k=0$, and $m>1$ it has been proved in \cite{CHVY} that the asymptotic behaviour is given by compactly supported stationary solutions independently of $\chi$.

\item[$\mathbf{m+k<1}$] This is the {\bf
attraction-dominated} regime. This regime is less understood.
Self-attraction is strong, and can overcome the regularising
effect of diffusion whatever $\chi>0$ is, but there also exist
global in time regular solutions under some smallness assumptions, see \cite{CPZ04, Sugi3, SugiKu06, CLW,BL,CW,LW, CaCo14}. However, there is no complete criteria in the literature up to date distinguishing between the two behaviours. 
\end{description}
\label{def:classification regimes}
\end{definition}
%%%%%%%%%%%%%%%%%%%%%%%%%%
We will here only concentrate on the fair-competition regime, and denote the corresponding energy functional by $\mathcal F_k[\rho] = \mathcal F_{1-k,k}[\rho]$. From now on, we assume $m+k=1$. Notice that the functional $\mF_k$ is homogeneous in this regime, i.e.,
\begin{equation*}%\label{dil}
\mF_k[\rho_\lambda] = \lambda^{-k} \mF_k[\rho]\, .
\end{equation*}
In this work, we wil first do a review of the main results known in one dimension about the stationary states and minimisers of the aggregation-diffusion equation in the fair-competition case. The novelties will be showing the functional inequalities independently of the flow and studying the long-time asymptotics of the equations \eqref{eq:KS} and \eqref{eq:KSresc} by exploiting the one dimensional setting. The analysis in the fair-competition regime depends on the sign of $k$:

\begin{definition}[Three different cases in the fair-competition regime]
\
\begin{description}
\item[$\mathbf{k<0}$] This is the {\bf porous medium case} with $m\in (1,2)$, where diffusion is small in regions of small densities. The classical porous medium equation, i.e.    $\chi=0$, is very well studied, see \cite{VazquezPME} and the references therein. For $\chi>0$, we have a dichotomy for existence of stationary states and global minimisers of the energy functional $\mF_k$ depending on a critical parameter $\chi_c$ which will be defined in \eqref{defchic}, and hence separate the sub-critical, the critical and the super-critical case, according to $\chi \lessgtr \chi_c$. These are the one dimensional counterparts to the case studied in \cite{BCL} where minimisers for the free energy functional were clarified. The case $k<0$ is discussed in Section \ref{sec: Porous Medium Case k neg}.

\item[$\mathbf{k=0}$] This is the {\bf logarithmic case}. There exists an explicit extremal density $\bar \rho_0 $ which realises the minimum of the functional $\mathcal F_0$ when
$\chi = 1$. Moreover, the functional $\mathcal F_0$ is bounded below but does not achieve its infimum for $0<\chi<1$ while it is not bounded below for $\chi > 1$. Hence, $\chi_c = 1$ is the critical parameter in the logarithmic case whose asymptotic behaviour was analysed in \cite{CaCa11} in one dimension and radial initial data in two dimensions. We refer to the results in \cite{CaDo14,EM16} for the two dimensional case.

\item[$\mathbf{k>0}$] This is the {\bf fast diffusion case} with $m\in (0,1)$, where diffusion is strong in regions of small densities. For any $\chi>0$, no radially symmetric non-increasing stationary states with bounded $k$th moment exist, and $\mF_k$ has no radially symmetric non-increasing minimisers. However, we have existence of self-similar profiles independently of $\chi>0$. The fast diffusion case is discussed in Section \ref{sec:Fast Diffusion Case k pos}.
\end{description}
\label{def:k}
\end{definition}

When dealing with the energy functional $\mF_k$, we work in the set of non-negative normalised densities,
$$
 \mY := \left\{ \rho \in L_+^1\left(\RR\right) \cap L^m \left(\RR\right) : ||\rho||_1=1\, , \, \int x\rho(x)\, dx=0\right\}.
$$
In rescaled variables, equation \eqref{eq:KSresc} is the formal gradient flow of the rescaled free energy functional $\mFr$, which is complemented with an additional quadratic confinement potential,
\begin{equation*}
\mFr[\rho]= \mathcal F_k [\rho] + \frac{1}{2}\mV[\rho]\, ,
\qquad
\mV[\rho] = \int_{\RR} |x|^2 \rho(x)\, dx \, .
\end{equation*}
Defining the set $\mY_2:=\left\{ \rho \in \mY:\mV[\rho]<\infty \right\}$,
we see that $\mFr$ is well-defined and finite on $\mY_2$. 
Thanks to the formal gradient flow structure in the Euclidean Wasserstein metric {\bf{W}}, we can write the rescaled equation \eqref{eq:KSresc} as
\begin{equation*}
\partial_t\rho
% =\partial_x \left( \rho\, \partial_x \mTr [\rho]\right)
=- \nabla_{\bf W} \mFr[\rho]\, .
%\label{eq:gradient flow resc}
\end{equation*}
% where $\mTr$ denotes the first variation of the rescaled energy functional,
% \begin{equation}\label{eq:1stvar resc}
%  \mTr[\rho](x) :=\mT_k[\rho](x) +\frac{|x|^2}{2}
% \end{equation}
% with $\mT_k$ as defined in \eqref{eq:1stvar}.\\

%%%%%%%%%%%%%%%%%%%%%%%%%%%%%%%%%%%%%%%%%%%%%%%%%%%%%%%%%%%%%%
%%%%%%%%%%%%%%%%%%%%%%%%%%%%%%%%%%%%%%%%%%%%%%%%%%%%%%%%%%%%%%
%%%%%%%%%%%%%%%%%%%%%%%%%%%%%%%%%%%%%%%%%%%%%%%%%%%%%%%%%%%%%%

In what follows, we will make use of a different characterisation of stationary states based on some integral reformulation of the necessary condition stated in Corollary \ref{cor:EL}. This characterisation was also the key idea in \cite{CaCa11} to improve on the knowledge of the asymptotic stability of steady states and the functional inequalities behind. 

\begin{lemma}[Characterisation of stationary states] \ \label{lem:char crit}
Let $k\in (-1,1)\backslash\{0\}$, $m=1-k$ and $\chi>0$. 
\begin{enumerate}[(i)]
 \item  Any stationary state $\bar \rho_k\in \mY$ of system \eqref{eq:KS} can be written in the form
\begin{equation} \label{eq:charac sstates}
\bar \rho_k(p)^m = \chi \int_{\RR}\int_{0}^1 |q|^{1-m}\bar \rho_k(p -sq)\bar \rho_k(p-sq+q)\, dsdq\, .
\end{equation}
Moreover, if such a stationary state exists, it satisfies $\mF_k[\bar \rho_k] = 0$.
\item Any stationary state $\bar \rho_{k,\resc}\in \mY_2$ of system \eqref{eq:KSresc} can be written in the form
\begin{equation} \label{eq:charac sstates resc}
\bar \rho_{k,\resc }(p)^m = \int_{\RR}\int_{0}^1 \left( \chi|q|^{1-m} + \dfrac{|q|^2}2 \right) \bar \rho_{k ,\resc}(p-sq)\bar \rho_{k ,\resc}(p-sq+q)\, dsdq\, .
\end{equation}
Moreover, it satisfies
\begin{equation}\label{eq:FV resc}
\mathcal F_{k,\resc}[\bar \rho_{k,\resc}] = \dfrac{m+1}{2(m-1)}\,  \mV[\bar \rho_{k,\resc}]
= \left(\frac{1}{2}-\frac{1}{k}\right)\,  \mV[\bar \rho_{k,\resc}]\,.
\end{equation}
\end{enumerate}
\end{lemma}

%%%%%%%%%%%%%%%%%%%%%%%%%%%%%%%%

\begin{proof}
We can apply the same methodology as for the logarithmic case (Lemma 2.3, \cite{CaCa11}). We will only prove \eqref{eq:charac sstates}, identity \eqref{eq:charac sstates resc} can be deduced in a similar manner. We can see directly from the equation that all stationary states of \eqref{eq:KS} in $\mY$ satisfy
$$
\partial_x \left(\bar \rho_k^m\right) + 2 \chi \bar \rho_k \partial_x \bar S_k =0\, .
$$
Hence, if $k \in (0,1)$, we can write for any test function $\varphi \in C_c^\infty(\RR)$
\begin{align*}
0 &= - \int_\RR \varphi'(p) \bar \rho_k^m(p)\, dp + 2\chi \iint_{\RR\times\RR} \varphi(x)|x-y|^{k-2}(x-y) \bar \rho_k(x)\bar \rho_k(y)\, dxdy\\
&=- \int_\RR \varphi'(p) \bar \rho_k^m(p)\, dp + \chi\iint_{\RR\times\RR} \left(\dfrac{\varphi(x)-\varphi(y)}{x-y}\right)|x-y|^{k} \bar \rho_k(x)\bar \rho_k(y)\, dxdy\, .
\end{align*}
%%%%%%%%%%%%%%%%%%%%%%%%%%%%%%%%%%%%%%%%%%%
%%%%%%%%%%%%%%%%%%%%%%%%%%%%%%%%%%%%%%%%%%%
%%%%%%%%%%%%%%%%%%%%%%%%%%%%%%%%%%%%%%%%%%%
%%%%%%%%%%%%%%%%%%%%%%%%%%%%%%%%%%%%%%%%%%%
%-------- Short version
%%%%%%%%%%%%%%%%%%%%%%%%%%%%%SSS
For $k \in (-1,0)$, the term $\partial_x \bar S_k$ is a singular integral, and thus writes
\begin{align*}
\partial_x \bar S_k(x)
= &\lim_{\eps\to 0} \int_{B^c(x,\eps)} |x-y|^{k-2}(x-y)\bar\rho_k(y)\, dy\\
= &\int_{\RR} |x-y|^{k-2}(x-y)\left(\bar\rho_k(y)-\bar\rho_k(x)\right)\, dy\, .
\end{align*}
%%%%%%%%%%%%%%%%%%%%%%%%%%%%%%%%%%%
The singularity disappears when integrating against a test function $\varphi \in C_c^\infty(\RR)$,
\begin{align}\label{nosingularity}
\int_{\RR} \varphi(x) \partial_x \bar S_k(x)\, dx
= \frac12\iint_{\RR\times\RR} \left(\dfrac{\varphi(x)-\varphi(y)}{x-y}\right)|x-y|^{k} \bar \rho_k(x)\bar \rho_k(y)\, dxdy\, .
\end{align}
In order to prove \eqref{nosingularity}, let us define 
$$
f_\eps(x):= \varphi(x) \int_{B^c(x,\eps)} \partial_x W_k(x-y)\bar \rho_k(y)\, dy.
$$
Then by definition of the Cauchy principle value, $f_\eps(x)$ converges to $\varphi(x) \partial_x \bar S_k(x)$ pointwise for almost every $x \in \RR$ as $\eps \to 0$. Further, we use the fact that $\bar\rho_k \in C^{0,\alpha}(\RR)$ for some $\alpha\in (-k,1)$ to obtain the uniform in $\eps$ estimate 
$$
|f_\eps(x)|\leq \left(\frac{2+k+\alpha}{k+\alpha}\right)|\varphi(x)|\, , \qquad \forall \, \, 0<\eps<1\, ,
$$ 
and therefore by Lebesgue's dominated convergence theorem,
\begin{align*}
\int_{\RR} \varphi(x) \partial_x \bar S_k(x)\, dx
&=\int_{\RR} \lim_{\eps\to 0} f_\eps(x)\bar \rho_k(x)\, dx 
= \lim_{\eps\to 0} \int_{\RR} f_\eps(x)\bar \rho_k(x)\, dx \\
&= \lim_{\eps\to 0} \iint_{|x-y|\geq \eps}\varphi(x)|x-y|^{k-2}(x-y)\bar \rho_k(x)\bar \rho_k(y)\, dxdy\\
&=\frac{1}{2}\lim_{\eps\to 0} \iint_{|x-y|\geq \eps} \left(\dfrac{\varphi(x)-\varphi(y)}{x-y}\right)|x-y|^{k} \bar \rho_k(x)\bar \rho_k(y)\, dxdy\\
&=\frac12\iint_{\RR\times\RR} \left(\dfrac{\varphi(x)-\varphi(y)}{x-y}\right)|x-y|^{k} \bar \rho_k(x)\bar \rho_k(y)\, dxdy\, .
\end{align*}
This concludes the proof of \eqref{nosingularity}.
Hence, we obtain for any $k \in (-1,1)\backslash\{0\}$,
\begin{align*}
0 
 &= - \int_\RR \varphi'(p) \bar \rho_k^m(p)\, dp + \chi \iint_{\RR\times\RR} \left(\dfrac{\varphi(x)-\varphi(y)}{x-y}\right)|x-y|^{k} \bar \rho_k(x)\bar \rho_k(y)\, dxdy\\
 &= - \int_\RR \varphi'(p) \bar \rho_k^m(p)\, dp + \chi \iint_{\RR\times\RR}\int_{0}^1 \varphi'\left((1-s)x+sy\right) |x-y|^{k} \bar \rho_k(x)\bar \rho_k(y)\, dsdxdy\\
 &= - \int_\RR \varphi'(p) \bar \rho_k^m(p)\, dp + \chi \int_{ \RR}\varphi'(p)\left\{\int_{\RR}\int_{0}^1 |q|^{k} \bar \rho_k(p - sq )\bar \rho_k(p-sq+q)\, dsdq\right\}\, dp \,
\end{align*}
and so \eqref{eq:charac sstates} follows up to a constant. Since both sides of \eqref{eq:charac sstates} have mass one, the constant is zero. To see that $\mF_k[\bar \rho_k]=0$, we substitute \eqref{eq:charac sstates} into \eqref{eq:functional} and use the same change of variables as above.\\
%%%%%%%%%%%%%%%%%%%%%

Finally, identity \eqref{eq:FV resc} is a consequence of various homogeneities.
 For every stationary state $\bar \rho_{k,\resc }$ of \eqref{eq:KSresc}, the first variation $\tfrac{\delta \mFr}{\delta \rho}[\bar \rho_{k,\resc }]=m/(m-1)\bar \rho_{k,\resc }^{m-1}+2\chi W_k\ast\bar \rho_{k,\resc }+ |x|^2/2$ vanishes on the support of $\bar \rho_{k,\resc }$ and hence it follows that for dilations $\bar \rho_\lambda(x):=\lambda \bar \rho_{k,\resc }(\lambda x)$ of the stationary state $\bar \rho_{k,\resc }$:  
\begin{align*}
 -k\mFr[\bar \rho_{k,\resc }] + \left( \frac{k}{2}-1\right)\mV[\bar \rho_{k,\resc }]
&=\left. \dfrac d{d\lambda} \mFr[\bar \rho_\lambda] \right|_{\lambda=1} \\
&=\int_{\RR} \left. \left(\frac{\delta \mFr}{\delta \rho}[\bar \rho_\lambda](x) \, \frac{d \bar \rho_\lambda}{d \lambda}(x) \right)\, dx \right|_{\lambda=1}
=0.
\end{align*}
In the fair-competition regime, attractive and repulsive forces are in balance $m+k=1$, and so \eqref{eq:FV resc} follows.
\end{proof}

%%%%%%%%%%%%%%%%%%%%%%%%%%%%%%%%%%%%%%%%%%%%%%%%%%%%%%%%%%%%%%
%%%%%%%%%%%%%%%%%%%%%%%%%%%%%%%%%%%%%%%%%%%%%%%%%%%%%%%%%%%%%%
%-----------OVERVIEW OF RESULTS
%%%%%%%%%%%%%%%%%%%%%%%%%%%%%%%%%%%%%%%%%%%%%%%%%%%%%%%%%%%%%%
Recall that stationary states in rescaled variables are self-similar solutions in original variables. Tables \ref{k<0}, \ref{k=0} and \ref{k>0} provide an overview of results proved in this paper and in \cite{CCH1} in one dimension.

%%%%%%%%%%%%%%%%%%%%%%%%%%%%%%%%%%%%%%%%%%%%%
%%%%%%%%%%%%%%%%%%%%%%%%%%%%%%%%%%%%%%%%%%%%%
%%%%%%%%%%%%%%%%%%%%%%%%%%%%%%%%%%%%%%k<0
\begin{table}[h!]
\scriptsize
\begin{center}
  \begin{tabular}{|ccc|ccc|ccc|}
    \hline
    %%%%%%%%%%%%%%%%%%%%%%%%%%%k<0
    \, & $\chi<\chi_c(k)$ &\,& \, & $\chi=\chi_c(k)$ &\,& \, & $\chi>\chi_c(k)$&\, \\ \hline
     \,& \,&\,& \,&\,& \,&\,&\,&\,\\
    \, &
     \parbox[t]{4.5cm}{%k<0, \chi<\chi_c Functional inequalities
     \textbf{Functional Inequalities:}
     \\
     \,
     \\
     \tabitem There are no stationary states in original variables, there are no minimisers for $\mF_k$ \cite[Theorem 2.9]{CCH1}..
     \\
     \tabitem In rescaled variables, all stationary states are continuous and compactly supported \cite[Theorem 2.9]{CCH1}.
     \\
     \tabitem There exists a minimiser of $\mFr$. Minimisers are symmetric non-increasing and uniformly bounded. Minimisers are stationary states in rescaled variables  \cite[Theorem 2.9]{CCH1}.
     \\
     \tabitem If $\bar \rho_{\resc}$ is a stationary state in rescaled variables, then all solutions of the rescaled equation satisfy 
     %\begin{flalign*}
      $\mFr[\rho] \geq \mFr[\bar \rho_{\resc}]$
     %\end{flalign*}
    (Theorem \ref{thm:HLSmresc}).
     \\
     \tabitem Stationary states in rescaled variables and minimisers of $\mFr$ are unique (Corollary \ref{cor:!subcrit}).
     }&\,
     &\, &\parbox[t]{4.5cm}{%k<0, \chi=\chi_c Functional inequalities
     \textbf{Functional Inequalities:}\\ \,
     \\
     \tabitem There exists a minimiser of $\mF_k$. Minimisers are symmetric non-increasing, compactly supported and uniformly bounded. Minimisers are stationary states in original variables \cite[Theorem 2.8]{CCH1}.
     \\
     \tabitem There are no stationary states in rescaled variables in $\mY_2$, and there are no minimisers of $\mFr$ in $\mY_2$
     %as $\inf_{\rho \in \mY}\mF_k[\rho]=-\infty$ 
     (Corollary \ref{cor:nosstatessupercrit} (ii)).
     \\
     \tabitem If $\bar \rho$ is a stationary state in original variables, then all solutions satisfy 
     %{equation}\label{ineq1}
     $\mF_k[\rho] \geq \mF_k[\bar \rho]=0$,
     %\end{equation}
     which corresponds to a variation of the HLS inequality (Theorem \ref{thm:HLSmequiv}).
     \\
     \tabitem Stationary states in original variables and minimisers of $\mF_k$ are unique up to dilations (Corollary \ref{cor:!crit}), and they coincide with the equality cases of $\mF_k[\rho] \geq 0$.
     }&\,
     &\, &\parbox[t]{4.5cm}{%k<0, \chi>\chi_c Functional inequalities
     \textbf{Functional Inequalities:}\\ \,
%      \\
%      \tabitem If $\bar \rho$ is a stationary state in original variables, then all solutions satisfy $\mF_k[\rho] \geq \mF_k[\bar \rho]=0$
%      (Theorem \ref{thm:HLSmequiv}).
     \\
     \tabitem There are no stationary states in original variables in $\mY$, and there are no minimisers of $\mF_k$  in $\mY$
     %as $\inf_{\rho \in \mY}\mF_k[\rho]=-\infty$ 
     (Corollary \ref{cor:nosstatessupercrit} (i)).
     \\
     \tabitem There are no stationary states in rescaled variables in $\mY_2$, and there are no minimisers of $\mFr$  in $\mY_2$
     %as $\inf_{\rho \in \mY}\mF_k[\rho]=-\infty$ 
     (Corollary \ref{cor:nosstatessupercrit} (ii)).
     }&\,
     \\ & & & & & & & &
     \\ 
     \hline
     \,& \,&\,& \,&\,& \, &\,&\,&\,\\
     \, & \parbox[t]{4.5cm}{%k<0, \chi<\chi_c Asymptotics
     \textbf{Asymptotics:}\\ \,
     \\
     \tabitem Under a stability condition solutions converge exponentially fast in Wasserstein distance towards the unique stationary state in rescaled variables with rate 1 (Proposition \ref{prop:cvsub-critical}).
     }&\,
     &\, & \parbox[t]{4.5cm}{%k<0, \chi=\chi_c, Asymptotics
     \textbf{Asymptotics:}\\ \,
     \\
     \tabitem Under a stability condition and for solutions with second moment bounded in time, we have convergence in Wasserstein distance (without explicit rate) to a unique (up to dilation) stationary state (Proposition \ref{prop:cvcritical2}).
     }&\,
     &\, & \parbox[t]{4.5cm}{%k<0, \chi>\chi_c Asymptotics
     \textbf{Asymptotics:}
     \\
     Asymptotics are not well understood yet. \\ \, \\
     \tabitem If there exists a time $t_0\geq0$ such that $\mF_k[\rho(t_0)]<0$, then $\rho$ blows up in finite time \cite{Sugi3, BCL}.
     \\
     \tabitem Numerics suggest that the energy of any solution becomes negative in finite time, but no analytical proof is known.
     }&\,
     \\ & & & & & & & &
     \\ 
     \hline 
       \end{tabular}
\end{center}
\caption{Overview of results in one dimension for $-1<k<0$.}
\label{k<0}
\end{table}

%%%%%%%%%%%%%%%%%%%%%%%%%%%%%%%%
%%%%%%%%%%%%%%%%%%%%%%%%%%%%%%%%
%%%%%%%%%%%%%%%%%%%%%%%%%%%%%%k=0
\begin{table}[h!]
\scriptsize
\begin{center}
  \begin{tabular}{|ccc|ccc|ccc|}
    \hline
     %%%%%%%%%%%%%%%%%%%%%%%%%%% k=0
     \,& $\chi<1$ &\,&\,& $\chi=1$ &\,& \,&$\chi>1$ &\,\\ 
     \hline
     \,& \,&\,& \,&\,& \,&\,&\,&\,\\
     \,&\parbox[t]{4.5cm}{%k=0, \chi<1, Functional inequalities
     \textbf{Functional Inequalities:}\\ \,
     \\
     \tabitem There are no stationary states in original variables, but self-similar profiles \cite{DoPe04, BlaDoPe06, CaDo12, CaDo14, EM16}.
     }&\,
     &\,&\parbox[t]{4.5cm}{%k=0, \chi=1, Funtional inequalities
     \textbf{Functional Inequalities:}\\ \,
     \\
     \tabitem If $\bar \rho$ is a stationary state in original variables, then all solutions satisfy $\mF_k[\rho] \geq \mF_k[\bar \rho]$, which corresponds to the logarithmic HLS inequality \cite{DoPe04, BlaDoPe06,CaCa11}.
     \\
     \tabitem Stationary states are given by dilations of Cauchy's density, $\bar \rho (x)=1/(\pi(1+|x|^2))$, wich coincide with the equality cases of the logarithmic HLS inequality. They all have infinite second moment \cite{DoPe04, BlaDoPe06,CaCa11}.
     }&\,
     &\,&\parbox[t]{4.5cm}{%k=0, \chi>1, Functional inequalities
     \textbf{Functional Inequalities:}\\ \,
     \\ 
     \tabitem Smooth fast-decaying solutions do not exist globally in time \cite{Nagai95,BiNa94,BlaDoPe06,CaPeSh07}.
     \\
     \tabitem There are no stationary states in original variables and there are no minimisers of $\mF_0$ in $\mY$ (Remark \ref{rmk:logsstates}).
     }&\,
     \\ & & & & & & & &
     \\
     \hline
     \,& \,&\,& \,&\,& \,&\,&\,&\,\\
      \,&\parbox[t]{4.5cm}{%k=0, \chi<1, Asymptotics
     \textbf{Asymptotics:}\\ \,
     \\
     \tabitem Solutions converge exponentially fast in Wasserstein distance towards the unique stationary state in rescaled variables \cite{CaCa11}.
     }&\,
     &\,& \parbox[t]{4.5cm}{%k=0, \chi=1, Asymptotics
     \textbf{Asymptotics:}\\ \,
     \\
     \tabitem Solutions converge in Wasserstein distance to a dilation of Cauchy's density (without explicit rate) if the initial second moment is infinite, and to a Dirac mass otherwise \cite{BKLN06, BlaCaMa07, CaCa11, BCC12, CF}. 
     }&\,
     &\,& \parbox[t]{4.5cm}{%k=0, \chi>1, Asymptotics
     \textbf{Asymptotics:}\\ \,
     \\
     \tabitem All solutions blow up in finite time provided the second moment is initially finite \cite{Herrero-Velazquez97, Raphael-Schweyer14}.
     }&\,
     \\ & & & & & & & &
     \\
     \hline 
            \end{tabular}
\end{center}
\caption{Overview of results in one dimension for $k=0$.}
\label{k=0}
\end{table}

%%%%%%%%%%%%%%%%%%%%%%%%%%%%%%%%%%%%%%%%%%%%%%%%
%%%%%%%%%%%%%%%%%%%%%%%%%%%%%%%%%%%%%%%%%%%%%%%%
%%%%%%%%%%%%%%%%%%%%%%%%%%%%%%%%%%%%%%%%k>0
\begin{table}[h!]
\scriptsize
\begin{center}
  \begin{tabular}{|ccc|}
    \hline
    \,& \,&\,\\
    %%%%%%%%%%%%%%%%%%%%%%%%% k>0
    \,& No criticality for $\chi$ &\,
    \\
    \,&\,&\,\\
    \hline
     \,& \,&\,\\
    \,&\parbox[t]{13.5cm}{%k>0 Funcitonal inequalities
     \textbf{Functional Inequalities:} \\ \,
     \\
     \tabitem There are no stationary states in original variables (Remark \ref{rmk:noE}). In rescaled variables, there exists a continuous symmetric non-increasing stationary state \cite[Theorem 2.11]{CCH1}.
     \\
     \tabitem There are no symmetric non-increasing global minimisers of $\mF_k$. Global minimisers of $\mFr$ can only exist in the range $0<k<\tfrac{2}{3}$ \cite[Theorem 2.11]{CCH1}.
     \\
     \tabitem If $\bar \rho_{\resc}$ is a stationary state in rescaled variables, then all solutions of the rescaled equation satisfy $\mFr[\rho] \geq \mFr[\bar \rho_{\resc}]$
     (Theorem \ref{thm:HLSmresc kpos}). Hence, for $0<k<\tfrac{2}{3}$, there exists a global minimiser for $\mFr$. 
     \\
     \tabitem For $0<k<\tfrac{2}{3}$, stationary states in rescaled variables and global minimisers of $\mFr$ are unique (Corollary \ref{cor:!kpos}).
     }&\,
    \\
    \,&\,&\,\\
    \hline
    \,&\,&\,\\
    \,&\parbox[t]{13.5cm}{%k>0 Asymptotics
      \textbf{Asymptotics:}\\ \,
     \\
     \tabitem Solutions converge exponentially fast in Wasserstein distance to the unique stationary state in rescaled variables with rate 1 (Proposition \ref{prop:cvinW2}).
     }&\,
    \\
    \,&\,&\,\\
  \hline
  \end{tabular}
\end{center}
\caption{Overview of results in one dimension for $0<k<1$.}
\label{k>0}
\end{table}

%%%%%%%%%%%%%%%%%%%%%%%%%%%%
%%%%%%%%%%%%%%%%%%%%%%%%%%%%
%%%%%%%%%%%%%%%%%%%%%%%%%%%%
%%%%%%%%%%%%%%%%%%%%%%%%%%%%
\subsection{Optimal Transport Tools}
\label{sec:Optimal Transport Tools}
%%%%%%%%%%%%%%%%%%%%%%%%%%%%
%%%%%%%%%%%%%%%%%%%%%%%%%%%%
%%%%%%%%%%%%%%%%%%%%%%%%%%%%
%%%%%%%%%%%%%%%%%%%%%%%%%%%%

This sub-section summarises the main results of optimal
transportation we will need. They were already used for the case of logarithmic HLS inequalities and the classical Keller-Segel model in 1D and radial 2D, see \cite{CaCa11}, where we refer for detailed proofs.

Let $\tilde \rho$ and $\rho$ be two density probabilities.
According to \cite{Brenier91,McCann95}, there exists a convex
function $\psi$ whose gradient pushes forward the measure
$\tilde\rho(a) da$ onto $\rho(x) dx$: $\psi'\#
\left(\tilde\rho(a)  da\right) = \rho(x)  dx$. This convex
function satisfies the Monge-Amp\`ere equation in the weak sense:
for any test function $\varphi\in C_b(\RR)$, the following
identity holds true
\begin{equation}
\int_{\RR} \varphi(\psi'(a)) \tilde\rho(a)\, da =
\int_{\RR} \varphi(x) \rho(x)\, dx\, . \label{eq:MongeAmpere2}
\end{equation}
The convex map is unique a.e. with respect to $\rho$ and it gives
a way of interpolating measures. In fact, the interpolating curve
$\rho_s$, $s\in [0,1]$, with $\rho_0=\rho$ and $\rho_1=\tilde
\rho $ can be defined as $\rho_s (x)\, dx= (s \psi' +(1-s)\id)(x)\#
\rho(x)\, dx$ where $\id$ stands for the identity map in $\RR$. This
interpolating curve is actually the minimal geodesic joining the
measures $\rho(x) dx$ and $\tilde\rho (x) dx$. The notion of convexity
associated to these interpolating curves is nothing else than
convexity along geodesics, introduced and called displacement
convexity in \cite{McCann97}. In one dimension the displacement
convexity/concavity of functionals is easier to check as seen in
\cite{CarrilloFerreiraPrecioso12, Carrillo-Slepcev09}. The convexity of the
functionals involved can be summarised as follows \cite{McCann97,CarrilloFerreiraPrecioso12}:
\begin{theorem}\label{convex-concave}
The functional $\mathcal U_m[\rho]$ is displacement-convex
provided that $m\geq 0$. The functional $\mathcal W_k[\rho]$ is displacement-concave if $k \in (-1,1)$.
\end{theorem}
This means we have to deal with convex-concave compensations.
On the other hand, regularity of the transport map is a
complicated matter. Here, as it was already done in \cite{CaCa11},
we will only use the fact that the Hessian measure $\det{\!}_H
D^2\psi(a) da$ can be decomposed in an absolute continuous part
$\det{\!}_A D^2\psi(a) da$ and a positive singular measure
(Chapter 4, \cite{Villani03}). Moreover, it is known that a convex
function $\psi$ has Aleksandrov second derivative $D^2_A \psi(a)$
almost everywhere and that $\det{\!}_A D^2\psi(a) = \det D^2_A
\psi(a)$. In particular we have $\det{\!}_H D^2\psi(a) \geq
\det{\!}_A D^2\psi(a)$. The formula for the change of variables
will be important when dealing with the internal energy
contribution. For any measurable function $U$, bounded below such
that $U(0) = 0$ we have \cite{McCann97}
\begin{equation}\label{eq:change variables}
\int_{\RR} U(\tilde\rho(x)) \, dx = \int_{\RR}
U\left(\dfrac{\rho(a)}{\det{\!}_A D^2\psi(a)}\right)\det{\!}_A
D^2\psi(a)\, da\, .
\end{equation}
Luckily, the complexity of Brenier's transport problem dramatically reduces in one dimension. More precisely, the
transport map $\psi'$ is a non-decreasing function, therefore it is
differentiable a.e. and it has a countable number of jump
singularities. The singular part of the positive measure $\psi''(x) \,dx$
corresponds to having holes in the support of the density $\rho$.
Also, the Aleksandrov second derivative of $\psi$ coincides with
the absolutely continuous part of the positive measure $\psi''(x)\, dx$
that will be denoted by $\psi_{\AC}''(x)\, dx$. Moreover, the a.e.
representative $\psi'$ can be chosen to be the distribution
function of the measure $\psi''(x) \, dx$ and it is of bounded variation
locally, with lateral derivatives existing at all points and
therefore, we can always write for all $a<b$
\begin{equation*}\label{tech1}
\psi'(b)-\psi'(a)= \int_{(a,b]}  \psi''(x)\, dx \geq \int_a^b
\psi_{\AC}''(x)\,dx
\end{equation*}
for a well chosen representative of $\psi'$.\\

The following Lemma proved in \cite{CaCa11} will be used to
estimate the interaction contribution in the free energy, and in
the evolution of the Wasserstein distance. 

\begin{lemma}\label{lem:interaction}
Let $\mathcal K:(0,\infty)\to \RR$ be an increasing and strictly
concave function. Then, for any $(a,b)$
\begin{equation} \label{eq:jensen 1D}
\mathcal K\left(  \frac{\psi'(b)-\psi'(a)}{b-a} \right) \geq
\int_{0}^1 \mathcal K \left( \psi_{\AC}''([a,b]_s) \right)\, ds\,
,
\end{equation}
where the convex combination of $a$ and $b$ is given by
$[a,b]_s=(1-s)a+sb$.
Equality is achieved in \eqref{eq:jensen 1D} if and only if the
distributional derivative of the transport map $\psi''$ is a
constant function.
\end{lemma}

Optimal transport is a powerful tool for reducing functional
inequalities onto pointwise inequalities ({\em e.g.} matrix
inequalities). In other words, to pass from microscopic
inequalities between particle locations to macroscopic
inequalities involving densities. We highlight for example the
seminal paper by McCann \cite{McCann97} where the displacement
convexity issue for some energy functional is reduced to the
concavity of the determinant. We also refer to the works of Barthe
\cite{Barthe97,Barthe98} and Cordero-Erausquin {\em et al.}
\cite{CoNaVi08}. The previous lemma will allow us to
connect microscopic to macroscopic inequalities by simple
variations of the classical Jensen inequality.

%%%%%%%%%%%%%%%%%%%%%%%%%%%%
%%%%%%%%%%%%%%%%%%%%%%%%%%%%
%%%%%%%%%%%%%%%%%%%%%%%%%%%%
%%%%%%%%%%%%%%%%%%%%%%%%%%%%
\section{Functional inequalities}
\label{sec: Functional inequalities}
%%%%%%%%%%%%%%%%%%%%%%%%%%%%
%%%%%%%%%%%%%%%%%%%%%%%%%%%%
%%%%%%%%%%%%%%%%%%%%%%%%%%%%
%%%%%%%%%%%%%%%%%%%%%%%%%%%%

The first part of analysing the aggregation-diffusion equations \eqref{eq:KS} and \eqref{eq:KSresc} is devoted to the derivation of functional inequalities which are all variants of the Hardy-Littlewood-Sobolev (HLS) inequality also known as the weak Young's inequality \cite[Theorem 4.3]{LieLo01}:
\begin{align}
&\iint_{\RR\times \RR} f(x) {|x-y|^{k}} f(y)\, dxdy \leq C_{HLS}(p,q,\lambda) \|f\|_{L^p} \|f\|_{L^q}\, ,  \label{eq:HLS}\\ & \dfrac1p + \dfrac1q  = 2 +  k \,  , \quad p,q>1\, ,\quad   k\in(-1,0) \, . \nonumber
\end{align}
%%%%%%%%%%%%%%%%%%%%%%%%%%%%%%%%%%%%%%%%%%%%%%%%%%HLSm

\begin{theorem}[Variation of HLS]\label{thm:HLSm}
Let $k \in (-1,0)$ and $m=1-k$. For $f \in L^1(\RR) \cap L^m(\RR)$, we have
\begin{equation}\label{eq:HLSm}
\left|\iint_{\RR\times \RR} f(x){|x-y|^{k}}f(y) dxdy\right| \leq C_* ||f||^{1+k}_1 ||f||^m_m,
\end{equation}
where $C_*=C_*(k)$ is the best constant.
\end{theorem}
%%%%%%%%%%%%%%%%%%%%%%%%%%%%%%%%%%
\begin{proof} The inequality is a direct consequence of the standard HLS inequality \eqref{eq:HLS} by choosing $p = q = \tfrac {2}{2+k}$, and of H\"{o}lder's inequality. For $k \in (-1,0)$ and for any $f \in L^1(\RR) \cap L^m(\RR)$, we have
$$
\left| \iint_{\RR \times \RR} f(x) |x-y|^k f(y) dx dy \right|
\leq C_{HLS} ||f||_p^2
\leq C_{HLS} ||f||_1^{1+k} ||f||_m^m.
$$
Consequently, $C_*$ is finite and bounded from above by $C_{HLS}$.
\end{proof}

For instance inequality \eqref{eq:HLSm} is a consequence of interpolation between $L^1$ and $L^m$. We develop in this section another strategy which enables to recover inequality \eqref{eq:HLSm}, as well as further variations which contain an additional quadratic confinement potential. This method involves two main ingredients:
\begin{itemize}
\item First it is required to know {\em a priori} that the inequality possesses some extremal function denoted {\em e.g.} by $\bar \rho(x)$ (characterised as a critical point of the energy functional). This is not an obvious task due to the intricacy of the equation satisfied by $\bar \rho(x)$. Without this {\em a priori} knowledge, the proof of the inequality remains incomplete. The situation is in fact similar to the case of convex functionals, where the existence of a critical point ensures that it is a global minimiser  of the functional. The existence of optimisers was shown in \cite{CCH1}.
\item Second we invoke some simple lemma at the microscopic level. It is nothing but the Jensen's inequality for the case of inequality \eqref{eq:HLSm} (which is somehow degenerated). It is a variation of Jensen's inequality in the rescaled case.
\end{itemize}

%%%%%%%%%%%%%%%%%%%%%%%%%%%%
%%%%%%%%%%%%%%%%%%%%%%%%%%%%
%%%%%%%%%%%%%%%%%%%%%%%%%%%%
%%%%%%%%%%%%%%%%%%%%%%%%%%%%
\subsection{Porous Medium Case $k<0$}
\label{sec: Porous Medium Case k neg}
%%%%%%%%%%%%%%%%%%%%%%%%%%%%
%%%%%%%%%%%%%%%%%%%%%%%%%%%%
%%%%%%%%%%%%%%%%%%%%%%%%%%%%
%%%%%%%%%%%%%%%%%%%%%%%%%%%%
%%%%%%%%%%%%%%%%%%%%%%%%%%%%
In the porous medium case, we have $k \in (-1,0)$ and hence $m\in(1,2)$. For $\chi=0$, this corresponds to the well-studied porous medium equation (see \cite{VazquezPME} and references therein). 
It follows directly from Theorem \ref{thm:HLSm}, that for all $\rho \in \mY$ and for any $\chi>0$,
\begin{equation*}%\label{HLSmcons1}
    \mF_k[\rho]\geq \frac{1-\chi C_*}{m-1} ||\rho||_m^m\, ,
\end{equation*}
where $C_*=C_*(k)$ is the optimal constant defined in \eqref{eq:HLSm}. Since global minimisers have always smaller or equal energy than stationary states, and stationary states have zero energy by Lemma \ref{lem:char crit}, it follows that $\chi \geq 1/C_*$. We define the \emph{critical interaction strength} by 
\begin{equation}\label{defchic}
 \chi_c(k):=\frac{1}{C_*(k)}\, ,
\end{equation}
and so for $\chi=\chi_c(k)$, all stationary states of equation \eqref{eq:KS} are global minimisers of $\mF_k$. From \cite[Theorem 2.8]{CCH1}, we further know that there exist global minimisers of $\mF_k$ only for critical interaction strength $\chi=\chi_c(k)$ and they are radially symmetric non-increasing, compactly supported and uniformly bounded. Further, all minimisers of $\mF_k$ are stationary states of equation \eqref{eq:KS}.\\
From the above, we can also directly see that for $0<\chi<\chi_c(k)$, no stationary states exist for equation \eqref{eq:KS}. Further, there are no minimisers of $\mF_k$. However, there exist global minimisers of the rescaled free energy $\mFr$ and they are radially symmetric non-increasing and uniformly bounded stationary states of the rescaled equation \eqref{eq:KSresc} \cite[Theorem 2.9]{CCH1}.

%%%%%%%%%%%%%%%%%%%%%%%%%%
%%%%%%%%%%%%%%%%%%%%%%%%%%
\begin{theorem}\label{thm:HLSmequiv}
 Let $k \in (-1,0)$ and $m=1-k$. If \eqref{eq:KS} admits a stationary density $\bar \rho_k$ in $\mY$, then for any $\chi>0$
 \begin{equation*} \label{fineqcritical}
 \mF_k[\rho]\geq 0, \quad \forall \rho \in \mY
 \end{equation*}
 with the equality cases given dilations of $\bar \rho_k$. In other words, for critical interaction strength $\chi=\chi_c(k)$, inequality \eqref{eq:HLSm} holds true for all $f \in L^1(\RR) \cap L^m(\RR)$.
\end{theorem}
%%%%%%%%%%%%%%%%%%%%%%
%%%%%%%%%%%%%%%%%%%%%%

\begin{proof}%[Proof of Theorem \ref{thm:HLSmequiv}]
For a given stationary state $\bar \rho_k \in \mY$ and solution $\rho \in \mY$ of \eqref{eq:KS}, we denote by $\psi$ the convex function whose gradient pushes forward the measure $\bar \rho_k(a) da$ onto $\rho(x) dx$: $\psi' \# \left(\bar\rho_k(a)  da\right) = \rho(x)  dx$. Using \eqref{eq:change variables}, the functional $\mathcal F_{k}[\rho]$ rewrites as follows:
\begin{align*}
\mF_{k}[\rho]& = \dfrac1{m-1}\int_\RR \left(\dfrac{\bar \rho_k(a)}{\psi_{ac}''(a)}\right)^{m-1} \bar \rho_k(a)\, da\\
&\quad+ \dfrac{\chi}{k} \iint_{\RR\times\RR}\left(\dfrac{\psi'(a)-\psi'(b)}{a-b}\right)^{k} |a-b|^{k} \bar \rho_k(a) \bar \rho_k(b) \, da db
\\
&= \dfrac1{m-1} \int_\RR \left(\psi_{ac}''(a)\right)^{1-m} \bar \rho_k(a)^m\, da\\
&\quad+ \dfrac{\chi}{1-m}  \iint_{\RR\times\RR} \left(\dfrac{\psi'(a)-\psi'(b)}{a-b}\right)^{1-m} |a-b|^{1-m} \bar \rho_k(a) \bar \rho_k(b) \, da db   \, ,
\end{align*}
where $\psi'$ non-decreasing. By Lemma \ref{lem:char crit} (i), we can write for any $\gamma \in \RR$,
$$
\int_\RR(\psi_{ac}''(a))^{-\gamma} \bar\rho_k(a)^m \, da=  \chi \iint_{\RR\times\RR} \bla \psi_{ac}''([a,b])^{-\gamma} \bra |a-b|^{1-m}\bar \rho_k(a) \bar \rho_k(b) \,  dadb\, ,
$$
where 
$$
\bla u([a,b]) \bra = \int_0^1 u([a,b]_s)\, ds
$$ 
and $[a,b]_s=(1-s)a+sb$ for any $a,b \in \RR$ and $u:\RR \to \RR_+$. Hence, choosing $\gamma=m-1$,
$$
\mF_{k}[\rho] =  \dfrac{\chi}{m-1} \iint_{\RR\times\RR} \left\{ \bla \psi_{ac}''([a,b])^{1-m} \bra - \left(\dfrac{\psi'(a)-\psi'(b)}{a-b}\right)^{1-m}\right\}   |a-b|^{1-m}\bar \rho_k(a) \bar \rho_k(b) \, da db\, .
$$
Using the strict concavity and increasing character of the power function $-(\cdot)^{1-m}$ and Lemma \ref{lem:interaction}, we deduce $ \mF_{k}[\rho]\geq 0$.
Equality arises if and only if the derivative of the transport map $\psi''$ is a constant function, i.e. when $\rho$ is a dilation of $\bar \rho_k$.\\
We conclude that if \eqref{eq:KS} admits a stationary state $\bar \rho_k \in \mY$, then $\mathcal{F}_k(\rho) \geq 0$ for any $\rho \in \mY$. This functional inequality is equivalent to \eqref{eq:HLSm} if we choose $\chi=\chi_c(k)$.
\end{proof}

\begin{remark}[Comments on the Inequality Proof]
In the case of critical interaction strength $\chi=\chi_c(k)$, Theorem \ref{thm:HLSmequiv} provides an alternative proof for the variant of the HLS inequality Theorem \ref{thm:HLSm} assuming the existence of a stationary density for \eqref{eq:KS}. More precisely, the inequalities $\mF_k[\rho]\geq 0$ and \eqref{eq:HLSm} are equivalent if $\chi=\chi_c(k)$. However, the existence proof \cite[Proposition 3.4]{CCH1} crucially uses the HLS type inequality \eqref{eq:HLSm}. If we were able to show the existence of a stationary density by alternative methods, e.g. fixed point arguments, we would obtain a full alternative proof of inequality \eqref{eq:HLSm}.
\end{remark}

\begin{remark}[Logarithmic Case]\label{rmk:logsstates} There are no global minimisers of $\mF_0$ in the logarithmic case $k=0$, $m=1$ except for critical interaction strength $\chi=1$. To see this, note that the characterisation of stationary states  \cite[Lemma 2.3]{CaCa11} which corresponds to Lemma \ref{lem:char crit}(i) for the case $k \neq 0$, holds true for any $\chi>0$. Similarly, the result that the existence of a stationary state $\bar \rho$ implies the inequality $\mF_0[\rho]>\mF_0[\bar \rho]$ \cite[Theorem 1.1]{CaCa11} holds true for any $\chi>0$, and corresponds to Theorem \ref{thm:HLSmequiv} in the case $k \neq 0$.  
 Taking dilations of Cauchy's density \eqref{eq:StSt log}, $\rho_\lambda(x)=\lambda \bar\rho_0\left(\lambda x\right)$, we have $\mF_0[\rho_\lambda]=(1-\chi)\log\lambda+\mF_0[\bar\rho_0]$, and letting $\lambda \to \infty$ for super-critical interaction strengths $\chi>1$, we see that $\mF_0$ is not bounded below. Similarly, for sub-critical interaction strengths $0<\chi<1$, we take the limit $\lambda \to 0$ to see that $\mF_0$ is not bounded below. Hence, there are no global minimisers of $\mF_0$ and also no stationary states (by equivalence of the two) except if $\chi=1$.
\end{remark}

Further, we obtain the following uniqueness result:

\begin{corollary}[Uniqueness in the Critical Case]\label{cor:!crit}
 Let $k \in (-1,0)$ and $m=1-k$. If $\chi=\chi_c(k)$, then there exists a unique stationary state (up to dilations) to equation \eqref{eq:KS}, with second moment bounded, and a unique minimiser (up to dilations) for $\mF_k$ in $\mY$.
\end{corollary}
\begin{proof}
 By \cite[Theorem 2.8]{CCH1}, there exists a minimiser of $\mF_k$ in $\mY$, which is a stationary state of equation \eqref{eq:KS}. Assume \eqref{eq:KS} admits two stationary states $\bar \rho_1$ and $\bar \rho_2$. By Lemma \ref{lem:char crit}, $\mF_k[\bar \rho_1]=\mF_k[\bar \rho_2]=0$. It follows from Theorem \ref{thm:HLSmequiv} that $\bar \rho_1$ and $\bar \rho_2$ are dilations of each other.
\end{proof}

%%%%%%%%%
%%%%%%%%%
%%%%%%%%%
%%%%%%%%%
A functional inequality similar to \eqref{eq:HLSm} holds true for sub-critical interaction strengths in rescaled variables:
\begin{theorem}[Rescaled Variation of HLS]\label{thm:HLSmresc}
  For any $\chi>0$, let $k \in (-1,0)$ and $m=1-k$. If $\bar \rho_{k,\resc} \in \mY_2$ is a stationary state of \eqref{eq:KSresc}, then we have for any solution $\rho \in \mY_2$,
 $$
 \mFr[\rho] \geq \mFr[\bar \rho_{k,\resc}]
 $$
 with the equality cases given by $ \rho=\bar \rho_{k,\resc}$.
\end{theorem}
The proof is based on two lemmatas: the characterisation of steady states Lemma \ref{lem:char crit} and a microscopic inequality. The difference with the critical case lies in the nature of this microscopic inequality: Jensen's inequality needs to be replaced here as homogeneity has been broken.
To simplify the notation, we denote by $u_{ac}(s):=\psi_{ac}''\left([a,b]_s\right)$ as above  with $[a,b]_s:=(1-s)a+sb$ for any $a,b \in \RR$. We also introduce the notation
$$
\bla u \bra:=\dfrac{\psi'(a)-\psi'(b)}{a-b}= \int_0^1 \psi''([a,b]_s)\, ds
$$
with $u(s):=\psi''\left([a,b]_s\right)$. Both notations coincide when $\psi''$ has no singular part. Note there is a little abuse of notation since $\psi''$ is a measure and not a function, but this notation allows us for simpler computations below.
\begin{lemma} \label{lem:micro ineq resc}
Let $\alpha, \beta>0$ and $m>1$. For any $a, b \in \RR$ and any convex function $\psi:\RR \to \RR$:
\begin{equation}
 \alpha  \bla \psi''([a,b]) \bra^{1-m}   +
\beta(1-m) \bla \psi''([a,b]) \bra^2 \leq
(\alpha+2\beta)  \bla \left(\psi_{ac}''([a,b])\right)^{1-m} \bra   -
\beta(m+1)\, , \label{eq:Jensen modified m}
\end{equation}
where equality arises if and only if $\psi''\equiv  1$ a.e.
\end{lemma}
\begin{proof} 
We have again by Lemma \ref{lem:interaction},
\begin{align*}
(\alpha+2\beta)\bla u \bra^{1-m} \leq (\alpha+2\beta ) \bla u_{ac}^{1-m} \bra\, ,
\end{align*}
thus
\begin{equation*}
 \alpha  \bla u \bra^{1-m}   +
\beta (1-m) \bla u \bra^2 \leq \,
(\alpha+2\beta ) \bla u_{ac}^{1-m}  \bra\\
-\beta\left[2 \bla u \bra^{1-m} + (m-1) \bla u \bra^2 \right] \, .
\end{equation*}
We conclude since the quantity in square brackets verifies
$$
\forall X>0:\; 2  X^{1-m} + (m-1) X^2\geq  m+1 \, .
$$
Equality arises if and only if $u$ is almost everywhere
constant and $\bla u \bra = 1$.
\end{proof}

\begin{proof}[Proof of Theorem \ref{thm:HLSmresc}.]
We denote by $\bar \rho = \bar \rho_{k,\resc} \in \mY_2$ a stationary state of \eqref{eq:KSresc} for the sake of clarity. Then for any solution $\rho \in \mY_2$ of \eqref{eq:KSresc}, there exists a convex function $\psi$ whose gradient pushes forward the measure $\bar \rho(a) da$ onto $\rho(x) dx$, 
$$\psi' \# \left(\bar\rho(a)  da\right) = \rho(x)  dx.$$ 
Similarly to the proof of Theorem \ref{thm:HLSmequiv}, the functional $\mFr[\rho]$ rewrites as follows:
\begin{align*}
\mathcal F_{k,\resc}[\rho]& = \dfrac1{m-1}\int_\RR (\psi_{ac}''(a))^{1-m} \bar \rho(a)^m\, da \\
&\quad+ \dfrac{\chi}{k} \iint_{\RR\times\RR} \left(\dfrac{\psi'(a)-\psi'(b)}{a-b}\right)^{k}|a-b|^{k} \bar \rho(a) \bar \rho(b) \, da db 
\\& \quad  + \dfrac14 \iint_{\RR\times\RR} \left(\dfrac{\psi'(a) - \psi'(b)}{a-b}\right)^2 |a-b|^2 \bar \rho(a) \bar \rho(b) \, da db \, .
\end{align*}
From the characterisation of steady states Lemma \ref{lem:char crit} (ii), we know that for all $\gamma \in \RR$:
$$
 \int_\RR(\psi_{ac}''(a))^{-\gamma}\bar \rho(a)^m\, da= \iint_{\RR\times\RR} \bla \psi_{ac}''([a,b]) ^{-\gamma}\bra  \left( \chi|a-b|^{1-m} + \dfrac{|a-b|^2}2 \right)\bar \rho(a) \bar \rho(b) \, dadb \, .
$$
Choosing $\gamma=m-1$, we can rewrite the energy functional as
\begin{align*}
(m-1)\mathcal F_{k,\resc}[\rho] =  &\iint_{\RR\times\RR} \bla \psi_{ac}''([a,b]) ^{1-m}\bra  \left( \chi|a-b|^{1-m} + \dfrac{|a-b|^2}2 \right)\bar \rho(a) \bar \rho(b) \, da db 
\\  &- \iint_{\RR\times\RR} \left( \bla \psi''([a,b]) \bra^{1-m} \chi|a-b|^{1-m}\right.\\
&\left.\qquad \qquad \quad + \bla \psi''([a,b]) \bra^{2} (1-m) \dfrac{|a-b|^2}{4 }\right) \bar \rho(a) \bar \rho(b) \, da db 
%     -  \dfrac{(1+m)|a-b|^2}{4(1-m)}  \right] &
\\
& \geq  (m+1)\iint_{\RR\times\RR}    \dfrac{|a-b|^2}4 \bar \rho(a) \bar \rho(b) \, da db \\
%& = \dfrac1{m-1}\int_\RR (\psi''(a))^{1-m} \mu(a)^m\, da  \\
%
  &=   \dfrac{m+1}{2}   \int_\RR |a|^2 \bar \rho(a)\, da  = (m-1) \mathcal F_{k,\resc}[\bar \rho]\, .
\end{align*}
Here, we use the variant of Jensen's inequality \eqref{eq:Jensen modified m} and for the final step, identity \eqref{eq:FV resc}. Again equality holds true if and only if $\psi''$ is identically one.
\end{proof}

%%%%%%%%%%%%%%%%%%%%%
\begin{remark}[New Inequality] Up to our knowledge, the functional inequality in Theorem \ref{thm:HLSmequiv} is not known in the literature. 
 Theorem \ref{thm:HLSmresc} makes a connection between equation \eqref{eq:KSresc} and this new general functional inequality by showing that stationary states of the rescaled equation \eqref{eq:KSresc} correspond to global minimisers of the free energy functional $\mFr$. The converse was shown in \cite[Theorem 2.9]{CCH1}.
\end{remark}
%%%%%%%%%%%%%%%%%%%%%%
%%%%%%%%%%%%%%%%%%%%%%
%%%%%%%%%%%%%%%%%%%%%%
As a direct consequence of Theorem \ref{thm:HLSmresc} and the scaling given by \eqref{scaling}, we obtain the following corollaries:

\begin{corollary}[Uniqueness in the Sub-Critical Case]\label{cor:!subcrit}
 Let $k \in (-1,0)$ and $m=1-k$. If $0<\chi<\chi_c(k)$, then there exists a unique stationary state with second moment bounded to the rescaled equation \eqref{eq:KSresc}, and a unique minimiser for $\mFr$ in $\mY_2$.
\end{corollary}
\begin{proof}
 By \cite[Theorem 2.9]{CCH1}, there exists a minimiser of $\mFr$ in $\mY_2$ for sub-critical interaction strengths $0<\chi<\chi_c(k)$, which is a stationary state of equation \eqref{eq:KSresc}. Assume \eqref{eq:KSresc} admits two stationary states $\bar \rho_1$ and $\bar \rho_2$. By Theorem \ref{thm:HLSmresc}, $\mFr[\bar \rho_1]=\mFr[\bar \rho_2]$ and it follows that $\bar \rho_1$ and $\bar \rho_2$ are dilations of each other.
\end{proof}

\begin{corollary}[Self-Similar Profiles]
For $0<\chi<\chi_c(k)$, let $k \in (-1,0)$ and $m=1-k$. There exists a unique (up to dilations) self-similar solution $\rho$ to \eqref{eq:KS} given by 
 $$
 \rho(t,x)=\left((2-k)t+1\right)^{\frac{1}{k-2}} \, u\left( \left((2-k)t+1\right)^{\frac{1}{k-2}} x\right)\, ,
 $$
 where $u$ is the unique minimiser of $\mFr$ in $\mY_2$.
\end{corollary}
%%%%%%%%%%%%%%%%
%%%%%%%%%%%%%%%%
%%%%%%%%%%%%%
%%%%%%%%%%%%%
%%%%%%%%%
%%%%%%%%%
\begin{corollary}[Non-Existence Super-Critical and Critical Case]\label{cor:nosstatessupercrit}
\begin{enumerate}[(i)]
 \item If $\chi>\chi_c(k)$, there are no stationary states of equation \eqref{eq:KS} in $\mY$, and the free energy functional $\mF_k$ does not admit minimisers in $\mY$. 
 \item If $\chi \geq \chi_c(k)$, there are no stationary states of the rescaled equation \eqref{eq:KSresc} in $\mY_2$, and the rescaled free energy functional $\mFr$ does not admit minimisers in $\mY_2$.
\end{enumerate}
\end{corollary}

\begin{proof}
 For critical $\chi_c(k)$, there exists a minimiser $\bar \rho \in \mY$ of $\mF_k$ by \cite[Theorem 2.8]{CCH1}, which is a stationary state of equation \eqref{eq:KS} by \cite[Theorem 3.14]{CCH1}. For $\chi>\chi_c(k)$, we have 
 $$
 \mF_k[\bar \rho]=\mU_m[\bar \rho]+\chi\mW_k[\bar\rho] < \mU_m[\bar \rho]+\chi_c(k)\mW_k[\bar\rho]=0
 $$
 since stationary states have zero energy by Lemma \ref{lem:char crit} (i). However, by Theorem \ref{thm:HLSmequiv}, if there exists a stationary state for $\chi>\chi_c(k)$, then all $\rho \in \mY$ satisfy $\mF_k[\rho]\geq 0$, which contradicts the above. Therefore, the assumptions of the theorem cannot hold and so there are no stationary states in original variables.
 Further, taking dilations $\rho_\lambda(x)=\lambda \bar \rho\left(\lambda x\right)$, we have $\mF_k[\rho_\lambda]=\lambda^{-k} \mF_k[\bar\rho]<0$, and letting $\lambda \to \infty$, we see that $\inf_{\rho \in \mY} \mF_k[\rho] = -\infty$, and so (i) follows.\\
 
 In order to prove (ii), observe that the minimiser $\bar \rho$ for critical $\chi=\chi_c(k)$ is in $\mY_2$ as it is compactly supported \cite[Corollary 3.9]{CCH1}. We obtain for the rescaled free energy of its dilations
 $$
 \mFr[\rho_\lambda]=\lambda^{-k}\mF_k[\bar \rho]+\frac{\lambda^{-2}}{2}\mV[\bar \rho] \to -\infty\, ,\qquad \text{as}\, \, \lambda \to \infty\, .
 $$
 Hence, $\mFr$ is not bounded below in $\mY_2$. Similarly, for $\chi=\chi_c(k)$, 
 $$
 \mFr[\rho_\lambda]=\frac{\lambda^{-2}}{2}\mV[\bar \rho] \to 0\, ,\qquad \text{as}\, \, \lambda \to \infty\, ,
 $$
 and so for a minimiser $\tilde \rho \in \mY_2$ to exist, it should satisfy $\mFr[\tilde \rho]\leq 0$. However, it follows from Theorem \ref{thm:HLSm} that $\mFr[\rho]\geq \frac{1}{2}\mV[\rho]>0$ for any $\rho \in \mY_2$, and therefore, $\mFr$ does not admit minimisers in $\mY_2$ for $\chi=\chi_c(k)$.\\
 Further, if equation \eqref{eq:KSresc} admitted stationary states in $\mY_2$ for any $\chi\geq \chi_c(k)$, then they would be minimisers of $\mFr$ by Theorem \ref{thm:HLSmresc}, which contradicts the non-existence of minimisers.
\end{proof}

%%%%%%%%%%%%%%%%%%%%%%%%
%%%%%%%%%%%%%%%%%%%%%%%%
%%%%%%%%%%%%%%%%
%%%%%%%%%%%%%%%%
\begin{remark}[Linearisation around the stationary density]
\label{rem:linearization fair}
We linearise the functional $\mF_k$ around the stationary distribution $\bar \rho_k$ of equation \eqref{eq:KS}.
For the perturbed measure $\mu_\eps = (\id + \eps\eta')\#\bar\mu_k$, with $d \bar\mu_k(x) = \bar \rho_k(x) \, dx$ and $d \mu_\eps (x) = \rho_\eps (x) \, dx$, we have
\begin{align*}
\mF_k[\rho_\eps]
& = \dfrac{\eps^2}2 m \left[ \int_\RR   \eta''(a) ^2 \bar \rho_k(a)^m\, da - \chi_c(k) \iint_{\RR\times\RR} \left(\dfrac{\eta'(a)-\eta'(b)}{a-b}\right)^{2} |a-b|^{1-m} \bar \rho_k(a) \bar \rho_k(b) \, dadb    \right]\\
&\quad+ o(\eps^2) \\
& = \dfrac{\eps^2}2 m \chi_c(k)\iint_{\RR\times\RR}
 \left\{ \bla \eta''([a,b])^{2} \bra  - \bla \eta''([a,b]) \bra^{2}
 \right\} |a-b|^{1-m} \bar \rho_k(a) \bar \rho_k(b) \, dadb + o(\eps^2)\, .
\end{align*}
We define the local oscillations (in $L^2$) of functions over intervals as
\[ \osc_{(a,b)}(v): = \int_{t = 0}^1  \left\{ v([a,b]_t) - \big\langle v([a,b]) \big\rangle\right\} ^2\, dt \geq 0 \, . \]
The Hessian of the functional $\mF_k$ evaluated at the stationary density $\bar \rho_k$ then reads
\[ D^2 \mF_k[\bar \rho_k](\eta,\eta) =  m\chi_c(k)  \iint_{\RR\times\RR}
 \osc_{(a,b)} ( \eta'') |a-b|^{1-m} \bar \rho_k(a) \bar \rho_k(b) \, dadb\, \geq0. \]
 %%%%%%%%%%
 %%%%%%%%%%
 Similarly, we obtain for the rescaled free energy
 \begin{align*}
\mathcal F_{k,\resc}[\rho_\eps]
& = \mathcal F_{k,\resc}[\bar \rho_k] 
+  \dfrac{\eps^2}2 m  \int_\RR   \eta''(a) ^2 \bar \rho_k(a)^m\, da \\
&\quad - \dfrac{\eps^2}2 m \chi  \iint_{\RR\times\RR} \left(\dfrac{\eta'(a)-\eta'(b)}{a-b}\right)^{2} |a-b|^{1-m} \bar \rho_k(a) \bar \rho_k(b)\, dadb   \\
& \quad + \dfrac{\eps^2}4  \iint_{\RR\times\RR} \left(\dfrac{\eta'(a) - \eta'(b)}{a-b}\right)^2 |a-b|^2  \bar \rho_k(a) \bar \rho_k(b)\, dadb  + o(\eps^2) \\
& = \mathcal F_{k,\resc}[\bar \rho_k]\\
&\quad +  \dfrac{\eps^2}2
\left[m \chi  \iint_{\RR\times\RR}
\left\{ \bla \eta''([a,b])^2 \bra   -  \bla \eta''([a,b]) \bra^2 
 \right\} |a-b|^{1-m}  \bar \rho_k(a) \bar \rho_k(b)\, dadb 
\right.\\
& \quad + \left.   \iint_{\RR\times\RR}
\left\{ \frac m2 \bla \eta''([a,b])^2 \bra  + \frac12 \bla \eta''([a,b]) \bra^2   \right\} |a-b|^2  \bar \rho_k(a) \bar \rho_k(b)\, dadb 
\right] + o(\eps^2) 
\end{align*}
to finally conclude
\begin{align*}
\mathcal F_{k,\resc}[\rho_\eps] & = \mathcal F_{k,\resc}[\bar \rho_k] \\
&\quad + \dfrac{\eps^2}2
\left[  \iint_{\RR\times\RR} \osc_{(a,b)} ( \eta'') \left( m \chi |a-b|^{1-m} + \frac m2 |a-b|^2 \right) \bar \rho_k(a) \bar \rho_k(b)\, dadb  \right. \\
& \quad +  \left. \frac{m+1}2 \iint_{\RR\times\RR} \left( \eta'(a) - \eta'(b) \right)^2  \bar \rho_k(a) \bar \rho_k(b)\, dadb   \right] + o(\eps^2)\, ,
\end{align*}
and hence, the Hessian evaluated at the stationary state $\bar \rho_k$ of \eqref{eq:KSresc} is given by the expression
\begin{align*}
D^2 \mF_{k,\resc}[\bar \rho_k ](\eta,\eta) =   &\iint_{\RR\times\RR} \osc_{(a,b)} ( \eta'') \left( m \chi |a-b|^{1-m} + \frac m2 |a-b|^2 \right) \bar \rho_k(a) \bar \rho_k(b)\, dadb  \\
&+  (m+1) \int_{\RR}  \eta'(a)^2 \bar \rho_k (a) \, da   \, \geq0 \, .
\end{align*}
We have naturally that the functional $\mF_{k,\resc}$ is locally uniformly convex, with the coercivity constant $m+1$. However, the local variations of $\mF_{k,\resc}$ can be large in the directions where the Brenier's map $\eta$ is large in the $\mathcal C^3$ norm. Interestingly enough the coercivity constant does not depend on $\chi$, even in the limit $\chi\nearrow \chi_c(k)$.
\end{remark}

%%%%%%%%%%%%%%%%%%%%%%%%%%%%%%%
%%%%%%%%%%%%%%%%%%%%%%%%%%%%%%%
%%%%%%%%%%%%%%%%%%%%%%%%%%%%%%%
%%%%%%%%%%%%%%%%%%%%%%%%%%%%%%%
\subsection{Fast Diffusion Case $k>0$}
\label{sec:Fast Diffusion Case k pos}
%%%%%%%%%%%%%%%%%%%%%%%%%%%%%%%
%%%%%%%%%%%%%%%%%%%%%%%%%%%%%%%
%%%%%%%%%%%%%%%%%%%%%%%%%%%%%%%
%%%%%%%%%%%%%%%%%%%%%%%%%%%%%%%

Not very much is known about the fast diffusion case where $k \in (0,1)$ and hence $m=1-k\in (0,1)$, that is diffusion is fast in regions where the density of particles is low. In \cite{CCH1}, we showed that equation \eqref{eq:KS} has no radially symmetric non-increasing stationary states with $k$th moment bounded, and there are no radially symmetric non-increasing global minimisers for the energy functional $\mF_k$ for any $\chi>0$. By \cite[Theorem 2.11]{CCH1}, there exists a continuous radially symmetric non-increasing stationary state of the rescaled equation \eqref{eq:KSresc} for all $\chi>0$. In this sense, there is no criticality for the parameter $\chi$. We provide here a full proof of non-criticality by optimal transport techniques involving the analysis of the minimisation problem in rescaled variables, showing that global minimisers exist 
in the right functional spaces for all values of the critical parameter and that they are indeed stationary states - as long as diffusion is not too fast. 
More precisely, global minimisers with finite energy $\mFr$ can only exist in the range $0<k<\tfrac{2}{3}$, that is $\tfrac{1}{3}<m<1$ \cite{CCH1}. This restriction is exactly what we would expect looking at the behaviour of the fast diffusion equation ($\chi=0$) \cite{VazquezFDE}. In particular, for $k \in (0,1)$ and $m=1-k\in(0,1)$, radially symmetric non-increasing stationary states, if they exist, are integrable and have bounded $k$th moment \cite[Remarks 4.6 and 4.9]{CCH1}. By \cite[Remark 4.11]{CCH1} however, their second moment is bounded and $\rho^m\in L^1\left(\RR\right)$ if and only if  $k < 2/3$, in which case they belong to $\mY_{2}$ and their rescaled free energy is finite. This restriction corresponds to $\tfrac{1}{3}<m<1$ and coincides with the regime of the one-dimensional fast diffusion equation ($\chi=0$) where the Barenblatt profile has second moment bounded and its $m$th power is integrable \cite{BDGV}. 
Intuitively, adding attractive interaction to the dynamics helps to counteract the escape of mass to infinity. However, the quadratic confinement due to the rescaling of the fast-diffusion equation is already stronger than the additional attractive force since $k<2$ and hence, we expect that the behaviour of the tails is dominated by the non-linear diffusion effects even for $\chi>0$ as for the classical fast-diffusion equation.\\

Using completely different methods, the non-criticality of $\chi$ has also been observed in \cite{CL,CieslakLaurent10} for the limiting case in one dimension taking $m=0$, corresponding to logarithmic diffusion, and $k=1$.
The authors showed that solutions to \eqref{eq:KS} with $(m = 0,k = 1)$ are globally defined in time for all values of the parameter $\chi>0$.
\\

In order to establish equivalence between global minimisers and stationary states in one dimension, we prove a type of reversed HLS inequality providing a bound on $\int \rho^m$ in terms of the interaction term $\int (W_k \ast \rho)\rho$. The inequality gives a lower bound on the rescaled energy $\mFr$:

\begin{theorem}\label{thm:HLSmresc kpos}
  Let $k \in (0,1)$, $m=1-k$ and $\chi>0$. Then $\bar \rho \in \mY_{2,k}$
   is a stationary state of \eqref{eq:KSresc} if and only if for any solution $\rho \in \mY_{2,k}$ we have the inequality
   $$
   \mFr[\rho]\geq \mFr[\bar \rho]
   $$
   with the equality cases given by $\rho=\bar \rho$. 
\end{theorem}
The above theorem implies that stationary states in $\mY_{2,k}$ of the rescaled equation \eqref{eq:KSresc} are mimimisers of the rescaled free energy $\mFr$. Since the converse is true by \cite[Theorem 2.11]{CCH1}, it allows us to establish equivalence between stationary states of \eqref{eq:KSresc} and minimisers of $\mFr$.
To prove Theorem \ref{thm:HLSmresc kpos}, we need a result similar to Lemma \ref{lem:micro ineq resc}:

\begin{lemma} \label{lem:micro ineq resc kpos}
Let $\alpha, \beta>0$ and $m \in (0,1)$. For any $a, b \in \RR$ and any convex function $\psi:\RR \to \RR$:
\begin{equation}
\left(\alpha+\beta\right)  \bla \left(\psi_{ac}''([a,b])\right)^{1-m} \bra 
\, \leq \, \alpha \bla \psi''([a,b]) \bra^{1-m} + \frac{\beta(1-m)}{2}\, \bla \psi''([a,b]) \bra^2 + \frac{\beta(m+1)}{2}\, , \label{eq:Jensen k pos}
\end{equation}
where equality arises if and only if $\psi''\equiv  1$ a.e.
\end{lemma}
\begin{proof}
Denote $u(s):=\psi''\left([a,b]_s\right)$ with $[a,b]_s:=(1-s)a+sb$ and we write $u_{ac}$ for the absolutely continuous part of $u$.
We have by Lemma \ref{lem:interaction},
\begin{align*}
(\alpha+\beta)\bla u_{ac}^{1-m} \bra \leq (\alpha+\beta ) \bla u \bra^{1-m}\, .
\end{align*}
Further by direct inspection, 
$$
\forall X>0:\;   \frac{1}{m-1}\, X^{1-m} +\frac{1}{2} \, X^2\geq  \frac{m+1}{2(m-1)} \, ,
$$
thus
$$
(\alpha+\beta)\bla u_{ac}\bra^{1-m} \leq \alpha\bla u \bra^{1-m} +\frac{\beta(1-m)}{2}\, \bla u\bra^2+ \frac{\beta(m+1)}{2}
$$
and equality arises if and only if $u$ is almost everywhere
constant and $\bla u \bra = 1$.
\end{proof}

\begin{proof}[Proof of Theorem \ref{thm:HLSmresc kpos}.] For a stationary state $\bar \rho \in \mY_{2,k}$ and any solution $\rho \in \mY_{2,k}$ of \eqref{eq:KSresc}, there exists a convex function $\psi$ whose gradient pushes forward the measure $\bar \rho(a) da$ onto $\rho(x) dx$ 
$$\psi' \# \left(\bar\rho(a)  da\right) = \rho(x)  dx.$$ 
From characterisation \eqref{eq:charac sstates resc} we have for any $\gamma \in \RR$,
\begin{align*}
 &\int_\RR \left(\psi_{ac}''(t,a)\right)^{-\gamma} \bar \rho_k(a)^m\, da = 
 \iint_{\RR\times\RR}\left(\chi|a-b|^{1-m} +\frac{|a-b|^2}{2}\right) \bla \psi_{ac}''(t,(a,b))^{-\gamma}\bra \bar \rho_k(a) \bar \rho_k(b) \, da db \, .
\end{align*}
Choosing $\gamma=m-1$, the functional $\mFr[\rho]$ rewrites similarly to the proof of Theorem \ref{thm:HLSmresc}:
\begin{align*}
\mathcal F_{k,\resc}[\rho] 
= &\, \dfrac1{m-1}\int_\RR (\psi_{ac}''(a))^{1-m} \bar \rho(a)^m\, da \\
&\quad+ \dfrac{\chi}{1-m} \iint_{\RR\times\RR} \left(\dfrac{\psi'(a)-\psi'(b)}{a-b}\right)^{1-m}|a-b|^{1-m} \bar \rho(a) \bar \rho(b) \, da db 
\\ & \quad + \dfrac14 \iint_{\RR\times\RR} \left(\dfrac{\psi'(a) - \psi'(b)}{a-b}\right)^2 |a-b|^2 \bar \rho(a) \bar \rho(b) \, da db 
\\ =  &\, \dfrac1{m-1}\iint_{\RR\times\RR} \bla \psi_{ac}''([a,b]) ^{1-m}\bra  \left( \chi|a-b|^{1-m} + \dfrac{|a-b|^2}2 \right)\bar \rho(a) \bar \rho(b) \, da db 
\\ & \quad- \dfrac1{m-1}   \iint_{\RR\times\RR} \left( \bla \psi''([a,b]) \bra^{1-m} \chi|a-b|^{1-m}   \right.\\
&\qquad \qquad \qquad \qquad \quad \left.+ \bla \psi''([a,b]) \bra^{2} (1-m) \dfrac{|a-b|^2}{4 }\right) \bar \rho(a) \bar \rho(b) \, da db 
\end{align*}
Now, using the variant of Jensen's inequality \eqref{eq:Jensen k pos} of Lemma \ref{lem:micro ineq resc kpos}, this simplifies to
\begin{align*}
\mFr[\rho] & \geq  \dfrac{m+1}{m-1}\iint_{\RR\times\RR}    \dfrac{|a-b|^2}4 \bar \rho(a) \bar \rho(b) \, da db 
  =   \dfrac{m+1}{2(m-1)}   \int_\RR |a|^2 \bar \rho(a)\, da  =  \mathcal F_{k,\resc}[\bar \rho]\, .
\end{align*}
Here, we used identity \eqref{eq:FV resc} for the final step. Again equality holds true if and only if $\psi''$ is identically one.
\end{proof}

\begin{remark}[Sign of the Rescaled Free Energy]
 In fact, $\mFr[\bar \rho]\leq 0$. Choosing $\rho_\lambda(x)=\lambda \bar \rho(\lambda x)$ a dilation of the stationary state, we obtain thanks to the homogeneity properties of the energy functional,
 $$
 \lambda^{-k} \mU_m[\bar \rho] + \lambda^{-k} \mW_k[\bar \rho] + \lambda^{-2} \mV[\bar \rho]= \mFr[\rho_\lambda] \geq \mFr[\bar \rho],
 $$
 and so we conclude that $\mFr[\bar \rho]$ must be non-positive for any stationary state $\bar \rho \in  \mY_{2}$ by taking the limit $\lambda \to \infty$.
\end{remark}

\begin{corollary}[Uniqueness]\label{cor:!kpos}
 Let $k \in \left(0,\tfrac{2}{3}\right)$ and $m=1-k$. For any $\chi>0$, there exists a unique stationary state with second and $k$th moment bounded to equation \eqref{eq:KSresc}, and a unique minimiser for $\mFr$ in $\mY_{2,k}$.
\end{corollary}
\begin{proof}
 By \cite[Theorem 2.11]{CCH1}, there exists a minimiser of $\mFr$ in $\mY_{2,k}$, which is a stationary state of equation \eqref{eq:KSresc}. Assume \eqref{eq:KSresc} admits two stationary states $\bar \rho_1$ and $\bar \rho_2$ in $\mY_{2,k}$. By Theorem \ref{thm:HLSmresc kpos}, $\mFr[\bar \rho_1]=\mFr[\bar \rho_2]$ and so $\bar \rho_1=\bar \rho_2$.
\end{proof}

\begin{corollary}[Self-Similar Profiles]
 Let $k \in (0,1)$ and $m =1-k$. For any $\chi>0$, if $u$ is a symmetric stationary state of the rescaled equation \eqref{eq:KSresc}, then there exists a self-similar solution to \eqref{eq:KS} given by 
 $$
 \rho(t,x)=\left((2-k)t+1\right)^{\frac{1}{k-2}} \, u\left( \left((2-k)t+1\right)^{\frac{1}{k-2}} x\right).
 $$
\end{corollary}

%%%%%%%%%%%%%%%%%%%%%%%%%%%%
%%%%%%%%%%%%%%%%%%%%%%%%%%%%
%%%%%%%%%%%%%%%%%%%%%%%%%%%%
%%%%%%%%%%%%%%%%%%%%%%%%%%%%
\section{Long-Time Asymptotics}
\label{sec:LTA}
%%%%%%%%%%%%%%%%%%%%%%%%%%%%
%%%%%%%%%%%%%%%%%%%%%%%%%%%%
%%%%%%%%%%%%%%%%%%%%%%%%%%%%
%%%%%%%%%%%%%%%%%%%%%%%%%%%%

This part is devoted to the asymptotic behaviour of solutions, adapting the above computations, ensuring {\em e.g.} uniqueness of the functional ground state, at the level of the gradient flow dynamics. We will demonstrate convergence towards these ground states in Wasserstein distance under certain conditions, in some cases with an explicit rate. Our results rely on the fact that there is a simple expression for the Wasserstein distance in one dimension. Therefore, our methodology cannot be extended to dimension two or more so far except possibly under radial symmetry assumptions, which we would like to explore in future work.\\

We assume here that solutions are smooth enough so that the operations in this section are well-defined. Firstly, we require the mean-field potential gradient $\partial_x S_k (t,x)$ to be well-defined for all $t>0$ which is guaranteed if $\rho(t,x)$ has at least the same regularity at each time $t>0$ as provided by Definition \ref{def:sstates} for stationary states. From now on, we assume that solutions of \eqref{eq:KS} satisfy $\rho(t,x) \in C\left([0,T],C_{loc}^{0,\alpha}\left(\RR\right)\cap \mY \cap L^\infty\left(\RR\right)\right)$ with $\alpha \in (-k,1)$.\\
Secondly, certain computations in this section remain formal unless the convex Brenier map $\psi$ satisfying $\rho(t,x) dx = \partial_x\psi(t,x) \# \bar \rho_{k}(x) dx$ is regular enough. In the fast diffusion regime $k>0$, stationary states are everywhere positive \cite{CCH1}, and thus $\psi''$ is absolutely continuous. However, in the porous medium regime $k<0$, stationary states are compactly supported \cite{CCH1}, and therefore, the following computations remain formal depending on the regularity and properties of the solutions of the evolution problem. From now on, we assume that $\psi''$ is absolutely continuous whenever we talk about solutions of the evolution problems \eqref{eq:KS} or \eqref{eq:KSresc}.

\subsection{Porous Medium Asymptotics}

%%%%%%%%%%%%%%%%%%%%%%%%%%%%
%%%%%%%%%%%%%%%%%%%%%%%%%%%%
%%%%%%%%%%%%%%%%%%%%%%%%%%%%
%%%%%%%%%%%%%%%%%%%%%%%%%%%%
\subsubsection{The Critical Case $\chi = \chi_c(k)$}
\label{sec: LTA critical k neg}
%%%%%%%%%%%%%%%%%%%%%%%%%%%%
%%%%%%%%%%%%%%%%%%%%%%%%%%%%
%%%%%%%%%%%%%%%%%%%%%%%%%%%%
%%%%%%%%%%%%%%%%%%%%%%%%%%%%

In the critical case, the set of global minimisers coincides with the set of stationary states of equation \eqref{eq:KS} \cite[Theorem 2.8]{CCH1}, but as we will see, it is not clear whether this set is a global attractor in the Wasserstein sense or not. We will prove here a convergence result under some conditions, which provides a dynamical proof of uniqueness up to dilations. 
%%%%%%%%%%%
Recall that in the fair-competition regime, we have $\mF_k[\rho_\lambda]=\lambda^{-k} \mF_k[\rho]$ for any dilation $\rho_\lambda(x)=\lambda \rho(\lambda x)$, $\lambda \in \RR$ of a density $\rho \in \mY$, and so every stationary state provides in fact a family of stationary states by scale invariance.
%%%%%%%%%%%
Given a density $\rho \in \mY$, $|x|^2\rho(x) \in L^1_+(\RR)$, we define the rescaling $\rho_1$ by
\begin{equation}  \label{moment2norm}
\rho_{1}(x) := \sigma \rho (\sigma x)\, , \quad \sigma^2 = \mV[\rho] =  \int_{\RR} |x|^2 \rho (x)\, dx\, , 
\end{equation}
and so any stationary state $\bar \rho_k$ with finite second moment has a dilation $\bar \rho_{k,1}$ with normalised second moment $\mV[\bar \rho_{k,1}]=1$.
In particular, $\bar \rho_{k,1}$ provides a convenient representative for the family of stationary states formed by dilations of $\bar \rho_k$.
Our aim here is to show that although uniqueness is degenerate due to homogeneity, we have a unique representative $\bar \rho_{k,1}$ with second moment equal to one. We will present here a discussion of partial results and open questions around the long-time behaviour of solutions in the critical case.\\

We first recall the logarithmic case $(m = 1,k = 0)$, where the ground state is explicitly given by Cauchy's density $\bar \rho_0$ \eqref{eq:StSt log}. The second momentum is thus infinite, and the Wasserstein distance to some ground state cannot be finite if the initial datum has finite second momentum. For a $\rho(t)$ satisfying \eqref{eq:KS}, we have the estimate \cite{CaCa11}
\begin{equation*} \label{eq:convergence crit case log}
\dfrac d{dt} \bW(\rho(t),\bar \rho_0)^2 \leq 0\, ,
\end{equation*}
where equality holds if and only if $\rho(t)$ is a dilation of $\bar \rho_0$. This makes sense only if $\rho(0)$ has infinite second momentum, and is at finite distance from one of the equilibrium configurations. Notice that possible ground states (dilations of Cauchy's density) are all infinitely far from each other with respect to the Wasserstein distance,
$$
\bW\left(\rho_{\lambda_1}, \rho_{\lambda_2}\right)^2
= \frac{\left(\lambda_1-\lambda_2\right)^2}{\lambda_1\lambda_2} \mV[\bar \rho_0] =\infty.
$$
Dynamics have been described in \cite{BlaCaMa07} when the initial datum has finite second momentum: the solution converges to a Dirac mass as time goes to $+\infty$. However, this does not hold true in the porous medium case $k \in (-1,0)$, $m=1-k$, since stationary states are compactly supported by \cite[Corollary 3.9]{CCH1}. The case where the initial data is at a finite distance from some dilation of a thick-tail stationary state has been investigated in \cite{BCC12} in two dimensions. 

\begin{proposition}\label{prop:cvcritical1}
For $\chi=\chi_c(k)$, let $\rho(t)$ satisfy \eqref{eq:KS} in the porous medium case $k \in (-1,0)$ and $m=1-k$. If $\bar \rho_k$ is a stationary state of \eqref{eq:KS}, then the  evolution of the Wasserstein distance to equilibrium can be estimated by
\begin{equation} \label{eq:cvcritical}
\dfrac d{dt} \bW(\rho(t),\bar \rho_{k})^2 \leq (m-1) \mF_k[\rho(t)]\, ,
\end{equation}
where equality holds if and only if $\rho(t)$ is a dilation of $\bar \rho_k$.
\end{proposition}

\begin{proof}
Let $\phi$ be the convex Brenier map such that $\bar \rho_{k}(x) dx = \partial_x\phi(t,x) \# \rho(t,x) dx$ and denote by $\partial_x \psi(t,x)$ the reverse transport map, $\partial_x\phi(t, \partial_x \psi (t,a))=a$. Following \cite{CaCa11, Villani03} and using the regularity of $\rho(t,x)$ together with the argument as in the proof of Lemma \ref{lem:char crit} that allows for the singularity of the mean-field potential gradient to disappear, we have 
\begin{align*}
\dfrac12 \dfrac d{dt}\bW(\rho(t),\bar \rho_{k})^2 
& \leq\int_\RR (\phi'(t,x) - x)\left( \dfrac{\partial}{\partial x} \left( \dfrac m{m-1} \rho(t,x)^{m-1}\right) + 2\chi_c(k) \partial_x S_k(t,x) \right) \rho(t,x)\, dx \\
& = - \int_\RR \phi''(t,x) \rho(t,x)^m\, dx \\
&\quad + \chi_c(k) \iint_{\RR\times\RR} \left(\dfrac{\phi'(t,x) - \phi'(t,y)}{x-y}\right)|x-y|^{k}\rho(t,x)\rho(t,y)\, dxdy \\ 
& \quad + (m-1) \mF_k[\rho(t)] \\
& = - \int_\RR \left(\psi''(t,a)\right)^{-1} \left(\psi''(t,a)\right)^{1-m}   \bar \rho_{k} (a)^m\, da 
\\ & \quad + \chi_c(k) \iint_{\RR\times\RR} \left(\dfrac{\psi'(t,a) - \psi'(t,b)}{a-b}\right)^{k-1}|a-b|^{k}\bar \rho_{k} (a)\bar \rho_{k} (b)\, dadb\\
&\quad+ (m-1) \mF_k[\rho(t)]  
\end{align*}
to finally conclude that
\begin{align*}
\dfrac12 \dfrac d{dt}\bW(\rho(t),\bar \rho_{k})^2 & \leq  - \int_\RR \left(\psi''(t,a)\right)^{-m} \bar \rho_{k} (a)^m\, da 
\\ & \quad +  \chi_c(k) \iint_{\RR\times\RR}   \int_{s=0}^1 \left(\psi''(t,[a,b]_s)\right)^{-m} |a-b|^k \bar \rho_{k} (a)\bar \rho_{k} (b)\, ds da db \\
&\quad+ (m-1) \mF_k[\rho(t)]\, ,
\end{align*}
where we have crucially used the convexity of $(\cdot)^{-m}$ in the last step. We conclude as for the proof of Theorem \ref{thm:HLSmequiv} thanks to the characterisation \eqref{eq:charac sstates}.
\end{proof}

By definition of the critical value $\chi_c(k)$, the functional $\mF_k$ is everywhere non-negative. It vanishes if and only if $\rho$ is a dilation of some critical density. Therefore we cannot deduce from \eqref{eq:cvcritical} that the density $\rho(t)$ converges to some dilation of $\bar \rho_{k}$. However, we can show convergence in Wasserstein distance if we assume a rather restrictive uniform $W^{2,\infty}(\RR)$-stability estimate on the Brenier map $\psi$ connecting the solution density to the stationary state:
\begin{equation}
\psi''(t,x) \in L^\infty\left(\RR_+,L^\infty(\RR)\right)\quad \text{such that} \quad
||\psi''||_{L^\infty\left(\RR_+,L^\infty(\RR)\right)} \leq 1 + \dfrac1 m\, .
\label{eq:stabestimate1}
\end{equation}
This condition is equivalent to
\begin{equation}\forall t>0\quad \bla \psi''(t,(x,y)) \bra := \int_0^1 \psi''(t,[x,y]_s)\, ds \in \left(0,1 + \dfrac1 m\right], \, \quad \text{for a.e.}\, x,y \in \RR \, , \quad \forall t>0\, .
\label{eq:stabestimate2}
\end{equation}
where $[x,y]_s:=(1-s)x+sy$.
If we want to show convergence of a solution $\rho(t)$ to a stationary state $\bar \rho_k$ in Wasserstein distance, we need to investigate quantities that are comparable. 

\begin{proposition}\label{prop:cvcritical2}
For $\chi=\chi_c(k)$, let $\bar \rho_{k}$ be a stationary state of \eqref{eq:KS} in the porous medium case $k \in (-1,0)$, $m=1-k$. Let $\rho(t)$ be a solution such that $$\mV_\infty:=\lim_{t\to\infty} \mV[\rho(t)]<\infty\, ,$$ and we denote by $\psi$ the transport map from $\bar\rho_{k}$ onto the solution, $$\rho(t,x)dx=\partial_x \psi(t,x) \# \bar \rho_{k}(x)dx\, .$$ If $\psi$ satisfies the uniform stability estimate \eqref{eq:stabestimate1}, then
\begin{equation*} \label{eq:cvsub-critical}
\dfrac d{dt} \bW(\rho(t),\bar \rho_{k})^2 \leq 0 \, ,
\end{equation*}
where equality holds if and only if $\rho(t)$ is a dilation of $\bar \rho_{k}$. 
\end{proposition}

\begin{proof} Note that $\mV[\bar\rho_k]<\infty$ since $\bar \rho_k$ is compactly supported \cite[Corollary 3.9]{CCH1}. We compute the evolution of the Wasserstein distance along the gradient flow, denoting by $\phi$ the inverse transport map, $\partial_x\phi(t,x)=\partial_x\psi(t,x)^{-1}$, we proceed as in Proposition \ref{prop:cvcritical1}:
\begin{align*}
&\dfrac12 \dfrac d{dt}\bW(\rho(t),\bar \rho_{k})^2 \\
& \quad \quad \leq -  \int_\RR \phi''(t,x) \rho(t,x)^m\, dx +  {\chi_c(k)} \iint_{\RR\times\RR} \left(\dfrac{\phi'(t,x) - \phi'(t,y)}{x-y}\right) |x-y|^{k}\rho(t,x)\rho(t,y)\, dxdy \\ 
& \qquad \quad + \int_\RR \rho(t,x)^m\, dx -  \chi_c(k) \iint_{\RR\times\RR} |x-y|^{k}\rho(t,x)\rho(t,y)\, dxdy\, ,
\end{align*}
which we can rewrite in terms of the transport map $\psi'$ as
\begin{align*}
&\dfrac12 \dfrac d{dt}\bW(\rho(t),\bar \rho_{k})^2 \\
& \quad \quad \leq - \int_\RR \left(\psi''(t,a)\right)^{-m} \bar \rho_{k}(a)^m\, da 
+ \chi_c(k) \iint_{\RR\times\RR} \bla \psi''(t,(a,b))\bra^{-m}|a-b|^{1-m} \bar \rho_{k}(a) \bar \rho_{k}(b) \, da db \\ 
& \qquad \quad + \int_\RR \left(\psi''(t,a)\right)^{1-m} \bar \rho_{k}(a)^m\, da - \chi_c(k) \iint_{\RR\times\RR} \bla \psi''(t,(a,b))\bra^{1-m}|a-b|^{1-m} \bar \rho_{k}(a) \bar \rho_{k}(b) \, da db \, .
\end{align*}
Using the characterisation \eqref{eq:charac sstates}, we obtain for any $\gamma \in \RR$,
\begin{align*}
 &\int_\RR \left(\psi''(t,a)\right)^{-\gamma} \bar \rho_{k}(a)^m\, da = 
 \chi_c(k) \iint_{\RR\times\RR} \bla \psi''(t,(a,b))^{-\gamma}\bra |a-b|^{1-m}  \bar \rho_{k}(a) \bar \rho_{k}(b) \, da db \, .
\end{align*}
Hence, the dissipation of the distance to equilibrium can be written as
\begin{align*}
\dfrac12 \dfrac d{dt}\bW(\rho(t),\bar \rho_{k})^2  
\leq
\chi_c(k) \iint_{\RR\times\RR} |a-b|^k &\Big(- \bla \psi''(t,(a,b))^{-m} \bra + \bla \psi''(t,(a,b))^{1-m} \bra \Big.\\
&\Big. + \bla \psi''(t,(a,b)) \bra^{-m} - \bla \psi''(t,(a,b)) \bra^{1-m}  \Big) \bar \rho_{k}(a) \bar \rho_{k}(b) \, da db  \,.
\end{align*}
We now examinate the sign of the microscopic functional $J_m[u]$ defined for non-negative functions $u:(0,1)\to \RR_+$ by
\begin{align*} 
J_m[u] := - \bla u^{-m} \bra + \bla u^{1-m} \bra + \bla u \bra^{-m} - \bla u \bra^{1-m} \, . 
\label{eq:J m}
\end{align*}
The first two terms can be written as
\begin{equation*} 
- \bla u^{-m} \bra + \bla u^{1-m} \bra 
= - \alpha \bla u \bra^{-m} + \beta \bla u \bra^{1-m}  \, , 
\end{equation*}
where $\alpha = \bla u \bra^{m} \bla u^{-m} \bra$ and $\beta = \bla u \bra^{m-1}\bla u^{1-m} \bra$. By Jensen's inequality we have $\alpha\geq 1$, $\beta \geq 1$, and by interpolation we have $\beta \leq \alpha^{m/(m+1)}$. Therefore,
\begin{align*} 
&J_m[u] \leq j_m(\la u\ra ) =  \max_{\alpha\geq 1} \left\{ - \alpha \bla u \bra^{-m} + \alpha^{m/(m+1)} \bla u \bra^{1-m} \right\} + \bla u \bra^{-m} - \bla u \bra^{1-m} \, . 
\end{align*}
We can compute explicitly the maximal value in the above expression. The first order condition gives
\[  \alpha_{max} := \left(\dfrac{m}{m+1} \bla u \bra  \right)^{m+1}\, . \]
Since the function 
$$
g(\alpha):= - \alpha \bla u \bra^{-m} + \alpha^{m/(m+1)} \bla u \bra^{1-m}
$$
achieves its maximum at $\alpha_{max}\leq 1$ for $\bla u \bra \leq 1 + 1/m$ and is strictly decreasing for $\alpha>\alpha_{max}$, we have
$$
\max_{\alpha\geq 1} g(\alpha) = g(1), \qquad \text{for} \, \, \bla u \bra \leq 1 + 1/m
$$
and so we conclude $j_m(\la u\ra )=0$ for $\bla u \bra \leq 1 + 1/m$. Therefore
\begin{align*}
\dfrac12 \dfrac d{dt}\bW(\rho(t),\bar \rho_{k})^2 
& \leq \chi_c(k)
\iint_{\RR\times\RR} |a-b|^k J_m[\psi''(t,(a,b))]\bar \rho_{k}(a) \bar \rho_{k}(b) \, da db \\
&\leq \chi_c(k) \iint_{\RR\times\RR} |a-b|^k j_m[\bla\psi''(t,(a,b))\bra] 
\bar \rho_{k}(a) \bar \rho_{k}(b) \, da db =0
\end{align*}
thanks to the stability estimate \eqref{eq:stabestimate2}. 
To investigate the equality cases, note that $\beta= \alpha^{m/(m+1)}$ if and only if $u\equiv 1$ (looking at the equality cases in H\"{o}lder's inequality). Moreover, $\langle u \rangle \in (0,1+1/m]$ implies
$$
J_m[u] \leq   - \alpha \bla u \bra^{-m} + \alpha^{m/(m+1)} \bla u \bra^{1-m}  + \bla u \bra^{-m} - \bla u \bra^{1-m} \leq 0\, ,
$$
using $\alpha \geq 1$. Hence, if $J_m[u]=0$, then we must have $\beta= \alpha^{m/(m+1)}$, and so $u\equiv 1$. The converse is trivial by substituting into the expression for $J_m[u]$. Taking $u$ to be the Brenier map $\psi''$, we conclude that $\tfrac{d}{dt}\bW(\rho(t),\bar \rho_{k})^2=0$ if and only if $\rho=\bar \rho_{k}$.
\end{proof}

%%%%%%%%%%%%
%%%%%%%%%%%%
%%%%%%%%%%%%
The utility of the previous result for understanding the asymptotic behaviour of solutions depends of course on the set of initial data for which solutions satisfy the stability estimate \eqref{eq:stabestimate1} at all times. This set is rather difficult to characterise, and we do not know its size.\\

Let us now explore what we can say about the long-time behaviour of solutions in the general case. The first insight consists in calculating the evolution of the second moment. It follows from homogeneity that
\begin{equation} \label{eq:2nd mom diss}
\dfrac d{dt} \mV[\rho(t)]  = 2(m-1) \mF_k[\rho(t)]\, . 
\end{equation}
 Identity \eqref{eq:2nd mom diss} implies that the second moment is non-decreasing, and it converges to some value $\mV_\infty \in \RR_+\cup\{+\infty\}$. Following \cite{BCL} we discuss the dichotomy of $\mV_\infty< +\infty$ and $\mV_\infty= +\infty$. Let $\rho(t) \in \mY$ be a solution of \eqref{eq:KS} such that $|x|^2 \rho(t) \in L_+^1(\RR)$ for all $t>0$. Let $\bar \rho_k$ be a stationary state of \eqref{eq:KS} according to Definition \ref{def:sstates}. Note that $\mV[\bar \rho_k]<\infty$ since $\bar \rho_k$ is compactly supported by \cite[Corollary 3.9]{CCH1}.\\

\fbox{Case 1: $\mV_\infty< +\infty$} \quad
If the second moment $\mV[\rho(t)]$ converges to $\mV_\infty< +\infty$, then we deduce from \eqref{eq:2nd mom diss} that the energy functional $\mF_k[\rho(t)]$ converges to $\mF_k[\bar \rho_k] = 0$ since $\mF_k$ is non-increasing along trajectories. This is however not enough to conclude convergence of $\rho(t)$ to $\bar \rho_k$ and the question remains open.
Note further that in order to have convergence, we need to choose a dilation of $\bar \rho_k$ with second moment equal to $\mV_\infty$. For any dilation $\bar \rho_k^\lambda$ of $\bar \rho_k$, we have $\mV[\bar \rho_k^\lambda]=\mV[\bar \rho_k]/\lambda^2$, and so there exists a unique $\lambda_*$ such that $\mV[\bar\rho_k^{\lambda_*}]=\mV_\infty$. This would be the natural candidate for the asymptotic behaviour of the solution $\rho(t)$.\\

\fbox{Case 2: $\mV_\infty= +\infty$} \quad
If the second moment $\mV[\rho(t)]$ diverges to $\mV_\infty= +\infty$ however, the discussion is more subtle and we can give some further intuition. First of all, let us remark that one has to seek a convergence other than in Wasserstein distance since $\infty=\mV_\infty \neq \mV[\bar \rho_k]<\infty$. We can not exclude this case a priori however since a convergence in another sense may be possible in principle.  We use the homogeneity properties of the flow to derive refined inequalities. To do this, we renormalise the density as in \eqref{moment2norm}, but now with a time dependency in $\sigma$:
\begin{equation}  \label{moment2normt}
\hat \rho(t,y) = \sigma(t) \rho (t,\sigma(t) y)\, , \quad \sigma(t)^2 = \mV[\rho(t)] =  \int_{\RR} |x|^2 \rho(t,x)\, dx\, . 
\end{equation}
Then $\hat \rho$ satisfies the equation
\begin{align*}
\partial_t \hat \rho(t,y)  
= &\sigma(t) \partial_t \rho(t,x) + \dot \sigma(t) \left( \rho(t,x) + x \cdot \partial_x \rho(t,x)\right) \\
 = &\sigma(t) \left\{ \sigma(t)^{-2-m} \partial_{yy} \hat\rho(t,y)^m + 2 \chi_c(k) \sigma(t)^{-3+k} \partial_y \left( \hat \rho(t,y)\partial_y (W_k (y)*\hat \rho(t,y))  \right) \right\} \\
&+ \dfrac{\dot \sigma(t)}{\sigma(t)}\left(\hat \rho(t,y) + y \cdot \partial_y \hat \rho(t,y) \right)\, .
\end{align*}
By homogeneity of $\mF_k$, we have
\begin{equation}\label{Fsigma}
 \mF_k[\rho(t)] = \sigma(t)^{1-m} \mF_k[\hat \rho(t)]\, ,
\end{equation}
and so it follows from \eqref{eq:2nd mom diss} that $2 \sigma(t) \dot \sigma(t) = 2 (m-1) \mF_k[\rho(t)] = 2 (m-1) \sigma(t)^{1-m} \mF_k[\hat \rho(t)]$.
We deduce
\begin{align*}
\partial_t \hat \rho(t,y)  = &\sigma(t)^{-1-m} \left\{    \partial_{yy} \hat\rho(t,y)^m + 2 \chi_c(k)  \partial_y \left( \hat \rho(t,y)\partial_y (W_k (y)*\hat \rho(t,y))  \right) \right\}  \\
&+  \sigma(t)^{-1-m} (m-1) \mF_k[\hat \rho(t)] \left( \hat \rho(t,y) + y \cdot \partial_y \hat \rho(t,y) \right) \, .
\end{align*}
Alternatively, we get
\begin{align} 
\dfrac d{dt} \mF_k[\hat \rho(t)] 
& = \dfrac d{dt} \left\{ \sigma(t)^{m-1} \mF_k[\rho(t)]\right\} \notag\\
& = - \sigma(t)^{m-1} \int_\RR \rho(t,x) \left| \partial_x \left( \dfrac m{m-1}\rho(t,x)^{m-1} + 2 \chi_c(k) W_k (x)*\rho(t,x) \right) \right|^2\, dx \notag\\
& \quad  + (m-1)^2 \sigma(t)^{m-2}\sigma(t)^{-m} \mF_k[\hat \rho(t)]\mF_k[\rho(t)]\notag\\
& = \sigma(t)^{-1-m} \mG[\hat \rho]\, ,\label{rhohatdiss}
\end{align}
where 
\begin{equation*}\label{functionalG}
\mG[\hat \rho] := - \int_\RR \left| \partial_y \left( \dfrac m{m-1}\hat\rho(y)^{m-1} + 2 \chi_c(k) W_k (y)*\hat\rho(y) \right) \right|^2\hat\rho(y) \, dy + (m-1)^2 \mF_k[\hat \rho]^2\, . 
\end{equation*}

\begin{proposition} \label{prop:H critical}
The functional $\mH$ defined by $\mH[\rho]:=\mG[\hat \rho]$ on $\mY_2$ is zero-homogeneous, and everywhere non-positive. Moreover, $\mH[\rho] = 0$ if and only if $\rho$ is a stationary state of equation \eqref{eq:KS}.
\end{proposition}
\begin{proof}
Homogeneity follows from the very definition of $\mH$. Non-positivity is a consequence of the Cauchy-Schwarz inequality:
\begin{align}
\left|(m-1) \mF_k[\hat \rho] \right|^2
& = \left|- \int_\RR y \cdot\partial_y \left( \dfrac m{m-1}\hat\rho(y)^{m-1} + 2 \chi_c(k) W_k (y)*\hat\rho(y) \right) \hat \rho(y)\, dy  \right|^2\notag \\
& \leq \left( \int_\RR |y|^2 \hat \rho(y)\, dy \right) \left(\int_\RR \left| \partial_y \left( \dfrac m{m-1}\hat\rho(y)^{m-1} + 2 \chi_c(k) W_k (y)*\hat\rho(y) \right) \right|^2\hat\rho(y) \, dy\right)\, .\label{ineq:CS}
\end{align} 
If $\rho$ is a stationary state of equation \eqref{eq:KS}, so is $\hat \rho$ and it follows from \eqref{rhohatdiss} that $\mG[\hat \rho] = 0$. Conversely, if $\mG[\hat \rho] = 0$, then we can achieve equality in the Cauchy-Schwarz inequality \eqref{ineq:CS} above, and so the two functions $y$ and 
$$\partial_y \left( \dfrac m{m-1}\hat\rho(y)^{m-1} + 2 \chi_c(k) W_k (y)*\hat\rho(y) \right)$$ are proportional to each other. In other words, there exists a constant $\hat \pi$ such that for all $ y \in \RR$,
\begin{equation}\label{eq:hatEL}
 \partial_y \left( \dfrac m{m-1}\hat\rho(y)^{m-1} + 2 \chi_c(k) W_k (y)*\hat\rho(y) \right)  + \hat \pi y = 0\,.
\end{equation}
This equation is the Euler-Langrange condition of the gradient flow given by the energy functional $\mF_k + \hat \pi \mV$:
\begin{equation}\label{eq:pi}
 \partial_t u = \partial_y \left( u \, \, \partial_y \left(\frac{\delta}{\delta u} \left(\mF_k + \hat \pi \mV\right)[u]\right)\right)\, ,
\end{equation}
and since $\hat \rho$ satisfies \eqref{eq:hatEL}, it is a stationary state of equation \eqref{eq:pi}. Testing this equation against $y \hat \rho(y)$, we obtain
\[  \hat \pi = (m-1) \mF_k[\hat \rho] \geq 0\, . \]
Non-negativity of $\hat \pi$ follows from the variant of the HLS inequality Theorem \ref{thm:HLSm} since $\mF_k[\rho]\geq 0$ for any $\rho \in \mY$ if $\chi=\chi_c(k)$.
We will show $\hat \pi=0$ by contradiction. Assume $\hat \pi > 0$. Applying Theorem \ref{thm:HLSmresc} for $\mF_k[\cdot] + \hat \pi \mV[\cdot]$ instead of $\mF_k[\cdot] + \tfrac12 \mV[\cdot]$, we deduce that $\hat \rho$ is a minimiser of the rescaled energy $\mF_k[\cdot] + \hat \pi \mV[\cdot]$. In particular, this means that we have for any $u\in \mY_2$, 
$$ 
\mF_k[u] + \hat \pi \mV[u] \geq \mF_k[\hat \rho] + \hat \pi \mV[\hat \rho] = \hat \pi/(m-1) + \hat \pi > \hat \pi \, .
$$
However, \cite[Proposition 3.4 (i), Corollary 3.9]{CCH1} and homogeneity of $\mF_k$ provide a stationary state $\bar \rho_{k,1}\in \mY_2$ with unit second moment, which is also a global minimiser by \cite[Theorem 2.8]{CCH1}. Then choosing $u=\bar \rho_{k,1}$ in the above inequality yields $\mF_k[\bar \rho_{k,1}] + \hat \pi \mV[\bar \rho_{k,1}] = 0 + \hat \pi  $, a contradiction. Therefore we necessarily have $\hat \pi = 0$ and so $\mF_k[\hat \rho]=0$. By \eqref{Fsigma}, $\mF_k[\rho]=0$ and this implies that $\rho$ is a global minimiser of $\mF_k$ by Theorem \ref{thm:HLSm}, and consequently it is a stationary state of \eqref{eq:KS} by \cite[Theorem 2.8]{CCH1}.
\end{proof}

It would be desirable to be able to show that $\mH[\rho(t)] \to \mH[\bar \rho_{k,1}]$ as $t \to \infty$ to make appropriate use of the new energy functional $\mH$. But even then, similar to the first case, we are lacking a stability  result for $\mH$ to prove that in fact $\hat \rho(t)$ converges to $\bar \rho_{k,1}$. Here, in addition, we do not know at which rate the second moment goes to $+\infty$. \\

We conjecture that only the first case $\mV_\infty< +\infty$ is admissible. The motivation for this claim is the following: $\mF$ and $\mH$ have both constant signs, and vanish only when $\hat \rho = \bar \rho_{k,1}$. If the stability inequality
\begin{equation}\label{eq:stabest2}
\eta \mF_k[\hat \rho] \leq -\mH[\rho],  \quad \forall \rho
\end{equation}
were satisfied for some $\eta>0$, then we would be able to prove that $\mV_\infty< +\infty$. To see this, we derive a second-order differential inequality for $\omega(t) := \sigma(t)^{m+1}$. We have
\begin{align*}
\dot \omega(t) & = (m+1) \sigma(t)^{m}\dot \sigma(t)= (m+1)(m-1) \mF_k[\hat \rho(t)] \geq 0\, ,
\end{align*}
and so by \eqref{rhohatdiss},
\begin{align*}
\ddot\omega(t) & = (m+1)(m-1)  \omega(t)^{-1} \mH[\rho(t)] \leq 0\, .
\end{align*}
Here, non-positivity of $\ddot\omega(t)$ follows from Proposition \ref{prop:H critical}.
Therefore, the stability estimate \eqref{eq:stabest2}, if true, would imply that $\ddot \omega(t)  \leq - \eta \omega(t)^{-1} \dot \omega(t)$, hence 
$$
\dot \omega(t) \leq C - \eta \log \omega(t).
$$
Consequently, $\omega(t)$ would be bounded, and so we arrive at a contradiction with the assumption $\mV_\infty=+\infty$.

%%%%%%%%%%%%%%%%%%%%%%%%%%%%
%%%%%%%%%%%%%%%%%%%%%%%%%%%%
%%%%%%%%%%%%%%%%%%%%%%%%%%%%
%%%%%%%%%%%%%%%%%%%%%%%%%%%%
\subsubsection{The sub-critical case $\chi < \chi_c$}
\label{sec: LTA sub-critical k neg}
%%%%%%%%%%%%%%%%%%%%%%%%%%%%
%%%%%%%%%%%%%%%%%%%%%%%%%%%%
%%%%%%%%%%%%%%%%%%%%%%%%%%%%
%%%%%%%%%%%%%%%%%%%%%%%%%%%%

We know that in the logarithmic case ($m=1, k=0$), solutions to \eqref{eq:KS} converge exponentially fast towards a unique self-similar profile as $t\to \infty$, provided that the parameter $\chi$ is sub-critical ($\chi <1$) \cite{CaCa11}. A similar argument works in the porous medium regime $k \in (-1,0)$ under certain regularity assumptions as we will show below. Surprisingly enough, convergence is uniform as the rate of convergence does not depend on the parameter $\chi$. In particular, it was shown in \cite{CaCa11} for $k=0$ that we have uniform convergence in Wasserstein distance of any solution $\rho(t)$ for the rescaled system \eqref{eq:KSresc} to the equilibrium distribution $\bar \rho_0$ of \eqref{eq:KSresc},
\begin{align*}
\dfrac d{dt}\bW(\rho(t),\bar \rho_0)^2 \leq - 2 \bW(\rho(t),\bar \rho_0)^2\,  .
\end{align*}
A similar result has been obtained in two dimension in \cite{CaDo14}.\\

Studying the long-time behaviour of the system in the porous medium case $k<0$ is more subtle than the logarithmic case and we cannot deduce exponentially fast convergence from our calculations without assuming a uniform stability estimate, which coincides with \eqref{eq:stabestimate2}. But as in the critical case, we do not know how many initial data actually satisfy this condition. Note also that due to the additional confining potential, homogeneity has been broken, and so we cannot renormalise the second moment of minimisers as we did in the critical case. As in the critical case, stationary states of the rescaled equation \eqref{eq:KSresc} are compactly supported by \cite[Corollary 3.9]{CCH1}.

\begin{proposition}\label{prop:cvsub-critical}
For sub-critical interaction strength $0<\chi<\chi_c(k)$, let $\rho(t)$ be a solution to \eqref{eq:KSresc} in the porous medium case $k \in (-1,0)$, $m=1-k$ 
and $\bar \rho_k$ a stationary state of \eqref{eq:KSresc}. If the transport map $\psi$ given by $\rho(t,x)dx=\partial_x \psi(t,x) \# \bar \rho_k(x)dx$ satisfies the uniform stability estimate \eqref{eq:stabestimate1}, then
\begin{equation*} \label{eq:cvsub-critical}
\dfrac d{dt} \bW(\rho(t),\bar \rho_k)^2 \leq -2 \bW(\rho(t),\bar \rho_k)^2 \, ,
\end{equation*}
where equality holds if and only if $\rho(t)$ is a dilation of $\bar \rho_k$. It follows that $$\lim_{t \to \infty} \mV[\rho(t)] = \mV[\bar \rho_k]\, .$$
\end{proposition}
%%%%%%%%%%%%%%%%%%Proof
\begin{proof} We compute the evolution of the Wasserstein distance along the gradient flow similar to the proof of Proposition \ref{prop:cvcritical2}, denoting by $\phi$ the inverse transport map, $\partial_x\phi(t,x)=\partial_x\psi(t,x)^{-1}$,
\begin{align*}
&\dfrac12 \dfrac d{dt}\bW(\rho(t),\bar \rho_k)^2 \\
& \leq -  \int_\RR \phi''(t,x) \rho(t,x)^m\, dx +  {\chi} \iint_{\RR\times\RR} \left(\dfrac{\phi'(t,x) - \phi'(t,y)}{x-y}\right) |x-y|^{k}\rho(t,x)\rho(t,y)\, dxdy \\ 
&  \quad + \int_\RR \rho(t,x)^m\, dx -  \chi \iint_{\RR\times\RR} |x-y|^{k}\rho(t,x)\rho(t,y)\, dxdy  \\
& \quad   + \dfrac 1 2 \iint_{\RR\times\RR} (\phi'(t,x) - \phi'(t,y))(x-y)\rho(t,x)\rho(t,y)\, dxdy  - \int_\RR |x|^2\rho(t,x)\, dx \, , 
\end{align*}
where we have used the fact that the centre of mass is zero at all times to double the variables:
\[\int_{\RR} \phi'(t,x) x\rho(t,x) \, dx = \dfrac12 \iint_{\RR\times\RR} (\phi'(t,x) - \phi'(t,y))(x-y)\rho(t,x)\rho(t,y)\, dxdy\, . \]
This rewrites as follows in terms of the transport map $\psi'$:
\begin{align*}
&\dfrac12 \dfrac d{dt}\bW(\rho(t),\bar \rho_k)^2 \\
& \leq - \int_\RR \left(\psi''(t,a)\right)^{-m} \bar \rho_k(a)^m\, da + \chi \iint_{\RR\times\RR} \bla \psi''(t,(a,b))\bra^{-m}|a-b|^{1-m} \bar \rho(a) \bar \rho_k(b) \, da db \\ 
& \quad + \int_\RR \left(\psi''(t,a)\right)^{1-m} \bar \rho_k(a)^m\, da - \chi \iint_{\RR\times\RR} \bla \psi''(t,(a,b))\bra^{1-m}|a-b|^{1-m} \bar \rho_k(a) \bar \rho_k(b) \, da db \\ 
&  \quad + \frac12 \iint_{\RR\times\RR} \bla \psi''(t,(a,b))\bra |a-b|^{2} \bar \rho_k(a) \bar \rho_k(b) \, da db  \\
& \quad - \frac12 \iint_{\RR\times\RR} \bla \psi''(t,(a,b))\bra^{2}|a-b|^{2} \bar \rho_k(a) \bar \rho_k(b) \, da db  \, .
\end{align*}
Using the characterisation \eqref{eq:charac sstates resc}, we obtain for any $\gamma \in \RR$,
\begin{align*}
 \int_\RR &\left(\psi''(t,a)\right)^{-\gamma} \bar \rho_k(a)^m\, da \\
 &= 
 \iint_{\RR\times\RR}\left(\chi|a-b|^{1-m} +\frac{|a-b|^2}{2}\right) \bla \psi''(t,(a,b))^{-\gamma}\bra \bar \rho_k(a) \bar \rho_k(b) \, da db \, .
\end{align*}
Hence, the dissipation of the distance to equilibrium can be written as
\begin{align*}
\dfrac12 \dfrac d{dt}&\bW(\rho(t),\bar \rho_k)^2 \\
& \leq
\chi \iint_{\RR\times\RR} |a-b|^k 
\left\{ - \bla \psi''(t,(a,b))^{-m} \bra + \bla \psi''(t,(a,b))^{1-m} \bra \right.\\
&\hspace{3.15 cm}+ \left. \bla \psi''(t,(a,b)) \bra^{-m} - \bla \psi''(t,(a,b)) \bra^{1-m}  \right\} \bar \rho_k(a) \bar \rho_k(b) \, da db \\ 
& \quad + \dfrac12  \iint_{\RR\times\RR} |a-b|^2 
\left\{ - \bla \psi''(t,(a,b))^{-m} \bra + \bla \psi''(t,(a,b))^{1-m} \bra \right.\\
&\hspace{3.15 cm}+ \left.  \bla \psi''(t,(a,b)) \bra  - \bla \psi''(t,(a,b)) \bra^{2}  \right\} \bar \rho_k(a) \bar \rho_k(b) \, da db \, .
\end{align*}
We now examinate the signs of the microscopic functionals $J_m[u]$ and $J_{m,2}[u]$ defined as follows for non-negative functions $u:(0,1)\to \RR_+$,
\begin{align} 
&J_m[u] := - \bla u^{-m} \bra + \bla u^{1-m} \bra + \bla u \bra^{-m} - \bla u \bra^{1-m} \, , 
\label{eq:J m resc} \\
&J_{m,2}[u] := - \bla u^{-m} \bra + \bla u^{1-m} \bra + \bla u \bra - \bla u \bra^{2} \, . \label{eq:J m 2 resc}
\end{align}
The first two terms in the functionals $J_m$ and $J_{m,2}$ are common. We can rewrite them as 
\begin{equation*} 
- \bla u^{-m} \bra + \bla u^{1-m} \bra 
= - \alpha \bla u \bra^{-m} + \beta \bla u \bra^{1-m}  \, , 
\end{equation*}
where $\alpha = \bla u \bra^{m} \bla u^{-m} \bra$ and $\beta = \bla u \bra^{m-1}\bla u^{1-m} \bra$. By Jensen's inequality we have $\alpha\geq 1$, $\beta \geq 1$, and by interpolation we have $\beta \leq \alpha^{m/(m+1)}$. Therefore,
\begin{align*} 
&J_m[u] \leq j_m(\la u\ra ) :=  \max_{\alpha\geq 1} g(\alpha) + \bla u \bra^{-m} - \bla u \bra^{1-m} \, , 
\\
&J_{m,2}[u] \leq j_{m,2}(\la u\ra ) :=  \max_{\alpha\geq 1} g(\alpha) + \bla u \bra - \bla u \bra^{2} \, , \end{align*}
where 
$$
g(\alpha):= - \alpha \bla u \bra^{-m} + \alpha^{m/(m+1)} \bla u \bra^{1-m}\, .
$$
We can compute explicitly the maximal value of $g$, and as before the first order condition gives
\[  \alpha_{max} = \left(\dfrac{m}{m+1} \bla u \bra  \right)^{m+1}\, . \]
It is straight forward to see that
$$
\max_{\alpha\geq 1} g(\alpha) = g(1) \qquad \text{for} \, \, \bla u \bra \leq 1 + 1/m\, ,
$$
and hence we obtain
\begin{align} 
& j_m(\la u\ra )  =   \begin{cases} \label{jm}
 0, &\quad \mathrm{if}\quad \bla u \bra \leq 1 + \dfrac{1}{m} \medskip \\
 \left(\dfrac{m}{m+1}  \right)^{m} \dfrac1{m+1}  \bla u \bra   + \bla u \bra^{-m} - \bla u \bra^{1-m}, &\quad \mathrm{if}\quad \bla u \bra \geq 1+ \dfrac{1}{m}  
\end{cases} \, , \\
& j_{m,2}(\la u\ra )  =   \begin{cases} \label{jm2}
 -  \bla u \bra^{-m} +   \bla u \bra^{1-m} + \bla u \bra - \bla u \bra^{2},  &\quad \mathrm{if}\quad \bla u \bra \leq 1 + \dfrac{1}{m} \medskip \\
 \left(\dfrac{m}{m+1}  \right)^{m} \dfrac1{m+1}  \bla u \bra   + \bla u \bra - \bla u \bra^{2}, &\quad \mathrm{if}\quad \bla u \bra \geq 1+ \dfrac{1}{m}  
\end{cases} \, .
\end{align}
We have $\lim_{+\infty} j_{m} = +\infty$, and $\lim_{+\infty} j_{m,2} = -\infty$. In addition, the function $j_{2,m}$ is non-positive and uniformly strictly concave: 
\begin{align*}
\forall \la u\ra \in \left(0,1 + \dfrac{1}m\right]\quad   j_{m,2}''(\la u\ra) & = m \la u\ra^{-m-2} \left( -(m+1) + (m-1)\la u\ra \right) -2 \\
& \leq - (m+1) \la u\ra^{-m-2} - 2
\, .
\end{align*}
Thus, $\forall \la u\ra \in \RR_+$,    $j_{m,2}''(\la u\ra) \leq -2$ and so the following coercivity estimate holds true:
\begin{equation}\label{eq:coercivity}
\forall \la u\ra \in \left(0,1 + \dfrac1 m\right]\, ,
\quad j_{m,2}(\la u \ra) \leq -  \left(\bla u \bra - 1\right)^2 \, .
\end{equation}
Furthermore, the function $j_m$ is everywhere non-negative.
The above analysis allows us to rewrite the dissipation in Wasserstein distance as
\begin{align*}
\dfrac12 \dfrac d{dt}\bW(\rho(t),\bar \rho_k)^2 & \leq
\iint_{\RR\times\RR} \chi|a-b|^k J_m[\psi''(t,(a,b))] \bar \rho_k(a) \bar \rho_k(b) \, da db \\
&\quad+ \dfrac12\iint_{\RR\times\RR}  |a-b|^2  J_{m,2}[\psi''(t,(a,b))] \bar \rho_k(a) \bar \rho_k(b) \, da db \\
&\leq \iint_{\RR\times\RR}  \chi|a-b|^k j_m[\bla\psi''(t,(a,b))\bra] \bar \rho_k(a) \bar \rho_k(b) \, da db \\
&\quad + \dfrac12 \iint_{\RR\times\RR} |a-b|^2  j_{m,2}[\bla\psi''(t,(a,b))\bra]\bar \rho_k(a) \bar \rho_k(b) \, da db 
\end{align*}
to finally conclude that
\begin{align*}
\dfrac12 \dfrac d{dt}\bW(\rho(t),\bar \rho_k)^2
& \leq - \dfrac12 \iint_{\RR\times\RR} |a-b|^2 \left( \bla \psi''(t,(a,b)) \bra- 1\right)^2  \, \bar \rho_k(a)\bar \rho_k(b) \, dadb,
\end{align*}
where the last inequality follows from \eqref{jm} and the coercivity property \eqref{eq:coercivity} thanks to the stability estimate \eqref{eq:stabestimate2}. This concludes the proof,
\begin{align*}
\dfrac d{dt}\bW(\rho(t),\bar \rho_k)^2 
&\leq -\iint_{\RR\times\RR} |a-b|^2 \left( \bla \psi''(t,(a,b)) \bra- 1\right)^2  \, \bar \rho_k(a)\bar \rho_k(b) \, dadb\\
&= - \iint_{\RR\times\RR} \left( \psi'(a)-a - \left(\psi'(b)-b\right)\right)^2  \, \bar \rho_k(a)\bar \rho_k(b) \, dadb,\\
&= -2  \int_{\RR} \left( \psi'(a)-a\right)^2  \, \bar \rho_k(a) \, da,
= -2\,  \bW(\rho(t),\bar \rho_k)^2,
\end{align*}
using the fact that $\rho(t)$ and $\bar \rho_k$ both have zero centre of mass.
\end{proof}

 \begin{remark}[Non-Existence of Stationary States]\label{no sstates}
  Proposition \ref{prop:cvsub-critical} motivates the rescaling in the sub-critical case since it means that there are no stationary states in original variables. Indeed, assume $\bar u$ is a stationary states of equation \eqref{eq:KS}, then its rescaling $\rho(t,x)=e^t\bar u(e^t x)$ is a solution to \eqref{eq:KSresc} and converges to $\delta_0$ as $t \to \infty$. \cite[Proposition 3.4 (ii)]{CCH1} on the other hand provides a stationary state $\bar \rho_k$, and the transport map $\partial_x\psi(t,x)$ pushing forward $\bar \rho_k$ onto $\rho(t,x)$ can be written as $\psi(t,x)=e^{-t}\phi(x)$ for some convex function $\phi$. Hence, for large enough $t>0$, $\psi(t,x)$ satisfies the stability estimate \eqref{eq:stabestimate1} and so eventually $\rho(t,x)$ converges to $\bar \rho_k$ by Proposition \ref{prop:cvsub-critical} which is not possible.
 \end{remark}

 %%%%%%%%%%%%%%%%%%%%%%%%%%%%
%%%%%%%%%%%%%%%%%%%%%%%%%%%%
%%%%%%%%%%%%%%%%%%%%%%%%%%%%
%%%%%%%%%%%%%%%%%%%%%%%%%%%%
\subsubsection{The super-critical case $\chi > \chi_c$}
\label{sec: LTA super-critical k neg}
%%%%%%%%%%%%%%%%%%%%%%%%%%%%
%%%%%%%%%%%%%%%%%%%%%%%%%%%%
%%%%%%%%%%%%%%%%%%%%%%%%%%%%
%%%%%%%%%%%%%%%%%%%%%%%%%%%%

Here, we investigate the possible blow-up dynamics of the solution in the super-critical case. In contrast to the logarithmic case $(m=1,k=0)$, for which all solutions blow-up when $\chi>\chi_c$, provided the second momentum is initially finite, see \cite{BlaDoPe06}, the picture is not so clear in the fair-competition regime with negative homogeneity $k<0$. There, the key identity is \eqref{eq:2nd mom diss}, which states in particular that the second momentum is a concave function.\\

It has been observed in \cite{BCL} that if the free energy is negative for some time $t_0$, $\mF_k[\rho(t_0)]<0$, then the second momentum is a decreasing concave function for $t>t_0$. So, it cannot remain non-negative for all time. Necessarily, the solution blows up in finite time. Whether or not the free energy could remain non-negative for all time was left open.
In \cite{Yao}, the author proved that solutions blow-up without condition on the sign of the free energy at initial time, but for the special case of the Newtonian potential, for which comparison principles are at hand.\\
In \cite{CG}, a continuous time, finite dimensional, Lagrangian numerical scheme of \cite{BCC08} was analysed. This scheme preserves the gradient flow structure of the equation. It was proven that, except for a finite number of values of $\chi$, the free energy necessarily becomes negative after finite time. Thus, blow-up seems to be a generic feature of \eqref{eq:KS} in the super-critical case.
However, we could not extend the proof of \cite{CG} to the continuous case for two reasons: firstly, we lack compactness estimates, secondly, the set of values of $\chi$ to be excluded gets dense as the number of particles in the Lagrangian discretisation goes to $\infty$.\\

Below, we transpose the analysis of \cite{CG} to the continuous level. We highlight the missing pieces.
Let us define the renormalised density $\hat \rho$ as in \eqref{moment2normt}. 
The following statement is the analogue of Proposition \ref{prop:H critical} in the super-critical case.

\begin{proposition}\label{prop:H super-critical}
The functional $\mH$ defined by $\mH[\rho]:=\mG[\hat \rho]$ on $\mY_2$ is zero-homogeneous, and everywhere non-positive. Moreover, it cannot vanish in the cone of non-negative energy:
\begin{equation}\label{eq:sign H}
\left( \mathcal F[\rho] \geq 0 \right) \Longrightarrow  \left( \mathcal H[\rho] < 0 \right)\, .
\end{equation}
\end{proposition}
%%%%%%%%%%%%%%%%%%%%%%%%%%%%%%%%%%%%%%%PROOF
\begin{proof}
We proceed as in the proof of Proposition \ref{prop:H critical}. Zero-homogeneity follows from the definition of $\mH$, and non-positivity is a direct consequence of the Cauchy-Schwarz inequality. It remains to show \eqref{eq:sign H}. Assume that $\rho$ is such that $\mathcal F[\rho] \geq 0$ and $\mathcal H[\rho] = 0 $. The latter condition ensures that there exists a constant $\hat \pi$ such that $\hat\rho$ is a critical point of the energy functional  $\mF + \hat \pi \mV$:
\begin{equation*}%\label{eq:hatEL supercrit}
 \partial_y \left( \dfrac m{m-1}\hat\rho(y)^{m-1} + 2 \chi W_k(y)*\hat\rho(y) \right)  + \hat \pi y = 0\,.
\end{equation*}
Testing this equation against $y\hat \rho(y)$, we obtain
\[  \hat \pi = (m-1) \mF_k[\hat \rho]=  (m-1) \sigma(t)^{m-1}\mF_k[ \rho]\geq 0\, . \]
Applying as in the proof of Proposition \ref{prop:H critical} a variant of Theorem \ref{thm:HLSmresc}, we obtain that $\hat \rho$ is a global minimiser of the energy functional  $\mF + \hat \pi \mV$. Here, the amplitude of the confinement potential $\hat \pi$  plays no role, but the sign $\hat \pi\geq 0$ is crucial. By \cite[Theorem 2.8]{CCH1}, there exists a stationary state $\bar \rho \in \mY_2$ for critical interaction strength $\chi=\chi_c(k)$. If $\chi>\chi_c(k)$, we have $\mF_k[\bar \rho]=\mU_m[\bar \rho]+\chi\mW_k[\bar \rho]<\mU_m[\bar \rho]+\chi_c(k)\mW_k[\bar \rho]=0$. Taking mass-preserving dilations of $\bar \rho$, we see immediately that the functional $\mF + \hat \pi \mV$ is not bounded below in the super-critical case. This is a contradiction with $\hat \rho$ being a minimiser. Hence, $\mH[\rho]<0$ and \eqref{eq:sign H} holds true.
\end{proof}
%%%%%%%%%%%%%%%%%%%%

As in Section \ref{sec: LTA critical k neg}, the following non-linear function of the second momentum,
\[ \omega(t) = \sigma(t)^{m+1}=\left( \int_{\RR} |x|^2 \rho(t,x)\, dx \right)^{\frac{m+1}2}\, ,  \]
satisfies the second order differential inequality,
\begin{equation}\label{eq:ddot omega} 
\ddot\omega(t) = (m^2-1) \omega(t)^{-1} \mH[ \rho(t)] \leq 0\, .
\end{equation}
In view of the property \eqref{eq:sign H} of the zero-homogeneous functional $\mathcal H$, it seems natural to ask whether there exists a positive constant $\delta>0$, such that
\begin{equation}\label{eq:conjecture H} \left( \mathcal F[\rho] \geq 0 \right) \Longrightarrow  \left( \mathcal H[\rho] < - \delta\right)\, .
 \end{equation}
If this would be the case, then \eqref{eq:ddot omega} could be processed as follows: assume that $\dot\omega(t)\geq 0$ for all $t$. This is equivalent to say that the free energy remains non-negative for all $t\geq 0$ using \eqref{eq:2nd mom diss}. Hence, assuming \eqref{eq:conjecture H} holds, \eqref{eq:ddot omega} becomes 
\begin{equation}\label{eq:ddot omega estimate}
 \ddot\omega(t) < -\delta(m^2-1) \omega(t)^{-1} <0.
\end{equation}
Multiplying by $\dot\omega(t)\geq 0$, and integrating between $0$ and $T$, we would get
\[ \frac12\dot\omega(T)^2 + \delta (m^2-1) \log \left( \omega(T)\right) \leq \frac12\dot\omega(0)^2 + \delta (m^2-1) \log \left( \omega(0)\right)\, . \]
Hence, for any $t>0$, 
\[ \omega(t) \leq \omega(0)\exp\left( \dfrac{\dot\omega(0)^2}{2\delta (m^2-1) } \right)\, . \]
Back to estimate \eqref{eq:ddot omega estimate}, we would conclude that $\omega$ is uniformly concave,
\[ \ddot\omega(t) \leq - \left(\frac{\delta(m^2-1)}{\omega(0)} \right) \exp\left( - \dfrac{\dot\omega(0)^2}{2\delta (m^2-1)} \right) <0\, . \]
Therefore, $\tfrac{d}{dt}\mV[\rho(t)]$ would become negative in finite time. This would be a contradiction with the everywhere non-negativity of the free energy by \eqref{eq:2nd mom diss}.
As a conclusion, the existence of positive $\delta>0$ as in \eqref{eq:conjecture H} implies unconditional blow-up. In \cite{CG}, existence of such $\delta$ is proven for a finite dimensional Lagrangian discretisation of $\mF_k$, and accordingly $\mathcal H$, except for a finite set of values for $\chi$.
Numerical simulations using the numerical scheme proposed in \cite{BCC08} clearly show that the energy has the tendency to become negative, even for positive initial data. Proving \eqref{eq:conjecture H} remains an open problem.

%%%%%%%%%%%%%%%%%%%%%%%%%%%
%%%%%%%%%%%%%%%%%%%%%%%%%%%
%%%%%%%%%%%%%%%%%%%%%%%%%%%
\subsection{Fast Diffusion Asymptotics}
\label{sec:LTA kpos}
%%%%%%%%%%%%%%%%%%%%%%%%%%%
%%%%%%%%%%%%%%%%%%%%%%%%%%%
%%%%%%%%%%%%%%%%%%%%%%%%%%%

In the fast diffusion case $k>0$, we are able to show a much stronger result: every stationary state of \eqref{eq:KSresc} is in fact a global attractor for any choice of interaction strength $\chi>0$. Investigating the evolution of the Wasserstein distance to equilibrium yields exponential convergence  with an explicit rate which is independent of the interaction strength $\chi>0$. In contrast to the porous medium case, where we required a stability estimate on Brenier's map, we do not need such an estimate here. As a consequence, we obtain an alternative proof of uniqueness of stationary states by a dynamical argument.

\begin{proposition}[Long-time asymptotics]\label{prop:cvinW2}
For $k \in (0,1)$ and $m=1-k$, if $\rho(t)$ has zero centre of mass initially and satisfies \eqref{eq:KSresc}, then the  evolution of the Wasserstein distance to the stationary states $\bar \rho_k$ of \eqref{eq:KSresc} can be estimated by
\begin{equation} \label{eq:cvsub-critical}
\dfrac d{dt} \bW(\rho(t),\bar \rho_k)^2 \leq -2 \bW(\rho(t),\bar \rho_k)^2
\end{equation}
for any interaction strength $\chi>0$. As a consequence, stationary states are unique if they exist.
\end{proposition}

%%%%%%%%%%%%%%%%%%

\begin{proof}%[Proof of Proposition \ref{prop:cvinW2}]
We compute the evolution of the Wasserstein distance along the gradient flow, denoting by $\phi$ the inverse transport map, $\partial_x\phi(t,x)=\partial_x\psi(t,x)^{-1}$. Proceeding as in the proof of Proposition \ref{prop:cvsub-critical}, we can write the dissipation of the distance to equilibrium as 
\begin{align*}
\dfrac12 \dfrac d{dt}\bW(\rho(t),\bar \rho_k)^2 & \leq
\chi \iint_{\RR\times\RR} |a-b|^k \left\{ - \bla \psi''(t,(a,b))^{-m} \bra + \bla \psi''(t,(a,b))^{1-m} \bra \right. \\
&\hspace{3.15 cm}\left. +\bla \psi''(t,(a,b)) \bra^{-m} - \bla \psi''(t,(a,b)) \bra^{1-m}  \right\} \bar \rho_k(a) \bar \rho_k(b) \, da db \\ 
& \quad + \dfrac12  \iint_{\RR\times\RR} |a-b|^2 \left\{ - \bla \psi''(t,(a,b))^{-m} \bra + \bla \psi''(t,(a,b))^{1-m} \bra \right. \\
&\hspace{3.45 cm}\left. +\bla \psi''(t,(a,b)) \bra  - \bla \psi''(t,(a,b)) \bra^{2}  \right\} \bar \rho_k(a) \bar \rho_k(b) \, da db \, .
\end{align*}
We now examine the signs of the microscopic functionals $J_m[u]$ and $J_{m,2}[u]$ defined as in \eqref{eq:J m resc} and \eqref{eq:J m 2 resc} for non-negative functions $u:(0,1)\to \RR_+$ by
\begin{align*} 
&J_m[u]: = - \bla u^{-m} \bra + \bla u^{1-m} \bra + \bla u \bra^{-m} - \bla u \bra^{1-m} \, , \\
&J_{m,2}[u] := - \bla u^{-m} \bra + \bla u^{1-m} \bra + \bla u \bra - \bla u \bra^{2} \, . 
\end{align*}
However, since $m<1$ we now have by convexity $\bla u \bra^{-m}- \bla u^{-m} \bra \leq 0$ and $ \bla u^{1-m} \bra  - \bla u \bra^{1-m} \leq 0$, hence
\begin{equation}\label{Jmbound}
J_m[u]\leq 0, \quad m\in(0,1).
\end{equation}
For the functional $J_{m,2}$, the first two terms can be written as 
\begin{equation*} 
- \bla u^{-m} \bra + \bla u^{1-m} \bra 
= - \alpha \bla u \bra^{-m} + \beta \bla u \bra^{1-m}  \, , 
\end{equation*}
where $\alpha = \bla u \bra^{m} \bla u^{-m} \bra$ and $\beta = \bla u \bra^{m-1}\bla u^{1-m} \bra$. As opposed to the proof of Proposition \ref{prop:cvsub-critical}, we now have $\beta \leq 1 \leq \alpha$ by Jensen's inequality since $m<1$, and therefore,
\begin{align*} 
\forall \la u\ra \in \RR_+, \quad
&J_{m,2}[u] \leq j_{m,2}(\la u\ra ) := - \bla u \bra^{-m} +  \bla u \bra^{1-m} + \bla u \bra - \bla u \bra^{2} \, .
\end{align*}
Note that $\lim_{+\infty} j_{m,2} = -\infty$. In addition, the function $j_{2,m}$ is non-positive and uniformly strictly concave: 
\begin{align*}
\forall \la u\ra \in \RR_+, \quad   j_{m,2}''(\la u\ra) & = -m (1+m)
\la u\ra^{-m-2} -m(1-m)\la u\ra^{-m-1} -2 \leq -2\, ,
\end{align*}
and hence
\begin{align}\label{jm2bound}
\forall \la u\ra \in \RR_+, \quad   j_{m,2}(\la u\ra) & \leq -\left(\la u\ra -1\right)^2\,.
\end{align}
From these estimates, we can deduce the exponential speed of convergence for the stationary state $\bar \rho_k$ by rewriting the dissipation to equilibrium as
\begin{align*}
\dfrac12 \dfrac d{dt}\bW(\rho(t),\bar \rho_k)^2 & \leq
\iint_{\RR\times\RR}  \chi|a-b|^k J_m[\psi''(t,(a,b))] \bar \rho_k(a) \bar \rho_k(b) \, da db \\
& \quad+
\iint_{\RR\times\RR} 
\dfrac12  |a-b|^2  J_{m,2}[\psi''(t,(a,b))]\bar \rho_k(a) \bar \rho_k(b) \, da db \\
&\leq \iint_{\RR\times\RR} \dfrac12  |a-b|^2  j_{m,2}[\bla\psi''(t,(a,b))\bra]\bar \rho_k(a) \bar \rho_k(b) \, da db \\
& \leq - \dfrac12 \iint_{\RR\times\RR} |a-b|^2 \left( \bla \psi''(t,(a,b)) \bra- 1\right)^2  \, \bar \rho_k(a)\bar \rho_k(b) \, dadb,
\end{align*}
where the last inequality follows from \eqref{Jmbound} and \eqref{jm2bound}. This concludes the proof,
\begin{align*}
\dfrac d{dt}\bW(\rho(t),\bar \rho_k)^2 
&\leq -\iint_{\RR\times\RR} |a-b|^2 \left( \bla \psi''(t,(a,b)) \bra- 1\right)^2  \, \bar \rho_k(a)\bar \rho_k(b) \, dadb\\
&= - \iint_{\RR\times\RR} \left( \psi'(a)-a - \left(\psi'(b)-b\right)\right)^2  \, \bar \rho_k(a)\bar \rho_k(b) \, dadb,\\
&= -2  \int_{\RR} \left( \psi'(a)-a\right)^2  \, \bar \rho_k(a) \, da,
= -2\,  \bW(\rho(t),\bar \rho_k)^2,
\end{align*}
using the fact that $\rho(t)$ and $\bar \rho_k$ both have zero centre of mass.
\end{proof}
%%%%%%%%%%%%%%%%%%%%%%%%%%%%
%%%%%%%%%%%%%%%%%%%%%%%%%%%%
\begin{remark}[Non-Existence of Stationary States]\label{rmk:noE}
 This result also provides a dynamical proof for the non-existence of stationary states for $k \in (0,2/3)$ in original variables. Indeed, if $\bar u$ were a stationary state of equation \eqref{eq:KS}, then its rescaled density $\rho(t,x)$ would converge to $\delta_0$ for large times. This contradicts the existence of a stationary state in rescaled variables \cite[Theorem 4.10]{CCH1} for $k \in (0,2/3)$ together with exponential convergence to equilibrium Proposition \ref{prop:cvinW2}.
\end{remark}

%%%%%%%%%%%%%%%%%%%%%%%%%%%%%%%%
%%%%%%%%%%%%%%%%%%%%%%%%%%%%%%%
%%%%%%%%%%%%%%%%%%%%%%%%%%%%%%%
\section{Numerical Simulations}
\label{sec:numerics}
%%%%%%%%%%%%%%%%%%%%%%%%%%%%%%%
%%%%%%%%%%%%%%%%%%%%%%%%%%%%%%%
%%%%%%%%%%%%%%%%%%%%%%%%%%%%%%%
There exists an illuminating way to rewrite the energy functional
$\mathcal F_k[\rho]$ due to the particular form of the transport map.
We use the Lagrangian transformation $\rho \mapsto  X$, where $X:(0,1)\to \RR$ denotes the pseudo-inverse of the cumulative distribution function (cdf) associated with $\rho$ \cite{Villani03, Gosse-Toscani, BCC08, CaCa11},
$$
X(\eta)=F^{-1}(\eta):=\inf\left\{x\, :\, F(x)\geq \eta \right\},
\quad F(x):=\int_{-\infty}^x \rho(y)\, dy\, .
$$
We introduce the parameter $r \in \left\{0,1\right\}$ as we are interested in both original ($r=0$) and rescaled ($r=1$) variables.
Integrating equations \eqref{eq:KS} and \eqref{eq:KSresc} over $(-\infty,X(t,\eta))$ with respect to the space variable yields
\begin{equation}\label{eq:newKS1}
 \partial_t \int_{-\infty}^{X(t,\eta)}\rho(t,y)\,dy
 =\left.\left[\partial_x \rho^m+2\chi\rho\partial_x \left(W_k\ast \rho\right) +r x \rho\right]\right|_{x=X(t,\eta)}\, .
\end{equation}
Differentiating the identity $F(t,X(t,\eta))=\eta$ with respect to $\eta$ twice yields 
$$
\rho(t,X(t,\eta))=\left(\partial_\eta X(t,\eta)\right)^{-1}\qquad \mbox{and} \qquad
\partial_x\rho(t,X(t,\eta))=-\partial_{\eta \eta}X(t,\eta)/\left(\partial_\eta X(t,\eta)\right)^3\,.
$$ 
Differentiating with respect to time, we obtain $\partial_t F(t,X(t,\eta))=-\partial_t X(t,\eta)/\partial_\eta X(t,\eta)$. This allows us to simplify \eqref{eq:newKS1},
\begin{equation*}
 \partial_tX(t,\eta)=-\partial\eta\left(\left(\partial_\eta X(t,\eta)\right)^{-m}\right)
 -2\chi \int_0^1 \left|X(t,\eta)-X(t,\tilde \eta)\right|^{k-2}\left(X(t,\eta)-X(t,\tilde \eta)\right)\, d\tilde \eta -r X(t,\eta)\,.
\end{equation*}
Similarly, the functionals $\mG_{k,0}:=\mF_k$ and $\mG_{k,1}:=\mFr$ read equivalently
\begin{align*}
\mG_{k,r}[X]= \frac{1}{m-1}\int_0^1(\partial_\eta X(\eta))^{1-m}\, d \eta 
+ \chi \int_0^1\int_0^1 \frac{|X(\eta)-X(\tilde \eta)|^k}{k}\, d \eta d \tilde \eta
+ \frac{r}{2} \int_0^1 |X(\eta)|^2\, d \eta\, .
\end{align*}
for $k\in (-1,1)\backslash\{0\}$, and
\begin{equation*}
\mG_{0,r}[X] = - \int_0^1 \log \left(\frac{dX}{d\eta}(\eta)\right)\,
d\eta + \chi\int_0^1 \!\! \int_0^1 \log |X(\eta) -
X(\tilde\eta)|\, d\eta d\tilde\eta
+ \frac{r}{2} \int_0^1 |X(\eta)|^2\, d \eta\, .
\end{equation*}
in the logarithmic case $k=0$. 
Intuitively, $X$ encodes the position of particles with respect to the partial mass $\eta \in (0,1)$, and the same homogeneity is preserved: $\mG_{k,0}[\lambda X]=\lambda^{k}\mG_{k,0}[X]$.

In Section \ref{sec: Functional inequalities}, we showed uniqueness of minimisers of the rescaled energy functional $\mFr[\rho]$ for $0<k<2/3$ and any $\chi>0$ (Corollary \ref{cor:!kpos}) and also for the sub-critical porous medium case $-1<k<0$, $\chi<\chi_c(k)$ (Corollary \ref{cor:!subcrit}). One may take these results as an indication that $\mFr[\rho]$ could in fact be displacement convex. As discussed in Section \ref{sec:Optimal Transport Tools}, $\mFr[\rho]$ is a sum of displacement convex and concave contributions and we do not know its overall convexity properties. We recall that the functionals related to the classical Keller-Segel models in two dimensions are displacement convex once restricted to bounded densities \cite{CLM14}. We will give some heuristics for the power-law potential case. If $\mG_{k,1}[X]$ were convex, then $\mFr[\rho]$ would be displacement convex \cite{Villani03, Carrillo-Slepcev09} and uniqueness of minimisers directly follows \cite{McCann97}. Taylor explanding $\mG_{k,1}$ around $X$ yields for any test function $\varphi \in C_c^\infty\left([0,1]\right)$,
$$
\mG_{k,1}[X+\epsilon \varphi]=
\mG_{k,1}[X] + \epsilon D_\varphi \mG_{k,1} [X]+\frac{\epsilon^2}{2} D^2_\varphi \mG_{k,1}[X] + O(\epsilon^3),
$$
where $D_\varphi \mG_{k,1}[X]= \int_0^1 \delta \mG_{k,1} [X](\eta)\, \varphi(\eta)\, d \eta$ with the first variation $\frac{\delta \mG_{k,1}}{\delta X}[X](\eta)$ given by
$$
\frac{\delta \mG_{k,1}}{\delta X}[X](\eta) =
 \partial_\eta \left( \left(\partial_\eta X\right)^{-m}\right)
+2\chi \int_0^1 |X(\eta)-X(\tilde \eta)|^{k-2} \left(X(\eta)-X(\tilde \eta)\right)\, d \tilde \eta + X(\eta)
$$
for $k \in (-1,1)/\{0\}$. However, the Hessian
\begin{align*}
D^2_\varphi \mG_{k,1}[X]
= &m\int_0^1 \left(\partial_\eta\varphi(\eta)\right)^2\left(\partial_\eta X(\eta)\right)^{-(m+1)}\, d \eta\\
+ &\chi (k-1) \, \int_0^1 \int_0^1 \left|X(\eta)-X(\tilde \eta)\right|^{k-2} \left(\varphi(\eta)-\varphi(\tilde \eta)\right)^2\,d\eta d \tilde \eta+ \int_0^1 \varphi(\eta)^2\,d \eta
\end{align*}
does not have a sign.
In other words, we cannot use this strategy to conclude overall convexity/concavity properties of the rescaled energy functional $\mFr$. It is an interesting problem to explore convexity properties of $\mG_{k,r}$ in a restricted set of densities such as bounded densities as in \cite{CLM14, Craig16}.\\ 

%%%%%%%%%%%%%%%%%%%%%%%%%%%%%%%%%%%%%%%%%
%%%%%%%%%%%%%%%%%%%%%%%%%%%%%%%%%%%%%%%%%
\subsection{Numerical Scheme} \label{Numerical Scheme}
%%%%%%%%%%%%%%%%%%%%%%%%%%%%%%%%%%%%%%%%
%%%%%%%%%%%%%%%%%%%%%%%%%%%%%%%%%%%%%%%%

To simulate the dynamics of $X$ we use a numerical scheme which was proposed in \cite{BCC08, CG} for the logarithmic case, and generalised to the one-dimensional fair-competition regime for the porous medium case $k\in(-1,0)$ in \cite{CGnew}. It can easily be extended to rescaled variables adding a confining potential, and works just in the same way in the fast diffusion case $k\in(0,1)$. We discretise the energy functional via a finite difference approximation of $X(\eta)$ on a regular grid. If $\left(X_i\right)_{1\leq i\leq n}$ are the positions of $n$ ordered particles sharing equal mass $\Delta \eta = 1/n$ such that $X_1<X_2<\cdots<X_n$, then we define the discretised energy functional by 
\begin{equation*}\label{discF}
 \mG_{k,r}^n\left[(X_i)\right]=\frac{\left(\Delta\eta\right)^{m}}{m-1}\sum_{i=1}^{n-1} \left(X_{i+1}-X_i\right)^{1-m}
+\chi\left(\Delta\eta\right)^{2}\sum_{1\leq i \neq j \leq n} \frac{\left|X_j-X_i\right|^k}{k}
+ r\frac{\Delta\eta}{2} \sum_{i=1}^n |X_i|^2
\end{equation*}
for $k \in (-1,1)\backslash\{0\}$, and by
\begin{equation*}\label{discFlog}
 \mG_{0,r}^n\left[(X_i)\right]=-\Delta\eta\sum_{i=1}^{n-1} \log\left(\frac{X_{i+1}-X_i}{\Delta\eta}\right)
+\chi\left(\Delta\eta\right)^2\sum_{1\leq i \neq j \leq n} \log\left|X_j-X_i\right|
+ r\frac{\Delta\eta}{2} \sum_{i=1}^n |X_i|^2
\end{equation*}
in the logarithmic case $k=0$. The Euclidean gradient flow of $\mG_{k,r}^n$ writes for $ 1<i<n$
\begin{align}\label{scheme1}
 \dot{X_i}=&-(\Delta\eta)^{m-1}\left(\left(X_{i+1}-X_i\right)^{-m} - \left(X_i-X_{i-1}\right)^{-m}\right)\notag\\
 &-2\chi\Delta\eta\sum_{1\leq j \neq i \leq n} \sign(i-j)\left|X_i-X_j\right|^{k-1}-rX_i\, ,
\end{align}
complemented with the dynamics of the extremal points
\begin{align}
 \dot{X_1}&=-(\Delta\eta)^{m-1}\left(X_2-X_1\right)^{-m}+2\chi\Delta \eta\sum_{j \neq 1}\left|X_j-X_1\right|^{k-1}-rX_1\label{scheme2}\, ,\\
 \dot{X_n}&=(\Delta\eta)^{m-1}\left(X_n-X_{n-1}\right)^{-m}-2\chi\Delta \eta\sum_{j \neq n}\left|X_j-X_n\right|^{k-1}-rX_n\label{scheme3}\, .
\end{align}
Equations \eqref{scheme2}-\eqref{scheme3} follow from imposing $X_0=-\infty$ and $X_{n+1}=+\infty$ so that the initial centre of mass $\sum_{i=1}^n X_i=0$ is conserved.
Working with the pseudo-inverse of the cummulative distribution function of $\rho$ also has the advantage that we can express the Wasserstein distance between two densities $\rho$ and $\tilde \rho$ in a more tractable way. More precisely, if $\psi'$ is the optimal map which transports $\tilde \rho$ onto $\rho$, then the Monge-Amp\'{e}re equation \eqref{eq:MongeAmpere2} is an increasing rearrangement. Let $F$ and $\tilde F$ be the cummulative distribution function of $\rho$ and $\tilde \rho$ respectively, with pseudo-inverses $X$ and $\tilde X$. Then we have
$$
\tilde F(x)=\int_{-\infty}^x \tilde \rho(y)\, dy
= \int_{-\infty}^{\psi'(x)} \rho(y)\, dy
=F \circ \psi'(x)\, .
$$
Hence the transport map is given explicitly by $\psi'=F^{-1} \circ \tilde F$, and we have for the Wasserstein distance
\begin{equation}\label{WW}
 \bW(\rho,\tilde \rho)^2 =\int_0^1 \left|\tilde F^{-1}(\eta)-F^{-1}(\eta)\right|^2 \, d \eta
 =\int_0^1 \left|\tilde X(\eta)-X(\eta)\right|^2 \, d \eta\, 
 = ||\tilde X-X||_2^2\, .
\end{equation}
This means that this numerical scheme can be viewed formally as the time discretisation of the abstract gradient flow equation \eqref{eq:gradient flow} in the Wasserstein-2 metric space, which corresponds to a gradient flow in $L^2\left( (0,1)\right)$ for the pseudo-inverse $X$,
$$
\dot{X}(t)=-\nabla_{L^2} \mG_{k,r}[X(t)]\, .
$$
Discretising \eqref{scheme1}-\eqref{scheme2}-\eqref{scheme3} by an implicit in time Euler scheme, this numerical scheme then coincides with a Jordan-Kinderlehrer-Otto (JKO) steepest descent scheme (see \cite{Otto, BCC08} and references therein). The solution at each time step of the non-linear system of equations is obtained by an iterative Newton-Raphson procedure.

%%%%%%%%%%%%%%%%%%%%%%%%%%%%%%%%%%
%%%%%%%%%%%%%%%%%%%%%%%%%%%%%%%%%%
\subsection{Results}
%%%%%%%%%%%%%%%%%%%%%%%%%%%%%%%%%
%%%%%%%%%%%%%%%%%%%%%%%%%%%%%%%%%
For the logarithmic case $k=0$, $m=1$, we know that the critical interaction strength is given by $\chi_c=1$ separating the blow-up regime from the regime where self-similar solutions exist \cite{DoPe04, BlaDoPe06, BKLN06}. As shown in \cite{CCH1}, there is no critical interaction strength for the fast diffusion regime $k>0$, however the dichotomy appears in the porous medium regime $k<0$ \cite{BCL,CCH1}. It is not known how to compute the critical parameter $\chi_c(k)$ explicitly for $k<0$, however, we can make use of the numerical scheme described in Section \ref{Numerical Scheme} to compute $\chi_c(k)$ numerically. 

% \vspace{-0.4cm}
%%%%%%%%%%%%%%%%%%%%%%%%%%%%%%%%%%%%%%%%%%%%
%%%%%%%%%%%%%%%%%%%%%%%%%%%%%%%%%%%%%%%%%%%%
%%%%fig2
\begin{figure}[h!]
\centering
\includegraphics[width=.8\textwidth]{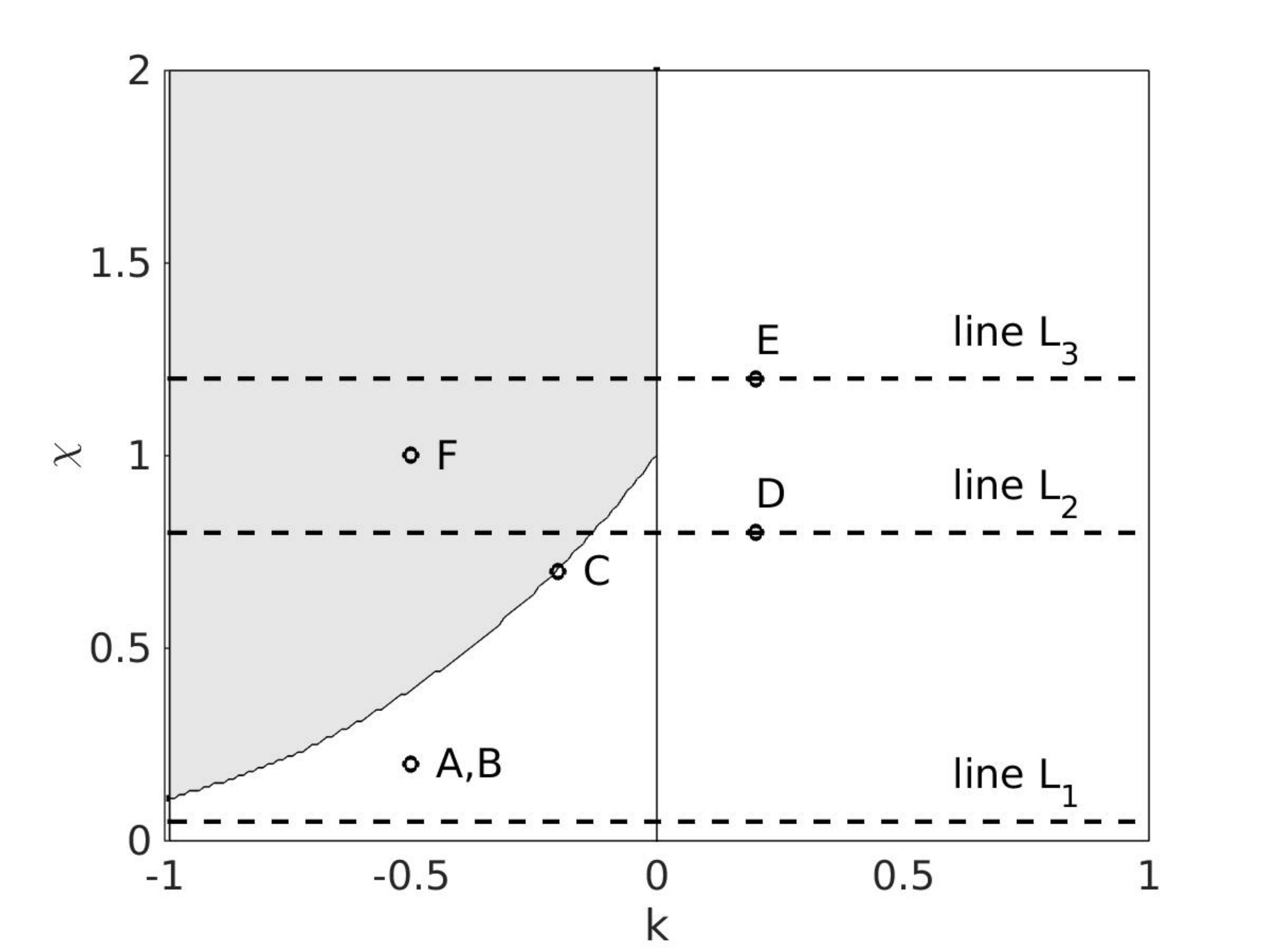}
\caption{\textbf{Regions of blow-up (grey) and convergence to self-similarity (white).} 
The notation refers to subsequent figures as follows: Lines $L_1$, $L_2$ and $L_3$ show the asymptotic profiles over the range $k\in(-1,1)$ for $\chi=0.05$, $\chi=0.8$ and $\chi=1.2$ respectively (Figure \ref{fig:criticalk_chi}). Point $A$ shows the density evolution at $(k,\chi)=(-0.5,0.2)$ in original variables (Figure  \ref{fig:chi=02_k=-05_resc=0}), and Point $B$ for the same choice of parameters $(k,\chi)=(-0.5,0.2)$ in rescaled variables (Figure \ref{fig:chi=02_k=-05_resc=1}). Points $C$, $D$ and $E$ correspond to simulations at $(-0.2,0.7)$ (Figure \ref{fig:chi=07_k=-02_resc=1}), $(0.2,0.8)$ (Figure \ref{fig:chi=08_k=02_resc=1}) and $(0.2,1.2)$ (Figure \ref{fig:chi=12_k=02_resc=1}) respectively in the parameter space $(k,\chi)$, all in rescaled variables. Point $F$ corresponds to simulations at $(k,\chi)=(-0.5,1.0)$ in original variables (Figure \ref{fig:chi=10_k=-05_resc=0}).}
\label{fig:parameterspace}
\end{figure}
%%%%%%%%%%%%%%%%%%%%%%%%%%%%%%%%%%%%%%%%%%
%%%%%%%%%%%%%%%%%%%%%%%%%%%%%%%%%%%%%%%%%%
% \vspace{-0.5cm}
Figure \ref{fig:parameterspace} gives an overview of the behaviour of solutions. In the grey region, we observe finite-time blow-up of solutions, whereas for a choice of $(k,\chi)$ in the white region, solutions converge exponentially fast to a unique self-similar profile. The critical regime is characterised by the black line $\chi_c(k)$, $-1<k\leq 0$, separating the grey from the white region. Note that numerically we have $\chi_c(-0.99)=0.11$ and $\chi_c(0)=1$. Figure \ref{fig:parameterspace} has been created by solving the rescaled equation \eqref{eq:KSresc} using the numerical scheme described above with particles equally spaced at a distance $\Delta \eta=10^{-2}$. 
For all choices of $k \in (-1,0)$ and $\chi \in (0,1.5)$, we choose as initial condition a centered normalised Gaussian with variance $\sigma^2=0.32$, from where we let the solution evolve with time steps of size $\Delta t=10^{-3}$.
We terminate the time evolution of the density distribution if one of the following two conditions is fullfilled: either the $L^2$-error between two consecutive solutions is less than a certain tolerance (i.e. we consider that the solution converged to a stationary state), or the Newton-Raphson procedure does not converge for $\rho(t,x)$ at some time $t<t_{max}$ because the mass is too concentrated (i.e. the solution sufficiently approached a Dirac Delta to assume blow-up). We choose $t_{max}$ large enough, and $\Delta \eta$ and $\Delta t$ small enough so that one of the two cases occurs. For Figure \ref{fig:parameterspace}, we set the maximal time to $t_{max}=10$ and the tolerance to $10^{-5}$. For a fixed $k$, we start with $\chi=0.01$ and increase the interaction strength by $0.01$ each run until $\chi=1.5$. This is repeated for each $k$ from $-0.99$ to $0$ in $0.01$ steps. For a given $k$, the numerical critical interaction strength $\chi_c(k)$ is defined to be the largest $\chi$ for which the numerical solution can be computed without blow-up until the $L^2$-error between two consecutive solutions is less than the specified tolerance. In what follows, we investigate the behaviour of solutions in more detail for chosen points in the parameter space Figure \ref{fig:parameterspace}.

%%%%%%%%%%%%%%%%%%%%%%%%%%%%%%%%%%%%%%%%%%%%%%
%%%%%%%%%%%%%%%%%%%%%%%%%%%%%%%%%%%%%%%%%%%%%%
%%fig2
\begin{figure}[h!]
\centering
\subfloat[]{\includegraphics[width=.5\textwidth]{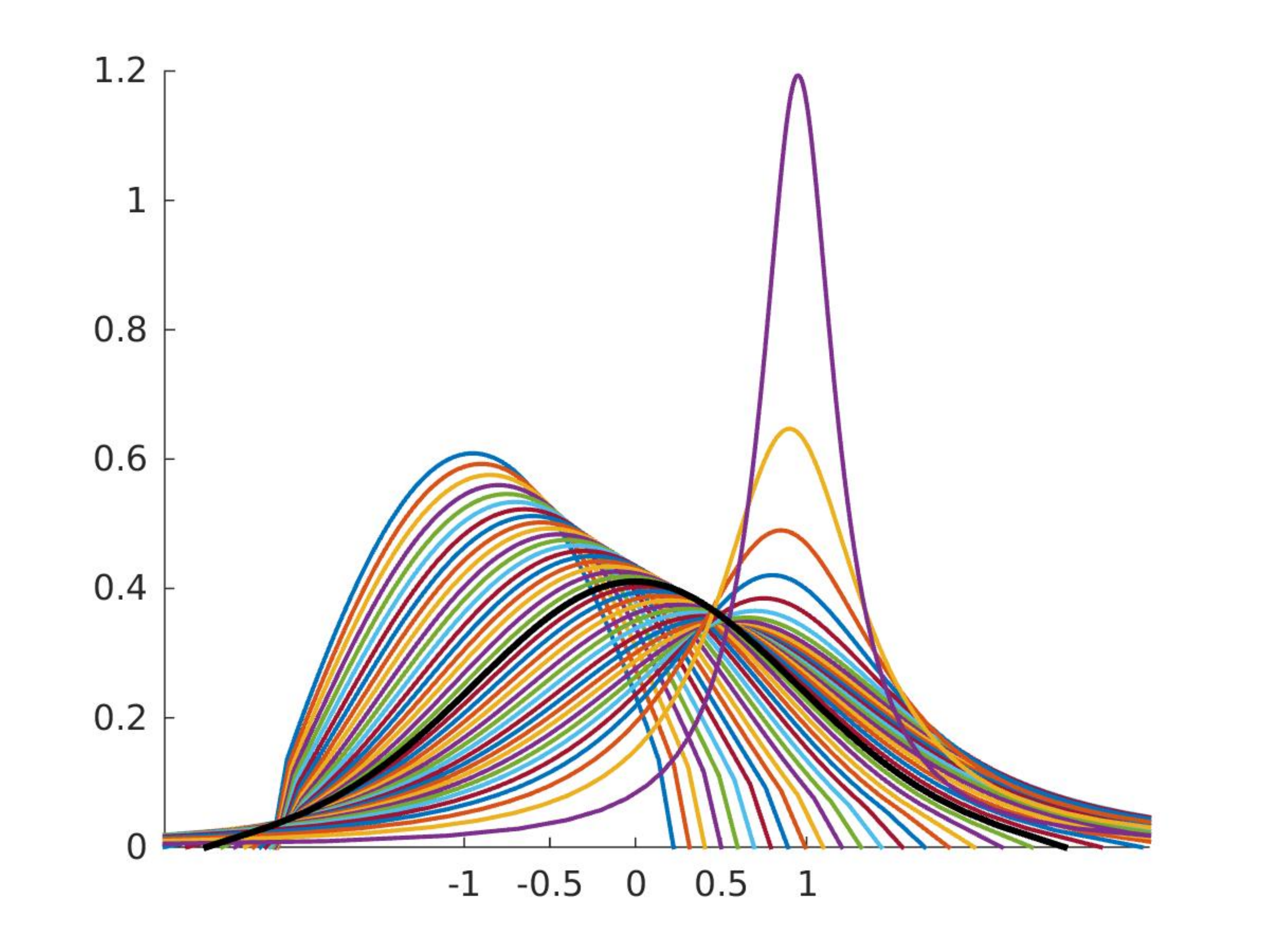}}%
\subfloat[]{\includegraphics[width=.5\textwidth]{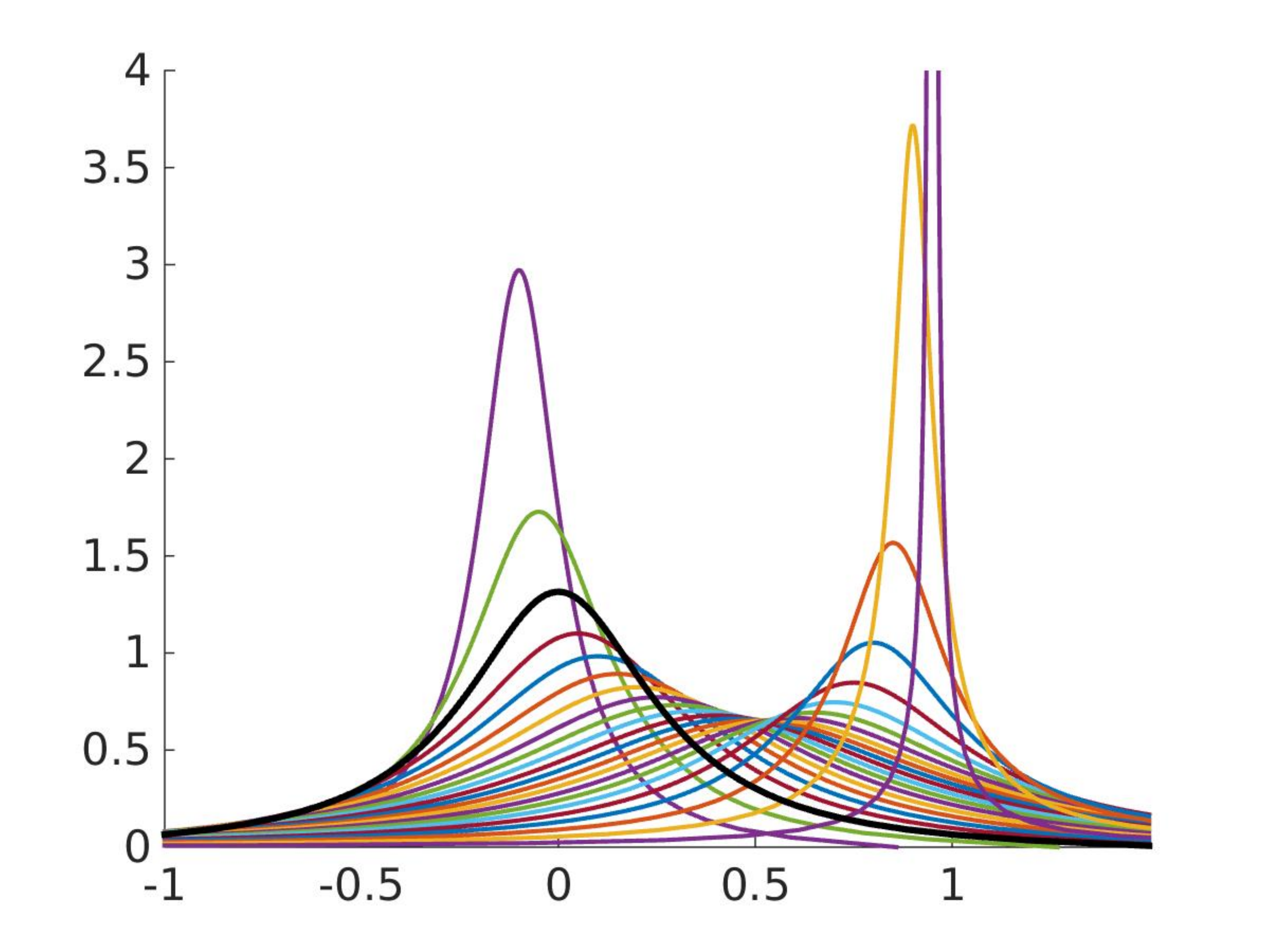}}%fig2a

\subfloat[]{\includegraphics[width=.5\textwidth]{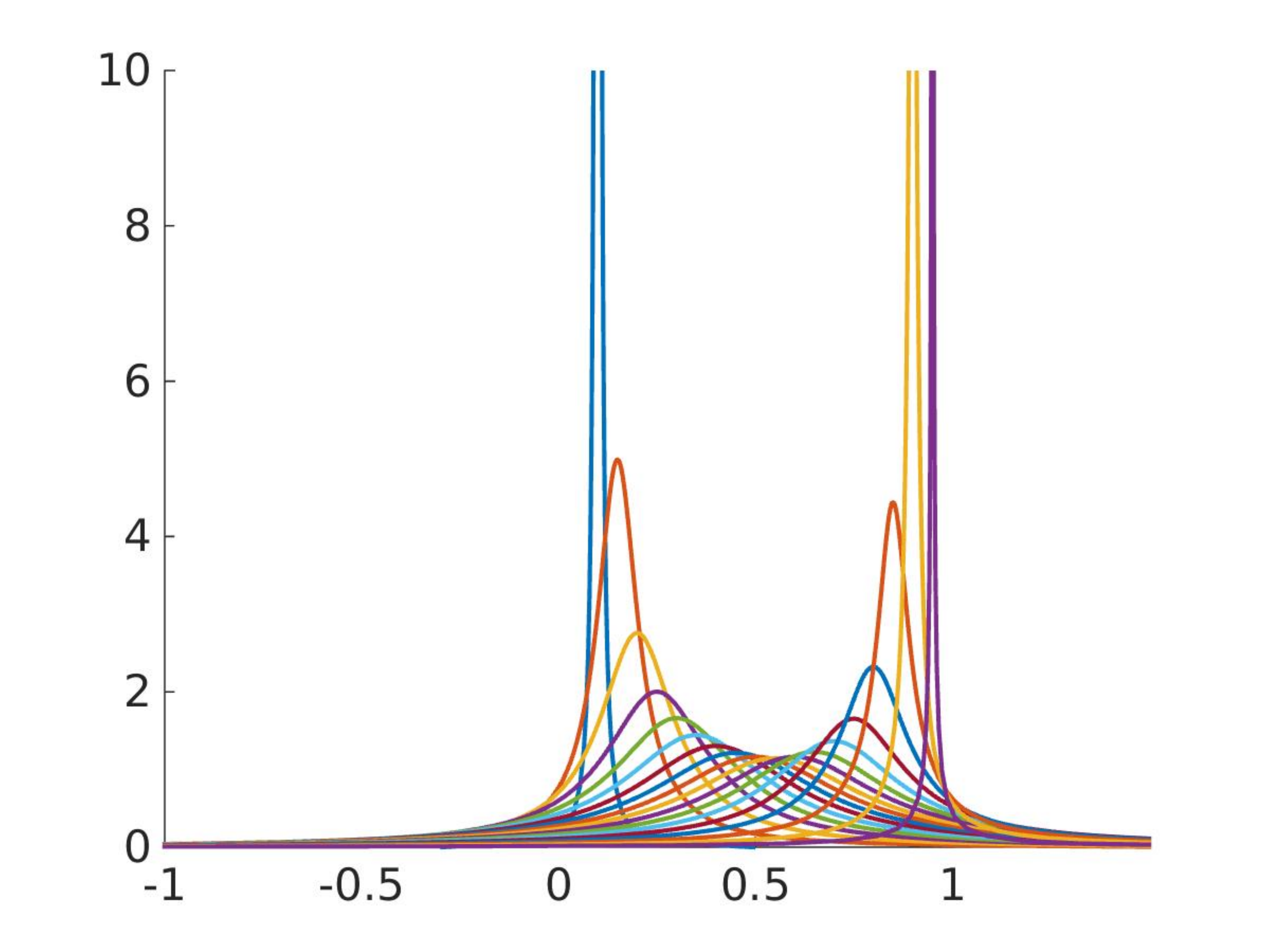}}%
\vspace{-0.2 cm}
\caption{Profiles of stationary states in rescaled variables ($r=1$) corresponding to lines $L_1$, $L_2$ and $L_3$ in Figure \ref{fig:parameterspace} for (a) $\chi=0.05$, (b) $\chi=0.8$ and (c) $\chi=1.2$ with $k$ ranging from $0.95$ to (a) $-0.95$,  (b) $-0.1$ and (c) $0.1$ in $0.05$ steps respectively. All stationary states are centered at zero, but are here displayed shifted so that they are centered at their corresponding value of $k$. The black curve indicates the stationary state for $k=0$.}\label{fig:criticalk_chi}
\end{figure}

\subsubsection{Lines $L_1$, $L_2$ and $L_3$}

Apart from points $A-F$ shown in Figure \ref{fig:parameterspace}, it is also interesting to observe how the asymptotic profile changes more globally as we move through the parameter space. To this purpose, we choose three different values of $\chi$ and investigate how the stationary profile in rescaled variables changes with $k$. Three representative choices of interaction strengths are given by lines $L_1$, $L_2$ and $L_3$ as indicated in Figure \ref{fig:parameterspace}, where $L_1$ corresponds to $\chi=0.05$ and lies entirely in the self-similarity region (white), $L_2$ corresponds to $\chi=0.8$ and captures part of the sub-critical region in the porous medium regime $k<0$ (white), as well as some of the blow-up regime (grey), and finally line $L_3$ which corresponds to $\chi=1.2$ and therefore captures the jump from the self-similarity (white) to the blow-up region (grey) at $k=0$. Note also that points $D$ and $E$ are chosen to lie on lines $L_2$ and $L_3$ respectively as to give a more detailed view of the behaviour on these two lines for the same $k$-value. The asymptotic profiles over the range $k \in (-1,1)$ for lines $L_1$, $L_2$ and $L_3$ are shown in Figure \ref{fig:criticalk_chi}, all with the same choice of parameters using time step size $\Delta t=10^{-3}$ and equally spaced particles at distance $\Delta \eta=10^{-2}$.

For each choice of interaction strength $\chi$, we start with $k=0.95$ and decrease $k$ in $0.05$ steps for each simulation either until $k=-0.95$ is reached, or until blow-up occurs and $(k,\chi)$ lies within the grey region. For each simulation, we choose as initial condition the stationary state of the previous $k$-value (starting with a centered normalised Gaussian distribution with variance $\sigma^2=0.32$ for $k=0.95$). As for Figure \ref{fig:parameterspace}, we terminate the time evolution of the density distribution for a given choice of $k$ and $\chi$ if either the $L^2$-error between two consecutive solutions is less than the tolerance $10^{-5}$, or the Newton-Raphson procedure does not converge. 
All stationary states are centered at zero. To better display how the profile changes for different choices of $k$, we shift each stationary state in Figure \ref{fig:criticalk_chi} so that it is centered at the corresponding value of $k$. The black curve indicates the stationary profile for $k=0$.

In Figure \ref{fig:criticalk_chi}(a), we observe corners close to the edge of the support of the stationary profiles for $k<0$. This could be avoided by taking $\Delta \eta$ and $\Delta t$ smaller, which we chose not to do here, firstly to be consistent with Figure \ref{fig:parameterspace} and secondly to avoid excessive computation times. 
For interaction strength $\chi=0.8$, the smallest $k$ for which the solution converges numerically to a stationary state is $k=-0.1$ (see Figure \ref{fig:criticalk_chi}(b)). This fits with what is predicted by the critical curve $\chi_c(k)$ in Figure \ref{fig:parameterspace} (line $L_2$).

In Figures \ref{fig:criticalk_chi}(b) and \ref{fig:criticalk_chi}(c), we see that the stationary profiles become more and more concentrated for $k$ approaching the critical parameter $k=k^*$ with $\chi=\chi_c(k^*)$, which is to be expected as we know that the stationary state $\bar \rho_k$ converges to a Dirac Delta as $k$ approaches the blow-up region. In fact, for $\chi=1.2$ the numerical scheme stops converging for $k=0.05$ already since the mass is too concentrated, and so we only display profiles up to $k=0.1$ in Figure \ref{fig:criticalk_chi}(c). Further, in all three cases $\chi=0.05$, $\chi=0.8$ and $\chi=1.2$ we observe that the stationary profiles become more and more concentrated as $k \to 1$. This reflects the fact that attractive forces dominate as the diffusivity $m$ converges to zero. Finally, note that we have chosen here to show only a part of the full picture for Figures \ref{fig:criticalk_chi}(b) and \ref{fig:criticalk_chi}(c), cutting the upper part. More precisely, the maximum of the stationary state for $k=0.95$ and $\chi=0.8$ in Figure \ref{fig:criticalk_chi}(b) lies at $75.7474$, whereas it is at $3,216.8$ for parameter choices $k=0.95$ and $\chi=1.2$ shown in Figure \ref{fig:criticalk_chi}(c).

%%%%%%%%%%%%%%%%%%%%%%%%%%%%%%%%%%%%%%%%%%%%%%%
%%%%%%%%%%%%%%%%%%%%%%%%%%%%%%%%%%%%%%%%%%%%%%%
%%%% chi=02_k=-05_resc=0_run1
%%%% fig6
\begin{figure}[h!]
\centering
\subfloat[]{\includegraphics[width=.45\textwidth]{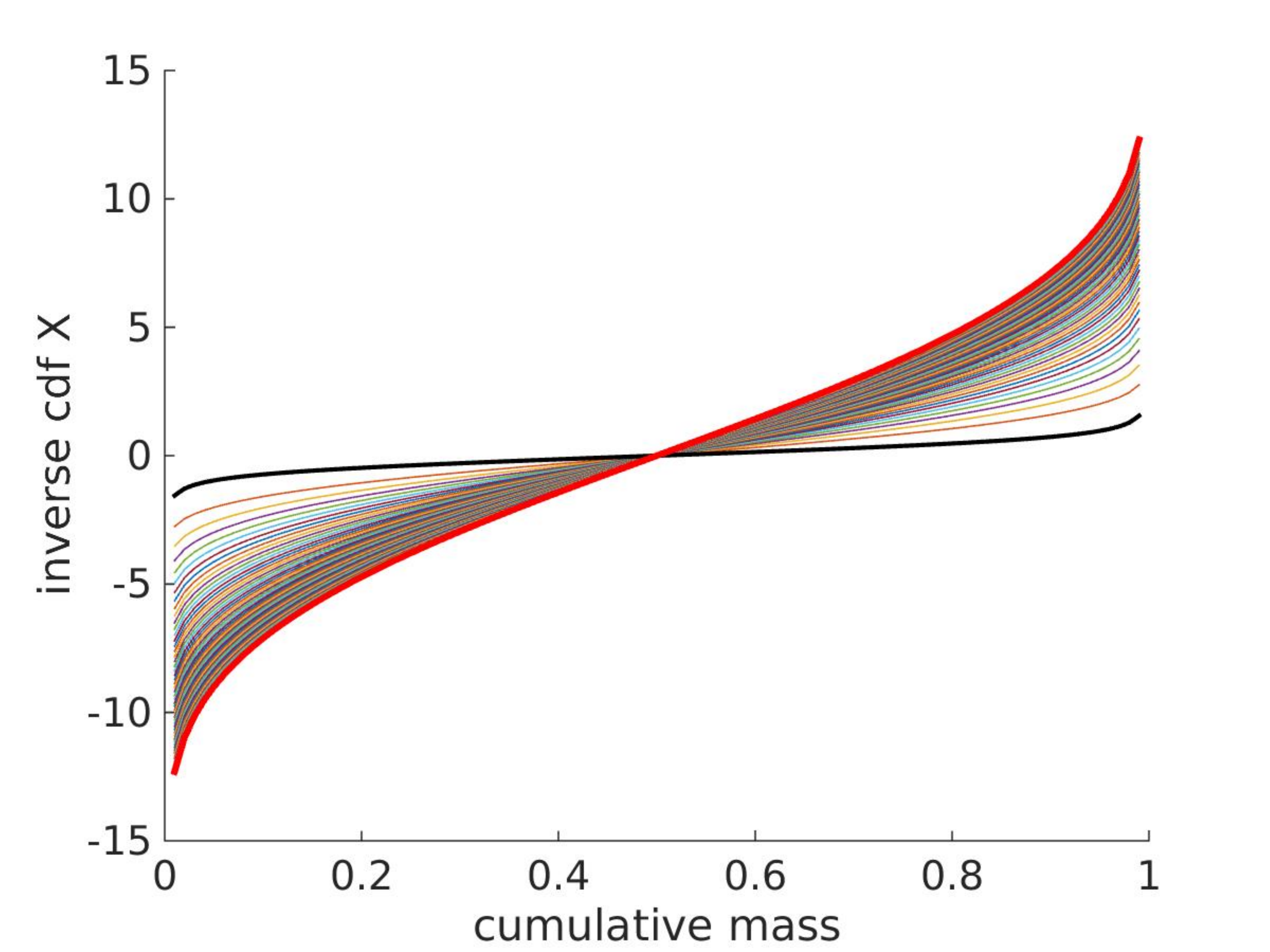}}%
\subfloat[]{\includegraphics[width=.45\textwidth]{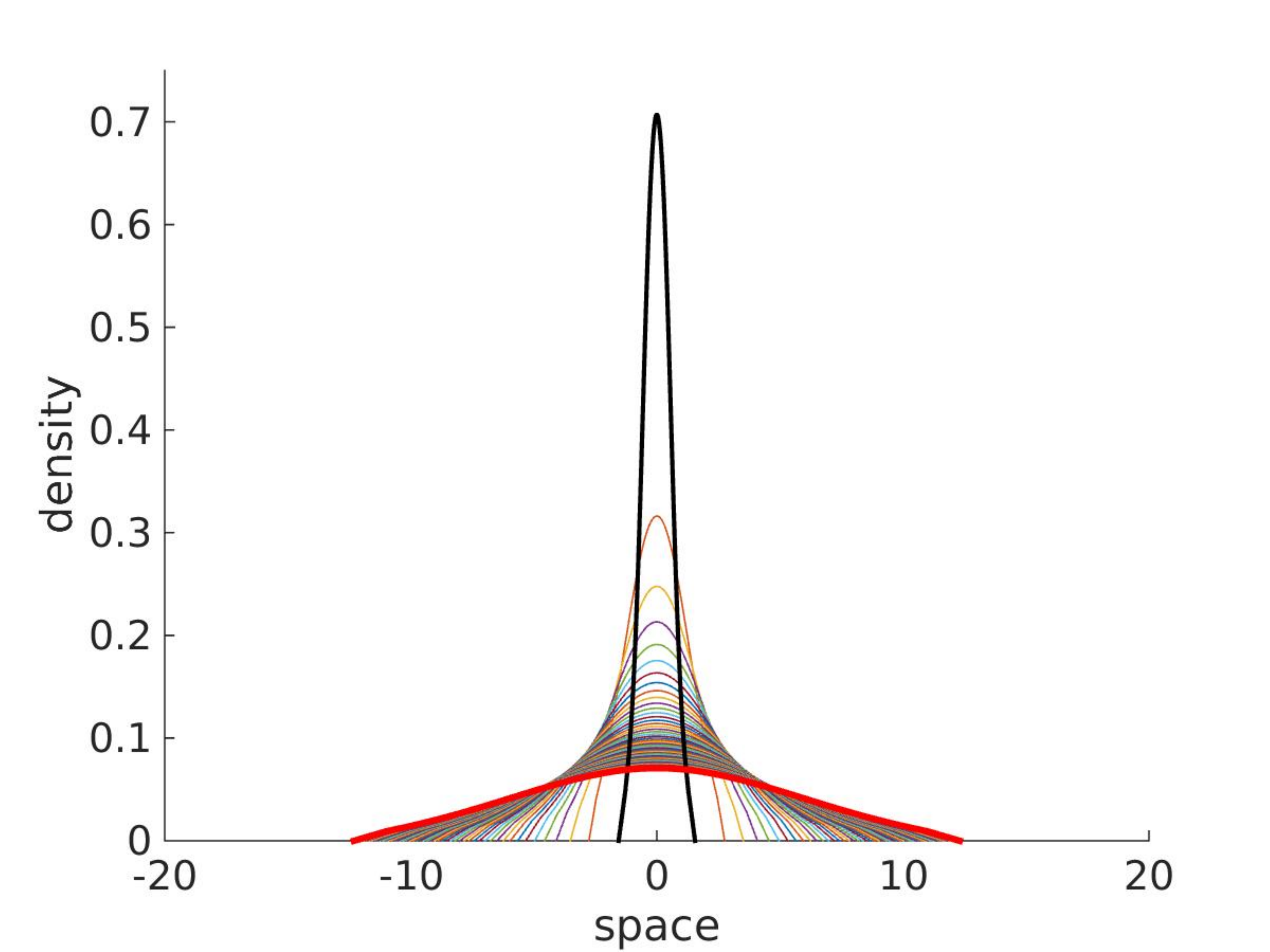}}%

\subfloat[]{\includegraphics[width=.45\textwidth]{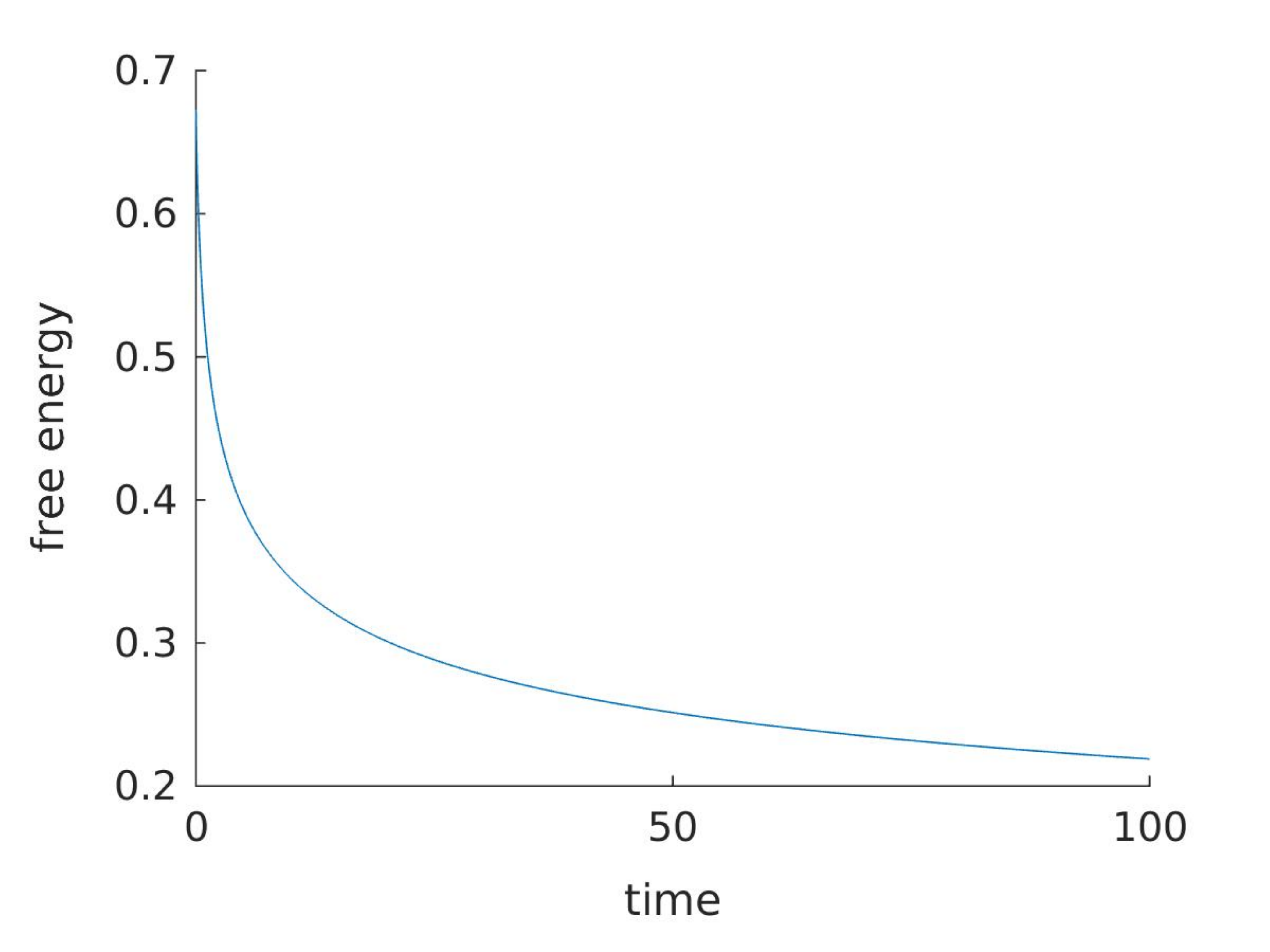}}%
\vspace{-0.2 cm}
\caption{Point $A$: $\chi=0.2$, $k=-0.5$, $r=0$.\\ (a) Inverse cumulative distribution function, (b) solution density, (c) free energy.}
\label{fig:chi=02_k=-05_resc=0}
\end{figure}

% \vspace{-0.8cm}
\subsubsection{Points $A$-$F$}
%%%%%%%%%%%%%%%%%%%%%%%%%%%%%%%%%%%%%%%%%%%%%%
%%%%%%%%%%%%%%%%%%%%%%%%%%%%%%%%%%%%%%%%%%%%%%
Let us now investigate in more detail the time-evolution behaviour at the points $A$--$F$ in Figure \ref{fig:parameterspace}. For $k=-0.5$ in the porous medium regime and sub-critical $\chi=0.2$ (point $A$ in Figure \ref{fig:parameterspace}), the diffusion dominates and the density goes pointwise to zero as $t \to \infty$ in original variables. Figure \ref{fig:chi=02_k=-05_resc=0}(a) and \ref{fig:chi=02_k=-05_resc=0}(b) show the inverse cumulative distribution function and the density profile for $(k,\chi)=(-0.5,0.2)$ respectively, from time $t=0$ (black) to time $t=100$ (red) in time steps of size $\Delta t=10^{-3}$ and with $\Delta \eta=10^{-2}$. We choose a centered normalised Gaussian with variance $\sigma^2=0.32$ as initial condition.
Figure \ref{fig:chi=02_k=-05_resc=0}(c) shows the evolution of the free energy \eqref{eq:functional} over time, which continues to decay as expected.

\newpage

%%%%%%%%%%%%%%%%%%%%%%%%%%%%%%%%%%%%%%%%
%%%%%%%%%%%%%%%%%%%%%%%%%%%%%%%%%%%%%%%%
%%%% chi=02_k=-05_resc=1_run2
%%%% fig7
\begin{figure}[ht]
\centering
\subfloat[]{\includegraphics[width=.45\textwidth]{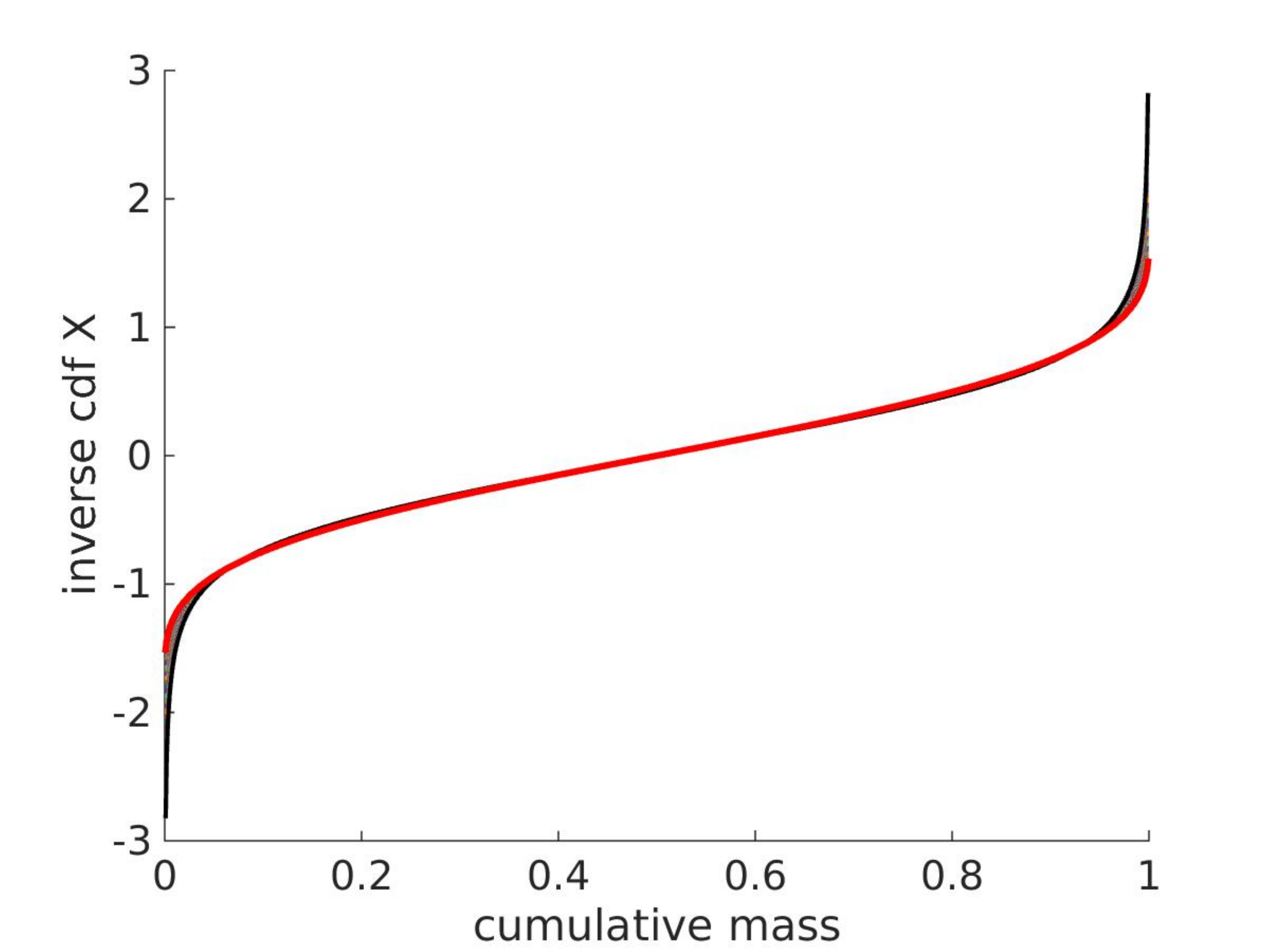}}%
\subfloat[]{\includegraphics[width=.45\textwidth]{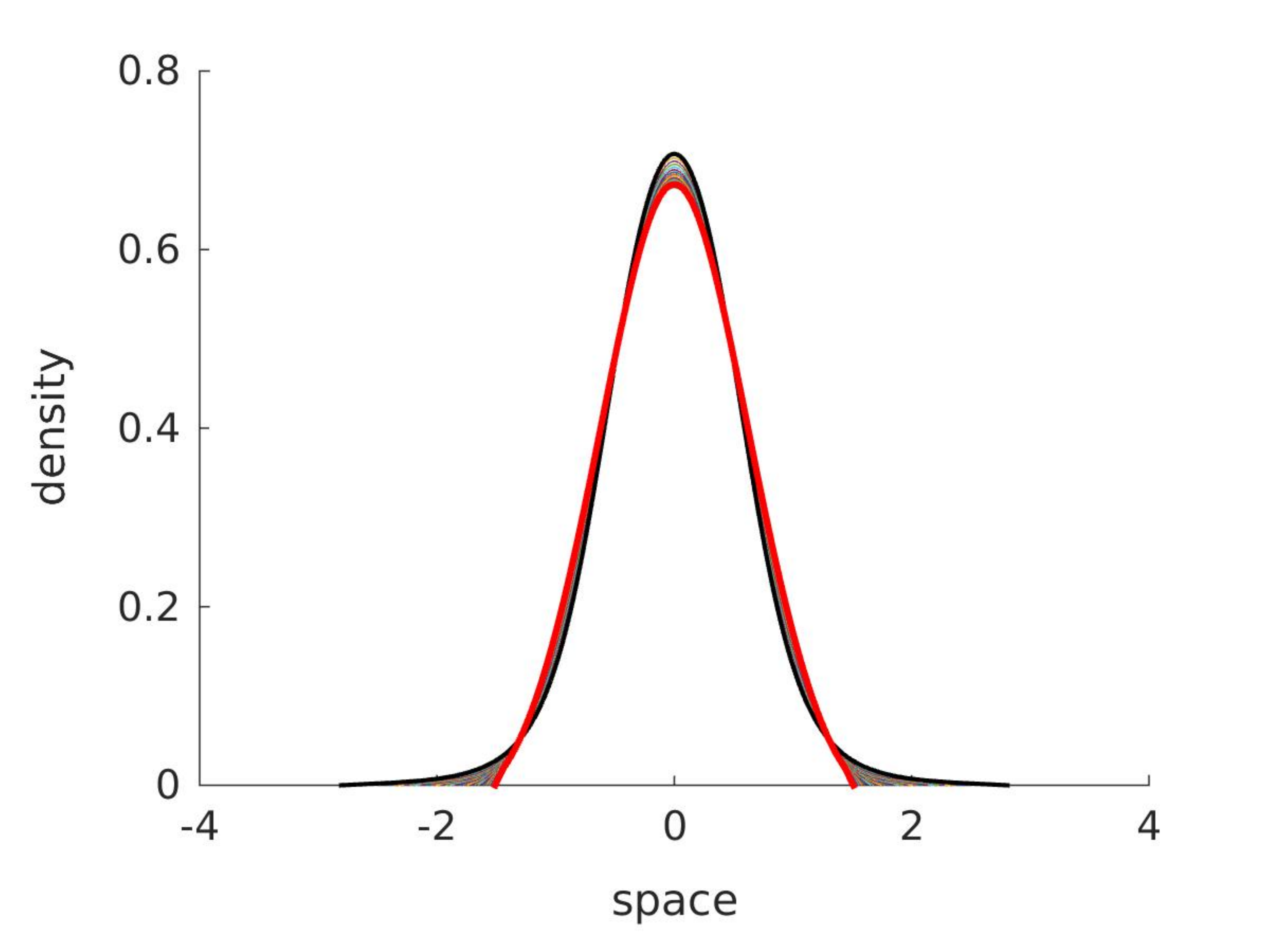}}%

% \vspace{-0.4cm}
\subfloat[]{\includegraphics[width=.45\textwidth]{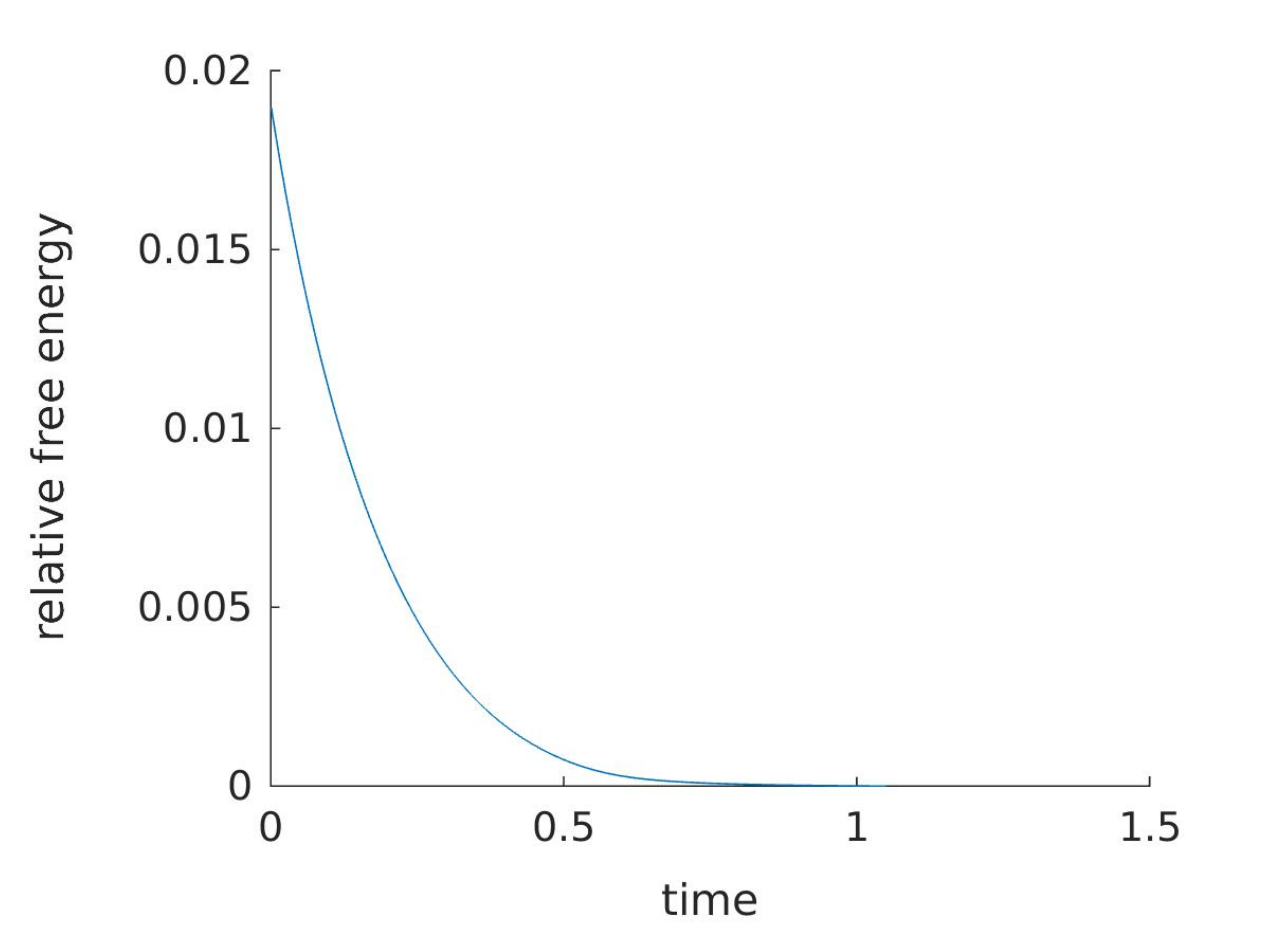}}%
\subfloat[]{\includegraphics[width=.45\textwidth]{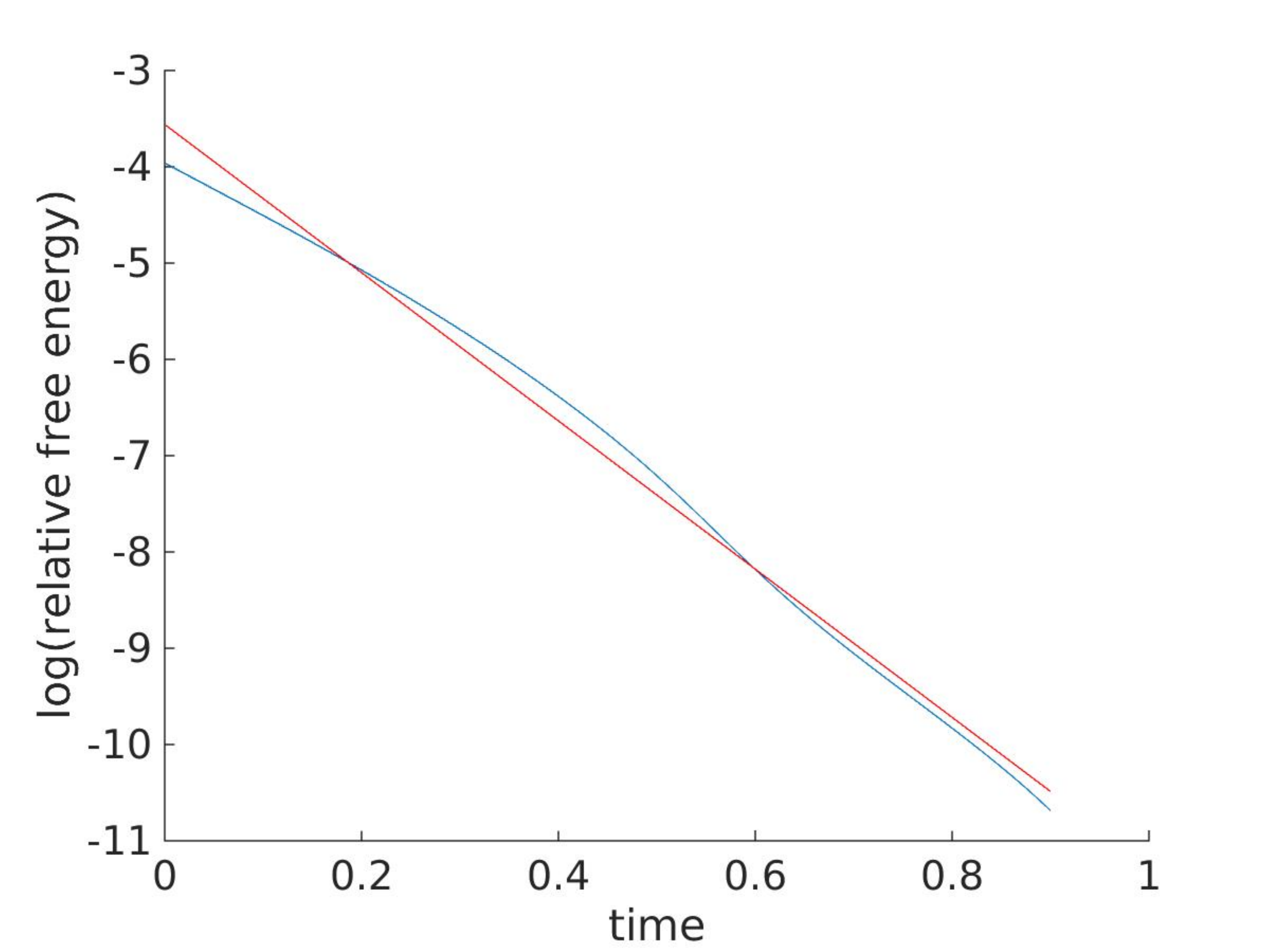}}%
%slope log(L2error)=-7.6965
% fitted between 0 and 0.9

% \vspace{-0.4cm}
\subfloat[]{\includegraphics[width=.45\textwidth]{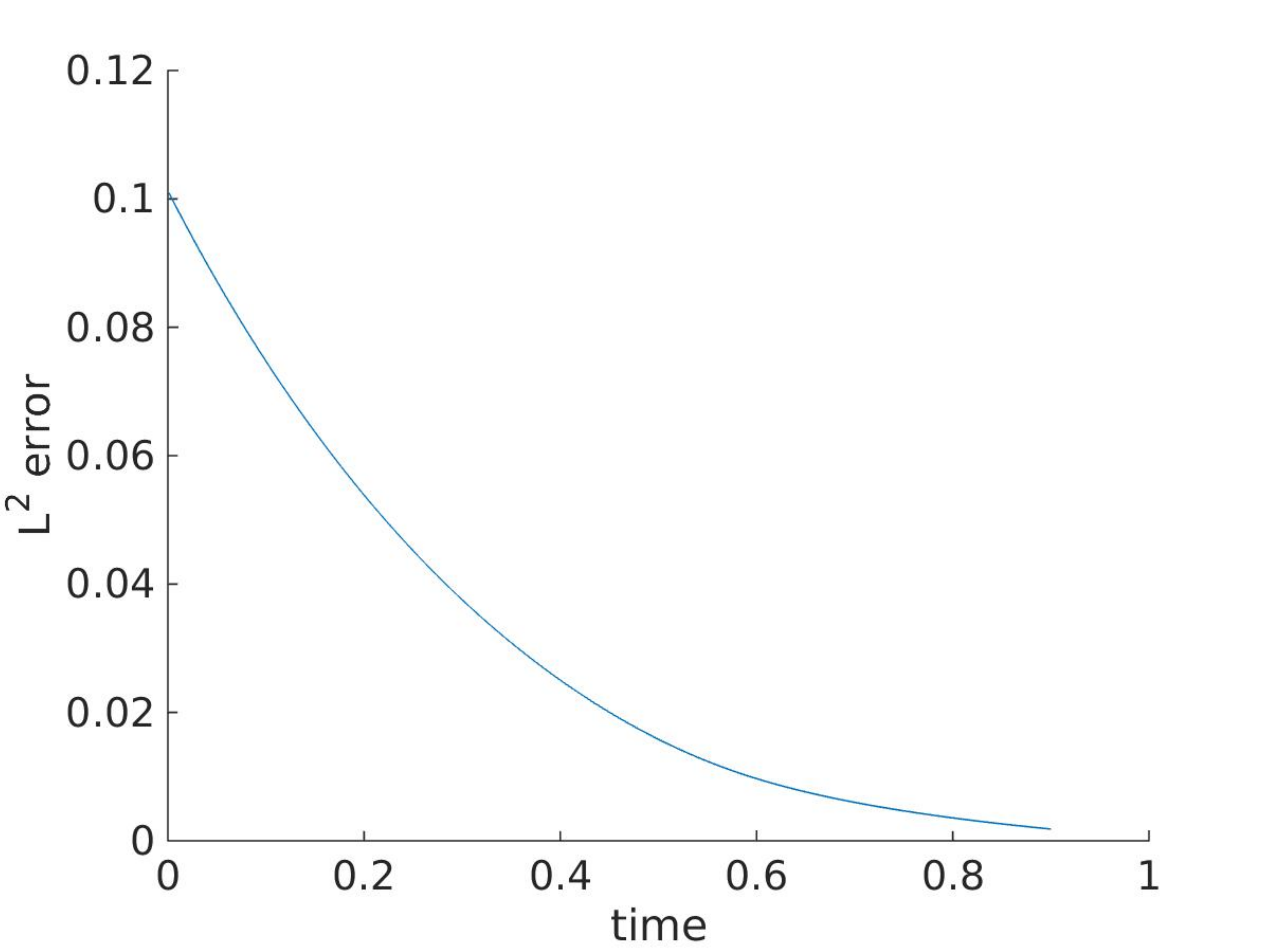}}%
\subfloat[]{\includegraphics[width=.45\textwidth]{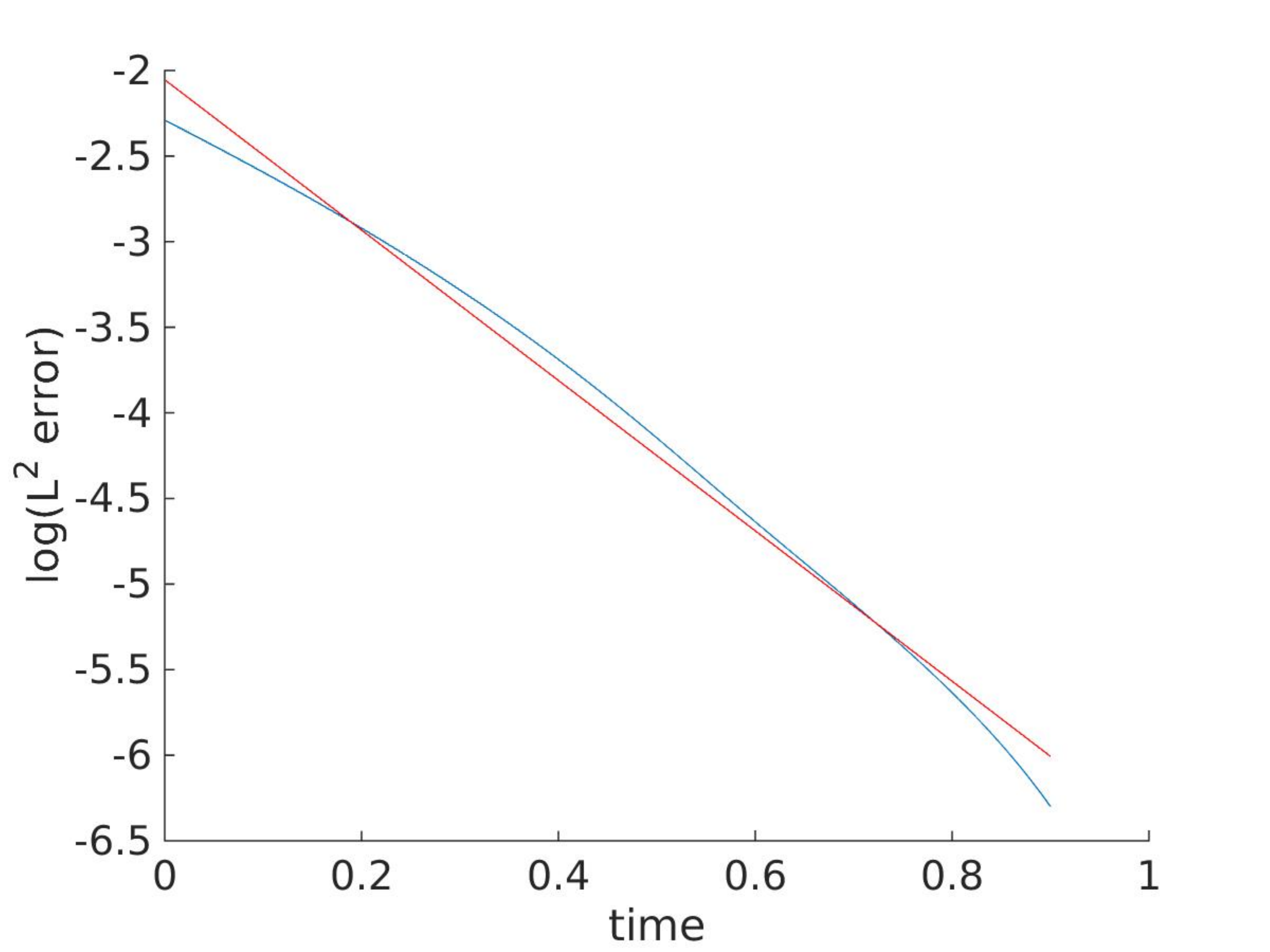}}%
%slope log(L2error)=-4.392
%fitted between 0 and 0.9
% \vspace{-0.2cm}
\caption{Point $B$: $\chi=0.2$, $k=-0.5$, $r=1$.\\ (a) Inverse cumulative distribution function from initial condition (black) to the profile at the last time step (red), (b) solution density from initial condition (black) to the profile at the last time step (red), (c) relative free energy, (d) log(relative free energy) and fitted line between times 0 and 0.9 with slope $-7.6965$ (red), (e) $L^2$-error between the solutions at time $t$ and at the last time step, (f) log($L^2$-error) and fitted line with slope $-4.392$ (red).}
\label{fig:chi=02_k=-05_resc=1}
\end{figure}
%%%%%%%%%%%%%%%%%%%%%%%%%%%%%%%%%%%%%%%%%%%
%%%%%%%%%%%%%%%%%%%%%%%%%%%%%%%%%%%%%%%%%%%
% We obtain a fitted line $y=a*x+b$ with $a=-7.6965$ and $b=-3.5604$ (Figure \ref{fig:chi=02_k=-05_resc=1}(d)).
%%%About fig7
 \newpage
For exactly the same choice of parameters $(k,\chi)=(-0.5,0.2)$ and the same initial condition we then investigate the evolution in rescaled variables (point $B$ in Figure \ref{fig:parameterspace}), and as predicted by Proposition \ref{prop:cvsub-critical}, the solution converges to a stationary state. See Figures \ref{fig:chi=02_k=-05_resc=1}(a) and \ref{fig:chi=02_k=-05_resc=1}(b) for the evolution of the inverse cumulative distribution function and the density distribution with $\Delta t=10^{-3}$ and $\Delta \eta=10^{-3}$ from $t=0$ (black) to the stationary state $\bar \rho$ (red). Again, we terminate the evolution as soon as the $L^2$-distance between the numerical solution at two consecutive time steps is less than a certain tolerance, chosen at $10^{-5}$. We see that the solution converges very quickly both in relative energy $|\mF_k[\rho(t)]-\mF_k[\bar \rho]|$ (Figure \ref{fig:chi=02_k=-05_resc=1}(c)) and in terms of the Wasserstein distance to the solution at the last time step $\bW\left(\rho(t),\bar \rho\right)$ (Figure \ref{fig:chi=02_k=-05_resc=1}(e)). To check that the convergence is indeed exponential as predicted by Proposition \ref{prop:cvsub-critical}, we fit a line to the logplot of both the relative free energy (between times $t=0$ and $t=0.9$), see Figure \ref{fig:chi=02_k=-05_resc=1}(d), and to the logplot of the Wasserstein distance to equilibrium, see Figure \ref{fig:chi=02_k=-05_resc=1}(f).
In both cases, we obtain a fitted line $y=-a*t+b$ with some constant $b$ and rate $a=7.6965$ for the relative free energy and rate $a=4.392$ for the Wasserstein distance to equilibrium. Recall that the $L^2$-error between two solutions $X(\eta)$ and $\tilde X(\eta)$ is equal to the Wasserstein distance between the corresponding densities $\rho(x)$ and $\tilde \rho(x)$ as described in \eqref{WW}. We observe a rate of convergence that is in agreement with \cite{CaCa11, CaDo14, EM16} for the logarithmic case $k=0$.\\

For parameter choices $k=-0.2$ and $\chi=0.7$ (point $C$ in Figure \ref{fig:parameterspace}), we are again in the sub-critical regime where solutions converge to a stationary state in rescaled variables according to Proposition \ref{prop:cvsub-critical}, see Figures \ref{fig:chi=07_k=-02_resc=1}(a) and \ref{fig:chi=07_k=-02_resc=1}(b). However, point $C$ is closer to the critical interaction strength $\chi_c(k)$ than point $B$ (numerically, we have $\chi_c(-0.2)=0.71$), and as a result we can observe that the stationary density $\bar \rho$ in Figure \ref{fig:chi=07_k=-02_resc=1}(b) (red) is more concentrated than in Figure \ref{fig:chi=02_k=-05_resc=1}(b). Here, we choose as initial condition a characteristic function supported on the ball centered at zero with radius $1/2$ (black, Figure \ref{fig:chi=07_k=-02_resc=1}(b)), and fix $\Delta t=10^{-3}$, $\Delta \eta=5*10^{-3}$ with tolerance $10^{-5}$. We observe that the solution converges very quickly to a stationary state both in relative free energy $|\mF_k[\rho(t)]-\mF_k[\bar \rho]|$ (Figure \ref{fig:chi=07_k=-02_resc=1}(c)) and in terms of the Wasserstein distance to equilibrium $\bW(\rho(t),\bar\rho)$ (Figure \ref{fig:chi=07_k=-02_resc=1}(e)). To investigate the exponential rate of convergence, we fit again a line to the logplot of both the relative free energy (here between times $t=0$ and $t=1.8$) see Figure \ref{fig:chi=07_k=-02_resc=1}(d), and the Wasserstein distance to equilibrium, see Figure \ref{fig:chi=07_k=-02_resc=1}(f). We obtain fitted lines $y=-a*t+b$ with some constant $b$ and rate $a=3.2407$ for the relative free energy, whereas the rate is $a=1.8325$ for the Wasserstein distance to equilibrium. \\
 
Next, we are looking at point $D$ in Figure \ref{fig:parameterspace}, which corresponds to the choice $(k,\chi)=(0.2,0.8)$ and is part of line $L_2$ (see Figure \ref{fig:criticalk_chi}(b)). Since point $D$ lies in the fast diffusion regime $k>0$, no critical interaction strength exists \cite{CCH1}, and so we look at convergence to self-similarity. Figures \ref{fig:chi=08_k=02_resc=1}(a) and \ref{fig:chi=08_k=02_resc=1}(b) display the evolution of the inverse cumulative distribution function and the density distribution from $t=0$ (black) to the stationary state $\bar \rho$ (red) in rescaled variables including the solutions at 50 intermediate time steps. We start with a characteristic function supported on a centered ball of radius $1/2$. Choosing $\Delta t=10^{-3}$ and $\Delta \eta=10^{-2}$ is enough. The density instantaneously becomes supported on the whole space for any $t>0$ as shown in the proof of \cite[Corollary 4.4]{CCH1}, which cannot be fully represented numerically since the tails are cut by numerical approximation, see Figure \ref{fig:chi=08_k=02_resc=1}(a)-(b). Again, we observe very fast convergence both in relative energy (Figure \ref{fig:chi=08_k=02_resc=1}(c)-(d)) and in Wasserstein distance to equilibrium (Figure \ref{fig:chi=08_k=02_resc=1}(e)-(f)) as predicted by Proposition \ref{prop:cvinW2}. A logplot of the relative free energy (Figure \ref{fig:chi=08_k=02_resc=1}(d)) and the Wasserstein distance to equilibrium (Figure \ref{fig:chi=08_k=02_resc=1}(f)) show exponential rates of convergence with rates $a=3.6904$ and $a=1.9148$ respectively for the fitted line $y=-a*t+b$ with some constant $b$ and for times $0.2 \leq t \leq 3.8$. \\

% \newpage

%%%% about fig10

For the same choice of $k=0.2$ in the fast diffusion regime, but with higher interaction strength $\chi=1.2$ (point $E$ in Figure \ref{fig:parameterspace}, which is part of line $L_3$, see Figure \ref{fig:criticalk_chi}(c)), we obtain a similar behaviour. Figures \ref{fig:chi=12_k=02_resc=1}(a) and \ref{fig:chi=12_k=02_resc=1}(b) show the inverse cumulative distribution function and the density distribution, both for the initial data (black), a characteristic supported on the centered ball of radius $1/2$, and for the stationary state $\bar \rho$ (red). Here we choose as before $\Delta t=10^{-3}$ and $\Delta \eta=10^{-2}$. We observe that \\

\newpage

%%%%%%%%%%%%%%%%%%%%%%%%%%%%%%%%%%%%%
%%%%%%%%%%%%%%%%%%%%%%%%%%%%%%%%%%%%%%
%%%%%%%%%%%%%%%%%%%%%%%%%%%%%%%%%%%%%%
%%%% chi=07_k=-02_resc=1_run1 
%%%% fig8
\begin{figure}[h!]
\centering
\subfloat[]{\includegraphics[width=.45\textwidth]{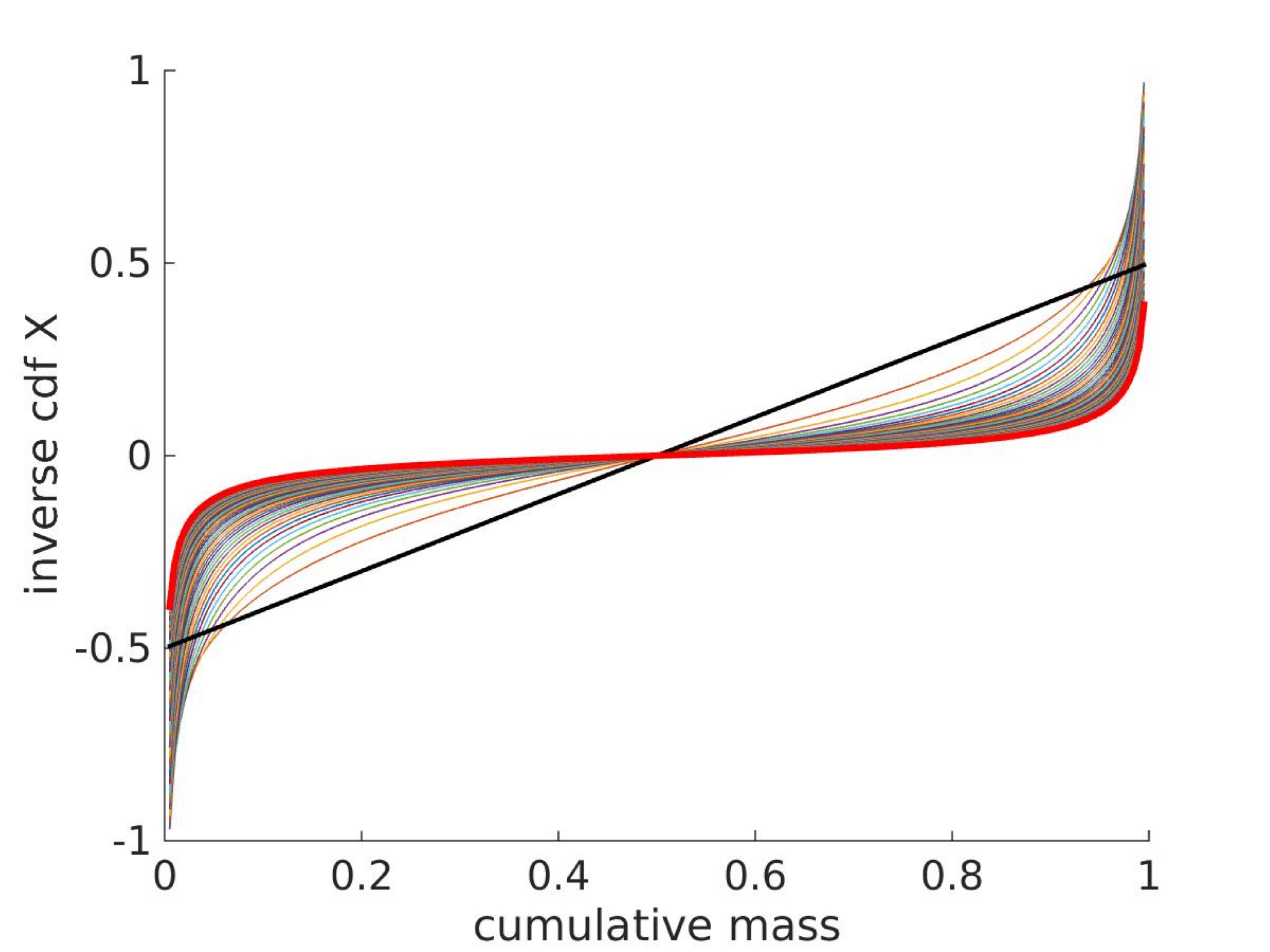}}%
\subfloat[]{\includegraphics[width=.45\textwidth]{chi=07_k=-02_resc=1_run1_fig1a.pdf}}%

% \vspace{-0.4cm}
\subfloat[]{\includegraphics[width=.45\textwidth]{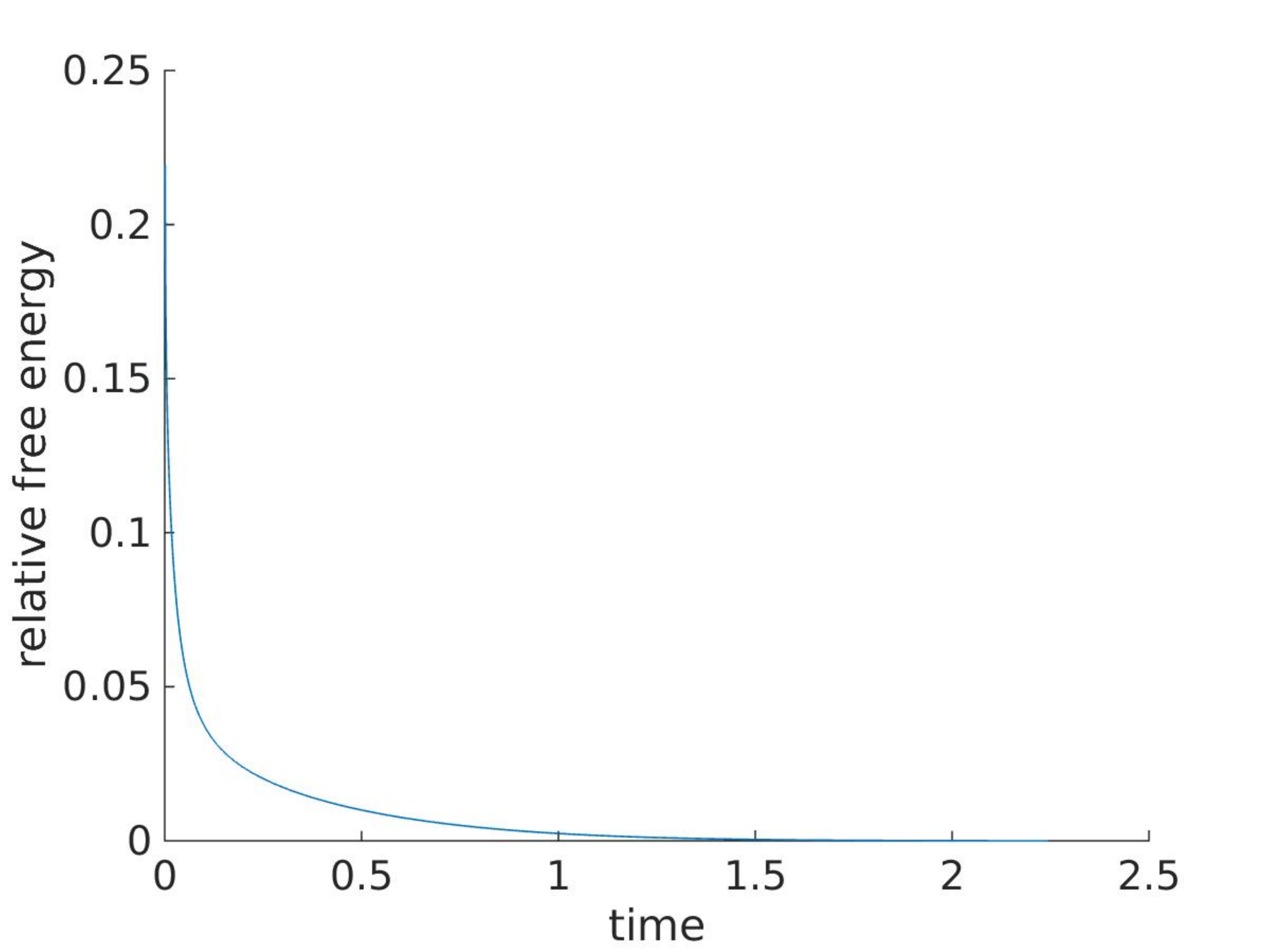}}%
\subfloat[]{\includegraphics[width=.45\textwidth]{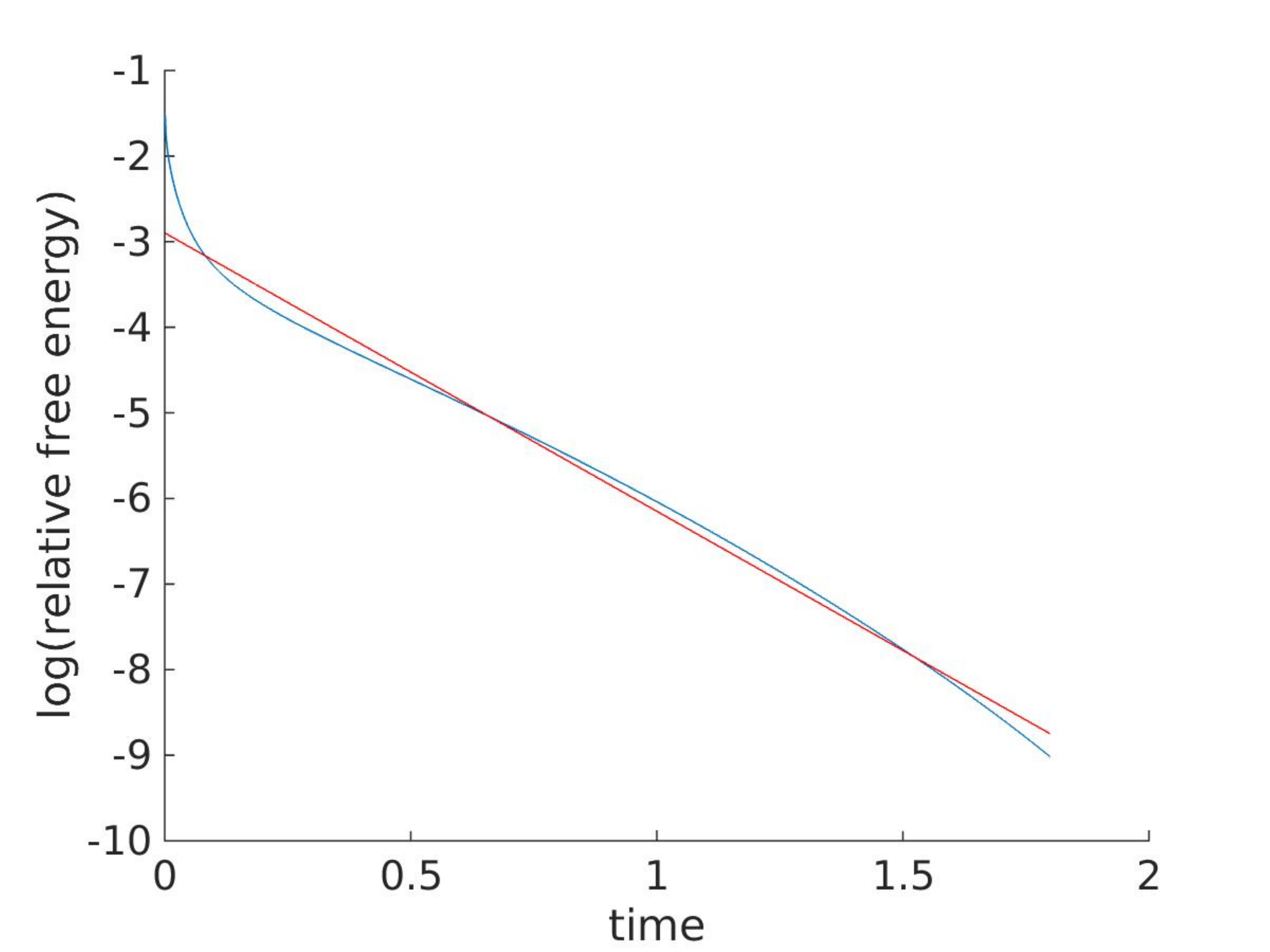}}%
%fitted between 0 and 1.8
%slope log(L2error)=-3.2522
% \vspace{-0.4cm}

\subfloat[]{\includegraphics[width=.45\textwidth]{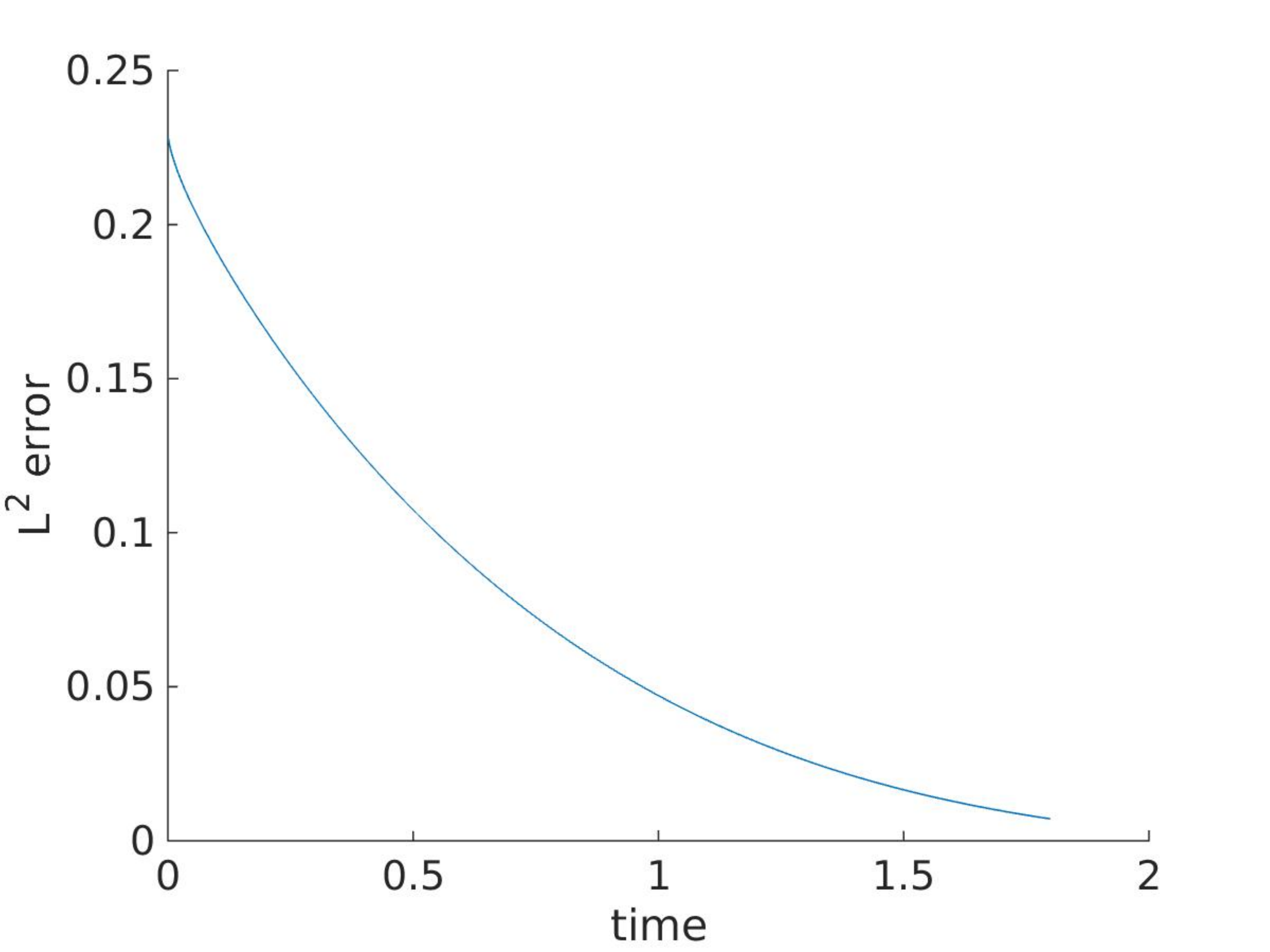}}%
\subfloat[]{\includegraphics[width=.45\textwidth]{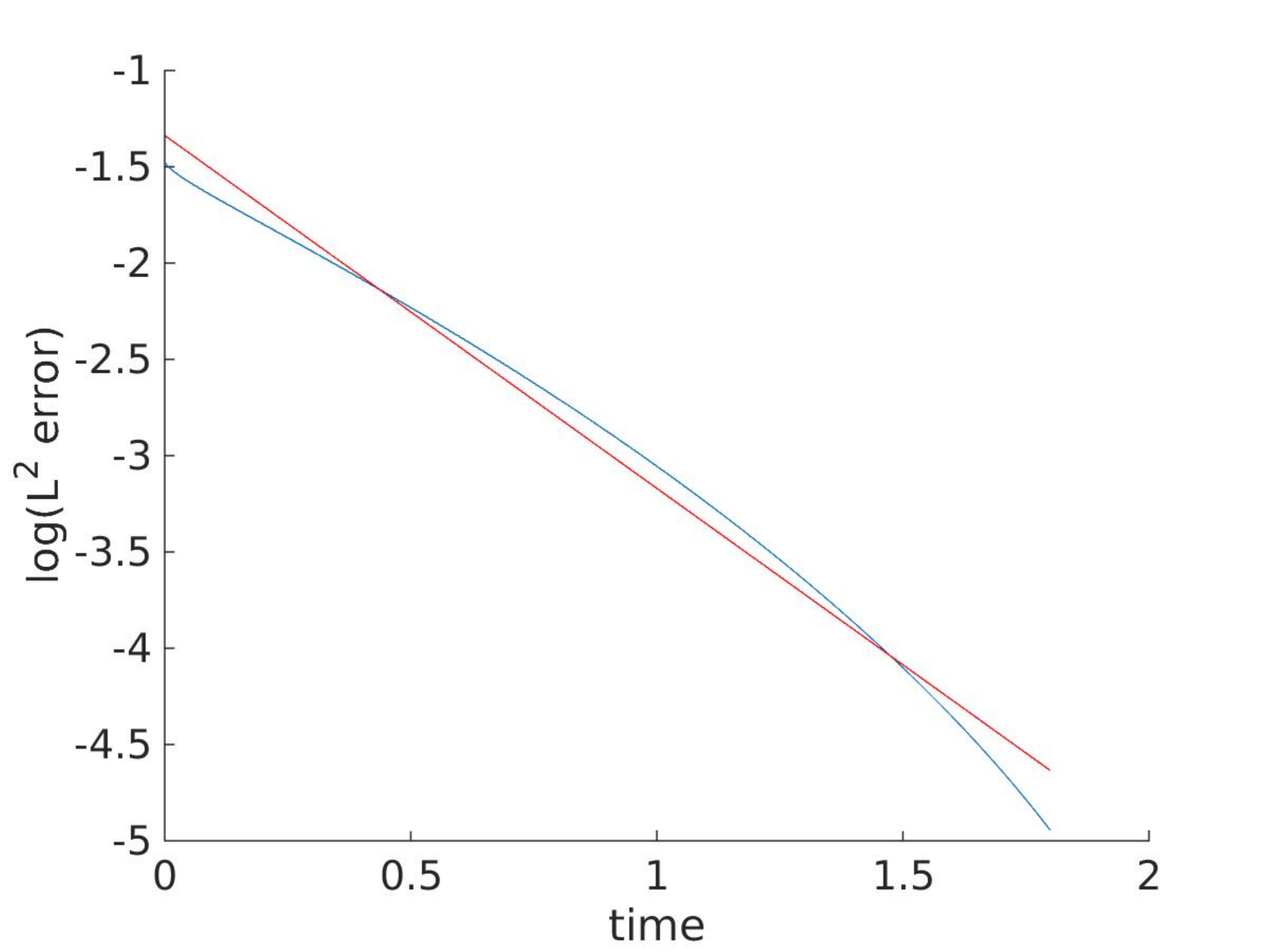}}%
%fitted between 0 and 1.8
%slope log(L2error)=-1.8325
% \vspace{-0.2cm}
\caption{Point $C$: $\chi=0.7$, $k=-0.2$, $r=1$.\\
(a) Inverse cumulative distribution function from initial condition (black) to the profile at the last time step (red),
(b) solution density from initial condition (black) to the profile at the last time step (red),
(c) relative free energy,
(d) log(relative free energy) and fitted line between times 0 and  1.8 with slope $-3.2522$ (red),
(e) $L^2$-error between the solutions at time $t$ and at the last time step,
(f) log($L^2$-error) and fitted line with slope $-1.8325$ (red).
}
\label{fig:chi=07_k=-02_resc=1}
\end{figure}
%%%%%%%%%%%%%%%%%%%%%%%%%%%%%%%%%%%%%%%%%%%
%%%%%%%%%%%%%%%%%%%%%%%%%%%%%%%%%%%%%%%%%%%
%We obtain a fitted line $y=-a*t+b$ with $a=3.2407$ and $b=-2.9076$ (Figure \ref{fig:chi=07_k=-02_resc=1}(d)). 
%%% about fig8
 \newpage
%%%%%%%%%%%%%%%%%%%%%%%%%%%%%%%%%%%%%%%%%%%%%%%%%%%%%%%%%
%%%%%%%%%%%%%%%%%%%%%%%%%%%%%%%%%%%%%%%%%%%%%%%%%%%%%%%%%
%%%% chi=08_k=02_resc=1 
%%%% fig9
\begin{figure}[h!]
\centering
\subfloat[]{\includegraphics[width=.45\textwidth]{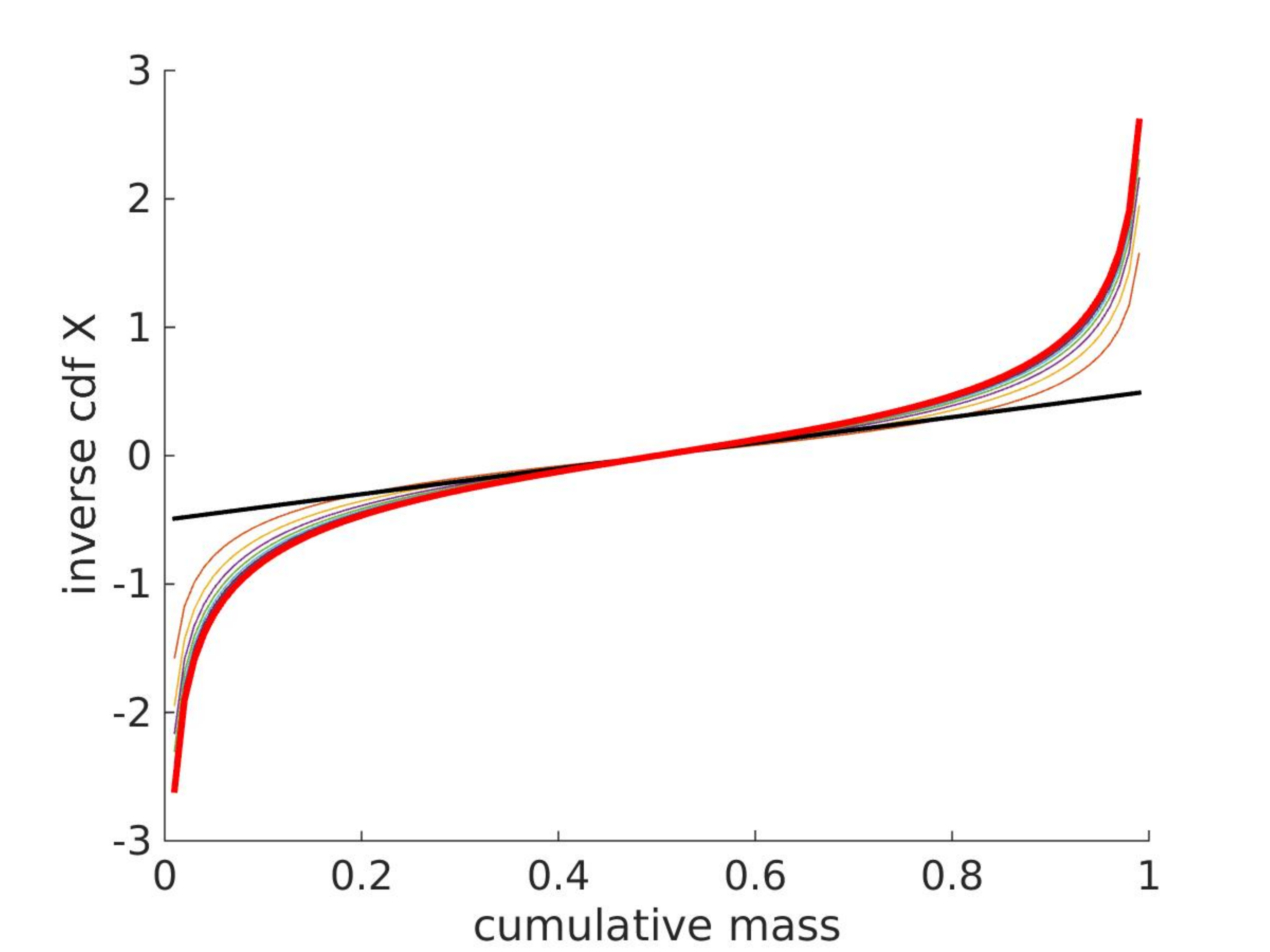}}%
\subfloat[]{\includegraphics[width=.45\textwidth]{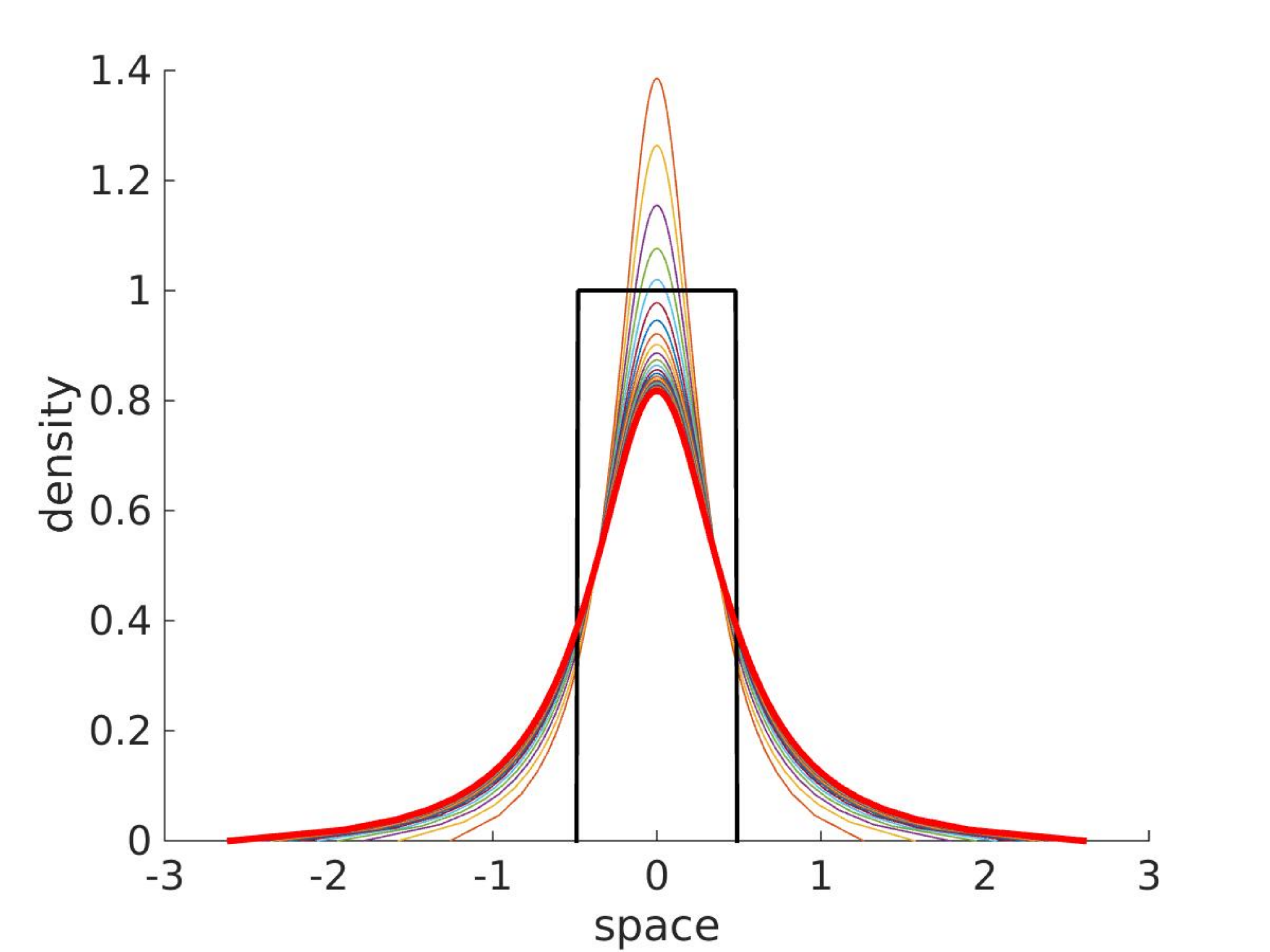}}%

% \vspace{-0.4cm}
\subfloat[]{\includegraphics[width=.45\textwidth]{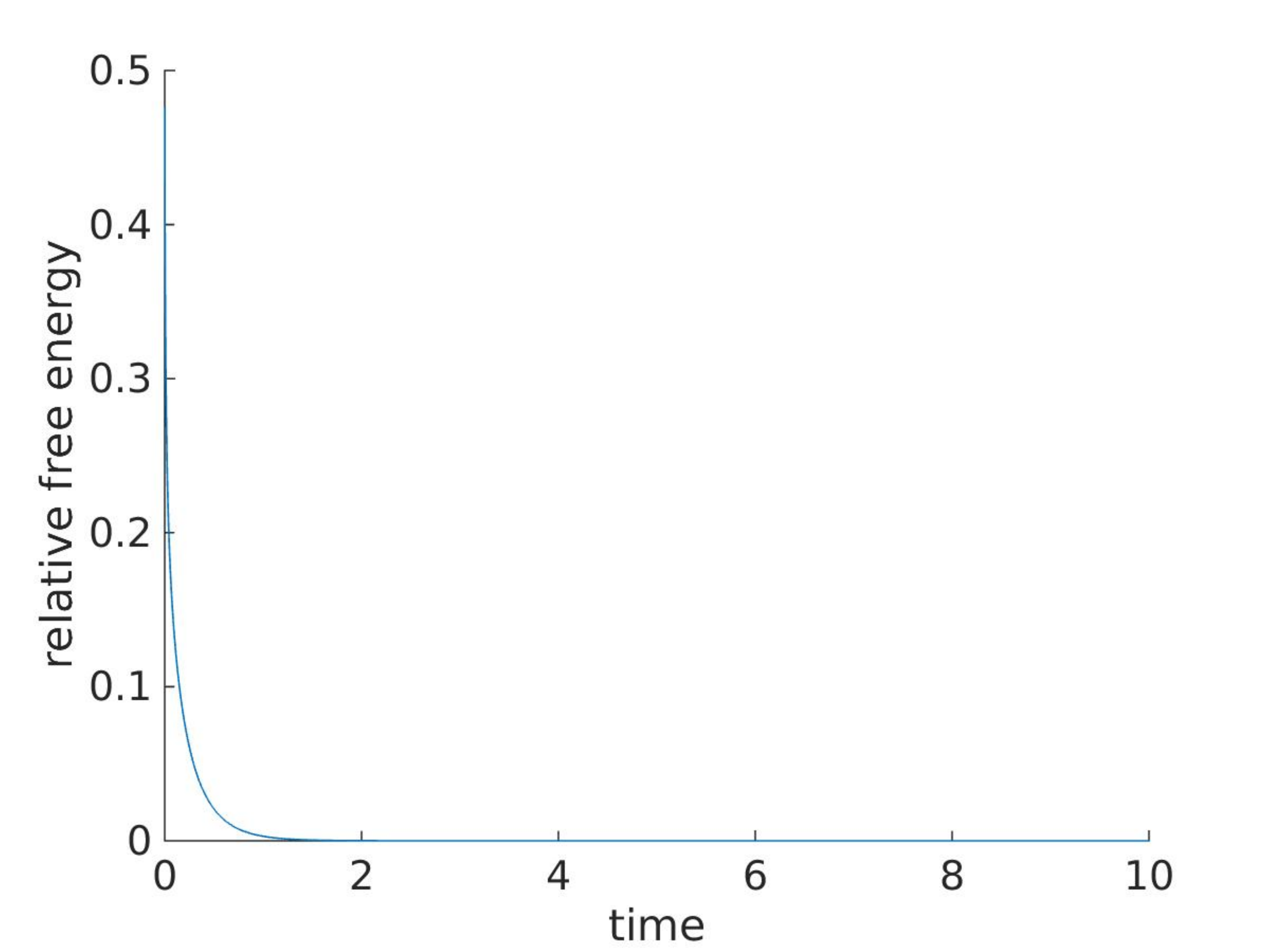}}%
\subfloat[]{\includegraphics[width=.45\textwidth]{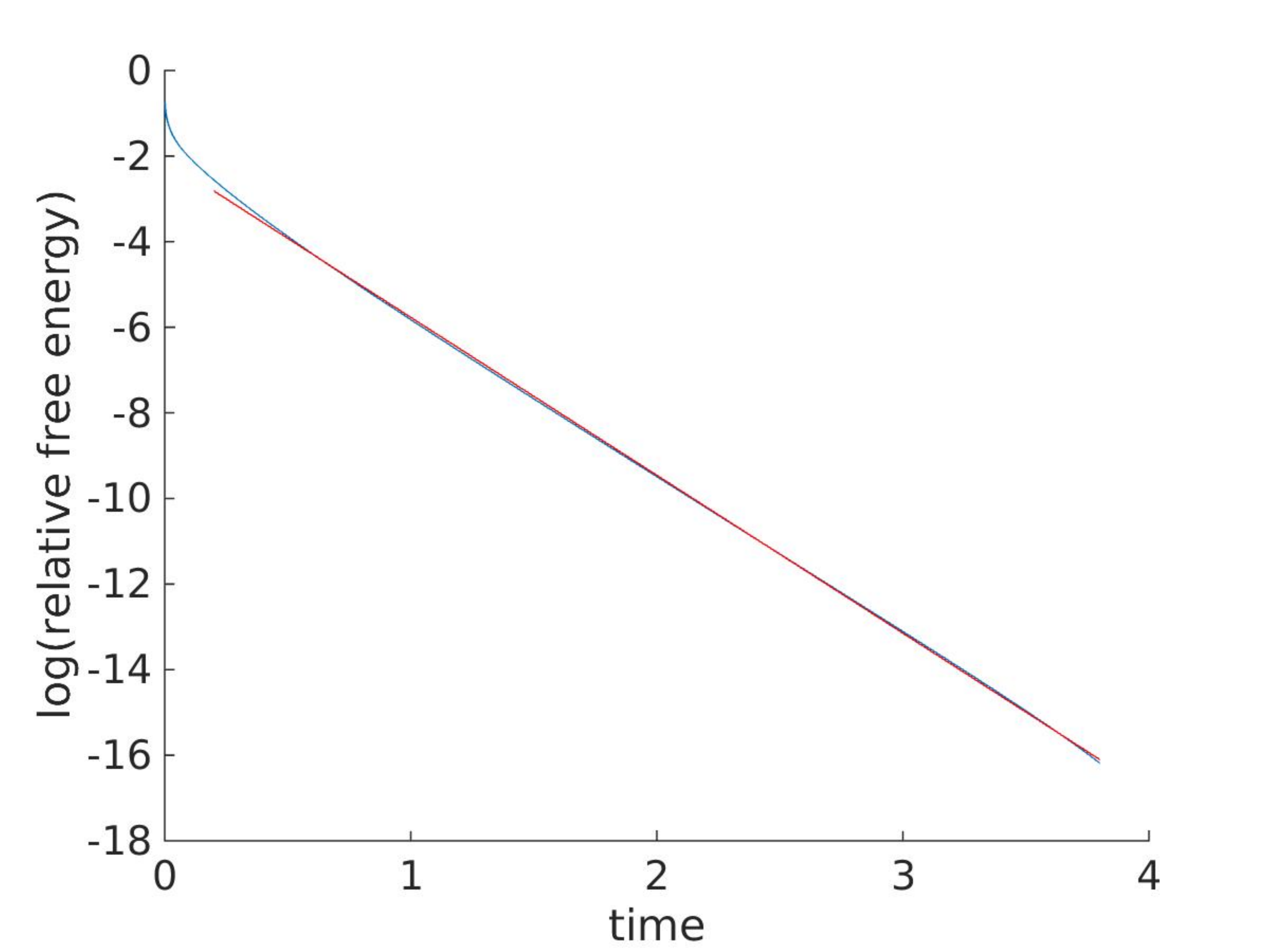}}%
%slope log(L2error)=-3.6904
% fitted between 0.2 and 3.8

% \vspace{-0.4cm}
\subfloat[]{\includegraphics[width=.45\textwidth]{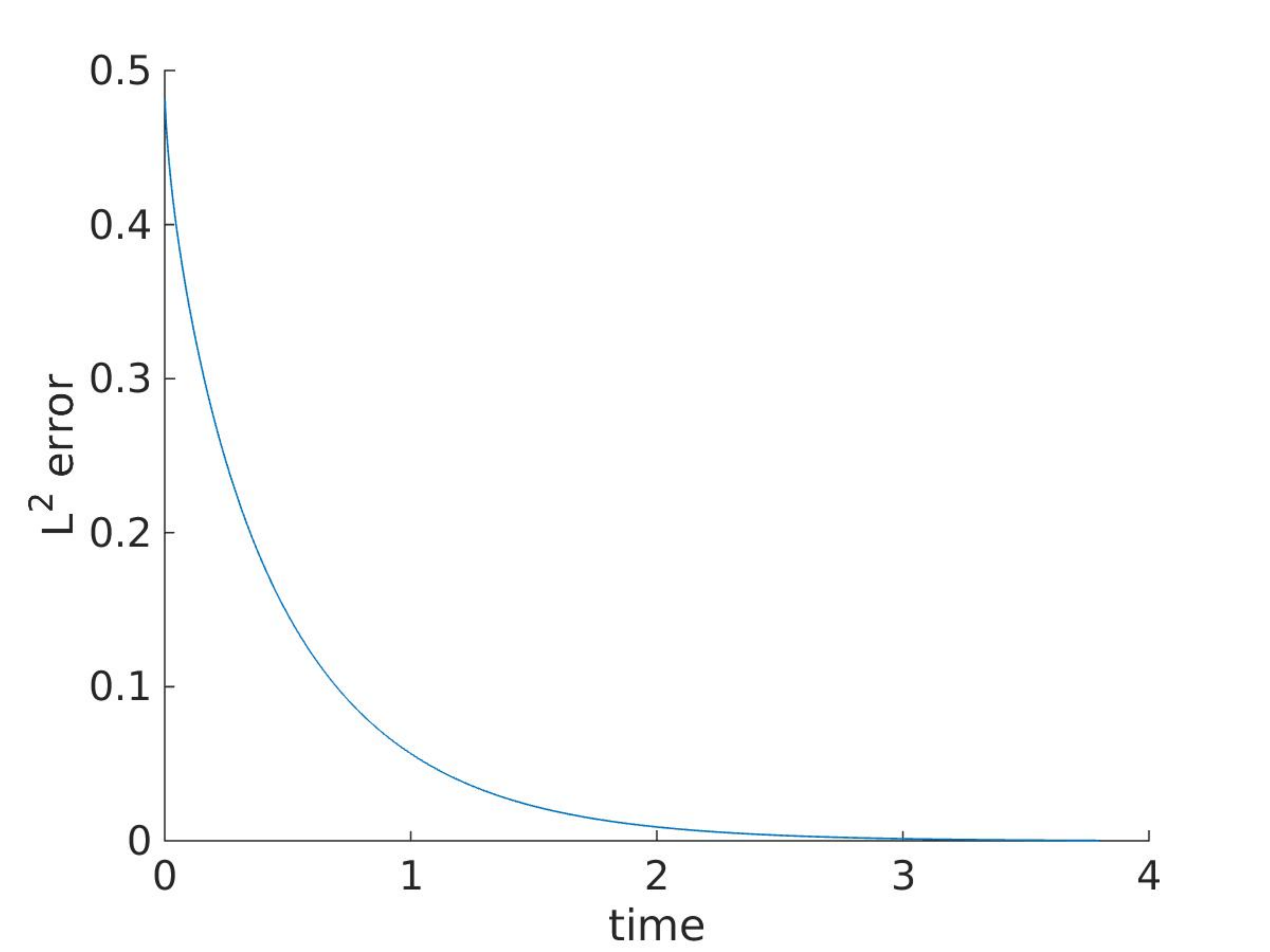}}%
\subfloat[]{\includegraphics[width=.45\textwidth]{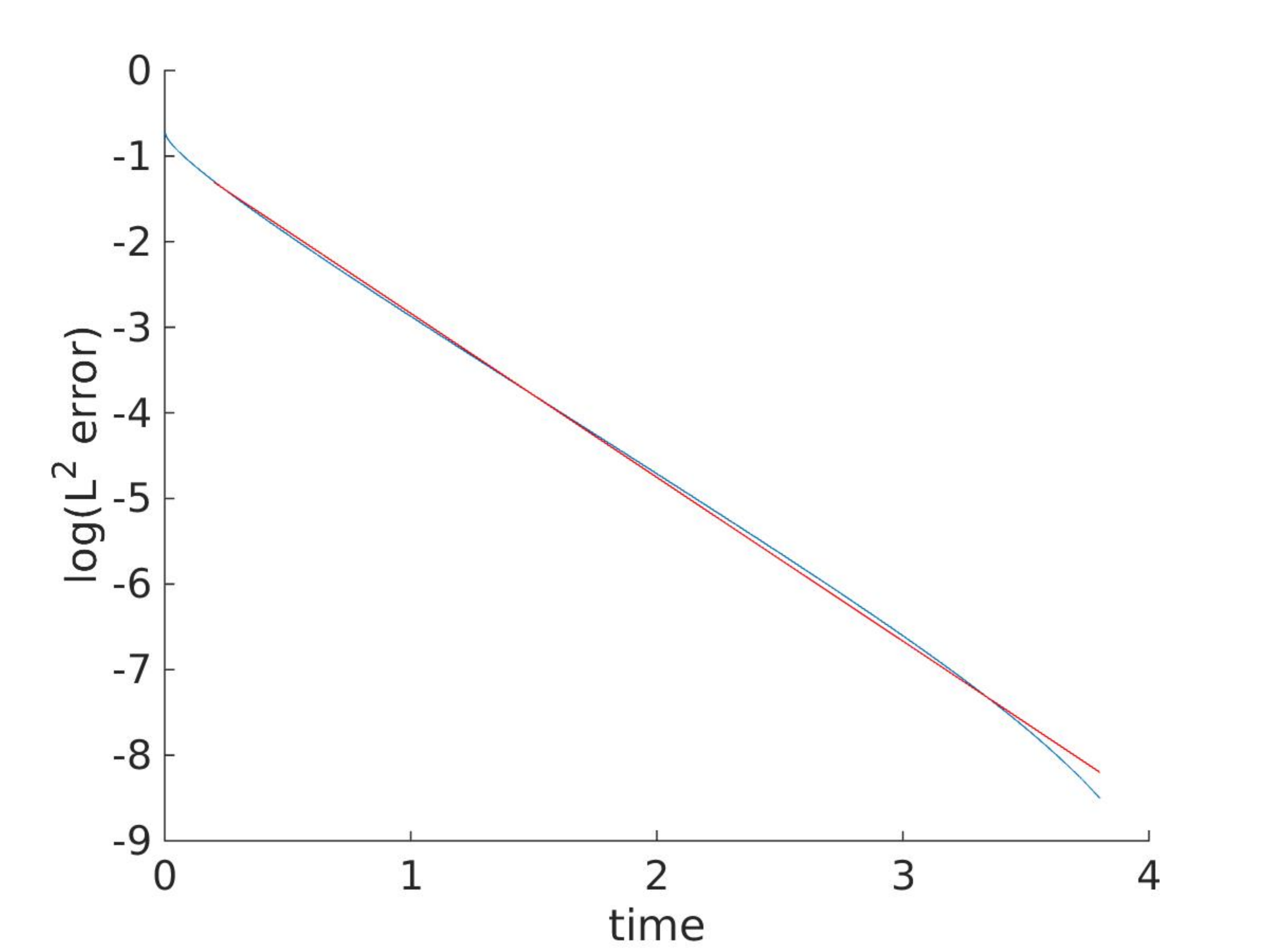}}%
%slope log(L2error)=-1.9148
% fitted between 0.2 and 3.8
% \vspace{-0.2cm}
\caption{Point $D$: $\chi=0.8$, $k=0.2$, $r=1$.\\
(a) Inverse cumulative distribution function from initial condition (black) to the profile at the last time step (red),
(b) solution density from initial condition (black) to the profile at the last time step (red),
(c) relative free energy,
(d) log(relative free energy) and fitted line between times 0.2 and 3.8 with slope $-3.6904$ (red),
(e) $L^2$-error between the solutions at time $t$ and at the last time step,
(f) log($L^2$-error) and fitted line between times 0.2 and 3.8 with slope $-1.9148$ (red).}
\label{fig:chi=08_k=02_resc=1}
\end{figure}
\newpage
%%%%%%%%%%%%%%%%%%%%%%%%%%%%%%%%%%%%%%%%%%%%%%%%
%%%%%%%%%%%%%%%%%%%%%%%%%%%%%%%%%%%%%%%%%%%%%%%%
%%%% chi=12_k=02_resc=1 
%%%% fig10
\begin{figure}[h!]
\centering
\subfloat[]{\includegraphics[width=.45\textwidth]{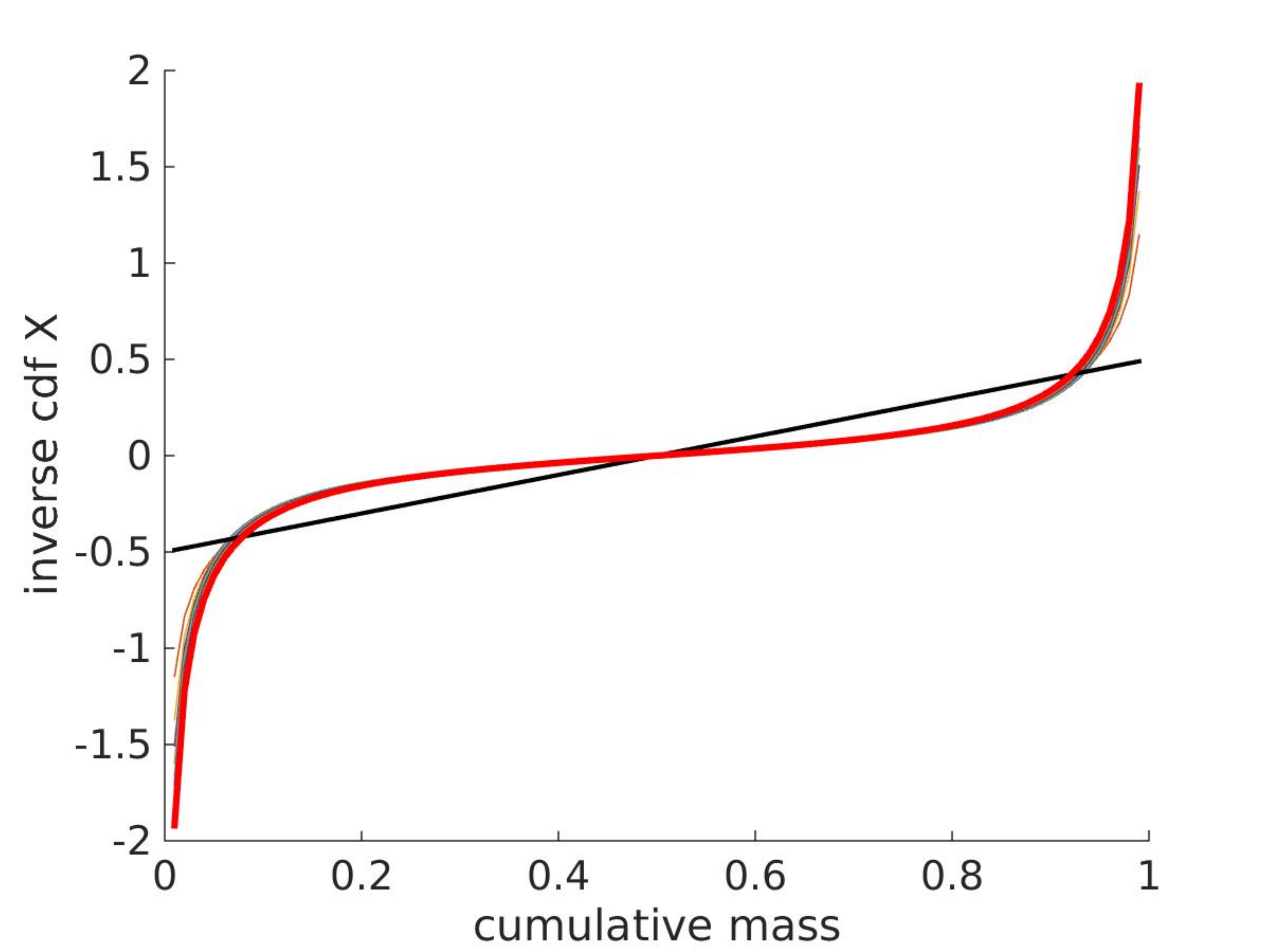}}%
\subfloat[]{\includegraphics[width=.45\textwidth]{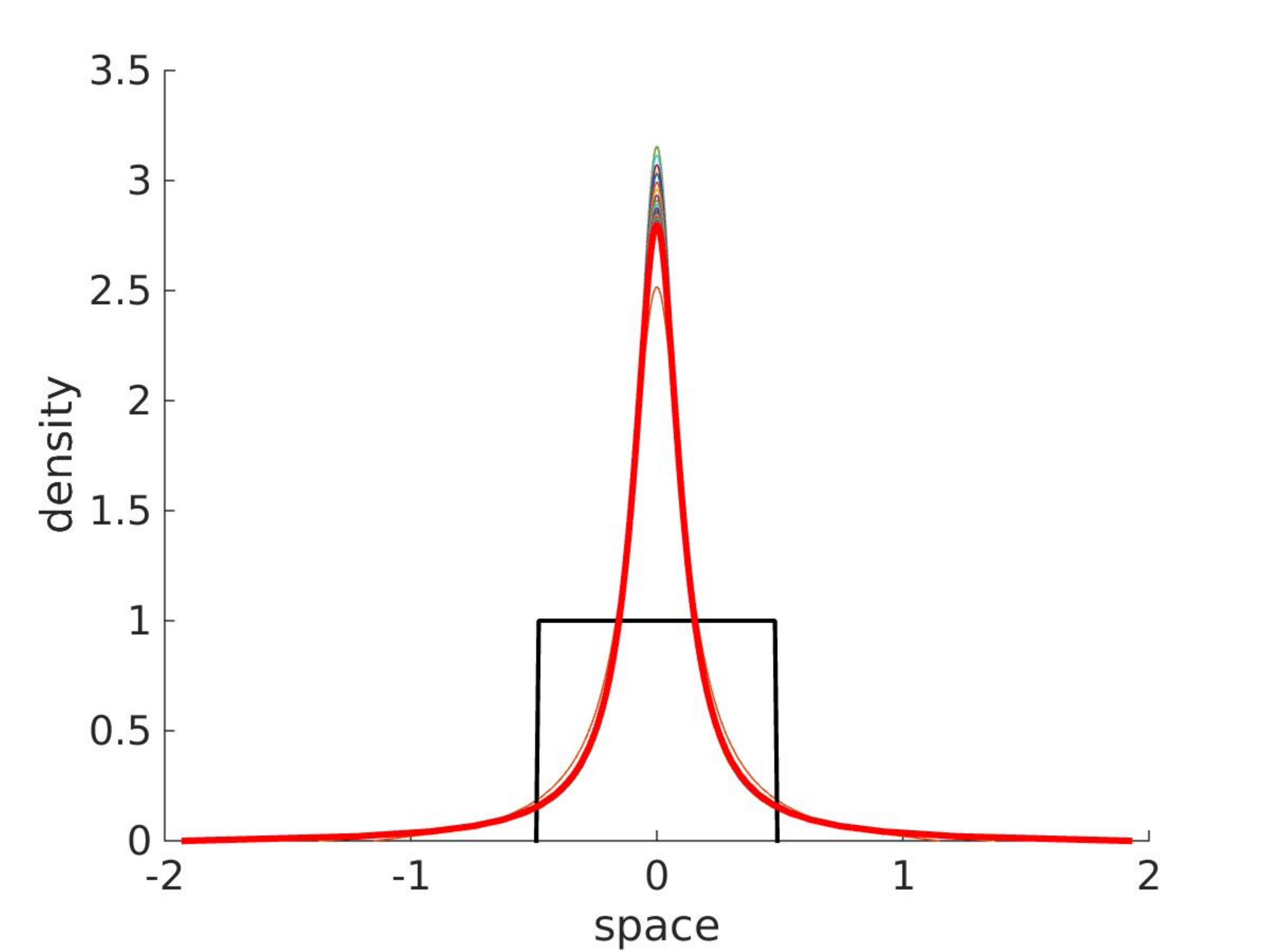}}%
% \vspace{-0.4cm}

\subfloat[]{\includegraphics[width=.45\textwidth]{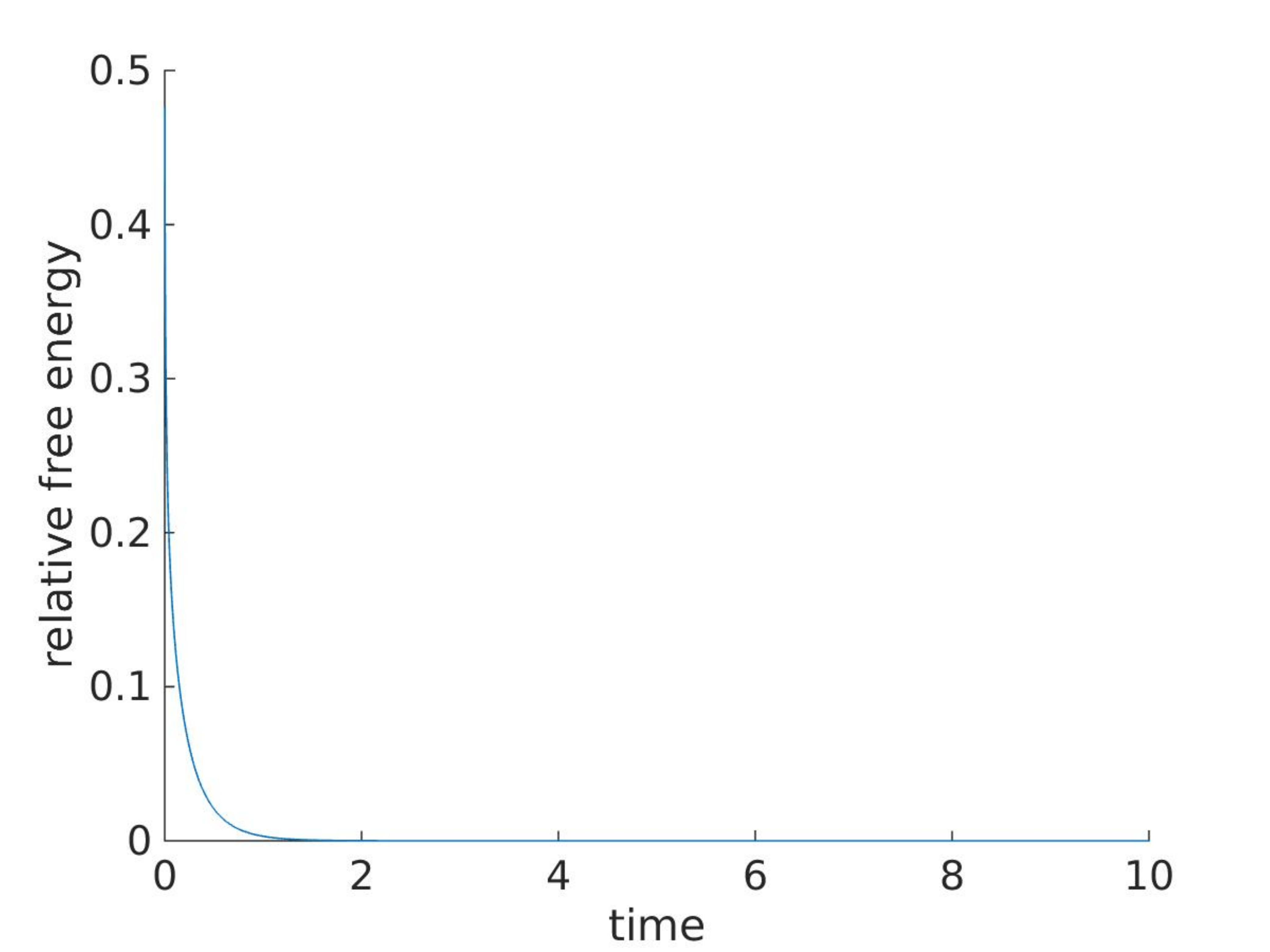}}%
\subfloat[]{\includegraphics[width=.45\textwidth]{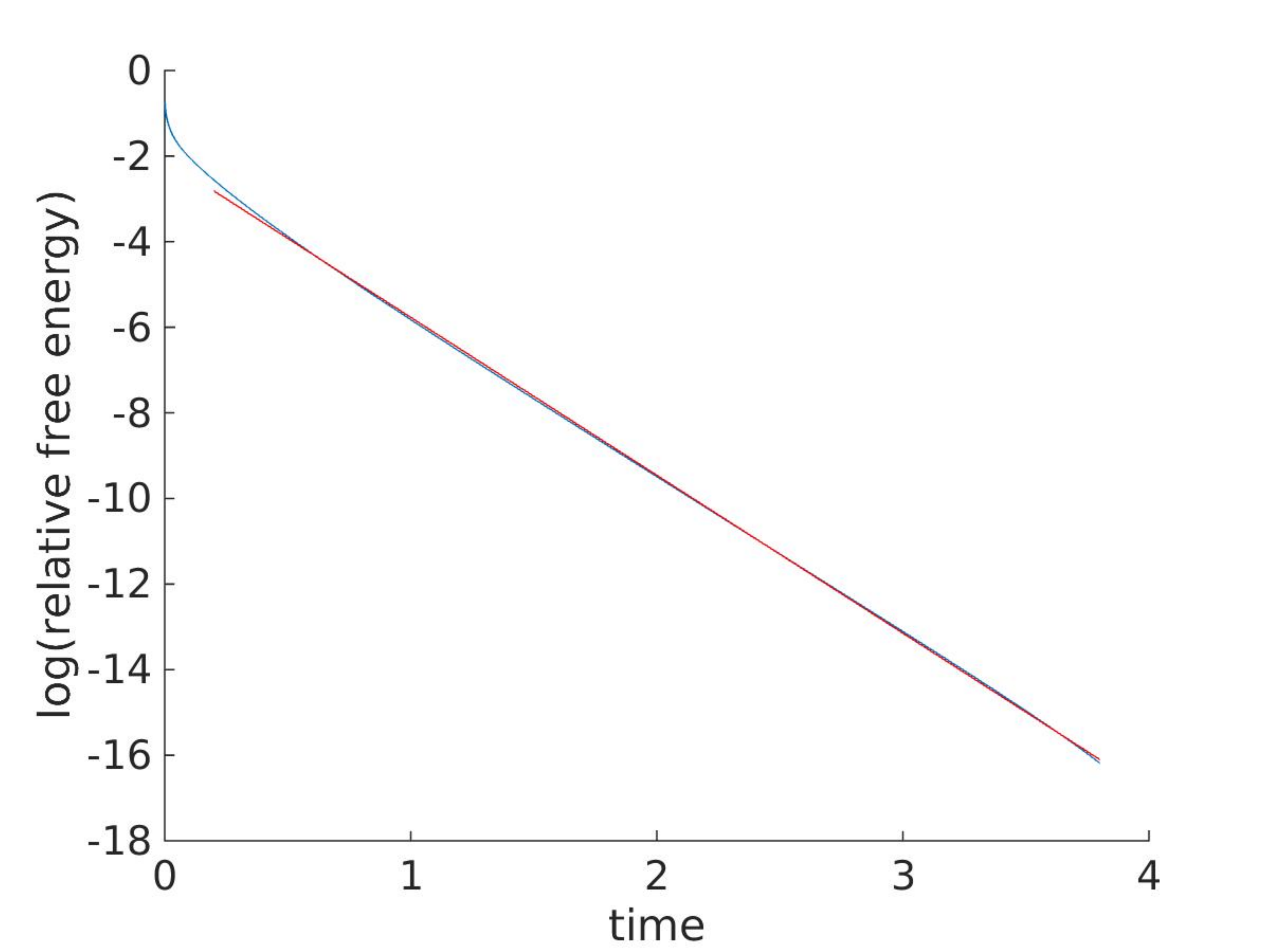}}%
%slope log(L2error)=-3.6898
% fitted between 0.3 and 3.5

% \vspace{-0.4cm}
\subfloat[]{\includegraphics[width=.45\textwidth]{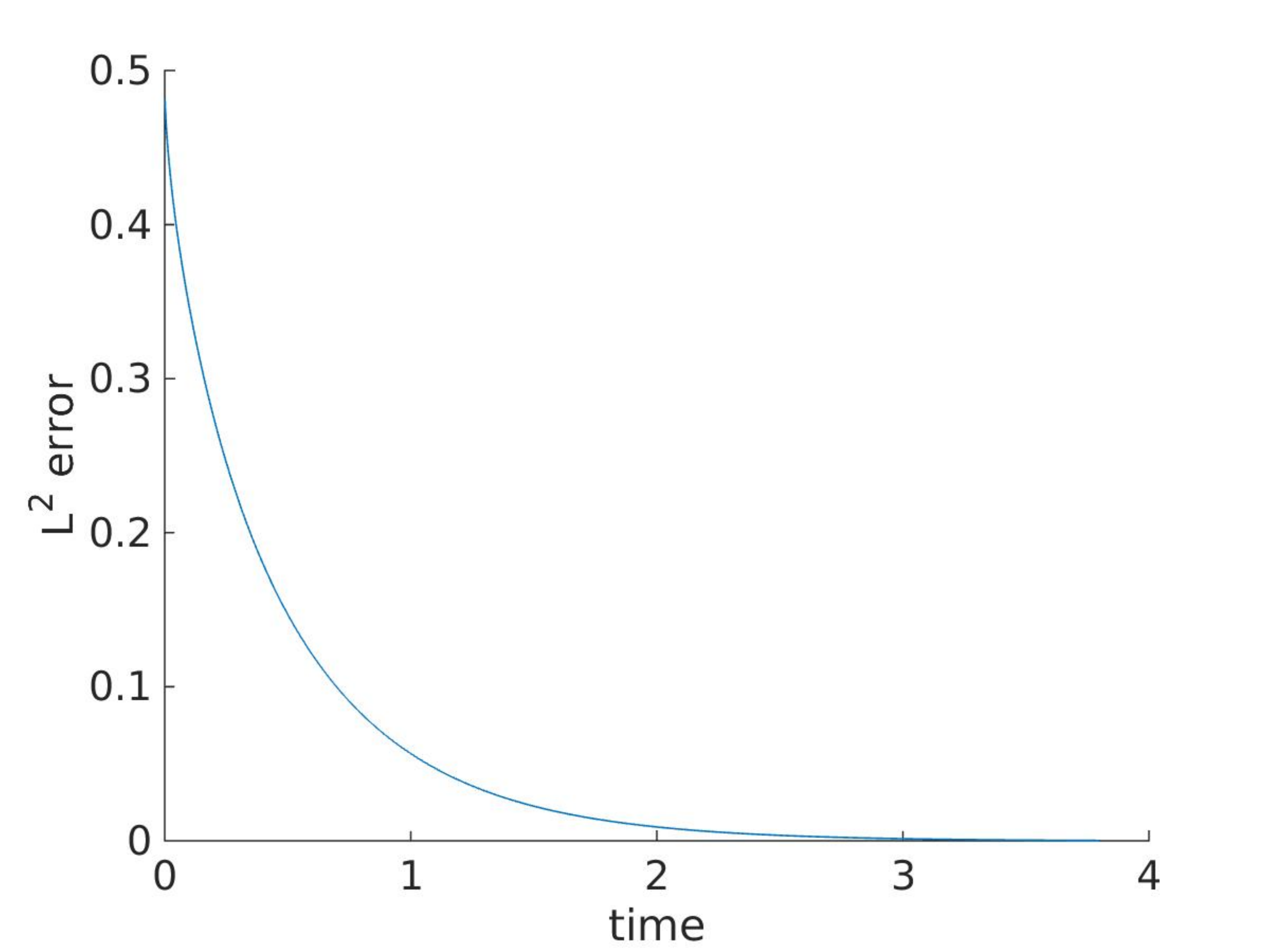}}%
\subfloat[]{\includegraphics[width=.45\textwidth]{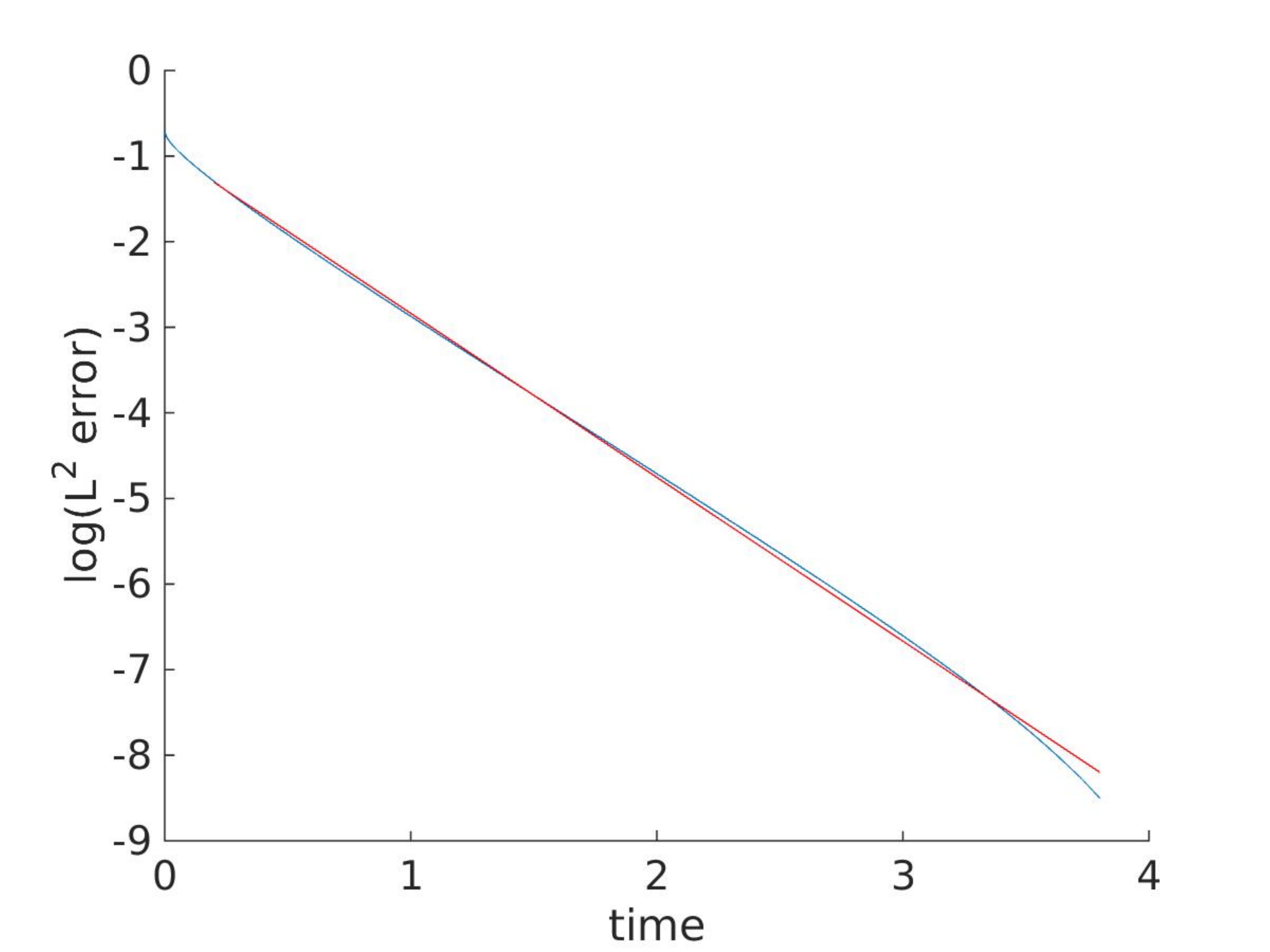}}%
%slope log(L2error)=-1.9593
% fitted between 0.3 and 3.5
% \vspace{-0.2cm}
\caption{Point $E$: $\chi=1.2$, $k=0.2$, $r=1$.\\
(a) Inverse cumulative distribution function from initial condition (black) to the profile at the last time step (red),
(b) solution density from initial condition (black) to the profile at the last time step (red),
(c) relative free energy,
(d) log(relative free energy) and fitted line between times $0.3$ and $3.5$ with slope $-3.6898$ (red),
(e) $L^2$-error between the solutions at time $t$ and at the last time step,
(f) log($L^2$-error) and fitted line between times $0.3$ and $3.5$ with slope $-1.9593$ (red).
}
\label{fig:chi=12_k=02_resc=1}
\end{figure}
% %%%%%%%%%%%%%%%%%%%%%%%%%%%%%%%%%%%%%%%%%%%
% %%%%%%%%%%%%%%%%%%%%%%%%%%%%%%%%%%%%%%%%%%%
% A logplot of the relative free energy shows an exponential rate of convergence of $a=-3.6898$ with fitted line $y=a*x+b$, $b=-4.1879$ between times $0.3 \leq t \leq 3.5$ (Figure \ref{fig:chi=12_k=02_resc=1}(d)).

\noindent
the stationary state for $\chi=1.2$ (Figure \ref{fig:chi=12_k=02_resc=1}(b)) is more concentrated than for $\chi=0.8$ (Figure \ref{fig:chi=08_k=02_resc=1}(b)), which is exactly what we would expect for decreasing $k$ as $\bar \rho$ approaches a Dirac Delta for $k \to 0$ if $\chi=1.2$, whereas it becomes compactly supported if $\chi=0.8$ as $k$ crosses the $\chi$-axis (see \cite[Corollary 3.9]{CCH1}). Again, we observe very fast convergence both in relative energy (Figure \ref{fig:chi=12_k=02_resc=1}(c)-(d)) and in Wasserstein distance to equilibrium (Figure \ref{fig:chi=12_k=02_resc=1}(e)-(f)) as predicted by Proposition \ref{prop:cvinW2}. A logplot of the relative free energy (Figure \ref{fig:chi=12_k=02_resc=1}(d)) and the Wasserstein distance to equilibrium (Figure \ref{fig:chi=12_k=02_resc=1}(f)) show exponential rates of convergence with rates $a=3.6898$ and $a=1.9593$ respectively for the fitted lines $y=-a*t+b$ and some constant $b$ between times $0.3 \leq t \leq 3.5$.

%%%%%%%%%%%%%%%%%%%%%%%%%%%%%%%%%%%%%%%%%%
%%%%%%%%%%%%%%%%%%%%%%%%%%%%%%%%%%%%%%%%%%
%%%% chi=10_k=-05_resc=0_run9
%%%% fig11
\begin{figure}[h!]
\centering
\subfloat[]{\includegraphics[width=.45\textwidth]{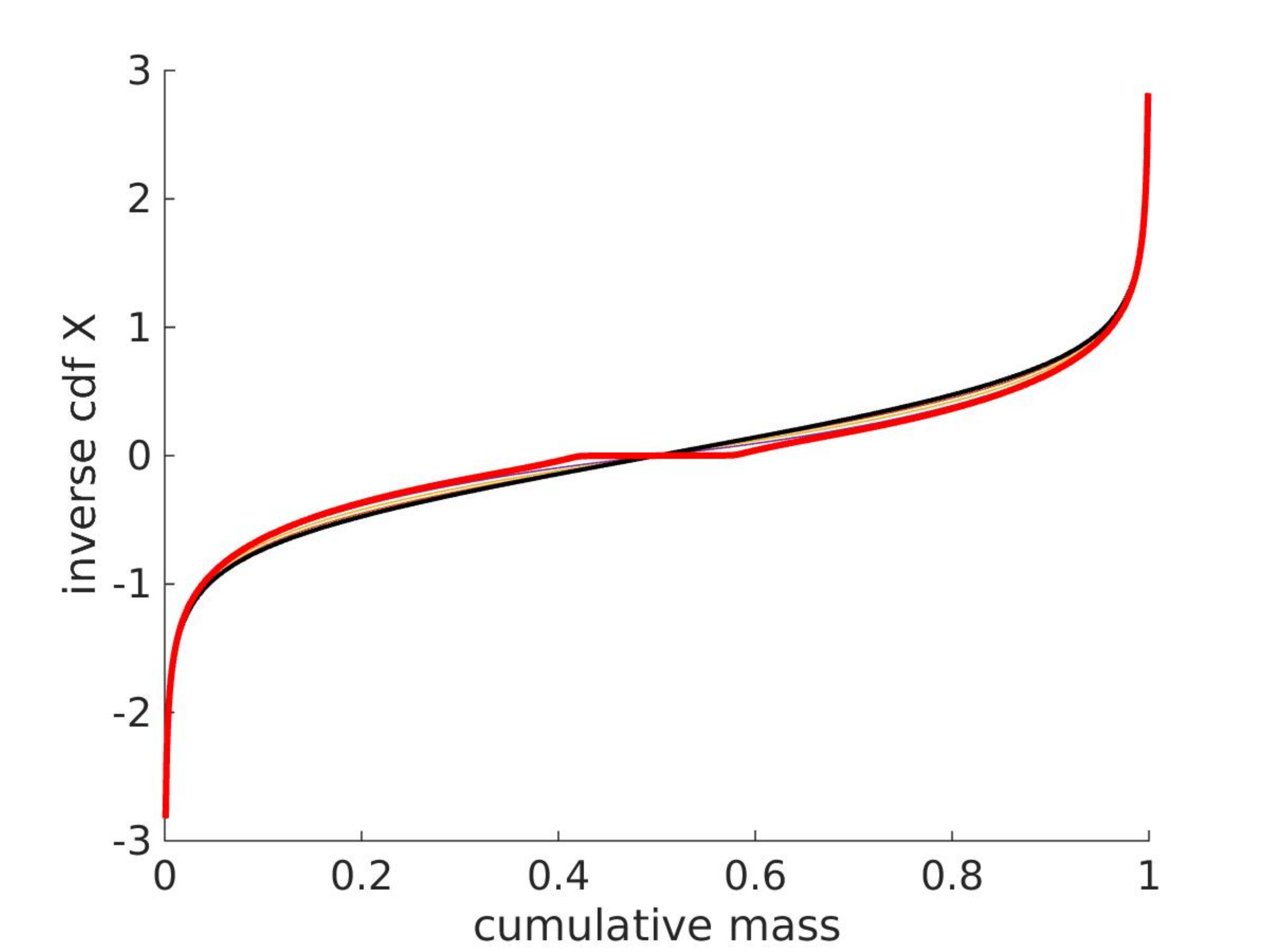}}%
\subfloat[]{\includegraphics[width=.45\textwidth]{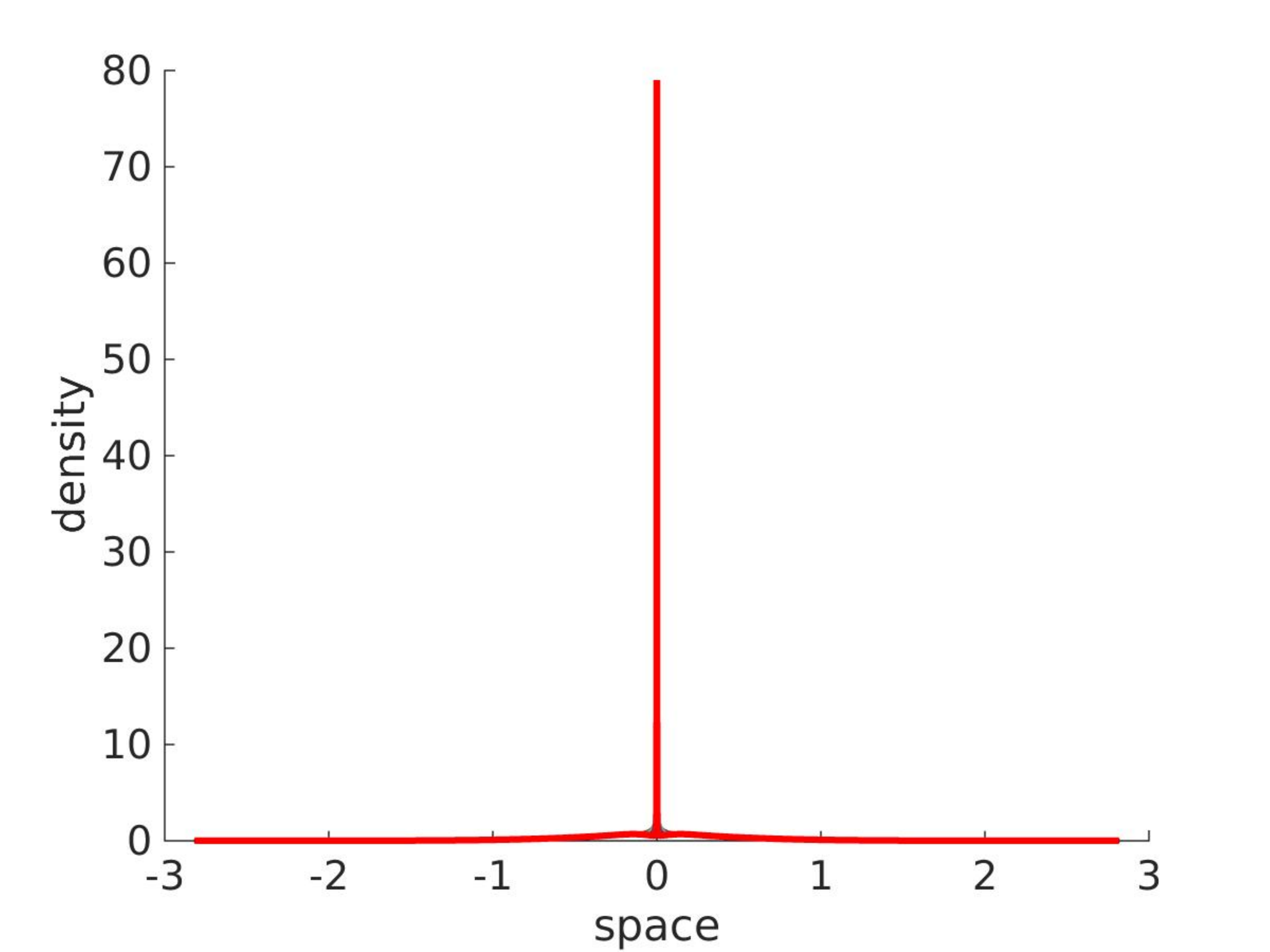}}%

 \vspace{-0.1cm}
\subfloat[]{\includegraphics[width=.45\textwidth]{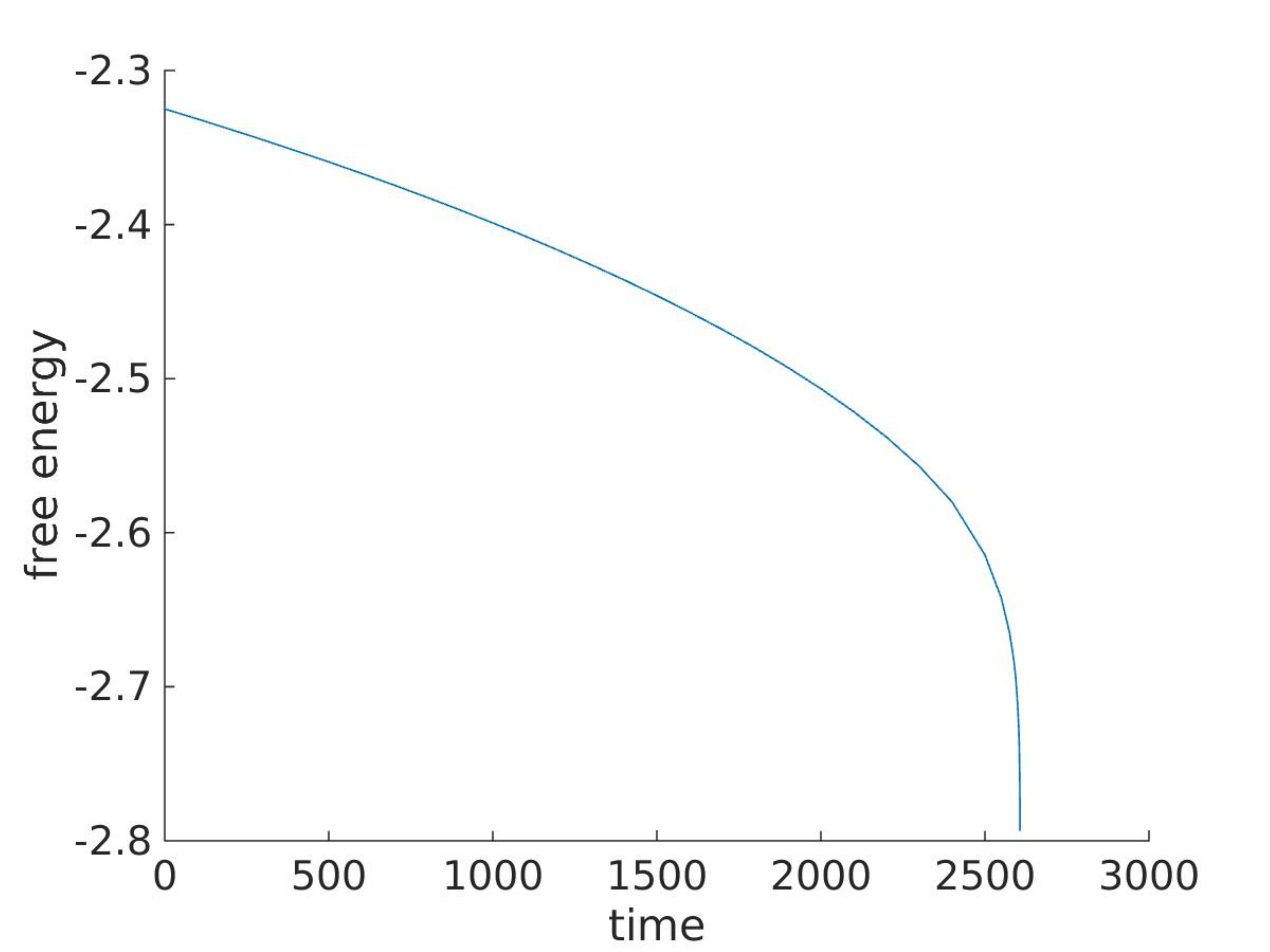}}%
 \vspace{-0.2cm}
\caption{Point $F$: $\chi=1$, $k=-0.5$, $r=0$.\\
(a) Inverse cumulative distribution function from initial condition (black) to the profile at the last time step (red),
(b) solution density from initial condition (black) to the profile at the last time step (red),
(c) free energy.}
\label{fig:chi=10_k=-05_resc=0}
\end{figure}
%%%%%%%%%%%%%%%%%%%%%%%%%%%%%%%%%%%%%%%%%%%%%%%%%
%%%%%%%%%%%%%%%%%%%%%%%%%%%%%%%%%%%%%%%%%%%%%%%%%
%%%% about fig11

Finally, let us investigate the behaviour for $(k,\chi)=(-0.5,1)$ in original variables (point $F$ in Figure \ref{fig:parameterspace}). Point $F$ lies in the porous medium regime and we expect blow-up as $\chi_c(-0.5)<1$, see Section \ref{sec: LTA super-critical k neg}. If the mass becomes too concentrated, the Newton-Raphson procedure does not converge and the simulation stops. We have therefore adapted the numerical scheme to better capture the blow-up. We fix $\Delta t=10^{-3}$ and $\Delta \eta=10^{-3}$ and take a centered normalised Gaussian with variance $\sigma^2=0.32$ as initial data. When the simulation stops, we divide the time step size $\Delta t$ by two and repeat the simulation, taking as initial condition the last density profile before blow-up. This process can be repeated any number of times, each time improving the approximation of an emerging Dirac Delta. The formation of a Dirac Delta in Figure \ref{fig:chi=10_k=-05_resc=0}(b) corresponds to the formation of a plateaux in \ref{fig:chi=10_k=-05_resc=0}(a). As expected from the analysis in Section \ref{sec: LTA super-critical k neg}, the free energy diverges to $-\infty$ (Figure \ref{fig:chi=10_k=-05_resc=0}(c)).

\section{Explorations in Other Regimes}\label{sec:otherregimes}

\subsection{Diffusion-Dominated Regime in One Dimension}

The numerical scheme described here gives us a tool to explore the asymptotic behaviour of solutions for parameter choices that are less understood. For example, choosing  $\chi=0.3$, $k=-0.5$ and $m=1.6$ in original variables ($r=0$), we observe convergence to a compactly supported stationary state, see Figure \ref{fig:chi=03_k=-05_m=16_resc=0}. This choice of parameters is within the diffusion-dominated regime since $m+k>1$ (see Definition \ref{def:classification regimes}). We choose as initial condition a normalised characteristic function supported on $B(0,15)$ from where we let the solution evolve with time steps of size $\Delta t=10^{-2}$ and particles spaced at $\Delta \eta=10^{-2}$. We let the density solution evolve until the $L^2$-error between two consecutive solutions is less than $10^{-7}$. Note that here $m+k=1.1$ is close to the fair-competition regime, for which $\chi_c\left(-0.5\right)=0.39$ (see Figure \ref{fig:parameterspace}).
%%%%%%%%%%%%%%%%%%%%%%%%%%%%%%%%%%%%%%%%%%
%%%%%%%%%%%%%%%%%%%%%%%%%%%%%%%%%%%%%%%%%%
%%%% chi=03_k=-05_m=16_run2
%%%% fig11
\begin{figure}[h!]
\centering
\subfloat[]{\includegraphics[width=.45\textwidth]{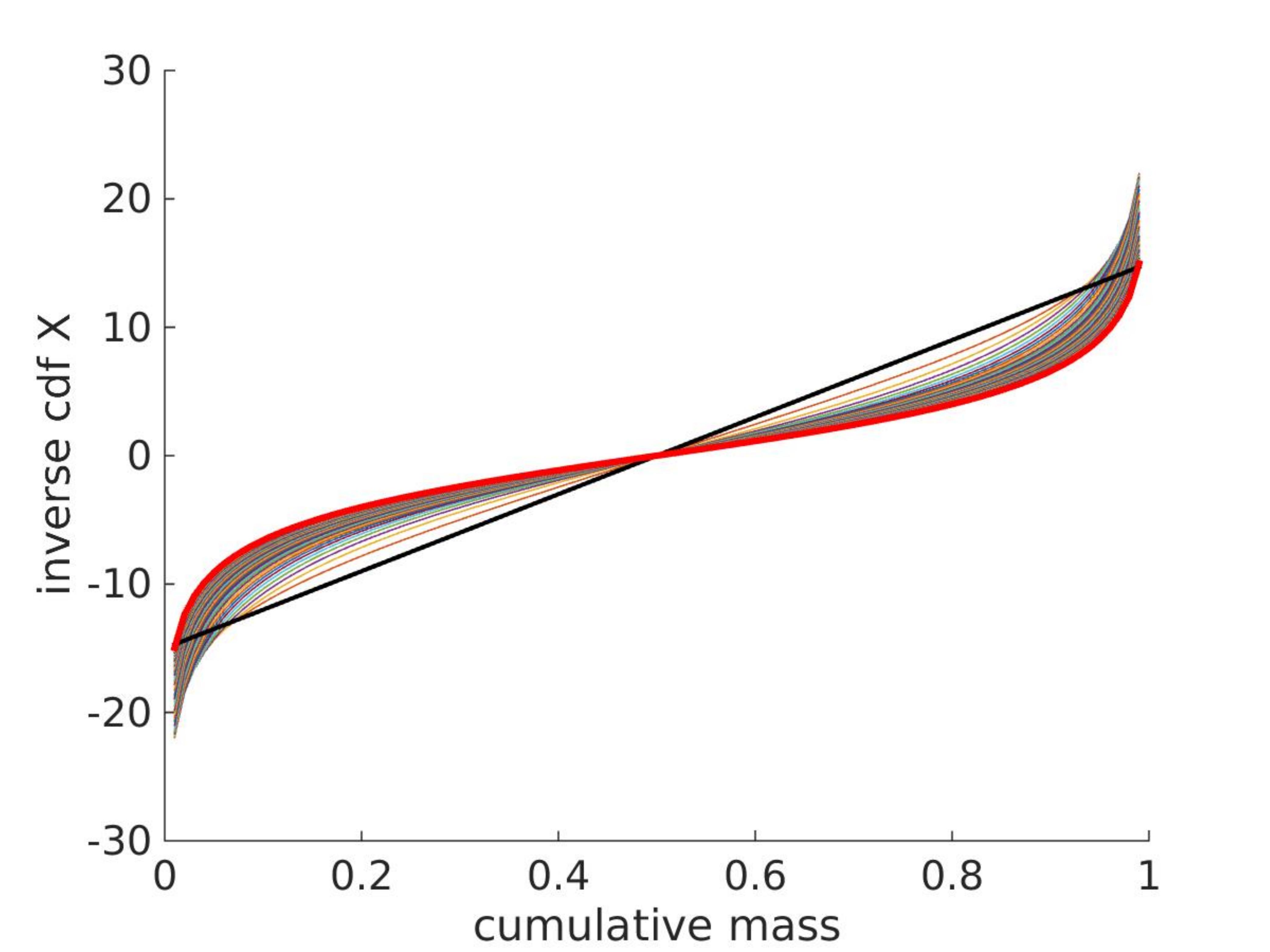}}%
\subfloat[]{\includegraphics[width=.45\textwidth]{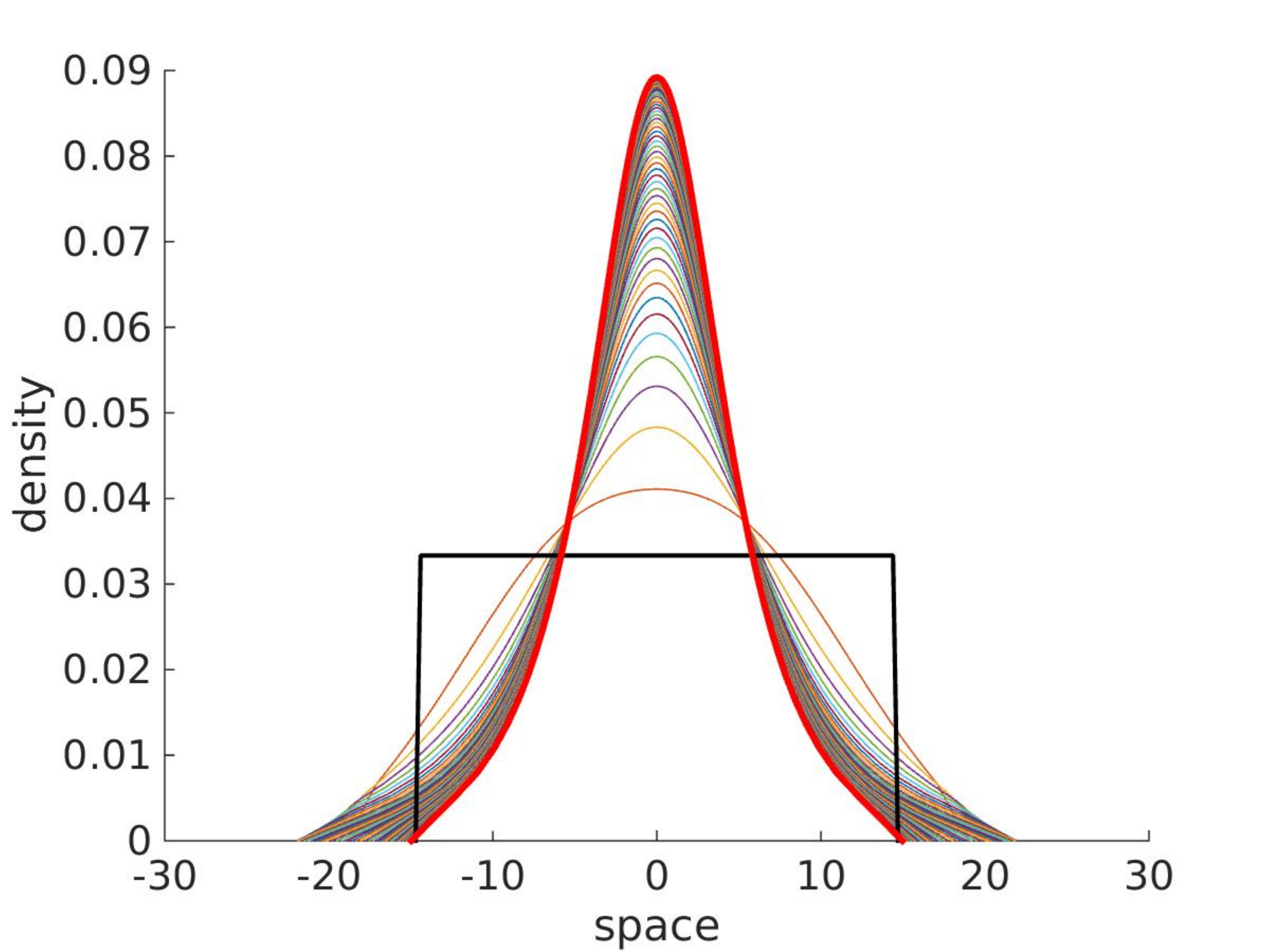}}%
\vspace{-0.4cm}
\subfloat[]{\includegraphics[width=.45\textwidth]{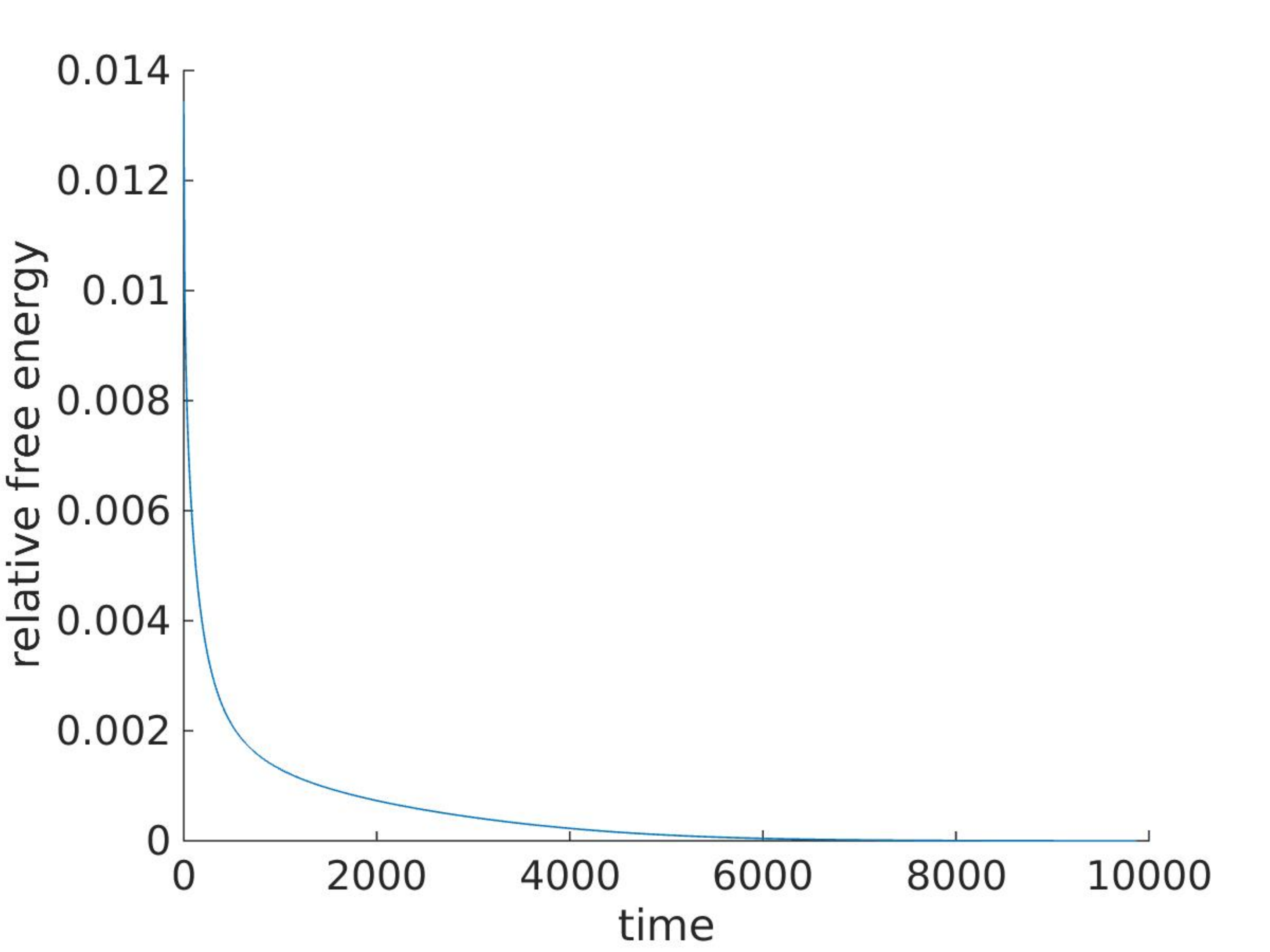}}%
\subfloat[]{\includegraphics[width=.45\textwidth]{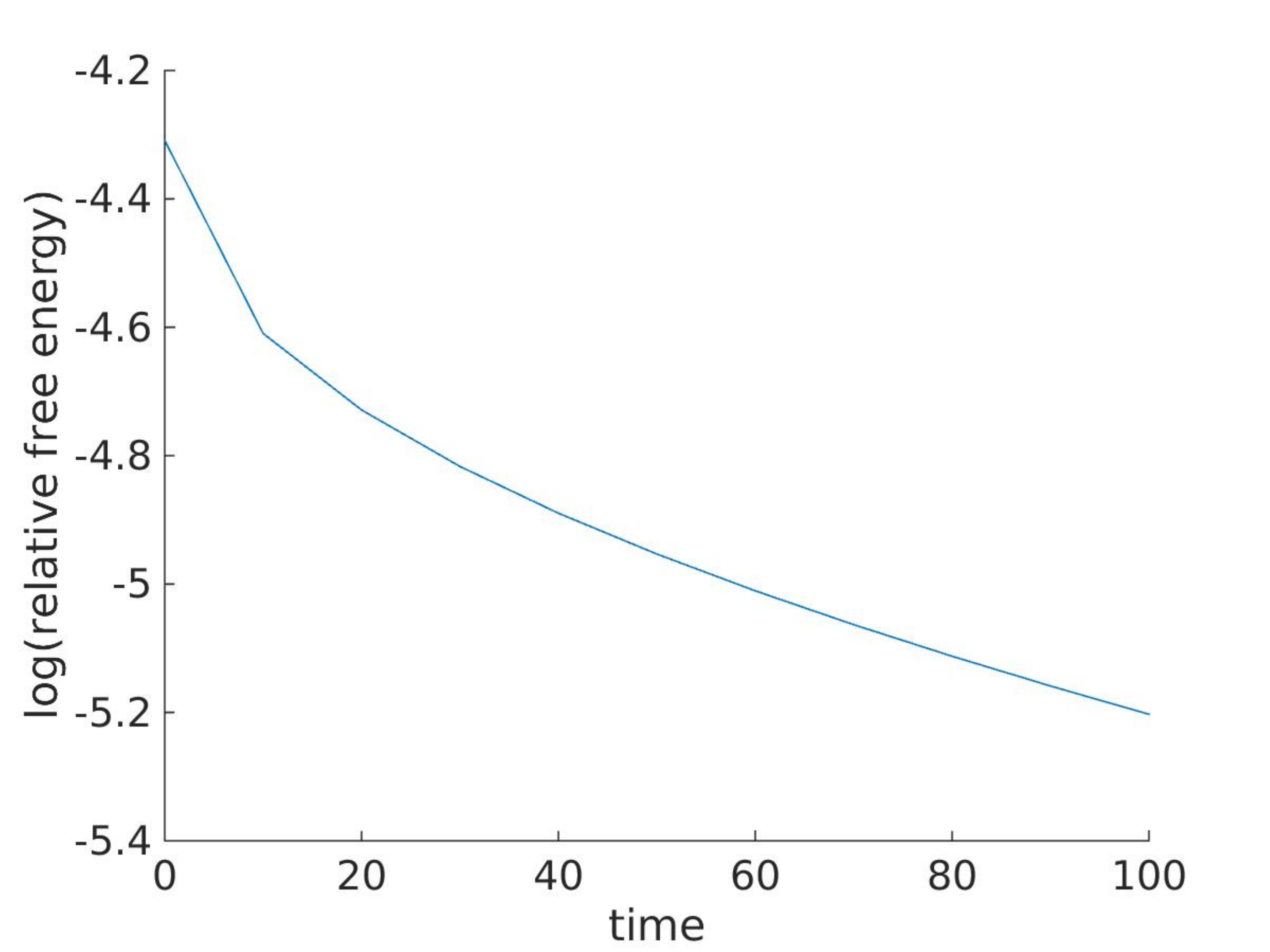}}%
\vspace{-0.2cm}
\caption{Diffusion-dominated regime: $\chi=0.3$, $k=-0.5$, $m=1.6$, $r=0$.\\
(a) Inverse cumulative distribution function from initial condition (black) to the profile at the last time step (red),
(b) solution density from initial condition (black) to the profile at the last time step (red),
(c) relative free energy,
(d) log(relative free energy).}
\label{fig:chi=03_k=-05_m=16_resc=0}
\end{figure}
%%%%%%%%%%%%%%%%%%%%%%%%%%%%%%%%%%%%%%%%%%%%
%%%%%%%%%%%%%%%%%%%%%%%%%%%%%%%%%%%%%%%%%%%%
%%%%%%%%%%%%%%%%%%%%%%%%%%%%%%%%%%%%%%%%%%%%

\subsection{Attraction-Dominated Regime in Any Dimension}

In the attraction-dominated regime $N(m-1)+k<0$ (corresponding to Definition \ref{def:classification regimes}) both global existence of solutions and blow-up can occur in original variables in dimension $N\geq1$ depending on the choice of initial data \cite{CPZ04, Sugi3, SugiKu06, CLW,BL,CW,LW, CaCo14}. Using the numerical scheme introduced in Section \ref{sec:numerics}, we can demonstrate this change of behaviour numerically in one dimension, see Figures \ref{fig:conj_chi=035_k=-05_m=133_run9b} (dispersion) and \ref{fig:conj_chi=035_k=-05_m=133_run8} (blow-up).\\

We will now investigate in more detail a special parameter choice ($m$, $k$) that belongs to the attraction-dominated regime. Instead of fixing $m$ and $k$ such that attractive and repulsive forces are in balance ($N(m-1)+k=0$), one may choose instead to investigate the regime where the free energy functional \eqref{eq:functional} is conformal invariant, corresponding to $m=2N/(2N+k)$. For $k<0$, this corresponds to the case $p=q=m$ in the HLS inequality \eqref{eq:HLS} for which the otimisers $\rho_{HLS}$ and the optimal constant $C_{HLS}$ are known explicitly \cite{Lieb83}. We have the following existence result:
\begin{theorem}
 Let $\chi>0$, $k \in (-N,0)$ and $m=2N/(2N+k)\in(1,2)$. Then the free energy functional $\mF_k$ admits a critical point in $\mY$.
\end{theorem}

\begin{proof}
Following the approach in \cite{CLW}, we rewrite the free energy functional \eqref{eq:functional} as a sum of two functionals
\begin{equation*}
 \mF_k[\rho]=\mF_k^1[\rho]+\mF_k^2[\rho]\, ,
\end{equation*}
where
\begin{align}
 \mF_k^1[\rho]:= &\frac{1}{N(m-1)} ||\rho||_m^m \left(1-\chi C_{HLS}\frac{N(m-1)}{(-k)}||\rho||_m^{2-m}\right)\notag\\
 =&\frac{2N+k}{N(-k)} ||\rho||_m^m \left(1-\chi C_{HLS}\frac{N}{2N+k}||\rho||_m^{2-m}\right)\, ,\label{eq:F1}
 \end{align}
and
 \begin{align}
 \mF_k^2[\rho]:=&\frac{\chi}{(-k)} \left(C_{HLS}||\rho||_m^2 - \iint_{\RR^N\times\RR^N} |x-y|^k\rho(x)\rho(y)\, dxdy\right)\, .\label{eq:F2}
\end{align}
By the HLS inequality \eqref{eq:HLS}, the second functional \eqref{eq:F2} is bounded below for any $\chi>0$,
$$
\mF_k^2[\rho]\geq0\, , \quad \forall \rho \in \mY\, ,
$$
and by \cite[Theorem 3.1]{Lieb83}, there exists a family of optimisers $\rho_{HLS,\lambda,c}$,
\begin{equation}\label{eq:rhoHLS}
 \rho_{HLS,\lambda,c}(x)=c\left(\frac{\lambda}{\lambda^2+|x|^2}\right)^{N/m}\, , \quad \lambda>0, c>0
\end{equation}
satisfying 
$
\mF_k^2\left[\rho_{HLS,\lambda,c}\right]=0
$
with the optimal constant $C_{HLS}$ given by
\begin{equation*}%\label{CHLS}
 C_{HLS}:=\pi^{-k/2} \left(\frac{\Gamma\left(\frac{N+k}{2}\right)}{\Gamma\left(N+\frac{k}{2}\right)}\right)
 \left(\frac{\Gamma\left(\frac{N}{2}\right)}{\Gamma\left(N\right)}\right)^{-(N+k)/N}\, .
\end{equation*}
The parameter $\lambda>0$ in \eqref{eq:rhoHLS} corresponds to the scaling that leaves the $L^m$-norm of $\rho_{HLS,\lambda,c}$ invariant. Since the first variation of the functional $\mF_k^1$ defined in \eqref{eq:F1} is given by
\begin{equation*}
 \frac{\delta \mF_k^1}{\delta \rho}[\rho](x)
 = \frac{2}{(-k)}\left(1-\chi C_{HLS} ||\rho||_m^{2-m}\right)\rho^{m-1}(x)
\end{equation*}
and since the $L^m$-norm of the optimiser can be calculated explicitly,
$$
||\rho_{HLS,\lambda,c}||_m=
c\left(\frac{2^{1-N} \pi^{\frac{N+1}{2}}}{\Gamma\left(\frac{N+1}{2}\right)}\right)^{1/m}\, ,
$$
there exists a unique choice of $(\lambda, c)=(\lambda^*, c^*)$ for each $\chi>0$ such that 
\begin{align*}
 \frac{\delta \mF_k^1}{\delta \rho}[\rho_{HLS,\lambda^*,c^*}](x)=0\, \quad
 \text{and} \quad
 \int_{\RR^N} \rho_{HLS,\lambda^*,c^*}(x)\,dx=1
\end{align*}
given by
\begin{equation}\label{eq:c*lambda*}
c^*(\chi):=\left(\frac{2^{1-N} \pi^{\frac{N+1}{2}}}{\Gamma\left(\frac{N+1}{2}\right)}\right)^{-1/m}\left(\chi C_{HLS}\right)^{1/(m-2)}\, , \quad
\lambda^*(\chi):=\left(\int_{\RR^N} \rho_{HLS,1,c^*(\chi)}(x)\,dx\right)^{2/k}\, .
%  \left(\lambda^*(\chi), c^*(\chi)\right)
%  =\left(\left(\int_{\RR^N} \rho_{HLS,1,c^*}(x)\,dx\right)^{2/(2+k)}, \left(\chi C(m,N)\right)^{1/(m-2)}\right)
\end{equation}
% with
% $$
% C(m,N):= \pi^{\left(\frac{N-1}{2}+\frac{1}{m}\right)} 2^{\left(\frac{(1-N)(2-m)}{m}\right)}\Gamma\left(\frac{N+1}{2}\right)^{\frac{m-2}{m}} 
% \left(\frac{\Gamma\left(\frac{N+k}{2}\right)}{\Gamma\left(\frac{2N+k}{2}\right)}\right)
%  \left(\frac{\Gamma\left(\frac{N}{2}\right)}{\Gamma\left(N\right)}\right)^{-(N+k)/N}\, .
% $$
Hence $\rho_{HLS,\lambda^*,c^*}$ is a critical point of $\mF_k$ in $\mY$.
\end{proof}

We can choose to leave $\lambda>0$ as a free parameter in \eqref{eq:rhoHLS}, only fixing $c=c^*(\chi)$ so that $\rho_{HLS,\lambda,c^*}$ is a critical point of $\mF_k$ with arbitrary mass. We conjecture that a similar result to \cite[Theorem 2.1]{CLW} holds true for general $k \in (-N,0)$ and $m=2N/(2N+k)$ for radially symmetric initial data:
%%%%%%%%%%%%%%%%%%%%%%%%%%%%%%%%%%%%%%%%%%
%%%%%%%%%%%%%%%%%%%%%%%%%%%%%%%%%%%%%%%%%%
%%%% conj_chi=035_k=-05_m=133_run9b
%%%% fig12
\begin{figure}[h!]
\centering
\subfloat[]{\includegraphics[width=.45\textwidth]{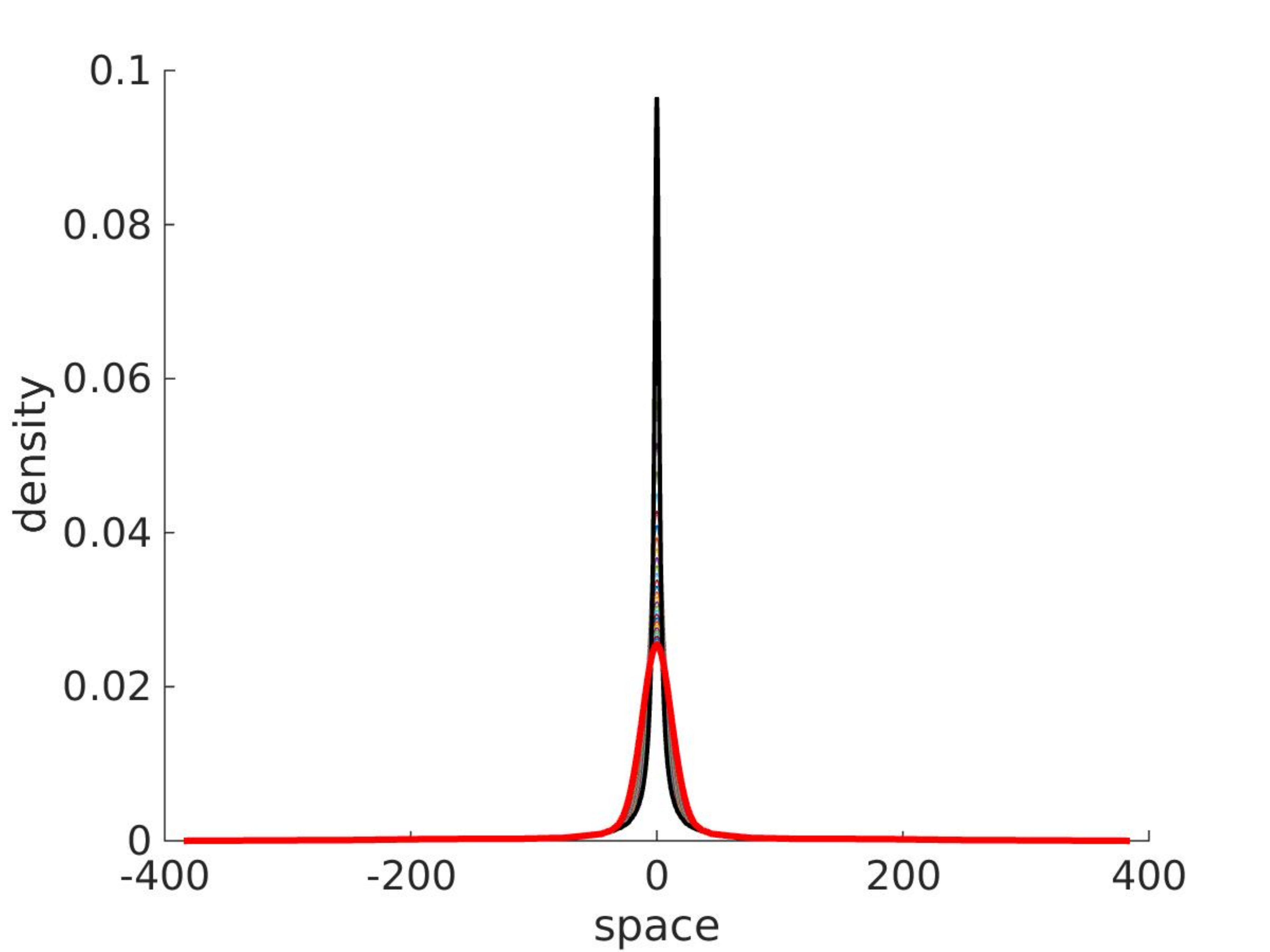}}%
\subfloat[]{\includegraphics[width=.45\textwidth]{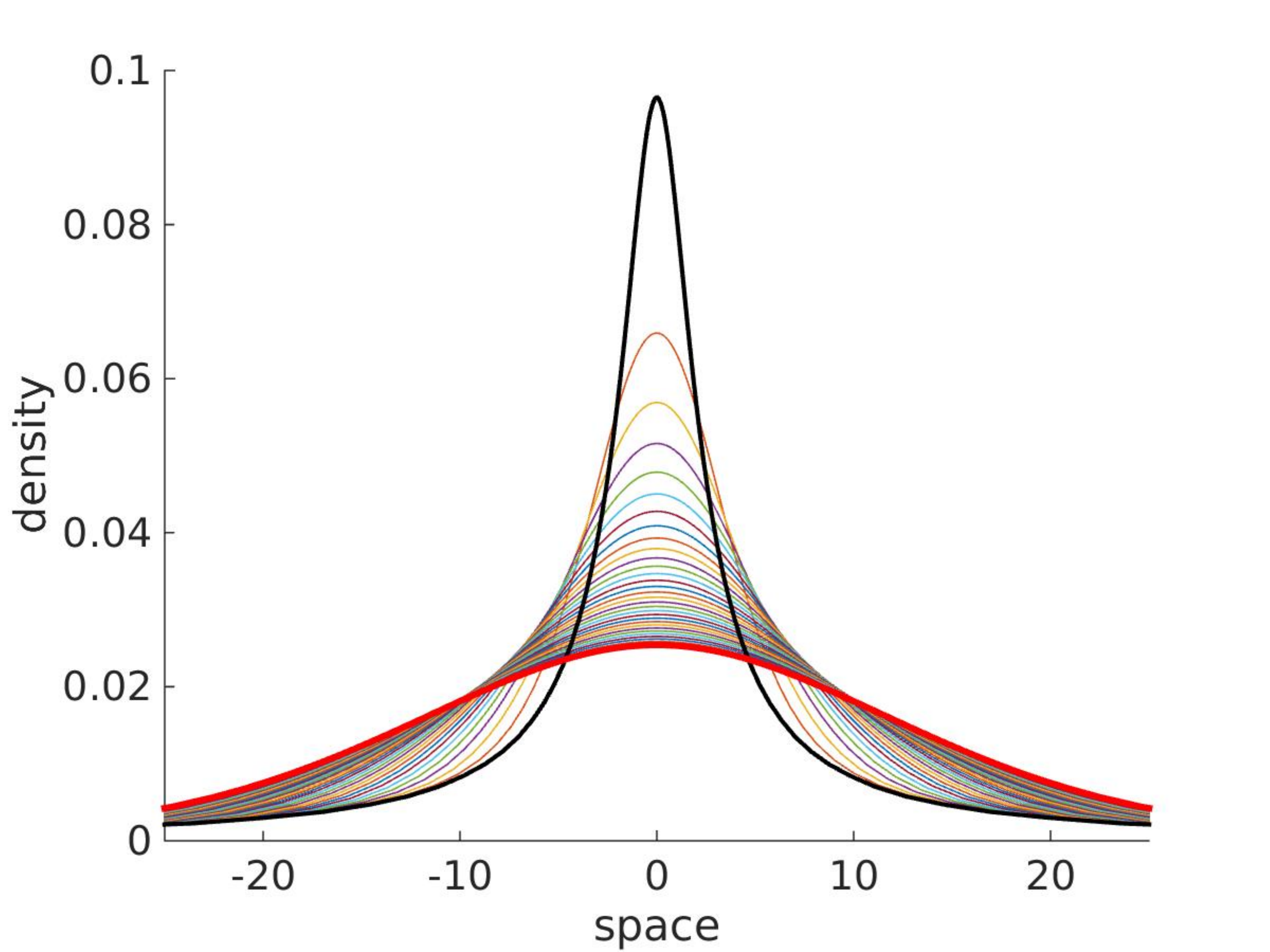}}%
 \vspace{-0.4cm}
\subfloat[]{\includegraphics[width=.45\textwidth]{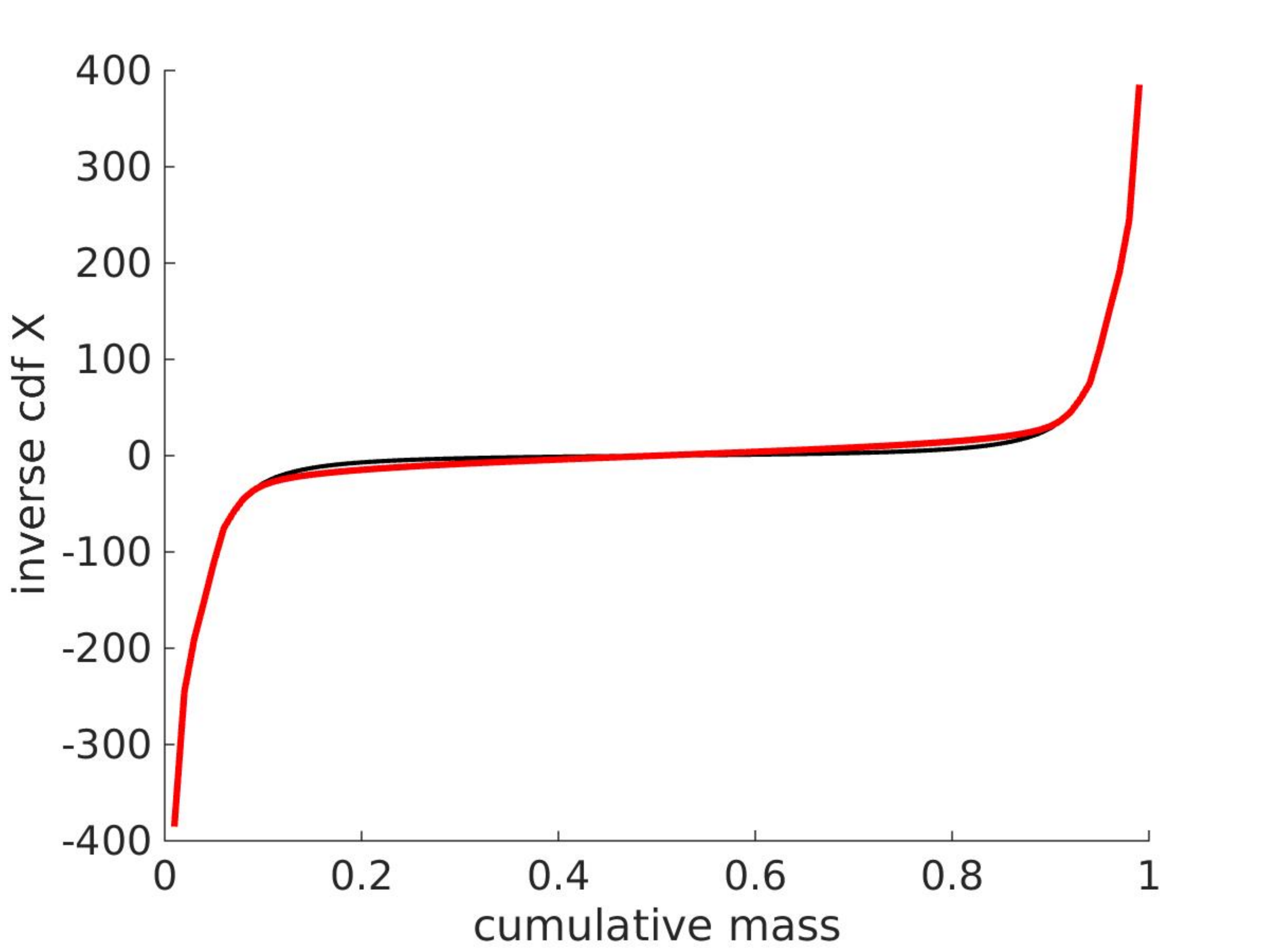}}%
\subfloat[]{\includegraphics[width=.45\textwidth]{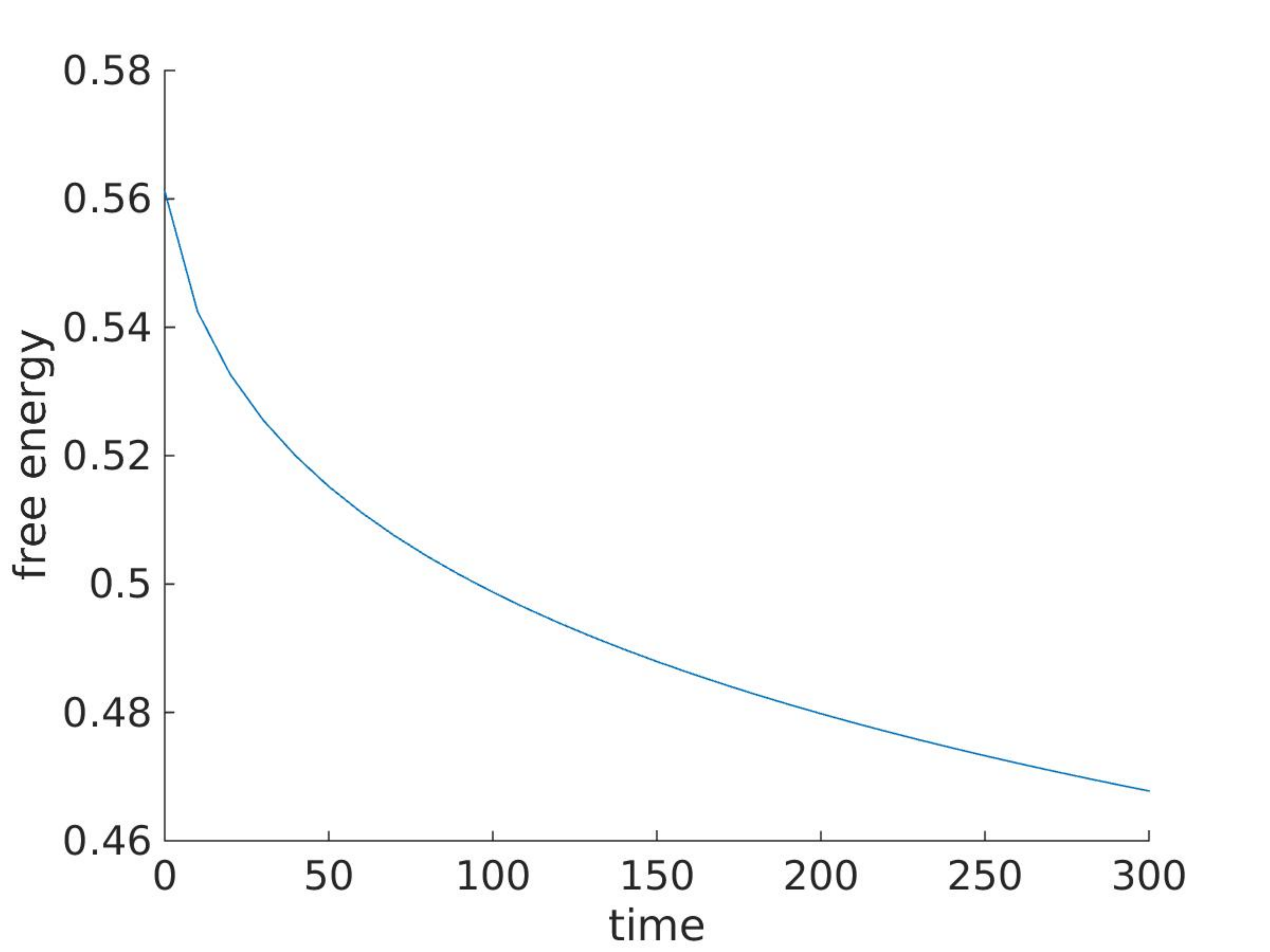}}%
 \vspace{-0.4cm}
% \subfloat[]{\includegraphics[width=.33\textwidth]{conj_chi=035_k=-05_m=133_run9b_fig6.pdf}}%
% \hfill
\subfloat[]{\includegraphics[width=.45\textwidth]{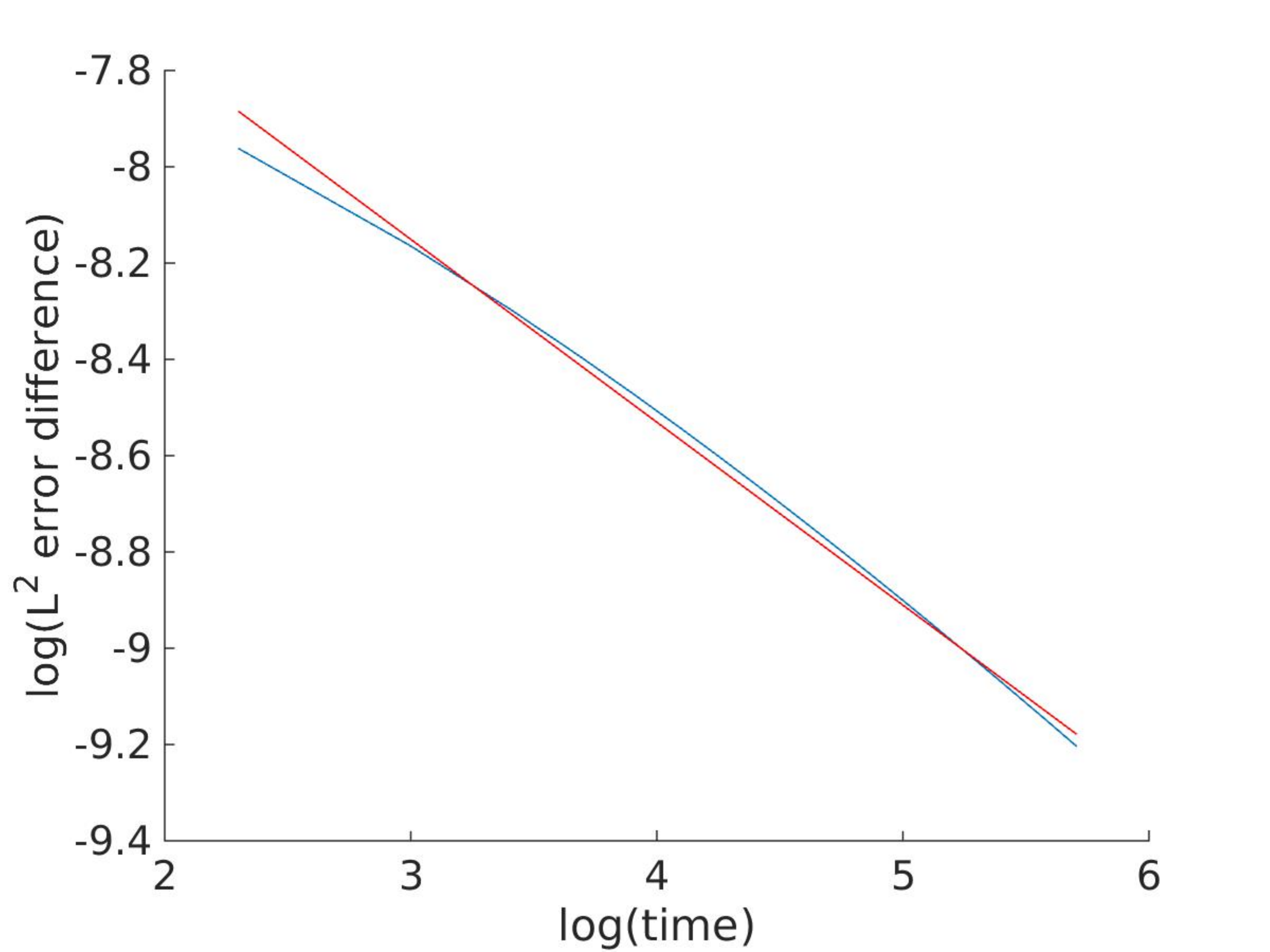}}%
\subfloat[]{\includegraphics[width=.45\textwidth]{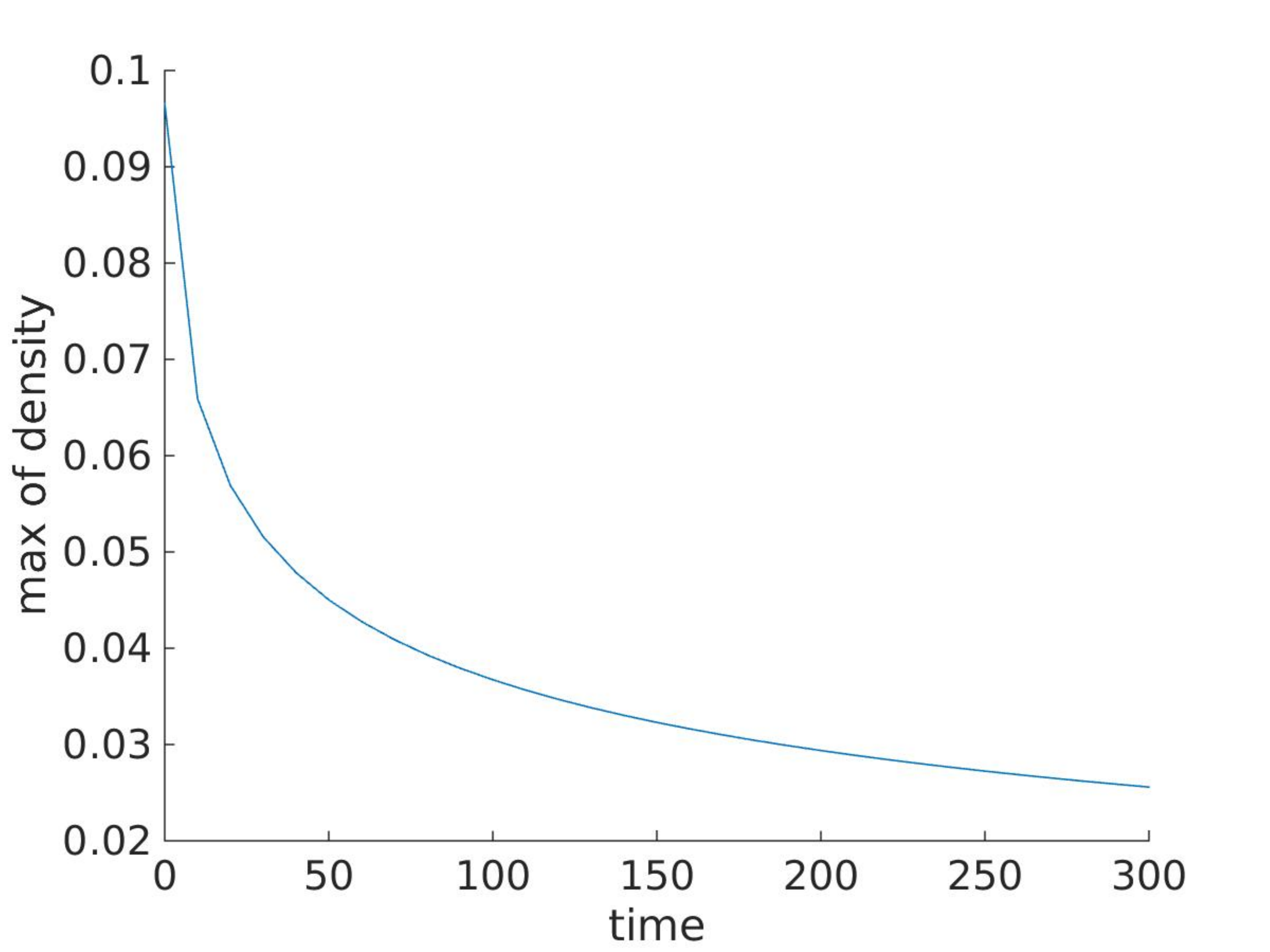}}%
 \vspace{-0.2cm}
\caption{Attraction-dominated regime: $\chi=0.35$, $k=-1/2$, $m=4/3$, $r=0$ with initial data $\rho(t=0,x)=\rho_{HLS,\lambda_0,c_0}(x)<\rho_{HLS,\lambda_0,c^*}(x)$ for all $x \in \RR$ with $c_0=0.4c^*$.\\
(a) Solution density from initial condition (black) to the profile at the last time step (red), 
(b) zoom of Figure (a),
(c) inverse cumulative distribution function from initial condition (black) to the profile at the last time step (red),
(d) free energy,
(e) log-log plot of the $L^2$-error difference between two consecutive solutions and fitted line with slope $-0.37987$,
(f) time evolution of $\max_{x}\rho(t,x)$
.}
\label{fig:conj_chi=035_k=-05_m=133_run9b}
\end{figure}
%%%%%%%%%%%%%%%%%%%%%%%%%%%%%%%%%%%%%%%%%%%%
%%%%%%%%%%%%%%%%%%%%%%%%%%%%%%%%%%%%%%%%%%%%
\begin{conjecture}[Global Existence vs Blow-up]\label{conj1}
 Let $\chi>0$, $k\in(-N,0)$ and $m=2N/(2N+k)$ in dimension $N\geq 1$. Assume the initial datum $\rho_0 \in \mY$ is radially symmetric.
 \begin{enumerate}[(i)]
  \item If there exists $\lambda_0>0$ such that 
  $$
  \rho_0(r)< \rho_{HLS,\lambda_0,c^*}(r)\, , \quad \forall\, r\geq 0, ,
  $$
  then any radially symmetric solution $\rho(t,r)$ of \eqref{eq:KS} with initial datum $\rho(0,r)=\rho_0(r)$ is vanishing in $L^1_{loc}\left(\RR^N\right)$ as $t \to \infty$.
   \item If there exists $\lambda_0>0$ such that 
  $$
  \rho_0(r)> \rho_{HLS,\lambda_0,c^*}(r)\, , \quad \forall\, r\geq 0\, ,
  $$
  then any radially symmetric solution $\rho(t,r)$ of \eqref{eq:KS} with initial datum $\rho(0,r)=\rho_0(r)$ must blow-up at a finite time $T^*$ or has a mass concentration at $r=0$ as time goes to infinity in the sense that there exist $R(t)\to 0$ as $t \to \infty$ and a positive constant $C>0$ such that
  $$
  \int_{B(0,R(t))} \rho(t,x)\, dx\geq C\, .
  $$
 \end{enumerate}
\end{conjecture}
%%%%%%%%%%%%%%%%%%%%%%%%%%%%%%%%%%%%%%%%%%%%
Further, we expect the following to be true analogous to \cite{CLW}:
\begin{conjecture}[Unstable Stationary State]\label{conj2}
 For any $\chi>0$, the density $\rho_{HLS,\lambda^*,c^*} \in \mY$ with $(\lambda^*,c^*)$ given by \eqref{eq:c*lambda*} is an unstable stationary state of equation \eqref{eq:KS}.
\end{conjecture}
%%%%%%%%%%%%%%%%%%%%%%%%%%%%%%%%%%%%%%%%%%
%%%%%%%%%%%%%%%%%%%%%%%%%%%%%%%%%%%%%%%%%%
%%%% conj_chi=035_k=-05_m=133_run8
%%%% fig13
\begin{figure}[h!]
\centering
\subfloat[]{\includegraphics[width=.45\textwidth]{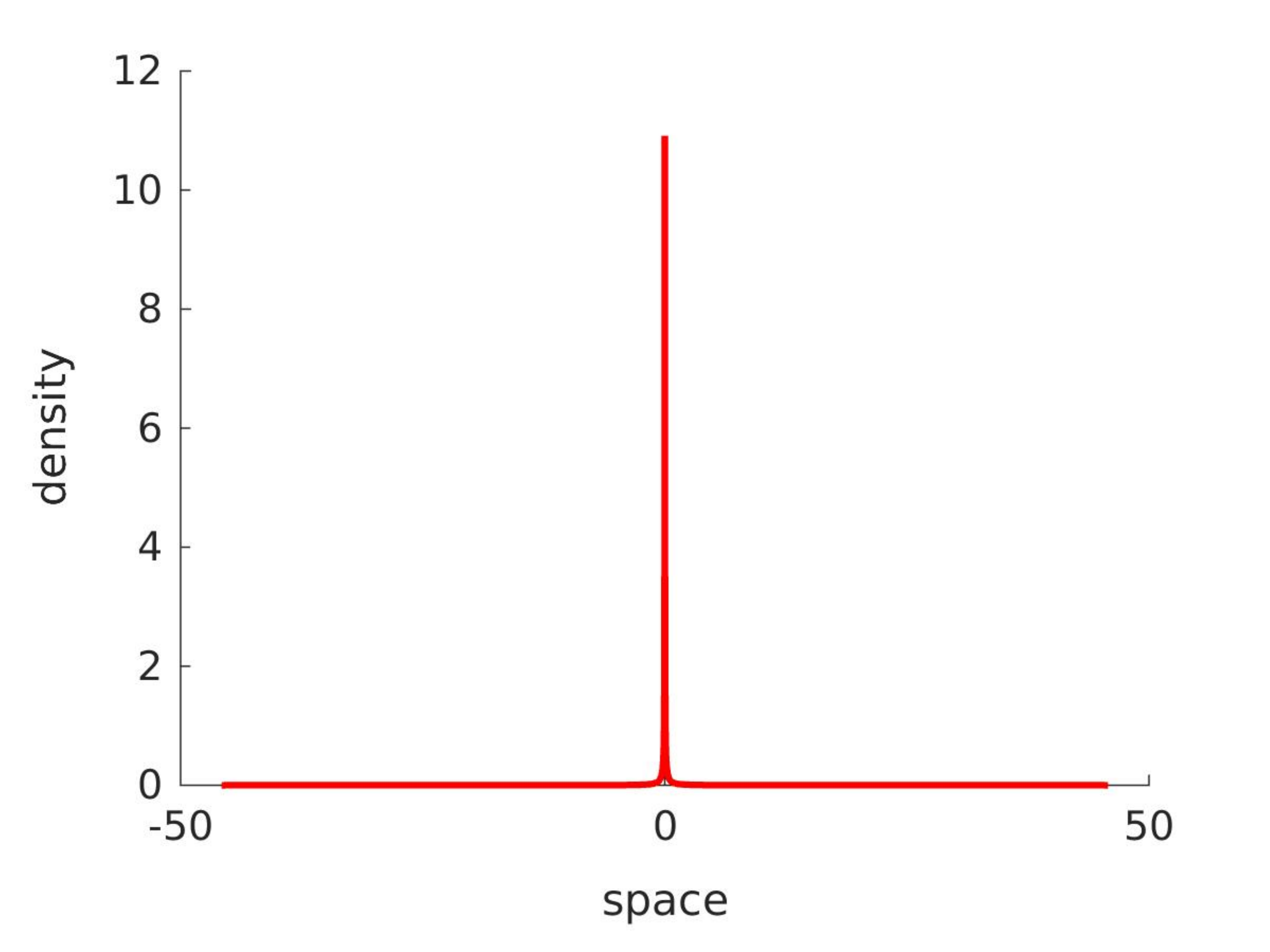}}%
\subfloat[]{\includegraphics[width=.45\textwidth]{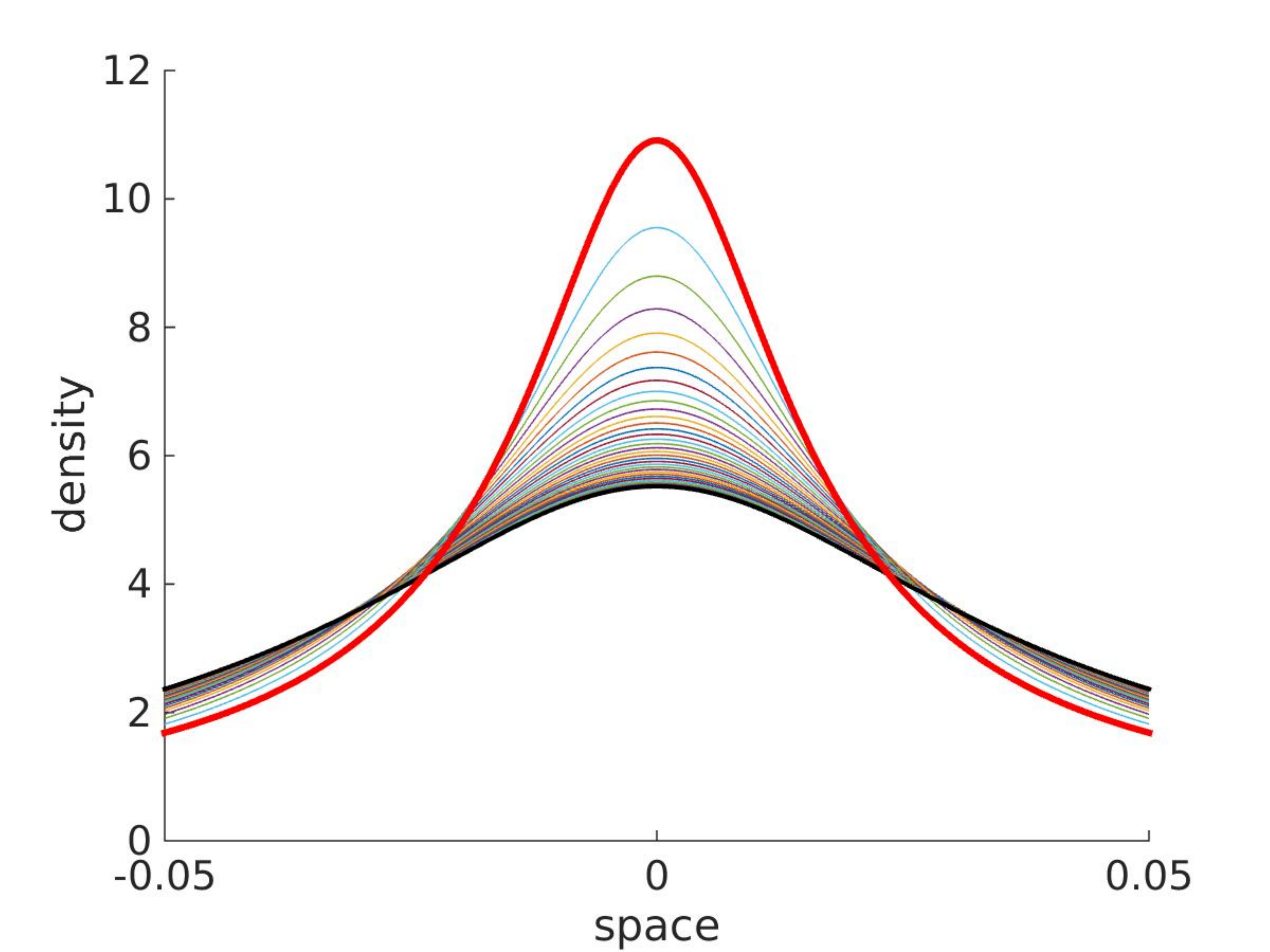}}%
\vspace{-0.4 cm}
\subfloat[]{\includegraphics[width=.45\textwidth]{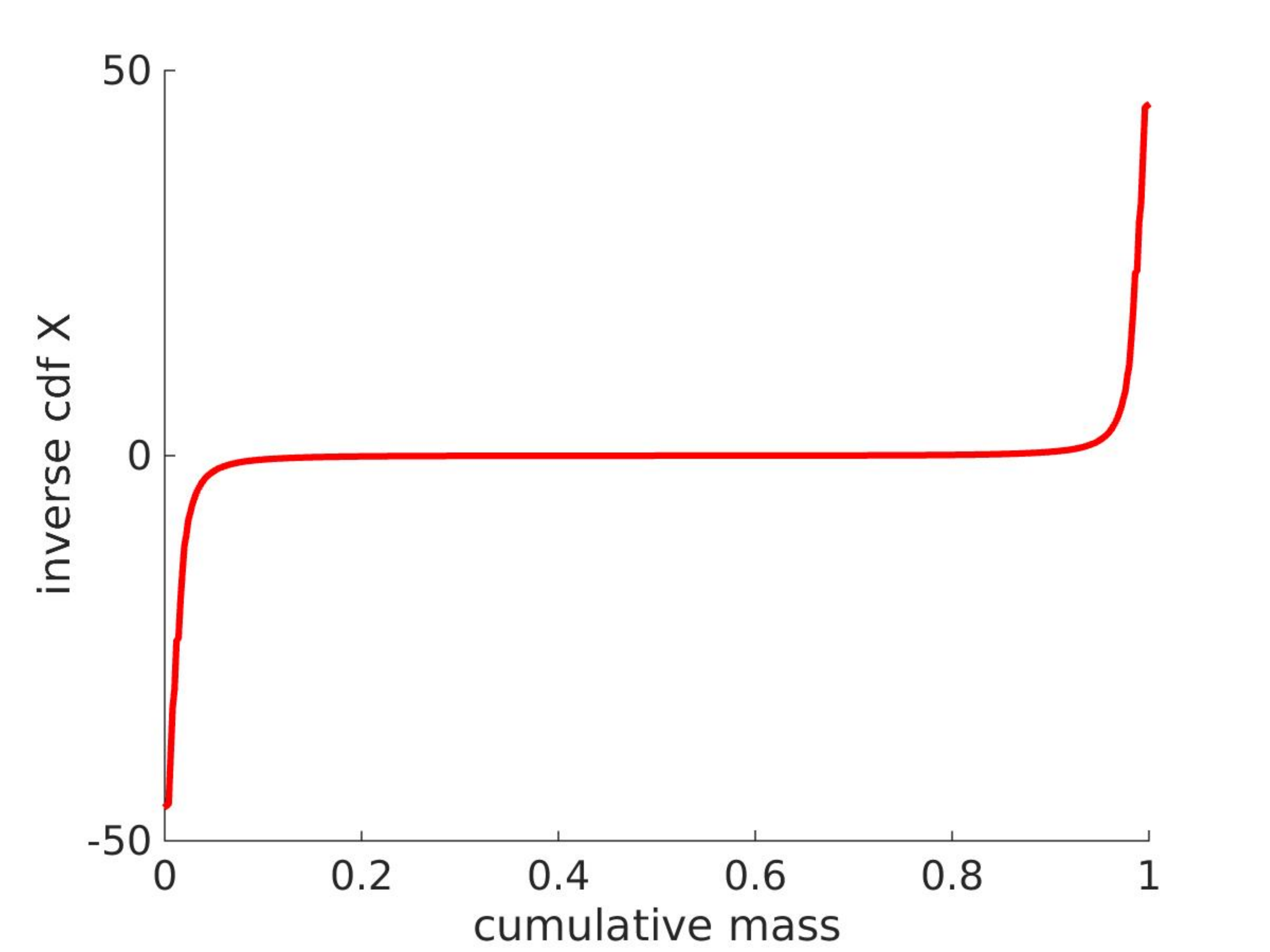}}%
\subfloat[]{\includegraphics[width=.45\textwidth]{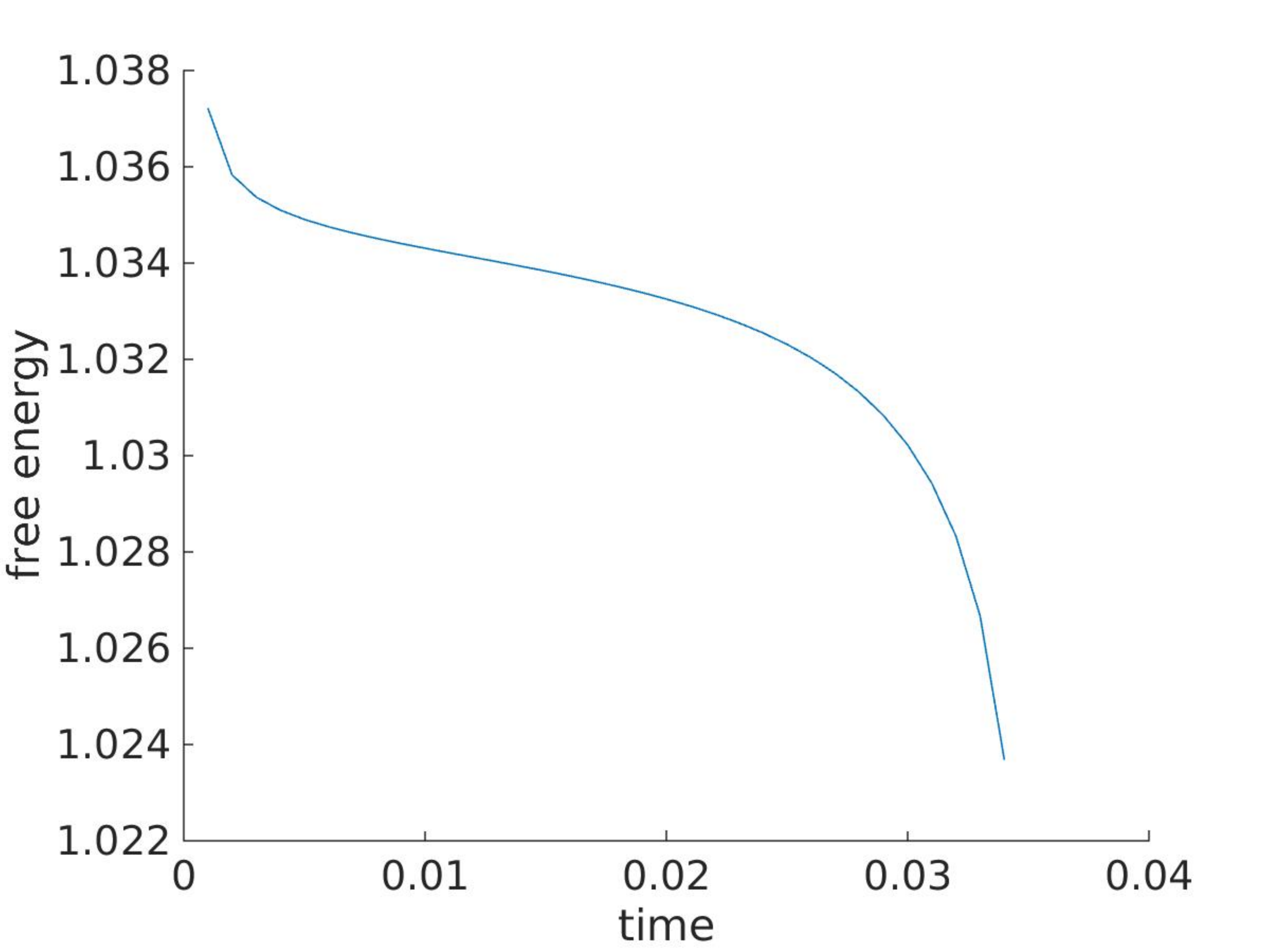}}%
\vspace{-0.3 cm}
\caption{Attraction-dominated regime: $\chi=0.35$, $k=-1/2$, $m=4/3$, $r=0$ with initial data $\rho(t=0,x)=\rho_{HLS,\lambda_0,c_0}(x)>\rho_{HLS,\lambda_0,c^*}(x)$ for all $x \in \RR$ with $c_0=1.1c^*$.\\
(a) Solution density from initial condition (black) to the profile at the last time step (red), 
(b) zoom of Figure (a),
(c) inverse cumulative distribution function from initial condition (black) to the profile at the last time step (red),
(d) free energy.}
\label{fig:conj_chi=035_k=-05_m=133_run8}
\end{figure}
%%%%%%%%%%%%%%%%%%%%%%%%%%%%%%%%%%%%%%%%%%%%
%%%%%%%%%%%%%%%%%%%%%%%%%%%%%%%%%%%%%%%%%%%%
%%%%%%%%%%%%%%%%%%%%%%%%%%%%%%%%%%%%%%%%%%%%

%%%%%%%%%%%%%%%%%%%%%%%%%%%%%%%%%%%%%%%%%%
%%%%%%%%%%%%%%%%%%%%%%%%%%%%%%%%%%%%%%%%%%
%%%% conj_chi=035_k=-05_m=133_run6
%%%% fig14
\begin{figure}[h!]
\centering
\subfloat[]{\includegraphics[width=.45\textwidth]{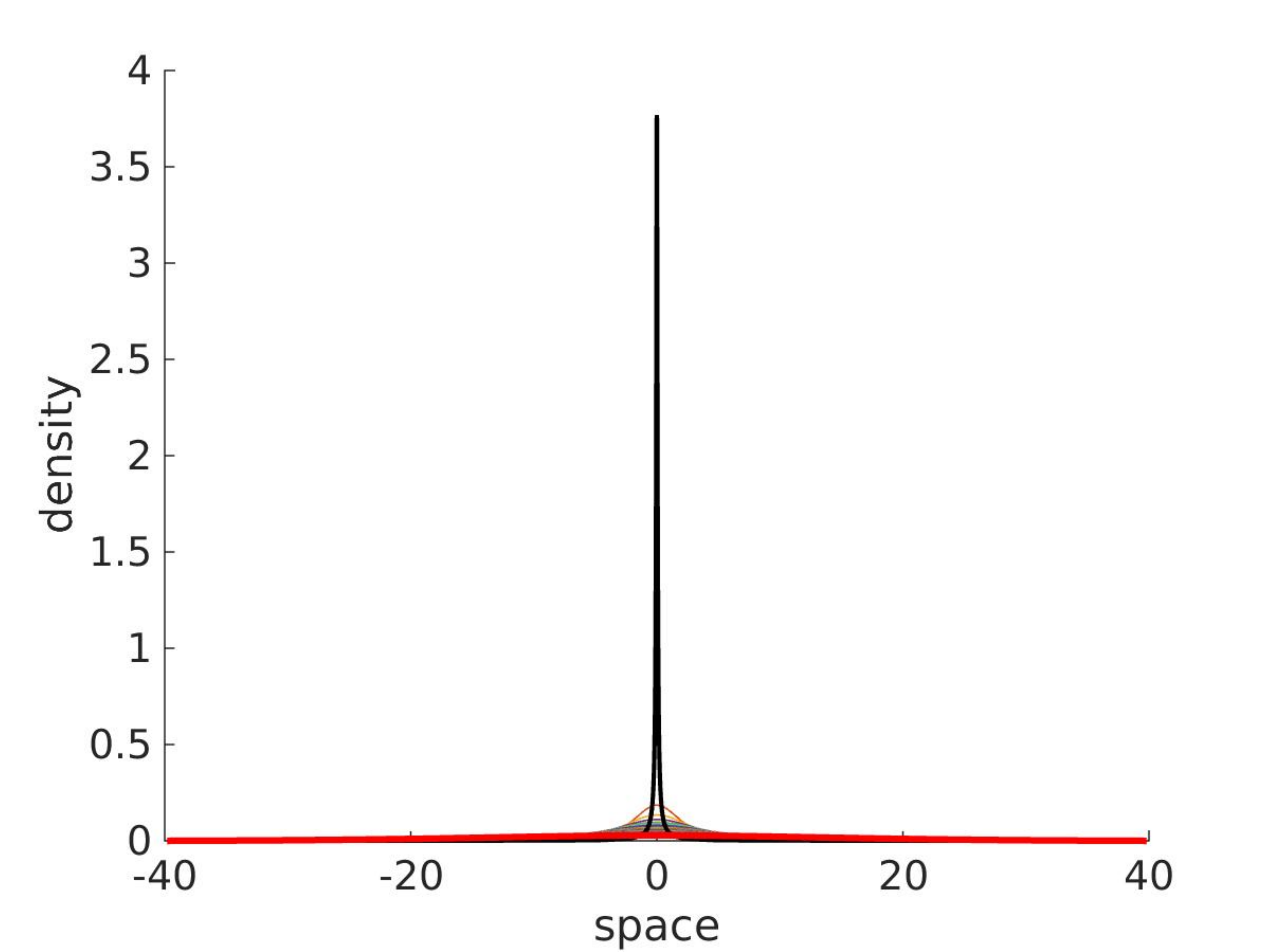}}%
\subfloat[]{\includegraphics[width=.45\textwidth]{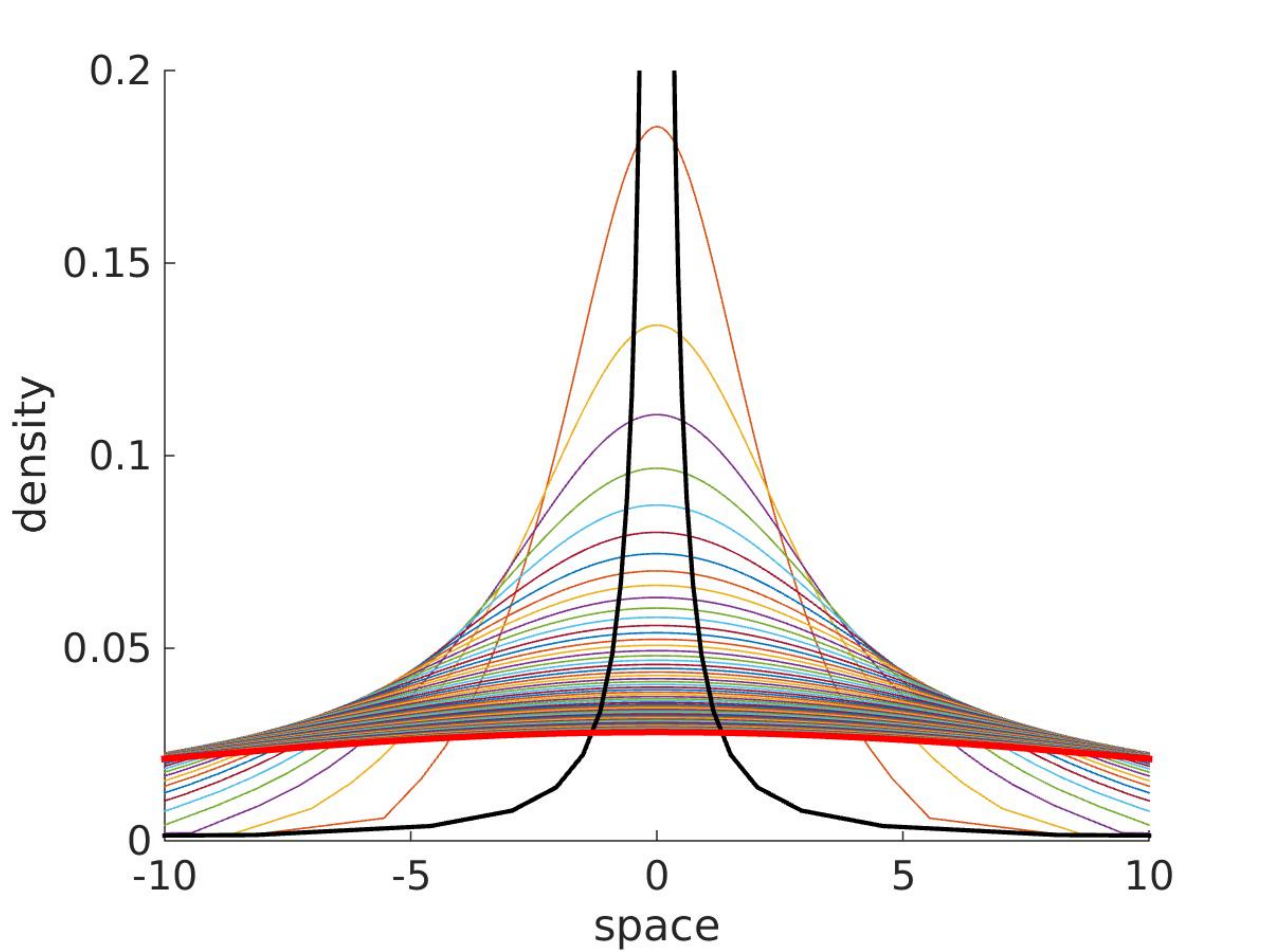}}%
\vspace{-0.4cm}
\subfloat[]{\includegraphics[width=.45\textwidth]{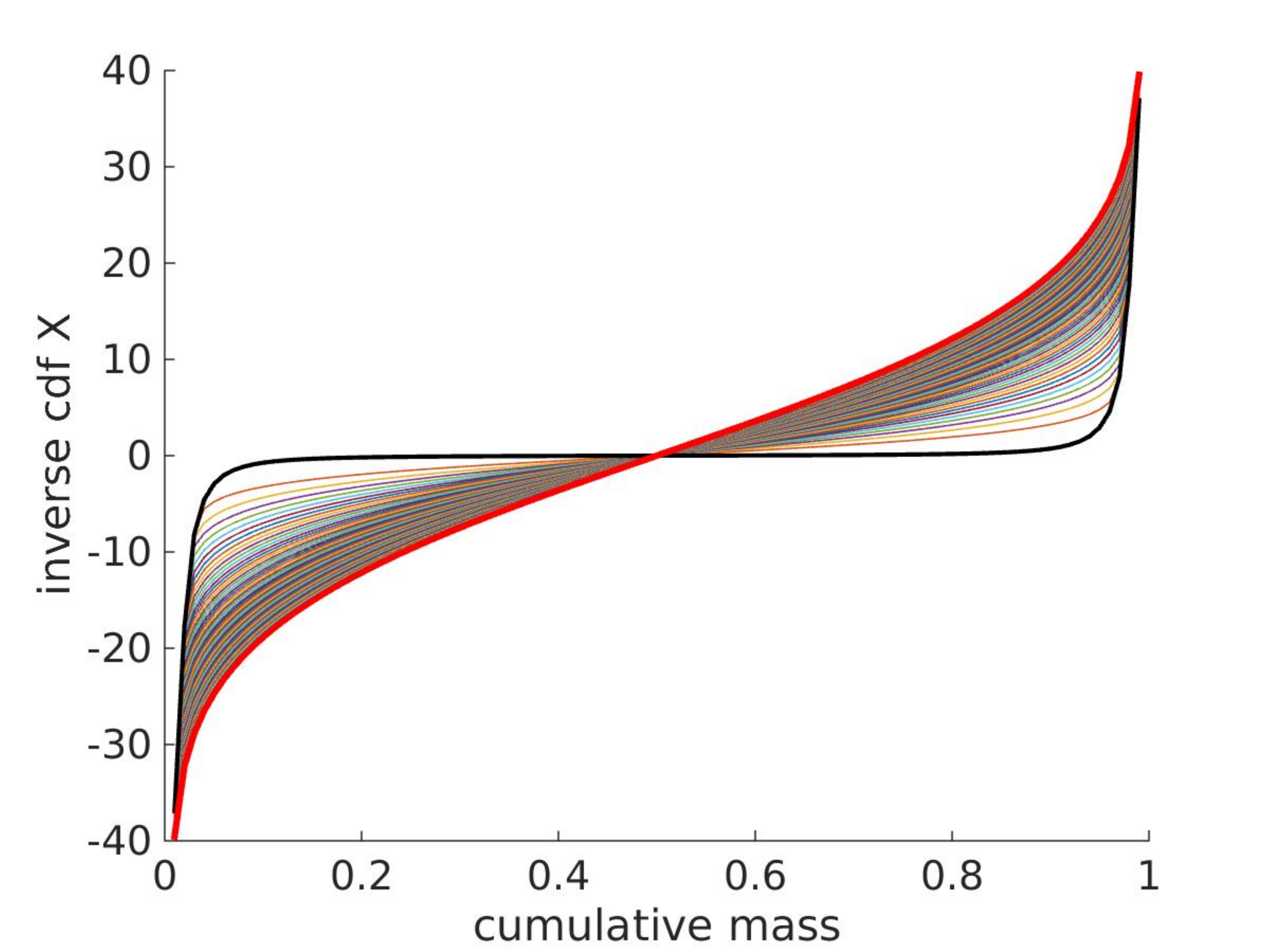}}%
\subfloat[]{\includegraphics[width=.45\textwidth]{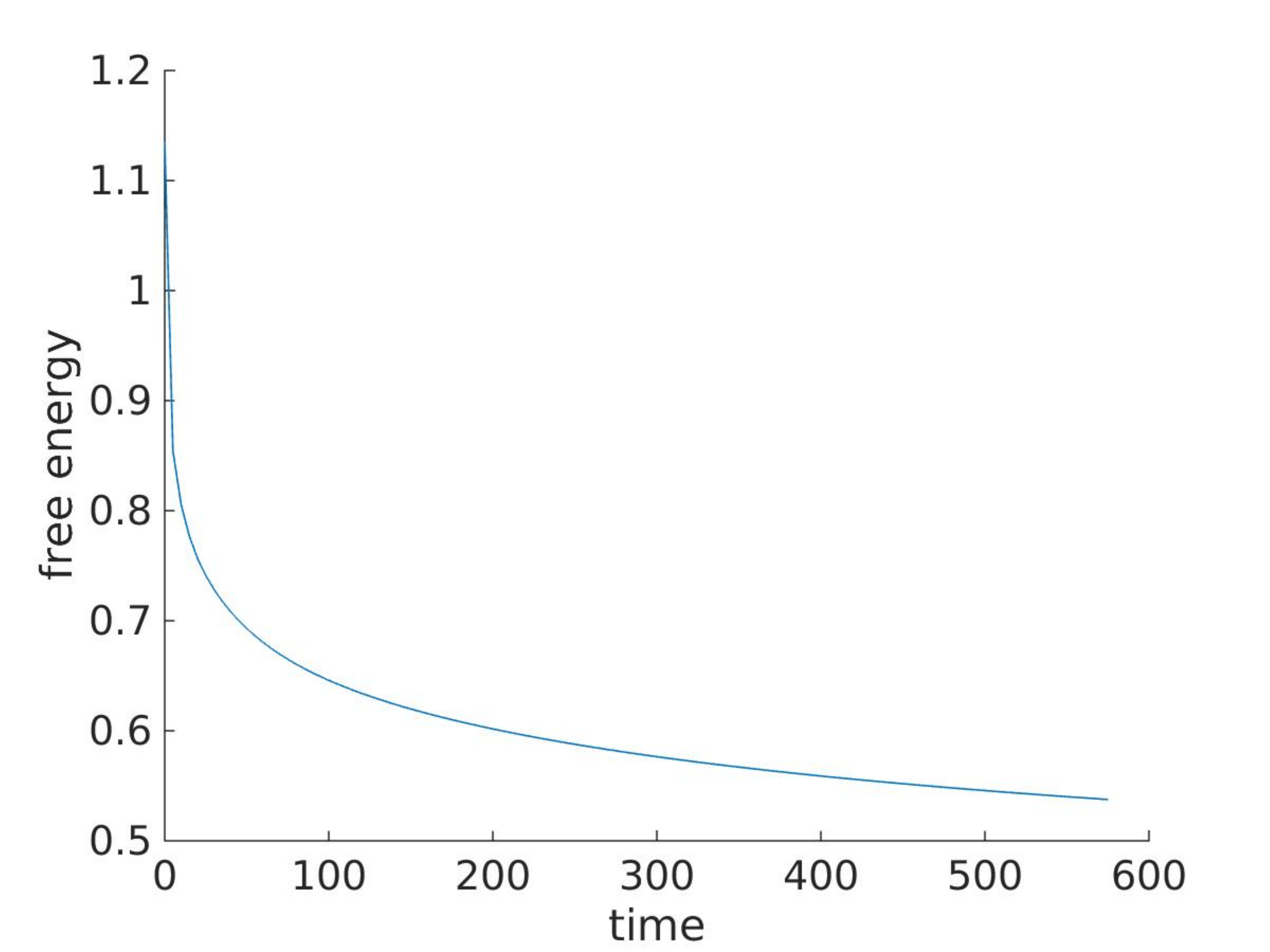}}%
\vspace{-0.4cm}
\subfloat[]{\includegraphics[width=.45\textwidth]{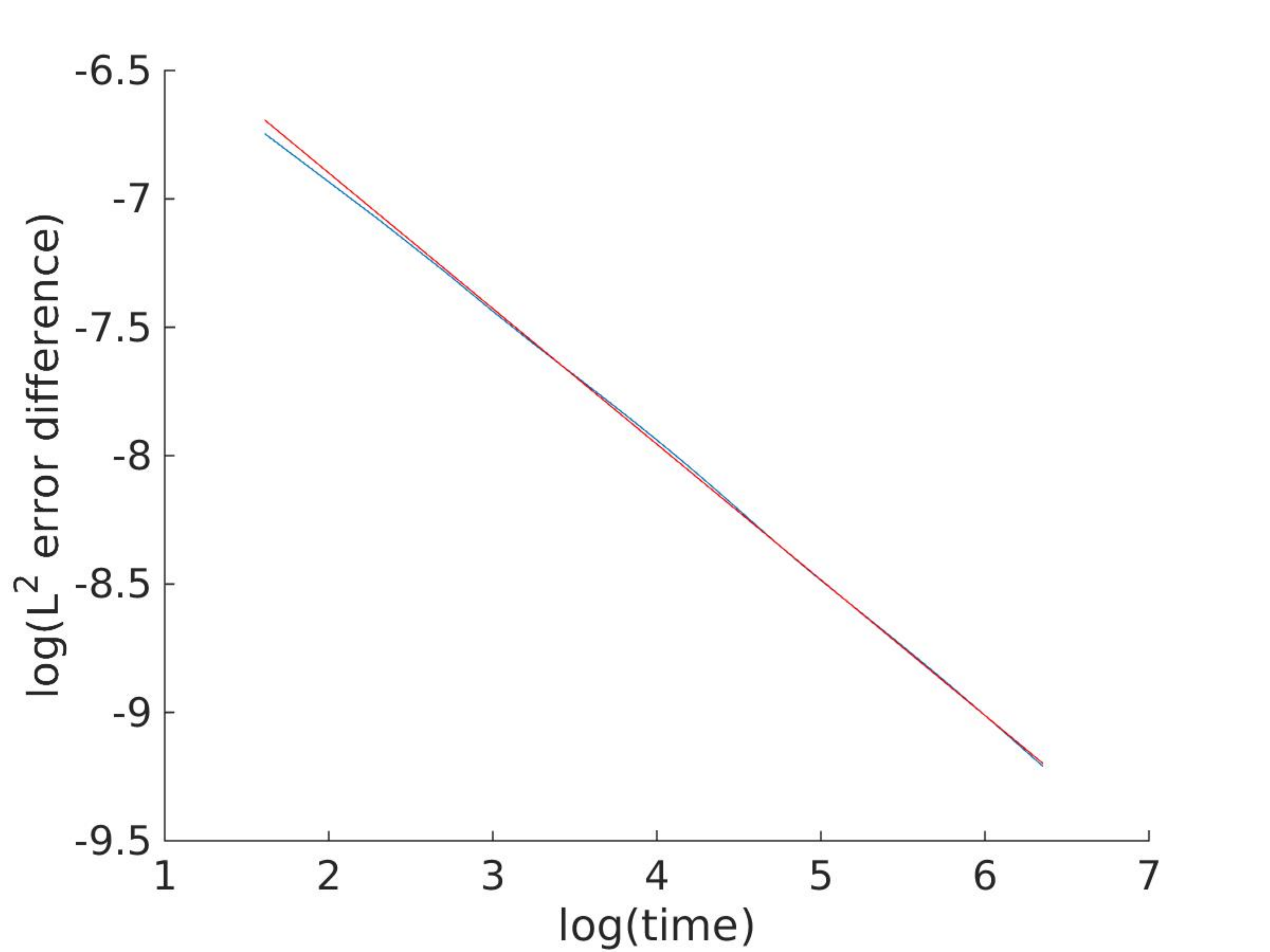}}%
\subfloat[]{\includegraphics[width=.45\textwidth]{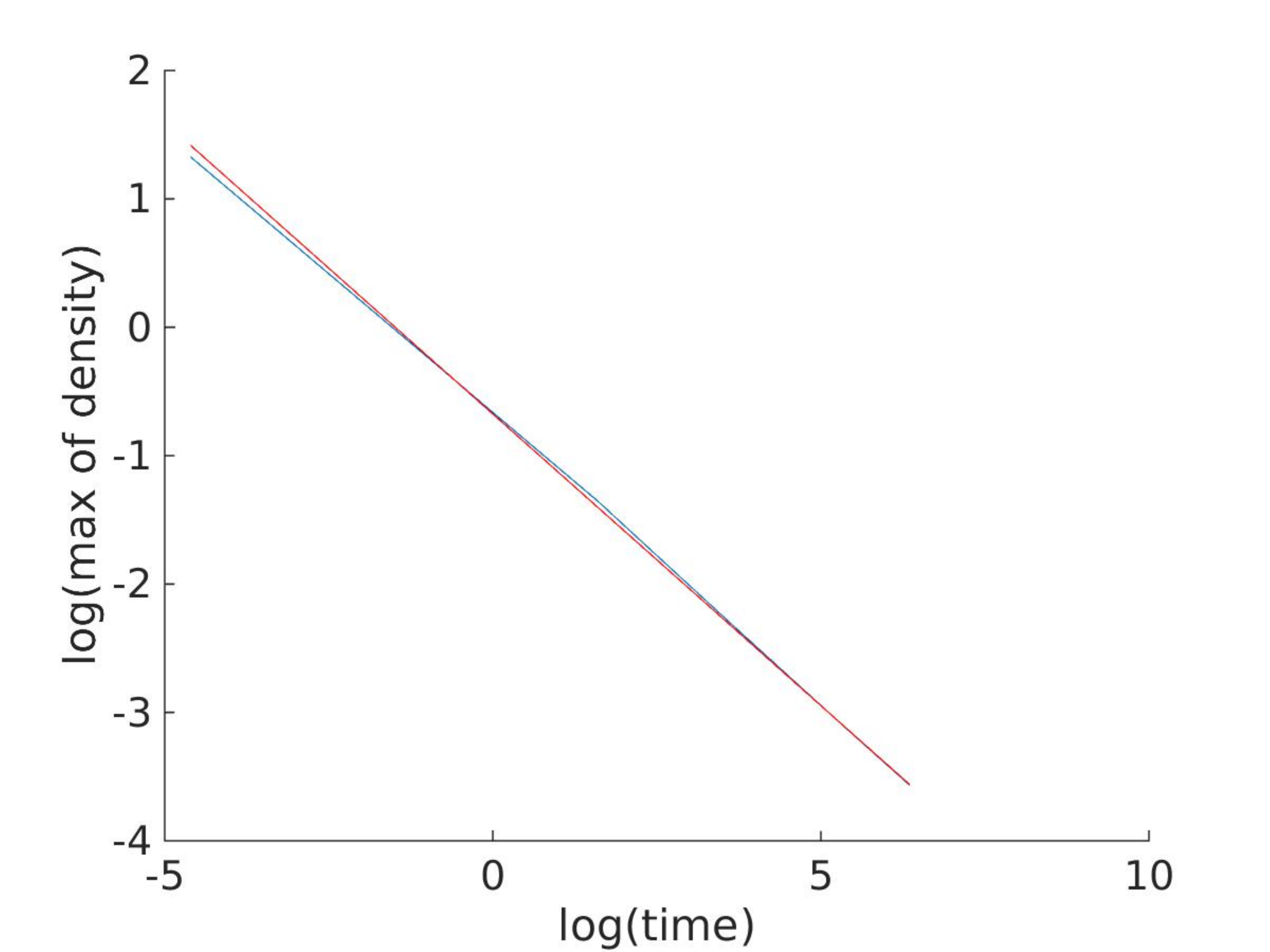}}%
\vspace{-0.2cm}
\caption{Attraction-dominated regime: $\chi=0.35$, $k=-1/2$, $m=4/3$, $r=0$ with initial data $\rho(t=0,x)=\rho_{HLS,\lambda^*,c^*}(x)$ given in \eqref{eq:rhoHLS}.\\
(a) Solution density from initial condition (black) to the profile at the last time step (red), 
(b) zoom of Figure (a),
(c) inverse cumulative distribution function from initial condition (black) to the profile at the last time step (red),
(d) free energy,
(e) log-log plot of the $L^2$-error difference between two consecutive solutions and fitted line with slope $-0.52817$,
% (f) time evolution of $\max_{x}\rho(t,x)$,
(f) log-log plot of $\max_{x}\rho(t,x)$ and fitted line with slope $-0.45431$
.}
\label{fig:conj_chi=035_k=-05_m=133_run6}
\end{figure}
%%%%%%%%%%%%%%%%%%%%%%%%%%%%%%%%%%%%%%%%%%%%
%%%%%%%%%%%%%%%%%%%%%%%%%%%%%%%%%%%%%%%%%%%%
%%%%%%%%%%%%%%%%%%%%%%%%%%%%%%%%%%%%%%%%%%%%
Numerically, we indeed observe the behaviour predicted in Conjecture \ref{conj1} for $N=1$. Using the scheme introduced in Section \ref{sec:numerics}, we choose as initial data the density $\rho_{HLS,\lambda_0,c_0}$ given by the optimisers of the HLS inequality \eqref{eq:rhoHLS}. For any choice of $c_0>0$, we fix $\lambda_0>0$ such that $\rho_{HLS,\lambda_0,c_0}$ has unit mass and is therefore in $\mY$. Note that $\rho_{HLS,\lambda_0,c_0}$ is not a critical point of $\mF_k$ unless $c_0=c^*$. Comparing with the stationary state $\rho_{HLS,\lambda_0,c^*}$, we have 
\begin{equation*}
 \sign\left(c^*-c_0\right)
 = \sign\left(\rho_{HLS,\lambda_0,c^*}(x)-\rho_{HLS,\lambda_0,c_0}(x)\right)\, ,
 \qquad \forall x \in \RR\, .
\end{equation*}
Note that the mass of the stationary state $\rho_{HLS,\lambda_0,c^*}$ is given by
$$
\int_{\RR^N} \rho_{HLS,\lambda_0,c^*(\chi)}(x)\,dx
=
\lambda_0^{-k/2}\int_{\RR^N} \rho_{HLS,1,c^*(\chi)}(x)\,dx\, ,
$$
which is equal to one if and only if $\lambda_0=\lambda^*$, that is $c_0=c^*$. If we choose $c_0<c^*$, then $\rho_0:=\rho_{HLS,\lambda_0,c_0}<\rho_{HLS,\lambda_0,c^*}$ and according to Conjecture \ref{conj1}(i), we would expect the solution $\rho(t,r)$ to vanish in $L_{loc}^1\left(\RR\right)$. This is exactly what can be observed in Figure \ref{fig:conj_chi=035_k=-05_m=133_run9b} for the choice of parameters $\chi=0.35$, $k=-1/2$, $m=4/3$ in original variables ($r=0$) and with $c_0=0.4\,c^*$. 
Here, we chose time steps of size $\Delta t=10^{-2}$ and particles spaced at $\Delta \eta=10^{-2}$. We let the density solution evolve until the $L^2$-error between two consecutive solutions is less than $10^{-4}$ (plotting every $1000$ iterations).\\
For the same choice of initial data, but with $c_0=1.1 \,c^*>c^*$ we observe numerically that the solution density concentrates at $x=0$ as predicted by Conjecture \ref{conj1}(ii), see Figure \ref{fig:conj_chi=035_k=-05_m=133_run8}. The Newton-Raphson procedure stops converging once the mass it too concentrated. Here, we chose time steps of size $\Delta t=10^{-3}$ and particles spaced at $\Delta \eta=2*10^{-3}$.\\

One may also take as initial condition exactly the steady state $\rho_0=\rho_{HLS,\lambda^*,c^*}$, see Figure \ref{fig:conj_chi=035_k=-05_m=133_run6}. However, the numerical approximation of the initial data is only accurate up to $\Delta \eta=10^{-2}$ and we observe indeed pointwise convergence to zero, in accordance with the statement of Conjecture \ref{conj2} that the stationary state $\rho_{HLS,\lambda^*,c^*}$ is unstable. Again, we let the Newton-Raphson procedure evolve with time steps of size $\Delta t=10^{-2}$ until the $L^2$-error between two consecutive solutions is less than the tolerance $10^{-4}$. % (plotting every 500 iterations).
 
\vspace{-0.4cm}

%%%%%%%%%%%%%%%%%%%%%%%%%%%%%%%%%%%%%%%%%%%%%%%
%%%%%%%%%%%%%%%END%%%%%%%%%%%%%%%%%%%%%%%%%%%%%
%%%%%%%%%%%%%%%%%%%%%%%%%%%%%%%%%%%%%%%%%%%%%%%
%%%%%%%%%%%%%%%%%%%%%%%%%%%%%%%%%%%%%%%%%%%%%%%%%%%%%%%%%%%%%%%%%%%%
%
%
%    Acknowledgements
%
%
%%%%%%%%%%%%%%%%%%%%%%%%%%%%%%%%%%%%%%%%%%%%%%%%%%%%%%%%%%%%%%%%%%%%%%%%%%

\section*{Acknowledgements}
\small{VC received funding for this project from the European Research Council (ERC) under the European Union's Horizon 2020 research and innovation programme (grant agreement No 639638). JAC was partially supported by the Royal Society via a Wolfson Research Merit Award. FH acknowledges support from the EPSRC grant number EP/H023348/1 for the Cambridge Centre for Analysis. The authors are very grateful to the Mittag-Leffler Institute for providing a fruitful working environment during the special semester \emph{Interactions between Partial Differential Equations \& Functional Inequalities}.}

\vspace{-0.2cm}

\bibliographystyle{abbrv}
\bibliography{biblio}
\end{document}